\documentclass[reqno]{amsart}

\usepackage{amssymb}
\usepackage{amscd}
\usepackage{enumerate}
\usepackage{graphicx}
\usepackage{vmargin}
\usepackage[hidelinks]{hyperref}
\usepackage{accents}

\setpapersize{A4}
\setmarginsrb{2.7cm}{2.7cm}{2.7cm}{2.7cm}{10pt}{20pt}{10pt}{20pt}

\theoremstyle{plain}
\newtheorem{thm}{Theorem}[section]
\newtheorem*{thm*}{Theorem}
\newtheorem{cor}[thm]{Corollary}
\newtheorem{lem}[thm]{Lemma}

\theoremstyle{definition}
\newtheorem{remark}[thm]{Remark}
\newtheorem{remarks}[thm]{Remarks}
\newtheorem{convention}[thm]{Convention}
\newtheorem{example}[thm]{Example}
\newtheorem{examples}[thm]{Examples}
\newtheorem{notation}[thm]{Definition}
\newtheorem{notations}[thm]{Definitions}
\newtheorem{defn}[thm]{Definition}
\newtheorem{defns}[thm]{Definitions}

\newcommand{\Notation}{Definition}

\newcommand{\Q}{{\mathbb{Q}}}
\newcommand{\N}{{\mathbb{N}}}
\newcommand{\PP}{{\mathbb{P}}}
\newcommand{\R}{{\mathbb{R}}}
\newcommand{\Z}{{\mathbb{Z}}}

\newcommand{\cB}{{\mathcal{B}}}
\newcommand{\cF}{{\mathcal{F}}}
\newcommand{\cG}{{\mathcal{G}}}
\newcommand{\cI}{{\mathcal{I}}}
\newcommand{\cJ}{{\mathcal{J}}}
\newcommand{\cL}{{\mathcal{L}}}
\newcommand{\cO}{{\mathcal{O}}}
\newcommand{\cP}{{\mathcal{P}}}
\newcommand{\cQ}{{\mathcal{Q}}}
\newcommand{\cR}{{\mathcal{R}}}
\newcommand{\KS}{{\mathcal{KS}}}

\newcommand{\fF}{{\mathfrak{F}}}

\newcommand{\hx}{{\hat{x}}}
\newcommand{\hB}{{\widehat{B}}}
\newcommand{\hT}{{\widehat{T}}}
\newcommand{\hf}{{\widehat{f}}}
\newcommand{\hF}{{\widehat{F}}}
\newcommand{\hH}{{\widehat{H}}}
\newcommand{\hI}{{\widehat{I}}}
\newcommand{\hS}{{\widehat{S}}}
\newcommand{\hX}{{\widehat{X}}}
\newcommand{\hq}{{\widehat{q}}}
\newcommand{\hw}{{\widehat{w}}}
\newcommand{\hy}{{\widehat{y}}}
\newcommand{\hz}{{\widehat{z}}}
\newcommand{\hR}{{\widehat{R}}}
\newcommand{\hpi}{{\widehat{\pi}}}
\newcommand{\hmu}{{\widehat{\mu}}}

\newcommand{\bt}{{\mathbf{t}}}
\newcommand{\bv}{{\mathbf{v}}}
\newcommand{\bx}{{\mathbf{x}}}
\newcommand{\by}{{\mathbf{y}}}
\newcommand{\bz}{{\mathbf{z}}}
\newcommand{\bp}{{\mathbf{p}}}
\newcommand{\bq}{{\mathbf{q}}}
\newcommand{\bJ}{{\mathbf{J}}}
\newcommand{\bP}{{\mathbf{P}}}
\newcommand{\bQ}{{\mathbf{Q}}}
\newcommand{\bT}{{\mathbf{T}}}
\newcommand{\bxi}{{\boldsymbol{\xi}}}

\newcommand{\barf}{\overline{f}}
\newcommand{\barg}{\overline{g}}
\newcommand{\bD}{\overline{D}}

\newcommand{\tB}{\widetilde{B}}
\newcommand{\tD}{\widetilde{D}}
\newcommand{\tmu}{{\tilde{\mu}}}

\newcommand{\diam}{\operatorname{diam}}
\newcommand{\Rem}{\operatorname{Rem}}
\newcommand{\Int}{\operatorname{Int}}
\newcommand{\lhe}{\operatorname{lhe}}
\newcommand{\rhe}{\operatorname{rhe}}
\newcommand{\NBT}{\operatorname{NBT}}
\newcommand{\id}{\operatorname{id}}

\newcommand{\ingam}{{{\accentset{\circ}{\gamma}}}}
\newcommand{\inR}{{{\accentset{\circ}{R}}}}
\newcommand{\invlim}{\varprojlim}
\newcommand{\thr}[1]{\left< #1 \right>}
\newcommand{\Infs}[1]{#1^\infty}
\newcommand{\Inf}[1]{\Infs{\left(#1\right)}}
\newcommand{\floor}[1]{\left\lfloor #1 \right\rfloor}
\newcommand{\yst}{\left((\by, s), t\right)}
\newcommand{\sj}{^{(j)}}
\newcommand{\ystj}{\left((\by\sj, s_j), t_j\right)}
\newcommand{\tPhi}{\Phi_*}
\newcommand{\htPhi}{\widehat{\Phi}_*}
\newcommand{\sD}{{D^{{\scriptstyle\dagger}}}} 
\newcommand{\sDt}{{D^{{\scriptstyle\dagger}}_{{t}}}}
\newcommand{\sDast}{{D^{{\scriptstyle\dagger}}_\ast}}
\newcommand{\syma}{\mu}
\newcommand{\symb}{\nu}
\newcommand{\re}{\alpha}

\graphicspath{{./}{figures/}}

\begin{document}

\title[Natural extensions of unimodal maps]{Natural extensions of unimodal maps: virtual 
sphere homeomorphisms and prime ends of basin boundaries.}
\date{January 2020}
\author{Philip Boyland}
\address{Department of Mathematics\\University of Florida\\372 Little
    Hall\\Gainesville\\ FL 32611-8105, USA}
\email{boyland@ufl.edu}
\author{Andr\'e de Carvalho}
\address{Departamento de Matem\'atica Aplicada\\ IME-USP\\ Rua Do Mat\~ao
    1010\\ Cidade Universit\'aria\\ 05508-090 S\~ao Paulo SP\\ Brazil}
\email{andre@ime.usp.br}
\author{Toby Hall}
\address{Department of Mathematical Sciences\\ University of Liverpool\\
    Liverpool L69 7ZL, UK}
\email{tobyhall@liverpool.ac.uk}

\begin{abstract} 
  Let~$\{f_t\colon I\to I\}$ be a family of unimodal maps with topological
  entropies $h(f_t)>\frac12\log 2$, and $\hf_t\colon\hI_t\to\hI_t$ be their
  natural extensions, where $\hI_t=\invlim(I,f_t)$. Subject to some regularity
  conditions, which are satisfied by tent maps and quadratic maps, we give a
  complete description of the prime ends of the Barge-Martin embeddings of
  $\hI_t$ into the sphere. We also construct a family $\{\chi_t\colon S^2\to
  S^2\}$ of sphere homeomorphisms with the property that each $\chi_t$ is a
  factor of $\hf_t$, by a semi-conjugacy for which all fibers except one
  contain at most three points, and for which the exceptional fiber carries no
  topological entropy: that is, unimodal natural extensions are virtually
  sphere homeomorphisms. In the case where~$\{f_t\}$ is the tent family, we
  show that~$\chi_t$ is a \emph{generalized pseudo-Anosov map} for the dense
  set of parameters for which $f_t$ is post-critically finite, so that
  $\{\chi_{t}\}$ is the completion of the unimodal generalized pseudo-Anosov
  family introduced in~\cite{gpa}.

\end{abstract}

\maketitle


\section{Introduction}
\subsection{Overview}
The study of continua and their rich topological structures goes back to the
first half of the $20^\text{th}$ century, and played a central r\^ole in the
early development of topology. Embeddings of continua in surfaces have also
been an important ingredient in dynamical systems theory: early examples
include Birkhoff's {\em remarkable curves}~\cite{birkhoff,LeC} and the
Cartwright-Littlewood Theorem~\cite{CL}. Williams~\cite{williams,williams2} was
the first to notice, in the late 1960s, that continua defined by inverse limits
are a useful tool in the study of dynamical systems: specially relevant here is
his discovery that a particular class of planar continua, the inverse limits of
expanding maps on graphs, describe planar, one-dimensional hyperbolic
attractors. In the early 1990s --- inspired in part by the importance of the
H\'enon family as a paradigm for the larger family of \emph{non-hyperbolic}
attractors --- Barge and Martin~\cite{bargemartin,param-fam} gave a method to
embed a wide class of inverse limits as attractors of planar homeomorphisms.
The inverse limits of unimodal maps of the interval such as those from the
quadratic and tent families are of particular importance for the H\'enon
family. These inverse limits are the chief objects of study here. A simple
example is shown in Figure~\ref{fig:1001110} for expository purposes: it will
be used as a point of reference throughout the introduction.

\begin{figure}[htbp]
  \begin{center}
   \includegraphics[width=0.8\textwidth]{1001110}
  \end{center}
  \caption{An approximation of the inverse limit of a tent map with
  kneading sequence $(1001110)^\infty$, which is of rational interior
  type with height~$1/3$ (see below).}
  \label{fig:1001110}
\end{figure}

The prime ends of the complementary domains form an essential part of the
analysis of planar continua: in dynamical systems, they have been used in the
description of basin boundaries~\cite{AY,nusseyorke} and, in the wider context
of holomorphic dynamics, prime ends of the complements of Julia
sets~\cite{BO,CMTT,kiwi,rogers} have also been studied. In this paper we give a
complete description of the prime ends of the complementary domains of the
Barge-Martin embeddings of the inverse limits of families of unimodal maps.  We
believe that this constitutes the first complete analysis in the literature of
the nature of the embeddings of a continuously varying family of planar
attractors. (In subsequent work using more symbolic techniques, Anu\v{s}i\'{c}
and \v{C}in\v{c}~\cite{AC} reproduce most of the results here about the prime
ends of Barge-Martin embeddings in the specific case of tent map inverse
limits, and enhance our results in this case with additional topological
information concerning folding points and endpoints.)

The topology of unimodal inverse limits is exquisitely complicated. For the
tent family~$\{f_t\}$, results of Bruin and of Raines~\cite{bruin,raines} imply
that, when the parameter~$t$ is such that the critical orbit of $f_t$ is dense
(a full measure, dense $G_\delta$ set of parameters),  the inverse
limit $\hI_t$ is nowhere locally the product of a Cantor set and an interval,
and is therefore much more complicated than the example of
Figure~\ref{fig:1001110}. A striking statement of self-similarity is given by
Barge, Brucks and Diamond~\cite{BaBrDi}, who show that there is a
dense~$G_\delta$ set of parameters~$t$ for which every open subset of $\hI_t$
contains a homeomorphic copy of $\hI_s$ for \emph{every} $s\in[\sqrt2, 2]$.
Moreover, the Ingram conjecture posits that the inverse limits $\hI_t$ are
pairwise non-homeomorphic. This has been proved for non-core tent maps by
Barge, Bruin, and \v Stimac~\cite{Ingram}, while for core tent maps there are
known to be uncountably many homeomorphism classes (see for
example~\cite{ABC,GR}).

In the second part of the paper, we show that all of these inverse limit spaces
are virtually spheres: there are quotients $p_t\colon \hI_t\to S^2$ which
respect the natural extensions $\hf_t\colon\hI_t\to \hI_t$, and have the
property that, with the exception of at most one $x\in S^2$, the fiber
$p_t^{-1}(x)$ contains at most three points: moreover, the exceptional fiber
carries no topological entropy. There is therefore a family $\{\chi_t\}$ of
sphere homeomorphisms --- which is shown to vary continuously --- such that
\[
  \begin{CD} 
    S^2    @>\chi_t>>      S^2\\
    @Ap_t AA   @AAp_t A\\ 
    \hI_t  @>\hf_t>> \hI_t\\
    @V \pi_0 VV   @VV \pi_0 V\\ 
    I   @>f_t>>      I\\
  \end{CD}
\]
commutes (here $\pi_0\colon\hI_t\to I$ is projection onto the first
coordinate). In view of the mildness of the semi-conjugacies~$p_t$, this
suggests that the sphere is a natural space on which to study invertible
analogs of unimodal maps.

The sphere homeomorphisms~$\chi_t$ are best seen as generalizations of
Thurston's pseudo-Anosov maps~\cite{Thurston}. A pseudo-Anosov
homeomorphism~$\phi$ of a surface has a transverse pair of invariant singular
foliations, one stable and one unstable, which fill the surface. Collapsing the
stable foliation yields a graph, the \emph{train track}, which carries an
expanding map. Following Williams, the inverse limit of the train track map
yields a homeomorphism~$\Phi$ with a one-dimensional hyperbolic attractor. This
homeomorphism can alternatively be obtained by ``DA-ing'' the prongs of the
pseudo-Anosov foliations, i.e., splitting open the leaves ending at the pronged
singularities. The original pseudo-Anosov map~$\phi$ can be reconstructed
from~$\Phi$ by ``collapsing'' stable sets in the complement of the attractor.

The process used to construct the sphere homeomorphisms~$\chi_t$ formalizes and
generalizes this last collapse: we start with the Barge-Martin embedding of the
inverse limit of a unimodal map~$f_t$ as an attactor, and collapse
\emph{strongly stable sets} (see Definition~\ref{def:strong-stable}) to obtain
$\chi_t$. In Figure~\ref{fig:1001110} the semi-circular arcs represent
identifications, and are not part of the inverse limit, which consists only of
the horizontal arcs. With this in mind, the quotient can be seen --- although
not quite accurately --- as being obtained by collapsing, in each vertical
line, the closure of the segments in the complement of the attractor, and then
sewing up the outside in a dynamically coherent way. Because the inverse limit
of a general unimodal map can be much more complicated than that of a train
track map, the dynamics and invariant geometric structures of~$\chi_t$ are
correspondingly more complicated. In particular, the invariant stable and
unstable ``foliations'' are only defined in a measurable sense.

In the case where~$\{f_t\}$ is the tent family, the corresponding
family~$\{\chi_t\}$ is a completion of the family of \emph{generalized
pseudo-Anosov maps} which was constructed in~\cite{gpa} for post-critically
finite tent maps. (A generalized pseudo-Anosov is defined similarly to a
pseudo-Anosov, except that its invariant foliations can have infinitely many
singularities, provided they accumulate on only finitely many points: see
Definition~\ref{def:gpA}.) The earlier construction was explicit, and made
essential use of the existence of finite Markov partitions. The constructions
of this paper show not only  how generalized pseudo-Anosovs arise directly from
inverse limits and natural extensions --- with the leaves of the unstable
foliation of $\chi_t$ coming from the path components of the inverse limit
of~$f_t$ --- but also how they live within the richer class of homeomorphisms
which arise in the post-critically infinite case. These \emph{measurable
pseudo-Anosov homeomorphisms}, whose invariant foliations are only defined on a
full measure subset of the sphere, are the subject of articles in preparation.

For the countable set of \emph{NBT }parameters introduced in~\cite{HS}, the map
$\chi_t$ is an actual pseudo-Anosov homeomorphism. Thus the analysis of the
family $\chi_t$ also contributes to the question of the completion in the
$C^0$-topology of the set of all pseudo-Anosov homeomorphisms on a given
surface.

\subsection{Background}
We now proceed to a brief overview of some background theory in dynamics and
topology, which will enable us to give more precise statements of our main
results in Sections~\ref{sec:prime-end-summary} and~\ref{sec:semi-conj-summary}
below.

\subsubsection{Unimodal maps}
\label{sec:uni-background}
The study of unimodal maps of the interval, one of the simplest classes of
dynamical systems which exhibit complicated behavior, drove the development of
the theory of topological dynamical systems in the 1970s and beyond. A
continuous map $f\colon[a,b]\to[a,b]$ is said to be \emph{(non-core) unimodal}
if
\begin{enumerate}[(a)]
  \item $f(a)=f(b)=a$, and
  \item there is a turning point $c\in(a,b)$ such that $f$ is strictly
  increasing on~$[a,c]$ and strictly decreasing on~$[c,b]$. Moreover $f(x)>x$
  for $x\in(a,c]$.
\end{enumerate}
The qualification \emph{non-core} is important here: we will shortly replace
condition~(a) with an alternative version which corresponds to restricting the
domain to an invariant sub-interval (the \emph{core}) in which all of the
non-trivial dynamics is contained.

Prototypical examples of families of unimodal maps are the \emph{quadratic}
(also known as \emph{logistic}) and \emph{tent} families
$f_t\colon[0,1]\to[0,1]$ defined respectively by
\[ 
    f_{t}(x)=tx(1-x) \quad (0<t\le 4) \qquad\text{ and }\qquad 
    f_{t}(x) = t \min\{x,1-x\} \quad (0<t\le 2).
\]
The tent family is of particular theoretical importance, because any unimodal
map~$f\colon[a,b]\to[a,b]$ with positive \emph{topological entropy} $h(f)$ (a
numerical measure of the asymptotic rate at which the orbits of nearby points
diverge from each other~\cite{AKM}) is semi-conjugate to the tent map with
slope $t=\exp(h(f))$~\cite{MT,Parry}. That is, there is an increasing surjection $p\colon[a,b]\to [0,1]$ such that
\[
  \begin{CD} 
    [a,b]  @>f>> [a,b]\\
    @V p VV   @VV p V\\ 
    [0,1]   @>f_t>>     [0,1]\\
  \end{CD}
\]
commutes, so that the dynamics of~$f$ and of the tent map~$f_t$ agree once
certain intervals in the domain of~$f$ --- the nontrivial point preimages
$p^{-1}(x)$ --- have been collapsed. The semi-conjugacy~$p$ can be described
explicitly, by means either of a formula, or of a dynamical description of
exactly which intervals in the domain of~$f$ are collapsed.

\medskip

\emph{Kneading theory} is a key tool in the analysis of the dynamics of
unimodal maps. Points $x\in[a,b]$ are described by their \emph{itineraries}
$\iota(x)\in\{0,1\}^\N$, sequences of $0$s and $1$s which encode, for each
successive point on the orbit of~$x$, whether it is on the left (`$0$') or the
right (`$1$') of the turning point~$c$. (The details of how to encode~$c$
itself are largely unimportant in this paper, and are left for
Section~\ref{sec:unimodal}.) The itinerary of $f(c)$, the largest point in the
range of~$f$, is of particular importance, and is called the \emph{kneading
sequence}~$\kappa(f)$ of~$f$. The dynamics of~$f$ is largely determined by its
kneading sequence --- in particular, $\kappa(f)$ determines the topological
entropy~$h(f)$, and hence the particular tent map to which~$f$ is
semi-conjugate.

The \emph{unimodal order} (or \emph{parity-lexicographic order}) $\preceq$
on~$\{0,1\}^\N$ is defined (see Definition~\ref{def:unimodal-order}) to reflect
the ordering of the interval: if $x<y$, then $\iota(x)\preceq\iota(y)$.
However, it takes on more meaning when interpreted as an order on the space of
unimodal maps: if $\kappa(f)\preceq \kappa(g)$, then the dynamics of~$g$ is at
least as complicated as the dynamics of~$f$. In particular, $\kappa(f_t)$ is an
increasing function of~$t$ if~$f_t$ is either the quadratic or the tent family.

\medskip

The \emph{core} of a unimodal map $f\colon [a,b]\to[a,b]$ is the interval $J =
[f^2(c), f(c)]$. It satisfies $f(J) = J$: moreover, the orbit of
every~$x\in(a,b)$ falls into~$J$, so that all of the non-trivial recurrent
dynamics of~$f$ is contained in the core. For this reason, it is sensible ---
particularly when considering inverse limits --- to restrict the domain of a
unimodal map to its core. This corresponds to replacing the condition that
$f(a)=f(b)=a$ with the condition that $f(c) = b$ and $f(b)=a$.

In this paper we will be exclusively concerned with core unimodal maps, and
Definition~\ref{def:unimodal} reflects this. We impose some additional
conditions on our unimodal maps~$f$, which are stated in
Convention~\ref{conv:unimodal} below. These conditions are of two types:
\begin{enumerate}[(a)]
  \item Regularity conditions, expressed in a way which allows them to
  encompass both the quadratic family and the tent family. These conditions
  appear technical, but are standard in the theory of unimodal maps.
  \item The additional condition that~$f$ has topological entropy
  $h(f)>\frac{1}{2}\log 2$. This is an indecomposability condition: it is
  equivalent to the non-existence of a pair of subintervals $J_1$, $J_2$,
  disjoint except perhaps at their endpoints, with $f(J_1)\subset J_2$ and
  $f(J_2)\subset J_1$.
\end{enumerate}

Readers without a background in one-dimensional dynamics can substitute ``a
unimodal map satisfying the conditions of Convention~\ref{conv:unimodal}'' with
``a map from the quadratic or tent family with sufficiently large parameter",
without substantial loss.

\subsubsection{Inverse limits}
Let~$X$ be a compact metric space with metric~$d$, and let $f\colon X\to X$ be
continuous and surjective. The {\em inverse limit} of $f\colon X\to X$ is the
space of ``backwards orbits" of~$f$:
\[
\hX := \invlim(X,f) = \{\bx\in X^\N\,:\,f(x_{i+1})=x_i \text{ for all
}i\in\N\}.
\] 
We endow $\hX$ with a standard metric, also denoted~$d$, which induces its
 natural topology as a subspace of the product $X^\N$:
\[ 
  d(\bx, \by) = \sum_{i=0}^\infty \frac{d(x_i,y_i)}{2^i}.
\] 
Elements of $\hX$ are denoted with angle brackets, $\bx =
\thr{x_0,x_1,x_2,\ldots}$ and referred to as {\em threads}.

The {\em natural extension} of $f\colon X\to X$ is the homeomorphism $\hf\colon
\hX\to\hX$ defined by
\[
\hf(\thr{x_0,x_1,x_2,\ldots}) = \thr{f(x_0), x_0, x_1, x_2,\ldots}.
\]

The projection $\pi_0\colon \hX\to X$ is defined by $\pi_0(\bx)=x_0$.
Clearly $\pi_0\circ\hf = f\circ\pi_0$, so that $\pi_0$ semi-conjugates $\hf$
to~$f$.  It is straightforward to show that if $g\colon Y\to Y$ is an
invertible dynamical system and $p\colon Y\to X$ semi-conjugates~$g$ to~$f$,
then $p$ factors through $\pi_0$: therefore the natural extension is the
simplest invertible system which has $f$ as a factor.

\subsubsection{The Barge-Martin construction}
\label{sec:BM-background}
The Barge-Martin construction~\cite{bargemartin} provides a mechanism for
embedding the inverse limit $\hX$ of a dynamical system $f\colon X\to X$ as a
global attractor of a self-homeomorphism of a manifold, on which the
homeomorphism restricts to the natural extension of~$f$. We now give a brief
outline of the construction in the case of interest here, where $f\colon I\to
I$ is a unimodal map whose inverse limit is embedded as an attractor of a
sphere homeomorphism. Further details can be found in
Section~\ref{sec:barge-martin}.

Let~$T$ be a topological sphere, $D\subset T$ be a  closed disk containing a
copy of~$I$ in its interior, and~$\partial$ be a point of $T\setminus D$.
Construct a \emph{smash}~$\Upsilon\colon T\to T$, a near-homeomorphism (i.e.\ a
uniform limit of homeomorphisms) which
\begin{itemize}
  \item collapses~$D$ onto~$I$, in such a way that the preimage of each point of~$I$ is an arc in~$D$;
  \item fixes~$\partial$; and
  \item pushes points of $T\setminus(D\cup\{\partial\})$ ``inwards''
  towards~$I$.
\end{itemize}
Let $\barf\colon T\to T$ be an \emph{unwrapping} of~$f$: a near-homeomorphism
which 
\begin{itemize}
  \item sends~$I$ into~$D$ in such a way that $\Upsilon\circ\barf|_I = f$; and
  \item doesn't push any points of $T\setminus\{\partial\}$ too far
  ``outwards''.
\end{itemize}

Now consider the near-homeomorphism $H=\Upsilon\circ\barf\colon T\to T$. By
construction we have $H|_I = f$, so that the inverse limit $\hT = \invlim(T,
H)$ contains an embedded copy of~$\hI$, namely $\{\bx\in\hT\,:\,x_i\in I \text{
for all }i\}$, on which the action of the natural extension $\hH$ restricts to
$\hf$. Because the smash pushes points of~$T$ other than~$\partial$ towards~$I$
more strongly than the unwrapping pushes them away, every point of $\hT$ other
than $\thr{\partial, \partial, \partial, \dots}$ is attracted to the copy of
$\hI$ under iteration of $\hH$.

The key observation is that the inverse limit $\hT$ is itself a topological
sphere, as a consequence of the following theorem due to Morton Brown:
\begin{thm*}[Brown~\cite{brown}]
  Let~$X$ be a compact metric space, and $f\colon X\to X$ be a
  near-homeomorphism. Then $\invlim(X,f)$ is homeomorphic to~$X$.
\end{thm*}

The constructions in this paper depend crucially on the details of the
smash~$\Upsilon$ (see Section~\ref{sec:barge-martin}), and on the careful
definition of a particular choice of unwrapping~$\barf$, which is described in
Section~\ref{sec:unwrapping-main}.

\subsubsection{Prime ends}
\label{sec:prime-ends}
We will describe the Barge-Martin embedding of the inverse limit~$\hI$ of a
unimodal map~$f$ in the topological sphere~$\hT$ by means of Carath\'eodory's
theory of prime ends. Here we review some basic definitions in order to allow
for precise statements of our main results. We note that while the theory
of prime ends can be profitably approached from the viewpoint of conformal
mappings, the spaces which we will be dealing with have no natural complex
structure, and so we take a purely topological approach. The reader seeking a
more comprehensive introduction from this point of view could consult, for
example, Mather's paper~\cite{Mather}.

Let~$T$ be a topological $2$-sphere, and~$X$ be a non-empty, compact,
connected, non-separating proper subset of~$T$, so that the complement $U :=
T\setminus X$ is a topological open disk. (For most of our applications, $T$
will be the sphere $\hT$ of the Barge-Martin embedding, and $X$ will be the
embedded copy of~$\hI$.) Fix a point $\partial\in U$.

A {\em crosscut (in $(T, X)$)} is an arc $\xi$ in $T$ which is disjoint
from~$\partial$ and intersects $X$ exactly at the endpoints of~$\xi$. Such a
crosscut separates the open disk $U$ into two components, and we write
$U(\xi)$ for the component which doesn't contain~$\partial$. If $\xi_1$ and
$\xi_2$ are crosscuts, then we write $\xi_2<\xi_1$ to mean that
$U(\xi_2)\subset U(\xi_1)$. A {\em chain} is a sequence $(\xi_k)$ of disjoint
crosscuts with $\xi_{k+1}<\xi_k$ for each~$k$ and $\diam(\xi_k)\to0$ as
$k\to\infty$. Two chains $(\xi_k)$ and $(\xi_k')$ are {\em equivalent} if for
each~$k$ there is some~$K$ with $\xi_K <
\xi_k'$ and $\xi_K' < \xi_k$.

A {\em prime end (of $(T,X)$)} is an equivalence class of chains of
crosscuts in $(T,X)$.

Let~$\cP$ be a prime end of $(T, X)$. The {\em principal set} $\Pi(\cP)$
of~$\cP$ is the set of points $x\in X$ for which there is some
chain~$(\xi_k)$ representing~$\cP$ with $d(\xi_k, x)\to 0$ as $k\to\infty$.
The {\em impression} $\cI(\cP)$ of~$\cP$ is defined by
\[
  \cI(\cP) = \bigcap_{k\ge 0} \overline{U(\xi_k)},
\] where $(\xi_k)$ is a chain representing~$\cP$ (the definition is clearly
independent of the choice of chain). We therefore have
\[
  \emptyset \not=\Pi(\cP) \subseteq \cI(\cP) \subseteq X.
\] 

According to Carath\'eodory's classification, a prime end~$\cP$ is of the
\begin{description}
\item[First kind] if $\Pi(\cP)=\cI(\cP)$ is a point;
\item[Second kind] if $\Pi(\cP)$ is a point and is strictly contained in
 $\cI(\cP)$;
\item[Third kind] if $\Pi(\cP)=\cI(\cP)$ is not a point; and
\item[Fourth kind] if $\Pi(\cP)$ is not a point and is strictly contained in
 $\cI(\cP)$.
\end{description}

The language of rays is helpful in developing an intuitive understanding of
principal sets and impressions. A {\em ray} in $(T, X)$ is a continuous
injection $\sigma\colon[0,\infty)\to U$ with $d(\sigma(s),X)\to 0$
as $s\to\infty$. The {\em remainder} $\Rem(\sigma)$ of $\sigma$ is the set
$\overline{\sigma([0,\infty))} \cap  X$. We say that $\sigma$ {\em lands} (and
that its {\em landing point} is $x\in X$) if $\Rem(\sigma)=\{x\}$. A point
$x\in X$ is {\em accessible} if it is the landing point of some ray.

Let~$\cP$ be a prime end defined by a chain~$(\xi_k)$. A ray~$\sigma$ {\em
 converges to~$\cP$} if for every~$k$ there is some $t$ such that
 $\sigma([t,\infty))\subset U(\xi_k)$: in particular, this means that the image
 of~$\sigma$ intersects $\xi_k$ for all sufficiently large~$k$.

It can be shown that if $\sigma$ converges to~$\cP$, then $\Pi(\cP)\subseteq
\Rem(\sigma)\subseteq \cI(\cP)$: in particular, if $\sigma$ lands at an
accessible point $x\in X$, then $\Pi(\cP)=\{x\}$. Moreover, there are rays
$\sigma, \sigma'$ converging to~$\cP$ with $\Rem(\sigma)=\Pi(\cP)$ and
$\Rem(\sigma')=\cI(\cP)$. Thus the principal set and impression of~$\cP$ can be
seen, respectively, as the remainders of the ``tightest'' and ``loosest'' rays
converging to~$\cP$.

\medskip

Let $\PP$ denote the set of prime ends of $(T, X)$. There is a natural topology
on~$\PP$, with respect to which it is a topological circle: a basis for this
topology is given by the subsets $\cB(\xi)$ of $\PP$, defined for each
crosscut~$\xi$ to consist of all of the prime ends defined by chains $(\xi_k)$
with $\xi_k<\xi$ for some~$k$. (In fact, this is the subspace topology of a
natural topology on $\PP\cup U$, with respect to which this space is a compact
disk; and the definition above of a ray converging to a prime end is the normal
notion of convergence with respect to this topology.)

A homeomorphism $H\colon (T, X)\to (T, X)$ (such as the natural extension
$\hH\colon(\hT, \hI)\to(\hT, \hI)$ of the Barge-Martin construction) induces a
self-homeomorphism of the circle~$\PP$ which sends the prime end represented by
a chain $(\xi_k)$ to the prime end represented by $(H(\xi_k))$. The {\em prime
end rotation number} of $H\colon(T,X)\to(T,X)$ is the Poincar\'e rotation
number of this circle homeomorphism.

\subsubsection{Height}
Let~$\{f_t\}$ be a family of core unimodal maps of an interval~$I$ satisfying
the assumptions of Convention~\ref{conv:unimodal}, such as the quadratic or
tent family with topological entropy greater than $\frac{1}{2}\log 2$. The
Barge-Martin construction yields (abstract) sphere homeomorphisms
$\hH_t\colon\hT_t\to\hT_t$, having attractors~$\Lambda_t$ which are
homeomorphic to the inverse limits $\hI_t := \invlim(I, f_t)$, and restricted
to which the homeomorphisms~$\hH_t$ are conjugate to the natural extensions
of~$f_t$.

In the first part of the paper we study the prime ends of $(\hT_t, \Lambda_t)$.
In the second part we construct sphere homeomorphisms~$\chi_t$ by collapsing a
system of subsets of~$\hT_t$, which are permuted by~$\hH_t$, such that
\begin{itemize}
  \item each subset intersects~$\Lambda_t$ in at least one point; and
  \item each subset except at most one intersects~$\Lambda_t$ in only finitely
  many points.
\end{itemize}
The former of these properties ensures that there is a semiconjugacy~$p_t$ from
$\hf_t$ to $\chi_t$, and the latter that all but at most one of the fibers
of~$p_t$ is finite.

Both the structure of the prime ends and the construction of the semiconjugacy
(including the nature of its exceptional fiber) are heavily dependent on the
parameter~$t$, or, to be more precise, on the \emph{height} $q(f_t)$
of~$f_t$~\cite{HS} (Section~\ref{sec:height}). Dynamically, the height is the
prime end rotation number of $\hH_t\colon (\hT_t, \Lambda_t)\to(\hT_t,
\Lambda_t)$. It is an element of~$[0,1/2]$, dependent only on the kneading
sequence $\kappa(f_t)$ of~$f_t$, which decreases as $\kappa(f_t)$ increases in
the unimodal order, with each irrational height being realized by a single
kneading sequence, and each rational height being realized on a closed
\emph{height interval} of kneading sequences. See Figure~\ref{fig:height},
which shows how height varies in the quadratic family. The assumption that
$h(f_t)>\frac{1}{2}\log 2$ is, in fact, equivalent to $q(f_t)< 1/2$
(Lemma~\ref{lem:height-intervals}).

It follows that every unimodal map~$f$ is of one of three types:
\begin{description}
  \item[Irrational] when~$q(f)$ is irrational;

  \item[Rational interior] when $\kappa(f)$ is in the interior of the interval
  of kneading sequences of some rational height~$m/n$; or

  \item[Rational endpoint] when $\kappa(f)$ is an endpoint of the interval of
  kneading sequences of some rational height~$m/n$.
\end{description}

\begin{figure}[htbp]
  \begin{center}
   \includegraphics[width=0.5\textwidth]{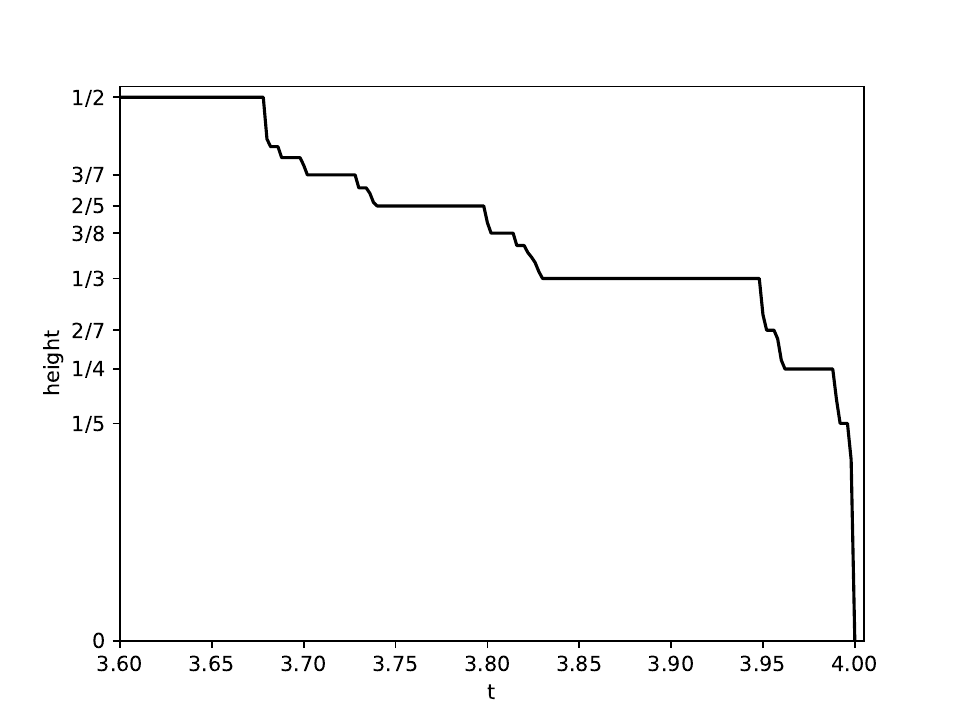}
  \end{center}
  \caption{Height $q(f_t)$ for the quadratic family $f_t(x)=tx(1-x)$.}
  \label{fig:height}
\end{figure}

\subsection{Prime ends of Barge-Martin attractors}
\label{sec:prime-end-summary}
The results of
Theorems~\ref{thm:summary-irrational},~\ref{thm:summary-rational-interior},
and~\ref{thm:summary-rational-endpoint}, together with
Remarks~\ref{rmk:irrat-accessible},~\ref{rmk:accessible-rat-interior},
and~\ref{rmk:accessible-rat-endpoint} are summarized in the following
statement, where we refer to a prime end of the first kind (whose impression is
a point) as \emph{trivial}. As discussed in Section~\ref{sec:uni-background},
the hypothesis of this theorem is satisfied by the tent and quadratic families
with topological entropy greater than $\frac{1}{2}\log 2$.

\begin{thm*}
  Let $\{f_t\}$ be a family of unimodal maps satisfying the assumptions of
  Convention~\ref{conv:unimodal}. Then the prime ends of the Barge-Martin
  attractor of~$f_t$ in the sphere satisfy the following.
  \begin{enumerate}[(a)]
     \item If $f_t$ is of \emph{irrational type}, then the set of non-trivial
     prime ends is a Cantor set. These non-trivial prime ends are of the second
     kind, with impression the whole attractor.

     \item If $f_t$ is of \emph{rational endpoint type} with height~$m/n$,
     then there are exactly~$n$ non-trivial prime ends, which are of the second
     kind, with impression the whole attractor.

     \item If $f_t$ is of \emph{rational interior type} with height~$m/n$, then
     there are exactly~$n$ non-trivial prime ends, whose impressions are the
     whole attractor. These are of the third kind (the principal set is also
     the whole attractor), unless $f_t$ belongs to a particular renormalization
     window at the start of the $m/n$ height interval, in which case they are
     of the fourth kind.

     \item If $f_t$ is of \emph{rational type} with height~$m/n$ then the
     attractor has~$n$ components of accessible points; while if~$f_t$ is of
     \emph{irrational type} then the attractor has infinitely many components
     of accessible points (countably many intervals and uncountably many
     points).
  \end{enumerate}
\end{thm*}

For the example of Figure~\ref{fig:1001110}, which depicts the inverse limit of
a tent map~$f_t$ of rational interior type with height~$1/3$, there are three
infinite ``tunnels'', corresponding to the three non-trivial prime ends, which
become stepwise narrower and narrower as they probe deeper and deeper into the
inverse limit. The natural extension stretches by a factor~$t$ in the
horizontal direction, contracts by a factor~$1/t$ in the vertical direction,
and bends the image around as dictated by~$f_t$ (see for example
Figure~\ref{fig:outside}): thus the three tunnels are permuted by the action,
and so are the three non-trivial prime ends, with rotation number~$1/3$ equal
to the height. By comparison, Figure~\ref{fig:height2_7} depicts an example of
rational interior type with height~$2/7$, where there are seven infinite
tunnels which are permuted with rotation number~$2/7$; and
Figure~\ref{fig:rhe1_3} depicts an example of rational endpoint type with
height~$1/3$: here there are infinitely many tunnels into the inverse limit,
but all of them are finite.

\begin{figure}[htbp]
  \begin{center}
   \includegraphics[width=0.7\textwidth]{height2_7}
  \end{center}
  \caption{An example of rational interior type with height~$2/7$.}
  \label{fig:height2_7}
\end{figure}

\begin{figure}[htbp]
  \begin{center}
   \includegraphics[width=0.7\textwidth]{rhe1_3}
  \end{center}
  \caption{An example of rational endpoint type with height~$1/3$.}
  \label{fig:rhe1_3}
\end{figure}

We now give an informal overview of the main steps in the proof of this
theorem, dropping the dependence on~$t$ for the sake of clarity. In
Section~\ref{sec:unwrapping} we construct an explicit unwrapping of~$f$ to be
used as the starting point of the Barge-Martin construction, based on the
\emph{outside map~$B\colon S\to S$} of~$f$ (Section~\ref{sec:outside-map}), a
monotone circle map which describes the ``thickened'' action of~$f$ as seen
from a circle around~$I$, whose rotation number is equal to the height $q(f)$
of~$f$. The explicit nature of the unwrapping provides a description of the
elements of $\hT\setminus\Lambda$ which makes it possible to construct
explicit chains of cross cuts to determine the prime ends (see for example
Figure~\ref{fig:gammas}). The key part of this process is the definition of a
homeomorphism $\Psi\colon
\hS\times[0,\infty)/(\hS\times\{0\})
\to
\hT\setminus\Lambda$ (Section~\ref{sec:Psi}), where $\hS$ is the inverse limit
of the outside map, which provides a coordinate system on $\hT\setminus\Lambda$
in which these crosscuts can be described in a straightforward way. Because the
space $\hS$ depends on the dynamics of the outside map, which is strongly
dependent on its rotation number (Theorem~\ref{thm:outside-dynamics}), the
structure of the prime ends themselves is strongly dependent on the height.

\subsection{Semi-conjugacy to a family of sphere homeomorphisms}
\label{sec:semi-conj-summary}

The results of Theorems~\ref{thm:cont-var} and~\ref{thm:gpA} are summarized in
the following statement.

\begin{thm*}
Let~$\{f_t\}$ be a family of unimodal maps satisfying the assumptions of
Convention~\ref{conv:unimodal}. Then there is a continuously varying
family~$\{\chi_t\colon S^2\to S^2\}$ of sphere homeomorphisms such that each
natural extension $\hf_t\colon\hI_t\to\hI_t$ is semi-conjugate to $\chi_t$, by
a semi-conjugacy all but one of whose fibers contains three or fewer points,
and only countably many of whose fibers contain three points.

If~$\{f_t\}$ is the tent family, then for each parameter~$t$ for which $f_t$ is
post-critically finite, the sphere homeomorphism~$\chi_t$ is a generalized
pseudo-Anosov map.
\end{thm*}

The exceptional fiber depends on the type of the unimodal map~$f_t$. In the
tent family case, where different parameters give rise to different kneading
sequences, there is a particularly clean description of these fibers:
\begin{itemize}
  \item if~$f_t$ is of irrational type, then the exceptional fiber is a Cantor set;
  \item if~$f_t$ is of rational interior type with height~$m/n$, then the exceptional fiber is finite with cardinality~$n$; and
  \item if~$f_t$ is of rational endpoint type with height~$m/n$, then the exceptional fiber is countable with~$n$ accumulation points.
\end{itemize}
In particular, the set of parameters for which the exceptional fiber is
infinite is a Cantor set. In the general case the description is more
complicated in an initial subinterval of each height interval, and the
exceptional fibre may contain arcs for parameters in these subintervals (see
Remark~\ref{rmk:semiconj-fibres}). In all cases no fiber of the semi-conjugacy
carries entropy, so no entropy is lost in the quotient.

For each parameter~$t$, the sphere homeomorphism~$\chi_t$ is constructed from
the Barge-Martin homeomorphism $\hH_t\colon\hT_t\to\hT_t$ by collapsing the
elements of an $\hH_t$-invariant decomposition~$\cG_t$ of~$\hT_t$. The
decompositions are dynamically defined, their elements being determined by the
\emph{strongly stable components} of~$\hH_t$. The homeomorphism $\Psi_t\colon \hS_t\times[0,\infty)/(\hS_t\times\{0\}) \to
\hT_t\setminus\Lambda_t$ enables these components to be described explicitly,
with their configuration determined by the type of the unimodal map~$f_t$ (see
Figures~\ref{fig:irrat-strong-stab},~\ref{fig:rat-gen-strong-stab},~\ref{fig:rat-NBT-strong-stab},
and~\ref{fig:rat-end-strong-stab}). From these descriptions, it can be shown
that each~$\cG_t$ is a non-separating, monotone, upper semicontinuous
decomposition, whose elements all intersect~$\Lambda_t$, with at most one
element intersecting~$\Lambda_t$ in more than three points. It follows from
Moore's theorem~\cite{moore} that the quotient space $\hT_t/\cG_t$ is itself a
sphere, and the quotient homeomorphism $\hH_t/\cG_t$ has the required
properties.

Since these quotient homeomorphisms are all defined on different abstract
spheres, some further work is needed to show that they can be conjugated to a
continuous family of homeomorphisms of a standard sphere. The key result here
is a theorem of Dyer and Hamstrom~\cite{DH}, Theorem~\ref{thm:DH}, which
requires, roughly speaking, that if we take the three-dimensional space
obtained by piecing together the spheres~$\hT_t$, then the decomposition of
this space obtained from the~$\cG_t$ is itself upper semicontinuous. That this
is the case follows once more from the explicit descriptions of the~$\cG_t$
(see Section~\ref{sec:cont-var}).

That the dynamics of the sphere homeomorphisms~$\chi_t$ closely mimic those of
the unimodal maps~$f_t$ is expressed by the following straightforward result
(see Theorem~\ref{thm:sphere-dynamics}).

\begin{thm*}
 Let~$f$ be a unimodal map satisfying the conditions
of Convention~\ref{conv:unimodal}, and $\chi\colon S^2\to S^2$ be the
corresponding semi-conjugate sphere homeomorphism. Then
\begin{enumerate}[(a)]
\item if~$f$ is topologically transitive then so is~$\chi$;
\item if~$f$ has dense periodic points, then so does~$\chi$;
\item $f$ and~$\chi$ have the same number of periodic orbits of each period, with
the exception that, provided $\kappa(f)\not=1\Infs{0}$,
\begin{itemize}
  \item $\chi$ has one more fixed point than~$f$, and
  \item if~$f$ is of rational type with
$q(\kappa(f))=m/n\in(0,1/2)$, then~$\chi$ has either one or two fewer period~$n$
orbits than~$f$.
\end{itemize}
\item $f$ and $\chi$ have the same topological entropy; and
\item if~$f$ preserves an ergodic Oxtoby-Ulam-measure, then~$\chi$
 preserves an ergodic
  Oxtoby-Ulam-measure with the same metric entropy.
\end{enumerate} In particular, if $f$ is a tent map of slope $t$, then~$\chi$
is topologically transitive, has dense periodic points, has topological
entropy~$\log(t)$, and has an invariant ergodic Oxtoby-Ulam-measure with metric
entropy~$\log(t)$.
\end{thm*}

\subsection{Acknowledgments}
  The authors are grateful for the support of FAPESP grants
  \mbox{2011/16265-8} and \mbox{2016/04687-9}, CAPES grant
  88881.119100/2016-01, and CAPES PVE grant \mbox{88881.068037/2014-01}. This
  research has also been supported in part by EU Marie-Curie IRSES
  Brazilian-European partnership in Dynamical Systems (FP7-PEOPLE-2012-IRSES
  318999 BREUDS). This work was supported by the Engineering and Physical
  Sciences Research Council (grant number EP/R024340/1).

\section{Preliminaries}
\label{sec:preliminaries}

\subsection{Unimodal maps}
\label{sec:unimodal} 
In this section we expand on the introductory material presented in
Section~\ref{sec:uni-background}, primarily to fix notation and conventions.
Our definition of unimodal maps reflects the fact that we will always consider
them to be defined on their cores.

\begin{defn}[Unimodal map, turning point]
\label{def:unimodal} 
  A \emph{unimodal map} is a continuous self-map
  \mbox{$f\colon[a,b]\to[a,b]$} of a compact interval~$[a,b]$, satisfying the
  following conditions:
  \begin{enumerate}[(a)]
    \item There is some $c\in(a,b)$, which is called the \emph{turning point}
      of~$f$, such that $f$ is strictly increasing on $[a,c]$ and strictly
      decreasing on $[c,b]$.
    \item $f(c)=b$ and $f(b)=a$.
  \end{enumerate}
\end{defn}

\begin{defn}[Itinerary]
\label{def:itinerary} 
  Let $f\colon[a,b]\to[a,b]$ be a unimodal map with turning point~$c$, and let
  $x\in [a,b]$. We say that an element $\syma$ of $\{0,1\}^\N$ is an {\em
  itinerary} of~$x$ if, for all $r\ge 0$,
  \begin{eqnarray*}
    \syma_r = 0 &\implies& f^r(x)\in[a,c], \text{ and}\\
    \syma_r = 1 &\implies& f^r(x)\in[c,b].
  \end{eqnarray*}
\end{defn} 
If the orbit~$\{f^r(x)\,:\,r\ge 0\}$ of~$x$ contains~$c$, then there is more
than one itinerary of~$x$. We will nevertheless abuse notation by writing
$\iota(x)=\syma$ to mean that~$\syma$ is an itinerary of~$x$.

\begin{defn}[Unimodal order]
\label{def:unimodal-order} 
  The {\em unimodal order} is a total order $\preceq$ defined on $\{0,1\}^\N$
  as follows. Let $\syma$ and $\symb$ be distinct elements of $\{0,1\}^\N$, and
  let $r\ge 0$ be least such that $\syma_r\not=\symb_r$. Then
  \[
    \syma\prec \symb \iff \sum_{i=0}^r \syma_i\,\,\text{ is even}.
  \]
\end{defn} 
The unimodal order reflects the ordering of points on the interval
$[a,b]$: if $x,y\in[a,b]$ have itineraries $\syma$ and~$\symb$ respectively,
and $x<y$, then $\syma \preceq \symb$.

\begin{defn}[Kneading sequence]
\label{def:kneading-sequence} 
  Let $f\colon[a,b]\to[a,b]$ be a unimodal map. The {\em kneading sequence}
  $\kappa(f)\in\{0,1\}^\N$ of~$f$ is the itinerary of~$b$ which is smallest
  with respect to the unimodal order.
\end{defn} 
Therefore $\kappa(f)$ is the unique itinerary of~$b$ unless the
turning point~$c$ is a periodic point of~$f$.  The choice of $\kappa(f)$ in the
periodic case has no particular significance: it is a convention which ensures
that the kneading sequence is well defined. It means that $\kappa(f) =
\Infs{W}$ for some word~$W$ whose length is the period of~$c$ and which
contains an even number of $1$s. (If $V$ and $W$ are words in the symbols $0$ and~$1$, we write $\Infs{W}$ and $V\Infs{W}$ for the
elements $WWW\ldots$ and $VWWW\ldots$ of $\{0,1\}^\N$.)

We recall the following definition and result (see for example~\cite{Devaney}),
which characterize the elements of $\{0,1\}^\N$ which are kneading sequences of
unimodal maps.

\begin{defn}[Maximal sequence]
\label{def:maximal} 
  An element $\syma$ of $\{0,1\}^\N$ is {\em maximal} if
  $\sigma^r(\syma)\preceq \syma$ for all $r\ge 0$,
  where~$\sigma\colon\{0,1\}^\N\to\{0,1\}^\N$ is the shift map.
\end{defn}

\begin{lem}
\label{lem:char-ks} 
  An element~$\syma$ of $\{0,1\}^\N$ is the kneading sequence of some unimodal
  map~$f$ if and only if $\syma$ is maximal and $\syma_0\syma_1=10$. \qed
\end{lem}

\begin{notation}[$\KS$]
\label{notation:KS} 
  We write $\KS\subset\{0,1\}^\N$ for the set of kneading sequences of unimodal
  maps: maximal sequences which start with the symbols~$10$.
\end{notation}

\begin{convention}[Standing assumptions for unimodal maps]
\label{conv:unimodal} 
  All unimodal maps~$f:[a,b]\to[a,b]$ in this paper will be assumed to satisfy
  the following conditions:
  \begin{enumerate}[(a)]
    \item $10\Infs{1} \prec \kappa(f)$.
    \item If~$\kappa(f)$ is not a periodic sequence, then distinct points
    of~$[a,b]$ cannot share a common itinerary.
    \item For each~$n>0$ and each $\mu\in\{0,1\}^\N$, there are at most two fixed points of $f^n$ with itinerary~$\mu$. If there are two such points, then $\kappa(f) = \sigma^r(\mu)$ for some~$r$.
  \end{enumerate}
\end{convention} 
Condition~(a) says that~$f$ cannot be subjected to a two-interval
renormalization: it is equivalent to requiring that the topological entropy
$h(f)$ of $f$ be greater than $\frac{1}{2}\log2$. Conditions~(b) and~(c) are
trivially satisfied by tent maps of slope greater than~1, for which no distinct
points share a common itinerary. It follows from standard results in the theory
of unimodal maps that they are also satisfied by quadratic maps, and indeed by
any $C^3$ unimodal map with non-flat turning point, no points of inflection,
and negative Schwarzian derivative.

Note that when we consider standard {\em families} of unimodal maps such as the
quadratic family and the tent family, we can apply a parameter-dependent affine
change of coordinates so that the core is constant throughout the family.

\medskip

The following notation will be useful:
\begin{defn}[$\re$, the point $\hx$ symmetric to $x$]
  Let~$f\colon[a,b]\to[a,b]$ be a unimodal map with turning point~$c$. We
  denote by~$\re$ the unique point of $(c,b]$ with $f(\re)=f(a)$. If
  $x\in[a,\re]$, we denote by $\hx$ the unique point of $[a,\re]$ which
  satisfies $f(\hx)=f(x)$, and $\hx\not=x$ unless $x=c$.
\end{defn}

Some necessary technical lemmas about the dynamics of unimodal maps, whose
proofs are routine, are presented in Appendix~\ref{app:technical}.

\subsection{Inverse limit attractors for unimodal maps: the Barge-Martin
  construction}
\label{sec:barge-martin} 
We now provide more details of the construction outlined in
Section~\ref{sec:BM-background}. The results stated are from~\cite{bargemartin}
and~\cite{param-fam}, restricted to the situation which is of interest in this
paper. Throughout this section $f\colon I\to I$ is a unimodal map defined on
the interval $I=[a,b]$. Recall that the aim of the construction is to embed the
inverse limit $\invlim(I,f)$ in the sphere, in such a way that it is a global
attractor of a homeomorphism which restricts to the natural extension $\hf$ on
$\invlim(I,f)$.

\begin{defns}[$S$, $x_u$, $x_\ell$, $T$, $\partial$, $\eta_y$]
\label{defn:circle-sphere}
  Let $S$ be the circle obtained by gluing together two copies of~$I$ at their
  endpoints. We denote the points of~$S$ by $x_u$ and $x_\ell$ for $x\in I$,
  depending on whether they come from the `upper' or `lower' copy of~$I$. We
  therefore have $a_u=a_\ell$ and $b_u=b_\ell$, and we will also denote these
  points of~$S$ with the symbols $a$ and $b$ respectively.

  Let~$T = S\times[0,1]/\!\!\sim$, where $\sim$ is the equivalence relation
  which identifies
  \begin{itemize}
    \item $(x_u,1)$ with $(x_\ell,1)$ for each~$x\in(a,b)$, and
    \item $(y, 0)$ with $(y', 0)$ for all $y, y'\in S$,
  \end{itemize}
  with the quotient topology. Then~$T$ is a two-sphere, which we endow with any
  metric~$d$ which induces its topology. Suppressing the equivalence relation,
  we will describe points of~$T$ by their ``coordinates'' $(y,s)\in
  S\times[0,1]$. We identify the subset $\{(x_u,1)\,:\,x\in I\} =
  \{(x_\ell,1)\,:\, x\in I\}$ with~$I$, so that $(x_u,1)=(x_\ell,1)=x$ for all $x\in I$, and denote by~$\partial$ the point
  of~$T$ corresponding to $S\times\{0\}$.

  $T$ decomposes into a continuously varying family of arcs $\{\eta_y\}_{y\in
  S}$ defined by $\eta_y(s) = (y,s)$, with initial points~$\eta_y(0)=\partial$
  and final points~$\eta_y(1)\in I$, whose images are mutually disjoint except
  at their initial points and perhaps at their final points. (See
  Figure~\ref{fig:MCN}. Here, for clarity, we have depicted~$T$ with the
  point~$\partial$ opened out into the circle~$S$.)
\end{defns}

\begin{figure}[htbp]
\begin{center}
\includegraphics[width=0.4\textwidth]{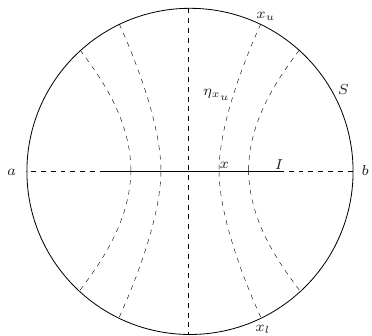}
\end{center}
\caption{The sphere~$T$ and the arc decomposition $\{\eta_y\}_{y\in S}$.}
\label{fig:MCN}
\end{figure}

\begin{defns}[The projection~$\tau$ and the smash~$\Upsilon$]
\label{def:retraction-smash} 
  The projection $\tau\colon S\to I\subset T$ is defined by $\tau(y) = (y,1)$.
  The {\em smash} $\Upsilon\colon T\to T$ is the near-homeomorphism defined by
  \[
    \Upsilon(y,s) =
    \begin{cases} 
      (y,2s) & \text{ if } s\in[0,1/2],\\ 
      (y,1) & \text{ if }s\in[1/2,1].
    \end{cases}
  \]
\end{defns}

\begin{defn}[Unwrapping]
\label{def:unwrapping} 
  An {\em unwrapping} of the unimodal map~$f$ is an
  orientation-preserving near-homeomorphism $\barf\colon T\to T$ with the
  properties that
  \begin{enumerate}[(a)]
  \item $\barf$ is injective on~$I$, and $\barf(I) \subseteq
    \{(y,s)\,:\, s\ge 1/2\}$,
  \item $\Upsilon\circ \barf|_I = f$, and
  \item $\barf(\partial)=\partial$, and for all $y\in S$ and all $s\in(0,1/2]$,
    the second component of $\barf(y,s)$ is~$s$.
  \end{enumerate}
\end{defn}

Given such an unwrapping, let $H=\Upsilon \circ \barf\colon T\to T$. Since~$H$
is a near-homeomorphism, the inverse limit $\hT=\invlim(T, H)$ is a topological
sphere by Brown's theorem. It has as a subset  
\[
  \hI = \invlim(I, H) = \invlim(I,f),
\]
 the two inverse limits being equal since $H|_I=f$ by
Definition~\ref{def:unwrapping}~(b). We reuse the notation~$\partial$
for the point $\partial = \thr{\partial, \partial, \ldots}$ of $\hT$.

Let $\hH\colon\hT\to\hT$ be the natural extension, which we refer to as the
{\em Barge-Martin homeomorphism} associated to the unwrapping~$\barf$. The
following theorem, from~\cite{bargemartin}, is a straightforward consequence of
the facts that $H|_I = f$, and that for all $(y,s)\not=\partial$, there is some
$r\ge 0$ with $H^r(y,s)\in I$.
\begin{thm}\mbox{}
\label{thm:unwrap-single}
  \begin{enumerate}[(a)]
  \item $\hH|_{\hI}\colon \hI\to\hI$ is topologically conjugate to the natural
  extension $\hf\colon \invlim(I, f)\to\invlim(I,f)$.
  \item For all $\bx\in\hT\setminus\{\partial\}$, the $\omega$-limit set
  $\omega(\bx, \hH)$ is contained in~$\hI$. \qed
  \end{enumerate}
\end{thm}

If we consider a parameterized family of unimodal maps, then the constructions
above can be done in a continuous way. Let $\{f_t\}_{t\in[0,1]}$ be a
continuously varying family of unimodal maps $I\to I$ (for each of which~$I$ is
the dynamical interval); and suppose that unwrappings $\barf_t$ of each $f_t$
are chosen in such a way that $\{\barf_t\}$ is a continuously varying family of
near-homeomorphisms of~$T$. Let $\hH_t\colon \hT_t\to \hT_t$ be the natural
extension of $H_t = \Upsilon\circ\barf_t\colon T\to T$; and let $\hI_t =
\invlim(I, H_t)$. A proof of the following result can be found
in~\cite{param-fam}, see also~\cite{barge}.

\begin{thm}
\label{thm:unwrap-family} 
  There are homeomorphisms $h_t\colon \hT_t\to S^2$ for each~$t$ (where~$S^2$
  is a standard model of the sphere) such that
  \begin{enumerate}[(a)]
    \item $h_t\circ \hH_t\circ h_t^{-1}\colon S^2\to S^2$ is a continuously
      varying family of homeomorphisms, and
    \item The attractors $h_t(\hI_t)$ vary Hausdorff continuously with~$t$.
    \qed
  \end{enumerate}
\end{thm}

\subsection{Independence of the unwrapping} In this paper we will carry out a
careful construction of a specific unwrapping $\barf\colon T\to T$ of each
unimodal map $f\colon I\to I$, which will enable us to describe precisely the
embedding of $\hI$ in~$\hT$, and hence the prime ends of $(\hT, \hI)$. A
natural and important question is therefore the extent to which the results
depend on the choice of unwrapping. We now state a theorem whose consequence is
that the results are, in fact, independent of the unwrapping.

If $\barf_0$ and $\barf_1$ are unwrappings of the same unimodal map~$f$, with
associated Barge-Martin homeomorphisms $\hH_0\colon
\hT_0\to\hT_0$ and $\hH_1\colon \hT_1\to \hT_1$ then we can identify $\hI =
\invlim(I,f) = \invlim(I,H_0)=\invlim(I,H_1)$ as a subset of both $\hT_0$ and
$\hT_1$. We say that $\barf_0$ and $\barf_1$ are {\em equivalent} if there is a
homeomorphism $\lambda\colon \hT_0\to\hT_1$ which restricts to the identity on
$\hI$. This means that $\hI$ is equivalently embedded in $\hT_0$ and~$\hT_1$;
and, since $\hH_0|_{\hI} = \hH_1|_{\hI} = \hf$, that $\lambda$ conjugates the
actions of $\hH_0$ and $\hH_1$ on $\hI$.

\begin{thm}
\label{thm:independence-unwrapping} 
  Any two unwrappings of a unimodal map~$f$ are equivalent.
\end{thm}

The proof can be found in Appendix~\ref{app:independence}.

\subsection{The height of a kneading sequence}
\label{sec:height}

{\em Height} is a function $q\colon\KS\to[0,1/2]$, introduced in~\cite{HS},
which will play a central role in this paper. (Recall from
Definition~\ref{notation:KS} that $\KS$ denotes the set of kneading sequences
of unimodal maps.) We will see that, for each unimodal map~$f$, the prime end
rotation number of the associated homeomorphism $\hH\colon
(\hT,\hI)\to(\hT,\hI)$ is equal to~$q(\kappa(f))$, and that the structure of
the prime ends depends strongly on whether $q(\kappa(f))$ is rational or
irrational, as does the exceptional fiber of the semi-conjugacy between $\hf$
and a sphere homeomorphism.

Height is defined using certain words $c_q$ associated to each rational
$q\in(0,1/2]$, which we now describe. These words, which are closely related to
Sturmian sequences, were introduced by Holmes and Williams~\cite{HoWi} in their
work on knot types in suspensions of Smale's horseshoe map, and developed by
the third author in~\cite{HS}: they also appear in a paper of Barge and
Diamond~\cite{BaDi3} on periodic orbits which are accessible from the
complement of the attracting set of H\'enon maps in cases where that attracting
set is homeomorphic to the inverse limit of a unimodal map.

\begin{defns}[The integers $\kappa_i(q)$ and the words $c_q$]
\label{def:cq} 
  Let $q\in(0,1/2]$, and let $L_q$ be the straight line in the plane which
  passes through the origin and has slope~$q$. For each~$i\ge 1$, define
  $\kappa_i(q)$ to be two less than the number of vertical lines
  $x=\mbox{integer}$ which~$L_q$ intersects for $y\in[i-1,i]$.

  If $q=m/n$ is rational (throughout the paper, when we write $m/n$ for a
  rational number, we always assume that $m$ and $n$ are coprime), define the
  word $c_q\in\{0,1\}^{n+1}$ by
  \[ 
    c_q = 10^{\kappa_1(q)} 11 0^{\kappa_2(q)} 11 \ldots 11 0^{\kappa_m(q)}1.
  \]
\end{defns}

It is straightforward (see~\cite{HS}) to obtain the following formula for
$\kappa_i(q)$: if $q=m/n$ is rational, then
\begin{equation}
\label{eq:kappa-i}
  \kappa_i(q) =
  \begin{cases}
    \floor{1/q}-1 & \text{ if }i=1, \text{ \ and}\\
    \floor{i/q}-\floor{(i-1)/q}-2  & \text{ if }2\le i\le m,
  \end{cases}
\end{equation} 
where $\floor{x}$ denotes the greatest integer which does not exceed~$x$. On
the other hand, if $q$ is irrational, then $\kappa_i(q)$ is given
by~(\ref{eq:kappa-i}) for all $i\ge 1$.

\begin{remark} 
  The fact that the formulae~(\ref{eq:kappa-i}) do not give $\kappa_i(q)$ in
  the rational case $q=m/n$ when $i>m$ is irrelevant, since we only make use of
  $\kappa_i(q)$ for $i\le m$.
\end{remark}

\begin{examples}[The words~$c_q$] 
  Figure~\ref{fig:cq} shows the line $L_{5/17}$ for $x\in[0,17]$. The numbers
  of intersections with vertical coordinate lines for $y\in[i-1,i]$ are $4$,
  $3$, $4$, $3$, and $4$ for $i=1$, $i=2$, $i=3$, $i=4$, and $i=5$. Hence
  $\kappa_1(5/17)=\kappa_3(5/17)=\kappa_5(5/17)=2$, while
  $\kappa_2(5/17)=\kappa_4(5/17)=1$. Therefore $c_{5/17}=100110110011011001$, a
  word of length~18.

  More generally, if $q=m/n$ then the word~$c_q$ is clearly palindromic, and
  contains $n-2m+1$ zeroes divided `as even-handedly as possible' into $m$
  (possibly empty) subwords, separated by~$11$. For example, for each~$n\ge 2$
  we have $c_{1/n} = 10^{n-1}1$; $c_{2/(2n+1)} = 10^{n-1}110^{n-1}1$;
  $c_{3/(3n+1)}=10^{n-1}110^{n-2}110^{n-1}1$; and
  $c_{3/(3n+2)}=10^{n-1}110^{n-1}110^{n-1}1$.
\end{examples}

\begin{figure}[htbp]
\begin{center}
\includegraphics[width=0.6\textwidth]{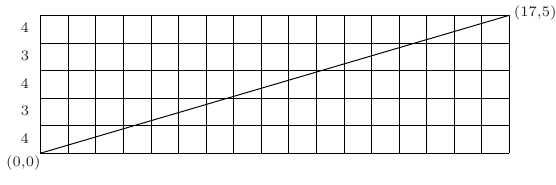}
\end{center}
\caption{$c_{5/17}=100110110011011001$.}
\label{fig:cq}
\end{figure}

The next lemma, from~\cite{HS}, is essential for the definition of height.

\begin{lem}
\label{lem:cq1-maximal} 
  $\Inf{c_q0}\in\KS$ for each rational $q\in(0,1/2]$. Moreover, the function
  $(0,1/2]\cap\Q\to\KS$ defined by $q\mapsto \Inf{c_q0}$ is strictly decreasing
  with respect to the unimodal order on~$\KS$. \qed
\end{lem}

\begin{defn}[Height]
\label{def-height} 
  Let $\syma\in\KS$. Then the {\em height} $q(\syma)\in[0,1/2]$ of $\syma$ is
  given by
  \[ 
    q(\syma) = \inf\left(\{q\in(0,1/2]\cap\Q\,:\, \Inf{c_q0} \prec
    \syma\}\cup\{1/2\}\right).
  \]
\end{defn}

By Lemma~\ref{lem:cq1-maximal}, the height function $q\colon\KS\to[0,1/2]$ is
decreasing with respect to the unimodal order on~$\KS$ and the usual order
on~$[0,1/2]$. The next result, also from~\cite{HS}, describes the interval of
kneading sequences with given rational height.

\begin{notation}[The words $w_q$]
\label{notn:wq} 
  For each $q=m/n\in(0,1/2)$, define $w_q\in\{0,1\}^{n-1}$ to be the word
  obtained by deleting the last two symbols of $c_q$; and
  $\hw_q\in\{0,1\}^{n-1}$ to be the reverse of $w_q$.
\end{notation}

Statements~(a), (b), and (c) of the following lemma can be found in~\cite{HS},
while~(d) is contained in results of~\cite{gpa} (see also Lemma~11.5
of~\cite{AC} for a self-contained proof).

\begin{lem}[Characterization of kneading sequences of given height]\mbox{}
\label{lem:height-intervals}
  \begin{enumerate}[(a)]
    \item For each irrational $q\in[0,1/2]$, there is a unique $\syma\in\KS$
      with $q(\syma)=q$, namely
      \[\syma=10^{\kappa_1(q)}110^{\kappa_2(q)}110^{\kappa_3(q)}11\ldots.\]
    \item Let $\syma\in\KS$. Then $q(\syma)=0$ if and only if
      $\syma=1\Infs{0}$; and $q(\syma)=1/2$ if and only if $\syma\preceq
      10\Infs{1}$.
    \item Let $\syma\in\KS$ and $q=m/n\in(0,1/2)\cap\Q$. Then $q(\syma)=q$ if
    and only if
    \[
    \Inf{w_q1} \preceq \syma \preceq 10\Inf{\hw_q1}.
    \] 
    Moreover, if $q(\mu)=q$ and $\syma\not=\Inf{w_q1}$ then $c_q$ is an initial
    subword of~$\syma$; and if~$\syma$ is periodic, then either
    $\syma=\Inf{w_q1}$ or $\syma=\Inf{w_q0}$, or its period is at least~$n+2$.
    \item Let $q=m/n\in(0,1/2)\cap\Q$. then $10\Inf{\hw_q1}$ is pre-periodic to
    $\Inf{w_q1}$: that is, there is some $r$ with $\sigma^r(10\Inf{\hw_q1}) =
    \Inf{w_q1}$. \qed
  \end{enumerate} 
\end{lem} By Lemma~\ref{lem:height-intervals}~(b), Condition~(a) of
Convention~\ref{conv:unimodal} says exactly that $q(\kappa(f))\in[0,1/2)$.

The endpoints of the intervals of kneading sequences of given rational height
will play an important role, as will the kneading sequences $\Inf{c_q0}$ used
in the definition of height. The acronym NBT in the following notation stands
for `no bogus transitions' and reflects the original motivation of height.

\begin{notations}[$\lhe(q)$, $\rhe(q)$, $\NBT(q)$, $\KS(q)$]
\label{def:lhe,rhe,etc} 
  Let $q\in(0,1/2)\cap\Q$. We write $\lhe(q)=\Inf{w_q1}$,
  $\rhe(q)=10\Inf{\hw_q1}$, and $\NBT(q)=\Inf{c_q0}$. We write $\KS(q)$ for the
  set of kneading sequences $\syma$ with height~$q$ (i.e.\ with $\lhe(q)\preceq
  \syma\preceq \rhe(q)$). In the special case $q=0$, we write
  $\lhe(q)=1\Infs{0}$ and $\rhe(q)$ is undefined.
\end{notations}

\begin{example} 
  Let $q=2/7$, with $c_q=10011001$, $w_q=100110$, and $\hw_q=011001$. Then we
  have \mbox{$\lhe(2/7)=\Inf{1001101}$}, $\rhe(2/7)= 10\Inf{0110011}$, and
  \mbox{$\NBT(2/7) =\Inf{100110010}$}. A kneading sequence $\syma$ lies
  in~$\KS(2/7)$ if and only if $\Inf{1001101} \preceq \syma\preceq
  10\Inf{0110011}$.
\end{example}

In addition to the characterization of Lemma~\ref{lem:height-intervals}, there
is a straightforward algorithm which calculates $q(\syma)$ for any kneading
sequence which has rational height: see Section~3.2 of~\cite{HS}\footnote{A
script to carry out this calculation can be found at
\url{http://www.maths.liv.ac.uk/cgi-bin/tobyhall/horseshoe}}

As stated at the beginning of this section, the structure of the prime ends of
$(\hT, \hI)$ depends on whether $q(\kappa(f))$ is rational or irrational. The
rational case $q(\kappa(f))=m/n$ also splits into two subcases: one in which
$\kappa(f)$ is either an endpoint of $\KS(m/n)$ or is equal to $\Inf{w_q0}$ (a
consecutive kneading sequence to $\lhe(m/n)$), and one in which neither of
these happens. The endpoint case further splits into subcases which, while they
yield the same results, are analyzed in quite different ways. These
observations motivate the following definitions (see Figure~\ref{fig:types}).

\begin{defns}[Irrational and rational types; interior and endpoint types;
    early, strict, and late types; tent-like and quadratic-like types; normal
    type; general and NBT types]
\label{def:type} 
  We say that a unimodal map~$f$ is of {\em irrational type} or of {\em
  rational type} according as $q(\kappa(f))$ is irrational or rational.

  In the rational case, with $q(\kappa(f))=m/n\in(0,1/2)$, we say that $f$ is
  of {\em (rational) endpoint type} if $\kappa(f)\in\{\lhe(m/n),
  \Inf{w_q0}, \rhe(m/n)\}$; and is of {\em (rational) interior type} otherwise.

  In the rational endpoint case we say that $f$ is of {\em left endpoint type}
  if $\kappa(f)\in\{\lhe(m/n), \Inf{w_q0}\}$, and of {\em right
  endpoint type} if $\kappa(f)=\rhe(m/n)$. 

  In the rational left endpoint case, we say that~$f$ is of {\em early endpoint
  type} if $\kappa(f)=\lhe(m/n)$ but $f^n(a)\not=a$; of {\em strict endpoint
  type} if $\kappa(f)=\lhe(m/n)$ and $f^n(a)=a$; and of {\em late endpoint
  type} if $\kappa(f)=\Inf{w_q0}$.

  In the left strict endpoint case, we say that~$f$ is of {\em tent-like type}
  if $b$ is the only period~$n$ point of~$f$ with itinerary $\lhe(m/n)$; and
  that it is of {\em quadratic-like type} if it has a second such period~$n$
  point (there cannot be more than two period~$n$ points with this itinerary by
  Convention~\ref{conv:unimodal}~(c)).

  We say that~$f$ is of {\em normal (endpoint) type} if it is either of right
  endpoint type, or of tent-like left strict endpoint type. (These are the only
  endpoint types which occur for tent maps.)

  In the rational interior case we say that $f$ is of {\em (rational) NBT type}
    if $f^{n+2}(c)=c$ --- in which case $\kappa(f)=\NBT(m/n)$
  --- and of {\em (rational) general type} otherwise.

  In the special case $m/n=0$ (i.e.\ $\kappa(f)=1\Infs{0}$), we declare~$f$ to
  be of tent-like strict left endpoint type.
\end{defns}

\begin{remark}\mbox{}
\label{rmk:terminology} To explain the terminology in the rational left
endpoint case, consider a full monotonic family $\{f_t\}$ of unimodal maps such
as the quadratic family, and let $t_1 = \inf\{t\,:\, \kappa(f_t)=\lhe(m/n)\}$.
Then a saddle-node bifurcation occurs at $t=t_1$ creating a semi-stable
period~$n$ orbit, which contains a point of itinerary $\lhe(m/n)$ and attracts
the orbit of the turning point. As~$t$ increases, this periodic orbit splits
into a stable-unstable pair of periodic orbits, both containing points of
itinerary~$\lhe(m/n)$. We follow this pair of periodic orbits until at $t=t_2$
the stable orbit contains the turning point. We still have
$\kappa(f_{t_2})=\lhe(m/n)$, but now $f_{t_2}^n(a)=a$. When we increase the
parameter further, the stable periodic orbit passes through the turning point
and the kneading sequence becomes $\Inf{w_q0}$. Therefore $f_t$ is of early
endpoint type for $t\in[t_1,t_2)$, of strict quadratic-like endpoint type for
$t=t_2$, and of late endpoint type for $t>t_2$ sufficiently close to~$t_2$.
There is no corresponding distinction at the right hand endpoint of the height
interval since, by Convention~\ref{conv:unimodal}~(b), if
$\kappa(f)=\rhe(m/n)$, which is not periodic, then $f(a)$ is necessarily
periodic of period~$n$ since it has the same itinerary $\Inf{\hw_q 1}$ as
$f^{n+1}(a)$.

 In the tent family, by contrast, $\kappa(f)=\lhe(m/n)$ only if $f^n(a)=a$,
 $\kappa(f)$ is never equal to $\Inf{c_q0}$, and there is only ever one point
 of any given itinerary. Therefore only the strict tent-like left hand endpoint
 case occurs. That is, only the first three rows of Figure~\ref{fig:types} are
 relevant for tent maps.

The reason for the distinction between rational general and rational NBT types
will become apparent in Section~\ref{sec:semi-conj}.
\end{remark}

\begin{figure}[htbp]
\begin{center}
\includegraphics[width=0.4\textwidth]{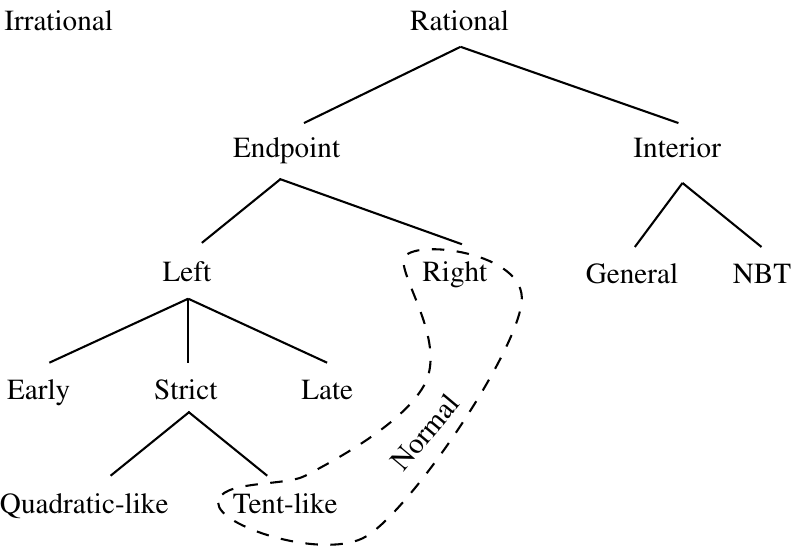}
\end{center}
\caption{Types of unimodal maps.}
\label{fig:types}
\end{figure}

\section{The unwrapping}
\label{sec:unwrapping-main}

In this section we construct an explicit unwrapping $\barf\colon T\to T$ of an
arbitrary unimodal map $f\colon[a,b]\to [a,b]$. This construction provides
explicit descriptions of the sphere $\hT$, the embedded inverse limit $\hI$,
and the homeomorphism $\hH\colon\hT\to\hT$ which restricts to the natural
extension $\hf$ on $\hI$.

The construction proceeds in two steps. In Section~\ref{sec:outside-map} we
recall from~\cite{gpa} the {\em outside map} $B\colon S\to S$ corresponding
to~$f$, which is obtained by ``fattening up'' the interval to give it some
two-dimensional structure (a closely related construction also appears
in~\cite{bruin2}). The unwrapping~$\barf$ itself is then constructed in
Section~\ref{sec:unwrapping}. It is the product of the outside map and the
identity on $\{(y,s)\,:\, s\le 1/2\}\subset T$, and is gradually changed in
$\{(y,s)\,:\,s\ge 1/2\}$ so that it satisfies the conditions of an unwrapping
(Lemma~\ref{lem:barf-is-unwrapping}). We finish with a description of the
elements of $\hT\setminus\hI$ (\Notation~\ref{notn:threads-U} and
Lemma~\ref{lem:threads-U}).

\subsection{The outside map}
\label{sec:outside-map} Let~$f\colon [a,b]\to [a,b]$ be a unimodal map with
turning point~$c\in(a,b)$. Recall that we denote by $\re$ the unique point
of~$(c,b]$ with $f(\re)=f(a)$.

Recall from Section~\ref{sec:barge-martin} that~$S$ denotes the circle obtained
by gluing together two copies of~$I$ at their endpoints; that points of~$S$ are
denoted $x_u$ or $x_\ell$ for $x\in I$; and that we write $a=a_u=a_\ell$ and
$b=b_u=b_\ell\in S$. We will use standard interval notation $(x,y)$, $[x,y]$,
etc.\ for subintervals of~$S$, the interval consisting of the arc which goes
counterclockwise, in the model of Figure~\ref{fig:MCN}, from the first point
listed to the second. Thus, for example, the interval~$[a,b]$ contains $x_\ell$
for all $x\in I$, while the interval~$[b,a]$ contains $x_u$ for all $x\in I$.

 The intuitive motivation for the definition of the outside map $B\colon S\to
 S$ is illustrated in Figure~\ref{fig:outside}. We add some two-dimensional
 structure to the unimodal map~$f$ as depicted on the left of the figure,
 regarding the image of~$[a,c)$ as lying underneath the image of~$(c,b]$. Then
 points which are {\em above} the interval, lying in~$(a,\re)$, get folded into
 the interior -- that is, they no longer remain on the outside. These points
 correspond to the interior of the interval~$\gamma=[\re_u,a]$ in~$S$ depicted
 on the right hand side of the figure, which is collapsed to a point by the
 outside map.  Other points above the interval, and all points below the
 interval, remain on the outside after one iteration, with points below $[a,c)$
 and above $[\re,b]$ being sent below the interval, and points below $(c,b]$
 being sent above the interval.

\begin{figure}[htbp]
\begin{center}
\includegraphics[width=0.8\textwidth]{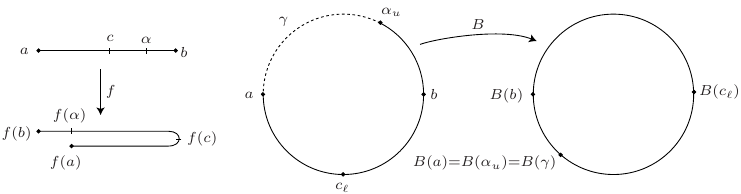}
\end{center}
\caption{The outside map~$B$ corresponding to a unimodal map~$f$.}
\label{fig:outside}
\end{figure}

This intuition leads to the following definition:

\begin{defn}[The outside map]
\label{def:outside-map} 
  Let $f\colon[a,b]\to[a,b]$ be a unimodal map. The {\em outside map
    \mbox{$B\colon S\to S$} corresponding to~$f$} is defined by
  \begin{equation}
  \label{eq:outside-map}
    \begin{array}{rcll} 
      B(x_\ell) &=& f(x)_\ell & \text{ if }x\in[a,c]\\
      B(x_\ell) &=& f(x)_u & \text{ if }x\in[c,b]\\ 
      B(x_u) &=& f(x)_\ell & \text{ if }x\in[\re,b], \text{ and}\\ 
      B(x_u) &=& f(a)_\ell & \text{ for all }x\in[a,\re].
    \end{array}
  \end{equation}
\end{defn}

The dynamics of the outside map plays a key role in the paper, and is discussed
in detail in Section~\ref{sec:outside-dynamics} below. For now we note that, by
the first three equations of~(\ref{eq:outside-map}),
\begin{equation}
\label{eq:commute}
  \tau\circ B(y) = f\circ\tau(y) \quad\text{ for all }y\in S\setminus\ingam,
\end{equation} 
where $\ingam=(\re_u,a)$ and $\tau\colon S\to I$ is the projection of
Definition~\ref{def:retraction-smash}, satisfying $\tau(x_\ell)=\tau(x_u)=x$.

\subsection{Definition of the unwrapping}
\label{sec:unwrapping} We now use the outside map to define an unwrapping of
the unimodal map~$f$.  Figure~\ref{fig:construct-unwrapping} shows the
sphere~$T$ (with~$\partial$ opened out into the circle~$S$), the
interval~$I\subset T$, the circle $S\times\{1/2\}$ (dashed line), and segments
of some of the arcs~$\eta_y$ (dotted lines). It also depicts an interval~$J$
with endpoints $(f(a)_\ell,\, 1/2)$ and $(b,1)$. The unwrapping will be
constructed so that as $x$ runs from~$a$ to $c$, $\barf(x_\ell,1)$ runs
along~$J$ with $\Upsilon(\barf(x_\ell,1)) = (f(x),1)$; while as $x$ runs from
$c$ to~$b$, $\barf(x_\ell,1)=(f(x),1)\in I$.  The interval~$J$ is defined by
$J=\{(\phi(s)_\ell, s)\,:\,s\in[1/2,1]\}$, where $\phi$ is the affine map of
Definition~\ref{def:unwrapping-unimodal} below.

\begin{figure}[htbp]
\begin{center}
\includegraphics[width=0.4\textwidth]{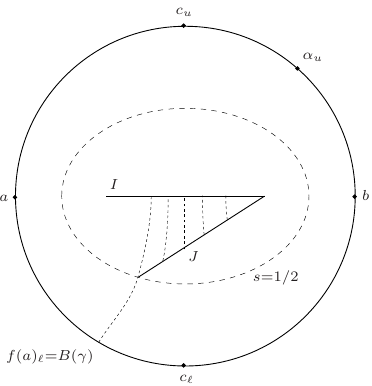}
\end{center}
\caption{Construction of the unwrapping $\barf$.}
\label{fig:construct-unwrapping}
\end{figure}

\begin{defn}[The map $\phi\colon\lbrack1/2,1\rbrack \to \lbrack f(a), b
\rbrack$ and the unwrapping~$\barf$ of a unimodal map~$f$]
\label{def:unwrapping-unimodal} \mbox{}\\
  Let $\phi\colon[1/2,1]\to[f(a),b]$ be the affine map
  \[
  \phi(s) = f(a) + (2s-1)(b-f(a))
  \] 
  with $\phi(1/2)=f(a)$ and $\phi(1)=b$. We define $\barf\colon T\to T$ as
  follows:
  \begin{enumerate}[(U1)]
    \item $\barf(y,s) = (B(y),s)$ for all $(y,s)\in S\times[0,1/2]$.
    \item If $y\in[c_\ell, \re_u]$ then $\barf(y,s) = (B(y),s)$ for all
      $s\in[0,1]$.
    \item If $y\in [\re_u, c_u]$ then
    \[
    \barf(y,s) =
    \begin{cases} (\phi(s)_\ell, s) & \text{ if }s\in[1/2,\,\,
    \phi^{-1}(f(\tau(y)))], \\ (f(\tau(y))_\ell, s) & \text{ if
    }s\in[\phi^{-1}(f(\tau(y))),\,\, 1].
    \end{cases}
    \]
    \item If $y\in [c_u,a]$ then
    \[
    \barf(y,s) =
    \begin{cases} (\phi(s)_\ell, s) & \text{ if } s\in[1/2,\,\,
    \phi^{-1}(f(\tau(y)))],\\ (f(\tau(y))_\ell, \phi^{-1}(f(\tau(y))) & \text{ if
    }s\in[\phi^{-1}(f(\tau(y))),\,\, 1].
    \end{cases}
    \]
    \item If $y\in [a,c_\ell]$ then
    \[
    \barf(y,s) =
    \begin{cases} (B(y),s) & \text{ if }s\in[1/2,\,\, \phi^{-1}(\tau(B(y)))],\\
    (B(y), \phi^{-1}(\tau(B(y)))) & \text{ if }s\in[\phi^{-1}(\tau(B(y))),\,\,
    1].
    \end{cases}
    \]
  \end{enumerate}
\end{defn}

\begin{remarks}\mbox{}
\label{rmk:barf}
  \begin{enumerate}[(a)]
    \item  If $y\not\in\ingam$ then the first component of $\barf(y,s)$ is
      equal to $B(y)$ for all~$s\in[0,1]$ by (U1), (U2), and~(U5).
    \item When parsing this definition, it is helpful to recall that, in order
      for $\barf$ to be an unwrapping, we must have $\Upsilon(\barf(y,1)) =
      f(\tau(y))$ for each~$y\in S$ (Definition~\ref{def:unwrapping}~(b)). The
      value $s=\phi^{-1}(f(\tau(y)))$ which appears in~(U3), (U4), and (U5) ---
      noting that in~(U5) we have $y\not\in\ingam$, so that
      $\tau(B(y))=f(\tau(y))$ by~\eqref{eq:commute} --- is the parameter of the
      point~$j_y$ of~$J$ which retracts to $f(\tau(y))$: therefore $\barf(y,1)$
      must lie on the decomposition arc~$\eta$ which passes through this point.
      According to the definition,
    \begin{enumerate}[(U1)]
    \addtocounter{enumii}{2}
      \item When $y\in[\re_u,c_u]$, the path $\{\barf(y,s)\,:\,s\in[1/2,1]\}$
        moves along~$J$ from $(f(a)_\ell, 1/2)$ until it reaches $j_y$, and
        then moves along~$\eta$ until it reaches~$I$;
      \item When $y\in[c_u,a]$, the path $\{\barf(y,s)\,:\, s\in[1/2,1]\}$ moves
        along~$J$ from $(f(a)_\ell, 1/2)$ until it reaches $j_y$, and then
        remains at this point;
      \item When $y\in[a,c_\ell]$, the path $\{\barf(y,s)\,:\, s\in[1/2,1]\}$ moves
        along~$\eta$ from $(B(y),1/2)$ until it reaches $j_y$, and then remains
        at this point.
    \end{enumerate}
  \end{enumerate}
\end{remarks}

\begin{lem}
\label{lem:barf-is-unwrapping} 
  $\barf$ is an unwrapping of $f$.
\end{lem}
\begin{proof} 
  A theorem of Youngs~\cite{youngs} states that any continuous monotone
  surjection $T\to T$ is a near-homeomorphism. Therefore $\barf$ is a
  near-homeomorphism (which is clearly orientation-preserving), since the
  preimage of each point of~$T$ under~$\barf$ is either a point or an arc. In
  fact, the only points of~$T$ whose preimages are not points are
  \begin{itemize}
    \item For each $s\in(0,1/2)$, the point $(f(a)_\ell,s)$, whose preimage is
      the arc $[\re_u,a]\times \{s\}$, and
    \item For each $s\in[1/2,1)$, the point $(\phi(s)_\ell, s)$ of~$J$, whose
      preimage is the arc
    \[
      [z(s)_u, w(s)_u] \times\{s\} \,\, \cup \,\,
      \{w(s)_u\}\times[s, 1] \,\, \cup \,\,
      \{w(s)_\ell\}\times [s,1],
    \]
    where $w(s)$ and $z(s)$ denote the points of $[a, c]$ and of $[c,b]$
    respectively with $f(w(s))=f(z(s)) = \phi(s)$. The first set in this union
    comes from~(U3) and~(U4), the second from~(U4), and the third from~(U5).
  \end{itemize}

   Moreover, $\barf$ satisfies condition~a) of Definition~\ref{def:unwrapping},
   since $\barf|_I$ is injective, with $\barf(y,1)\in I\cup J$ for all $y\in S$
   by (U2)\,--\,(U5); it satisfies condition~b) since $\Upsilon\circ\barf(y,1)
   = f\circ\tau(y)$ for all $y\in S$ by (U2)\,--\,(U5) and~\eqref{eq:commute};
   and it satisfies condition~c) by (U1). It is therefore an unwrapping of~$f$
   as required.
\end{proof}

\begin{defns}[$H$, $\hT$, $U$, $\hH\colon\hT\to\hT$, $\hB\colon\hS\to\hS$] 
  As in Section~\ref{sec:barge-martin}, set $H =\Upsilon\circ\barf\colon T\to
  T$, and observe that $H|_I=f$. Write
  \begin{eqnarray*}
    \hT &=& \invlim(T,H) \quad\text{ (a topological sphere),}\\
    \hI &=& \invlim(I,H) = \invlim(I,f), \quad\text{ and}\\ 
    U   &=& \hT\setminus \hI.
  \end{eqnarray*} 
  Let $\hH\colon\hT\to\hT$ be the natural extension of $H\colon T\to T$, so
  that $\hH|_{\hI} = \hf$, the natural extension of $f\colon I\to I$.

  Let $\hB\colon\hS\to\hS$ denote the natural extension of the outside map
  $B\colon S\to S$. This is a circle homeomorphism, since $\hS$ is a
  topological circle by Brown's theorem.
\end{defns}

We next introduce some notation for the elements of~$U$. The key fact here is
that if $y\in S$ and $s\in[0,1/2)$, then $H(y,s) =
  \Upsilon\circ\barf(y,s) = \Upsilon(B(y),s) = (B(y), 2s)$. Recall that we
  denote by~$\partial$ the element $\thr{\partial, \partial, \ldots}$ of~$\hT$.

\begin{notation}[Threads $\bT(\by,s)$ and $\bT(\by,s,k)$ in $U$]\mbox{}
\label{notn:threads-U}
  \begin{enumerate}[(a)]
  \item For each $\by\in\hS$ and $s\in(0,1)$, define
  \begin{equation}
  \label{eq:thread1}
  \bT(\by, s) = \thr{(y_0,s), (y_1,s/2), (y_2,s/4), \ldots} \in U.
  \end{equation}
  \item For each $\by\in\hS$, $s\in[1/2,1)$, and $k\ge 0$, define
  \begin{equation}
  \label{eq:thread2}
  \bT(\by, s, k) = \thr{f^k(H(y_0,s)), \ldots, f(H(y_0,s)), H(y_0,s), (y_0,s),
    (y_1,s/2), (y_2, s/4), \ldots} \in U.
  \end{equation}
  \end{enumerate}
\end{notation}
\begin{lem}
\label{lem:threads-U} 
  Every element of~$U\setminus\{\partial\}$ is equal to exactly one of the
  threads of \Notation~\ref{notn:threads-U}.
\end{lem}
\begin{proof} 
  Let $\bx\in U\setminus\{\partial\}$. Since $\bx\not\in\hI$, there is some
  least~$k\ge 0$ with $x_k\not\in I$: therefore $x_k = (y,s)$ for some $y\in S$
  and $s\in(0,1)$.

  If $k=0$ then $\bx = \bT(\by,s)$, where $y_i$ is the first component of~$x_i$
  for each~$i$. On the other hand, if $k\ge 1$ then, since $H(x_k)\in I$, we
  have $s\in[1/2,1)$; and $\bx = \bT(\by,s,k-1)$ where $y_i$ is the first
  component of~$x_{k+i}$ for each~$i$.
\end{proof}

The interesting entry of the threads $\bT(\by,s,k)$ is $H(y_0,s)\in I$, which
 is where the transition takes place from the dynamics of the outside map to
 the dynamics of the unimodal map.

\begin{remark} 
  The unwrapping~$\barf$ varies continuously with the unimodal map~$f$. It
  follows from Theorem~\ref{thm:unwrap-family} that if $\{f_t\}$ is a
  continuously varying family of unimodal maps, then the spheres constructed
  above can be identified with a standard model in such a way that the
  homeomorphisms $\hH_t$ and the attractors $\hI_t$ vary continuously.
\end{remark}

\section{Calculation of prime ends}
\label{sec:prime-end-calc} 

In this section we determine the prime ends of~$(\hT, \hI)$ for any unimodal
map~$f$ satisfying the conditions of Convention~\ref{conv:unimodal}. The main
tool that we use is an explicit homeomorphism~$\Psi$ from the open disk
$D=\hS\times[0,\infty)/(\hS\times\{0\})$ to~$U = \hT\setminus\hI$, which is
defined in Section~\ref{sec:Psi}. We will see that~$\Psi$ conjugates $\hH|_{U}$
to the product of $\hB$ and a simple push on~$D$
(Corollary~\ref{cor:Psi-conjugates}).

In Section~\ref{sec:loc-unif-land} we define the {\em locally uniformly landing
  set}~$\cR$, a subset of~$\hS$ with the property that~$\Psi$ extends
  continuously over $\cR\times\{\infty\}$.  In Section~\ref{sec:good-chain} we
  impose some additional conditions (which we show later are always satisfied),
  and use these to construct a homeomorphism between $\hS$ and the circle~$\PP$
  of prime ends.

The structure of the locally uniformly landing set for a specific unimodal
map~$f$ depends on the dynamics of the outside map $B\colon S\to S$, which is
discussed in Section~\ref{sec:outside-dynamics}. Armed with the results of that
section, we will be able to complete the calculation of prime ends. The details
of this calculation depend on whether $f$ is of irrational type, of rational
interior type, or of rational endpoint type, and these cases are presented in
Sections~\ref{sec:irrational},~\ref{sec:rational}, and~\ref{sec:rat-endpoint}
respectively.

\subsection{The homeomorphism \texorpdfstring{$\Psi\colon D\to U$}{Psi}}
\label{sec:Psi}
\begin{notations}[$D$, $\bD$, $\partial'$, $\hS_\infty$, $X_\infty$, the
    push~$\lambda$, the homeomorphisms $G\colon\bD\to\bD$ and $G\colon D\to D$]
  \label{notn:G-et-al}  \mbox{}\\   
    Write $D=\hS\times[0,\infty)/(\hS\times\{0\})$ and $\bD =
    \hS\times[0,\infty]/(\hS\times\{0\})$. We regard $D$ as a subset of
    $\bD$, and use coordinates $(\by,s)\in \hS\times[0,\infty]$ on $D$ and
    $\bD$: these coordinates are singular at $\partial'$, the point
    corresponding to $\hS\times\{0\}$.

     Write $\hS_\infty = \hS\times\{\infty\}
    \subset \bD$, the circle at infinity. Similarly, given any subset $X$ of
    $\hS$, we write $X_\infty = X\times\{\infty\}\subset \hS_\infty$.

Let $\lambda\colon[0,\infty]\to[0,\infty]$ be defined by
\[
  \lambda(s) =
  \begin{cases} 
    2s & \text{ if }s\in [0,1],\\ s+1 & \text{ if }s\in [1,\infty),\\
    \infty & \text{ if }s=\infty,
  \end{cases}
\] 
and $G\colon \bD\to \bD$ be the homeomorphism defined by $G(\by,s) = (\hB(\by),
\lambda(s))$. We denote the restriction to~$D$ with the same symbol, $G\colon
D\to D$.
\end{notations}

In this section we define an explicit homeomorphism $\Psi\colon D\to U$, which
is constructed in such a way that it conjugates $G\colon D\to D$ to $\hH\colon
U\to U$, thereby providing a coordinate system on~$U$ in which the action
of~$\hH$ is very easy to understand. We will see in subsequent sections
that~$\Psi$ extends over an open dense subset of the circle $\hS_\infty$ as a
homeomorphism into $\hT$. The non-trivial prime ends of $(\hT, \hI)$ can be
understood in terms of the action of $\Psi$ on rays in~$D$ which converge to
points of $\hS_\infty$ at which~$\Psi$ is discontinuous or not defined.

The surjectivity of~$\Psi$ will be an immediate consequence of its definition
(Lemma~\ref{lem:Psi-surjective}). To show that it is continuous and injective,
we first establish that it semi-conjugates $G$ and $\hH$
(Lemma~\ref{lem:conjugation}), and then use this semi-conjugacy to extend the
obvious continuity and injectivity on $\hS\times(0,1)$ over the rest of~$D$
(Corollary~\ref{cor:Psi-homeo}).

In order to define $\Psi$, it will be convenient to introduce the following
notation (in which it should be noted carefully that~$v$ is {\em not} the
fractional part of~$s$).
\begin{notation}[Splitting $s$ into parts] 
  We define $P\colon [0,\infty)\to \N \times [1/2,1)$ by $P(s) = (t,v)$, where
  $t=\floor{s}$ is the integer part of~$s$ and $v=(u+1)/2$, where~$u=s-t$ is
  the fractional part of~$s$.
\end{notation}

\begin{defn}[$\Psi\colon D\to U$] 
  Define $\Psi\colon D\to U$ by $\Psi(\partial') = \partial$ and
  \begin{equation}
  \label{eq:Psi}
    \Psi(\by,s) =
    \begin{cases}
      \bT(\by,s) & \text{ if }s\in(0,1),\\
      \bT(\hB^{-t}(\by), v, t-1), & \text{ where $(t,v) = P(s)$, if $s\ge1$.}
    \end{cases}
\end{equation}
\end{defn} 
Substituting \eqref{eq:thread2} into the formula for $\Psi(\by,s)$
in the case $s\ge 1$ yields the useful alternative expression
\begin{equation}
\label{eq:convenient}
  \Psi(\by,s) = \thr{f^{t-1}(H(y_t,v)),
    \ldots, f(H(y_t,v)), H(y_t,v), (y_t,v), (y_{t+1}, v/2), \ldots}
  \quad\text{ when }s\ge 1.
\end{equation}

Therefore the number of entries of $\Psi(\by,s)$ which are in~$I$ is equal to
the integer part~$t$ of~$s$. We will frequently use the following immediate
consequence of~(\ref{eq:convenient}):

\begin{equation}
\label{eq:convenient-2}
  \Psi(\by,s)_r = f^{t-1-r}(H(y_t,v))\qquad\text{for $0\le r\le t-1$, where
  $P(s)=(t,v)$}.
\end{equation}

When $y\in\gamma$, the definition of $\barf(y,s)$, and hence of $H(y,s)$,
depends on whether~$s$ is smaller or greater than $\phi^{-1}(f(\tau(y)))$ (see
(U3) and (U4) of Definition~\ref{def:unwrapping-unimodal}).
By~(\ref{eq:convenient}), the behavior of $\Psi(\by,s)$ therefore depends on
whether $v=(u+1)/2$ is smaller or greater than this value; that is, on whether
the fractional part~$u$ of $s$ is smaller or greater than
$2\phi^{-1}(f(\tau(y_t)))-1$. This value will frequently be significant in the
remainder of the paper, and the following notation will be useful.

\begin{notation}[The function $u\colon S\to\lbrack0,1\rbrack$]
\label{notn:u-y} 
  Define $u\colon S\to[0,1]$ by
  \[ 
    u(y) =
    \begin{cases} 
      2\phi^{-1}(f(\tau(y)))-1 & \text{ if }y\in\gamma,\\ 
      0 & \text{ otherwise.}
    \end{cases}
  \]
\end{notation} 
In particular $u(c_u)=1$ and $u(a)=u(\alpha_u)=0$. The following
lemma gives the key property of the function~$u$.

\begin{lem}
\label{lem:u-key} 
  Let $y\in S$ and $v\in[1/2,1)$, and write $u=2v-1\in[0,1)$.
  Then
  \[ H(y,v) =
  \begin{cases} 
    f(\tau(y)) & \text{ if }u\ge u(y),\\
    \phi(v) & \text{ if }u<u(y).
  \end{cases}
  \]
\end{lem}
\begin{proof} 
  If $y\not\in\gamma$ then necessarily $u\ge u(y)$, and
  $H(y,v)=\tau(B(y)) = f(\tau(y))$ by Remark~\ref{rmk:barf}~(a)
  and~(\ref{eq:commute}).

  If $y\in\gamma$ and $u\ge u(y)$, then $v\ge\phi^{-1}(f(\tau(y)))$, and hence
  the first component of $\barf(y,v)$ is $f(\tau(y))_\ell$ by~(U3) and~(U4) of
  Definition~\ref{def:unwrapping-unimodal}. Therefore $H(y,v)=\Upsilon\circ
  \barf(y,v) = f(\tau(y))$ as required.

  If $y\in\gamma$ and $u<u(y)$, then $v<\phi^{-1}(f(\tau(y)))$, and hence the
  first component of $\barf(y,v)$ is $\phi(v)_\ell$ by~(U3) and~(U4) of
  Definition~\ref{def:unwrapping-unimodal}. Therefore $H(y,v) = \phi(v)$ as
  required.
\end{proof}

\begin{remark} 
  If $y\in\gamma$ and $u=u(y)$ then $H(y,v)=f(\tau(y))=\phi(v)$. On the other
  hand, if $y\not\in\gamma$ and $u=u(y)=0$, then it need not be the case that
  $H(y,v)=\phi(v)$: we have $H(y,v)=f(\tau(y))$, while $\phi(v)=\phi(1/2) =
  f(a)$.
\end{remark}

\begin{lem}
\label{lem:Psi-surjective} 
  $\Psi\colon D\to U$ is  surjective.
\end{lem}
\begin{proof} 
  Recall (Lemma~\ref{lem:threads-U}) that every element
  of~$U\setminus\{\partial\}$ is either of the form $\bT(\by,s)$ for some
  $\by\in\hS$ and $s\in(0,1)$; or of the form $\bT(\by,s,k)$ for some
  $\by\in\hS$, $s\in[1/2,1)$ and $k\ge0$. In the former case we have
  $\bT(\by,s) = \Psi(\by,s)$; while in the latter case $\bT(\by,s,k) =
  \Psi(\hB^{k+1}(\by), k+1+(2s-1))$ by direct substitution into~(\ref{eq:Psi}),
  since $P(k+1+(2s-1))=(k+1, s)$. 

  Since $\partial=\Psi(\partial')$, this establishes the surjectivity of
  $\Psi$.
\end{proof}

\begin{lem}
\label{lem:conjugation} 
  $\hH\circ\Psi = \Psi\circ G\colon D\to U$.
\end{lem}

\begin{proof} 
  The proof is a straightforward calculation by cases. Let $(\by,s)\in D$. If
  $s=0$ then $(\by, s)=\partial'$, and $\hH(\Psi(\partial'))=\Psi(G(\partial'))
  = \partial$. We therefore assume that~$s>0$.
  \begin{enumerate}[(a)]
    \item If $s\in(0,1/2)$ then $H(y_0,s)=\Upsilon\circ\barf(y_0,s) = (B(y_0),
      2s)$ by~(U1) of Definition~\ref{def:unwrapping-unimodal} and
      Definition~\ref{def:retraction-smash}. Therefore
      \[
        \hH(\Psi(\by,s)) = \hH(\bT(\by,s)) =
        \thr{(B(y_0),2s),(y_0,s),(y_1,s/2),\ldots} = \bT(\hB(\by),2s) = \Psi(G(\by,s)).
      \]
    \item If $s\in[1/2,1)$ then $P(\lambda(s)) = P(2s) = P(1+(2s-1))=(1,s)$, so
      that
      \[
        \Psi(G(\by,s))=\Psi(\hB(\by),\lambda(s)) = \bT(\hB^{-1}(\hB(\by)),s,0)
        =
        \bT(\by,s,0) = \thr{H(y_0,s),(y_0,s),\ldots} = \hH(\Psi(\by,s)).
      \]
      \item If $s\in[1,\infty)$ then $\lambda(s)=s+1$, so that if $P(s)=(t,v)$ then
        $P(\lambda(s))=(t+1,v)$. Therefore
      \[
        \Psi(G(\by,s))=\Psi(\hB(\by), \lambda(s)) = \bT(\hB^{-(t+1)}(\hB(\by)), v, t) =
        \bT(\hB^{-t}(\by), v, t) = \hH(\Psi(\by,s)),
      \] 
      since $H(f^{t-1}(H(y_t,v))) = f^t(H(y_t,v))$.
  \end{enumerate}
\end{proof}

\begin{cor}
\label{cor:Psi-homeo} 
  $\Psi\colon D\to U$ is a homeomorphism.
\end{cor}
\begin{proof} 
  For each~$N\ge 1$, the restriction of $G^{-N}$ to $\hS\times[0,N+1)$ is a
  homeomorphism onto $\hS\times[0,1)$. Lemma~\ref{lem:conjugation} gives
  \begin{equation}
  \label{eq:HGPsi}
    \Psi|_{\hS\times[0,N+1)} = \hH^N\circ \Psi|_{\hS\times[0,1)} \circ
    G^{-N}|_{\hS\times[0,N+1)}.
  \end{equation} 
  Since $\Psi$ is evidently continuous and injective on $\hS\times[0,1)$, it is
  continuous and injective on $\hS\times [0,N+1)$ for each~$N$, and hence
  on~$D$. Therefore (using Lemma~\ref{lem:Psi-surjective}) $\Psi$ is a
  continuous bijection.

  $\Psi^{-1}$ is clearly continuous on $\Psi(\hS\times[0,1))$, and it is
            continuous on $\Psi(\hS\times(0,\infty))$ by invariance of domain.
\end{proof}

Combining Lemma~\ref{lem:conjugation} and Corollary~\ref{cor:Psi-homeo} gives:
\begin{cor}
\label{cor:Psi-conjugates} 
  $\Psi$ is a topological conjugacy between~$G\colon D\to D$ and $\hH\colon
   U\to U$. \qed
\end{cor}

\subsection{Extension to the circle at infinity}
\label{sec:loc-unif-land} We now investigate the extension of $\Psi$ to points
\mbox{$(\by,\infty)\in\hS_\infty$}.

\begin{defns}[The rays $R_\by$, the landing set~$\cL$, the landing function
    $\omega\colon\cL\to\hI$,\,\,$\sD$, $\Psi\colon \sD\to\hT$]
\label{def:landing-etc} 
  For each~$\by\in\hS$, let $R_\by\colon[0,\infty)\to U$ be the ray defined by
  $R_{\by}(s)=\Psi(\by,s)$. Define the {\em landing set} $\cL\subset\hS$ to be
  the set of $\by\in\hS$ for which $R_\by$ lands; and let
  $\omega\colon\cL\to\hI$ denote the {\em landing function}, which takes each
  $\by\in\cL$ to the landing point of $R_\by$. We write $\sD =
  D\cup\cL_\infty\subset\bD$, and extend $\Psi\colon D\to U$ to a function
  $\Psi\colon\sD\to\hT$ by setting $\Psi(\by,\infty)=\omega(\by)$ for each
  $\by\in\cL$.
\end{defns}

The main results of this section are:
\begin{enumerate}[(a)]
\item If all of the entries of the thread~$\by$ after the $(N+1)^\text{st}$ lie
  in $S\setminus\ingam$, then the first~$N+1$ entries of the thread
  $\Psi(\by,s)$ are independent of~$s$, provided that~$s\ge N+1$
  (Lemma~\ref{lem:formula-disjoint-gamma}). In particular
  (Corollary~\ref{cor:ray-lands}), $\by\in\cL$. For this reason we say that an
  element $\by$ of $\hS$ satisfying this condition is {\em landing of
  level~$N$}. We also show (Lemma~\ref{lem:landing-injective}) that the landing
  function~$\omega$ is injective on the set of all points which are landing of
  some level.
\item If all threads sufficiently close to~$\by$ are also landing of level~$N$
  (that is, if $\by$ has a locally uniformly landing neighborhood), then $\Psi$
  is continuous at~$(\by, \infty)$ (see
  Lemma~\ref{lem:Psi-extension-continuous} and
  Corollary~\ref{cor:Psi-maximal-extension}).
\end{enumerate}

\begin{defns}[Landing, $\cL_N$, uniformly landing, locally uniformly landing
    set~$\cR$]
\label{def:landing} 
    Let $N\in\N$. We say that $\by\in\hS$ is {\em landing of
    level~$N$} if $y_i\not\in\ingam$ for all $i>N$; and we write
    $\cL_N\subset\hS$ for the set of such points. (Therefore
    $\cL_0\subset\cL_1\subset\cL_2\subset\cdots$.) We say that a subset~$J$
    of~$\hS$ is {\em uniformly landing (of level~$N$)} if $J\subset\cL_N$. We
    write $\cR$ for the set of elements of~$\hS$ which have a uniformly landing
    neighborhood in~$\hS$.
\end{defns}

\begin{lem}
\label{lem:formula-disjoint-gamma} 
  Let $\by\in\cL_N$, and let $s\ge N+1$. Then,
  writing $P(s) = (t,v)$,
  \[
    \Psi(\by,s) = \thr{ f^N(\tau(y_N)), \ldots, f(\tau(y_N)), \tau(y_N),
      \tau(y_{N+1}),  \ldots, \tau(y_{t-2}),
    \tau(y_{t-1}), (y_t,v), (y_{t+1}, v/2), \ldots }.
  \]
\end{lem}
\begin{proof} 
  We have $t\ge N+1$. Now
  \begin{eqnarray*}
    \Psi(\by, s) &=&
    \thr{ f^{t-1}(H(y_t,v)), \ldots, f(H(y_t,v)), H(y_t,v), (y_t,v), (y_{t+1},
    v/2), \ldots } \\
    &=&
    \thr{ f^{t-1}(\tau(y_{t-1})), \ldots, f(\tau(y_{t-1})), \tau(y_{t-1}), (y_t,v),
    (y_{t+1}, v/2), \ldots }\\
    &=&
    \thr{ f^N(\tau(y_N)), \ldots, f(\tau(y_N)), \tau(y_N), \tau(y_{N+1}),  \ldots,
    \tau(y_{t-2}),
    \tau(y_{t-1}), (y_t,v), (y_{t+1}, v/2), \ldots }
  \end{eqnarray*} as required. Here the first equality
  is~(\ref{eq:convenient}); for the second, we use Remark~\ref{rmk:barf}~(a) to
  give that $\barf(y_t,v)_1 = B(y_t)=y_{t-1}$, since $y_t\not\in\ingam$, so
  that $H(y_t,v) =
  \Upsilon(\barf(y_t,v))=\tau(y_{t-1})$; and for the third, we use that
  $f(\tau(y_i))=\tau(B(y_i))=\tau(y_{i-1})$ for all~$i>N$
  by~(\ref{eq:commute}), since $y_i\not\in\ingam$ for these values of~$i$.
\end{proof}

\begin{cor}
\label{cor:ray-lands} 
  Let $\by\in\cL_N$. Then $\by\in\cL$ and
  \begin{equation}
  \label{eq:omega}
     \omega(\by) = \thr{ f^N(\tau(y_N)), \ldots, f(\tau(y_N)), \tau(y_N),
    \tau(y_{N+1}),
    \tau(y_{N+2}), \ldots }.
  \end{equation} 
  In particular, $\cR\subset\cL$. \qed
\end{cor}

\begin{remark} 
  Therefore $\bigcup_{N\ge 0}\cL_N \subset \cL$. We will see later that these
  two sets are equal, except in the late left endpoint case: see
  Remark~\ref{rmk:unusual-landing}.
\end{remark}

\begin{lem}
\label{lem:landing-injective} 
  Let $\cL' = \bigcup_{N\ge 0}\cL_N$.  Then $\omega\colon \cL'\to\hI$ is
  injective.
\end{lem}
\begin{proof} 
  Let $\by,\bz\in \cL'$ be such that $\omega(\by)=\omega(\bz)$. Pick~$N$ such
  that $\by, \bz \in\cL_N$. Then $\tau(y_{N+r}) =
  \tau(z_{N+r})$ for all $r\ge 0$ by~(\ref{eq:omega}). However, at least one of
  any two successive entries of a thread of $\hI$ must lie in~$[a,\re)$ (as
  $f([\re,b]) = [a,f(a)]$, and $f(a)<\re$ since $\kappa(f)\succ 10\Infs{1}$).
  Since $y_{N+r},z_{N+r}\not\in\ingam$ for all~$r$, it follows that
  $y_{N+r}=z_{N+r}$ for arbitrarily large~$r$, so that $\by=\bz$ as required.
\end{proof}

\begin{lem}
\label{lem:Psi-extension-continuous} 
  Let $J$ be a uniformly landing subset of~$\hS$. Then $\Psi|_{ J\times
  [0,\infty]}$ is continuous.
\end{lem}
\begin{proof} 
  Since $\Psi$ is continuous on~$D$ (Corollary~\ref{cor:Psi-homeo}), it
  suffices to prove continuity at points of $J_\infty$. So let $\by\in J$, and
  let~$N$ be such that~$J\subset\cL_N$. Pick sequences $\by^{(i)}\to\by$ in~$J$
  and $s^{(i)}\to\infty$ in $[0,\infty)$.

  Let $P(s^{(i)}) = (t^{(i)}, v^{(i)})$. Lemma~\ref{lem:formula-disjoint-gamma}
  gives that, for sufficiently large~$i$,
  \[
  \Psi(\by^{(i)}, s^{(i)}) = \thr{ f^N(\tau(y_N^{(i)})), \ldots,
  f(\tau(y_N^{(i)})), \tau(y_N^{(i)}),
  \tau(y_{N+1}^{(i)}), \ldots, \tau(y_{t^{(i)}-2}^{(i)}),
  \tau(y_{t^{(i)}-1}^{(i)}), (y_{t^{(i)}}^{(i)}, v^{(i)}), \ldots },
  \] which converges to $\Psi(\by,\infty)=\omega(\by)$ as
  \mbox{$i\to\infty$}. Similarly, it follows from~(\ref{eq:omega}) that
  $\Psi(\by^{(i)}, \infty) \to \Psi(\by,\infty)$ as $i\to\infty$.
\end{proof}

If~$J$ is uniformly landing but not open in $\hS$, then $\Psi$ (as opposed to
its restriction to $J\times[0,\infty]$) may not be continuous at $(\by,\infty)$
when $\by$ is a boundary point of $J$. However, $\Psi$ is continuous at
interior points of~$J_\infty$, and in particular is continuous at
$(\by,\infty)$ for all $\by$ in the locally uniformly landing set~$\cR$.

The following immediate corollary, which will be used frequently in the
remainder of the paper, states that $\Psi$ extends continuously and injectively
from the disk~$D$ over the locally uniformly landing set at~$\infty$. We
will see later that this is the maximal set over which~$\Psi$ has such an
extension.

\begin{notation}[$\tD$] 
\label{notn:tilde-disk}
  We write $\tD = D \cup \cR_\infty\subset\sD\subset\bD$.
\end{notation}

\begin{cor}
\label{cor:Psi-maximal-extension} 
  $\Psi\colon\tD\to\hT$ is injective and continuous. In particular, its
  restriction to any compact subset of~$\tD$ is a homeomorphism onto its image.
\end{cor}
\begin{proof} 
  $\Psi$ is injective since it is injective on~$D$ and on $\cR_\infty$
  (Corollary~\ref{cor:Psi-homeo} and Lemma~\ref{lem:landing-injective}), and
  $\Psi(\by,s)\in\hI$ if and only if $s=\infty$. It is continuous since it is
  continuous on~$D$ (Corollary~\ref{cor:Psi-homeo}) and on
  $\cR\times[0,\infty]$ (Lemma~\ref{lem:Psi-extension-continuous}).
\end{proof}

So far in this section we have been concerned with the behavior of
$\Psi(\by,s)$ as $s\to\infty$. Our final result is a technical lemma with a
different flavor: it states that if there are several consecutive entries in a
thread~$\by$ which do not lie in~$\ingam$, then one entry (and hence all
earlier entries) of the thread $\Psi(\by,s)$ is constant for $s$ in a
corresponding interval.

\begin{lem}
\label{lem:constant-block} 
  Let $\by\in\hS$, and suppose that $r\ge 1$ and $k\ge 1$ are such that
  $y_{r+i}\not\in\ingam$ for all $1\le i\le k$. Then
  \[
    \Psi(\by,s)_{r-1} = f(\tau(y_r)) \quad\text{ for all \ } s\in[r+u(y_r),\,
    r+k+1].
  \] 
  In particular, if $\by\in\cL_r$, so that $y_{r+i}\not\in\ingam$ for all $i\ge
  1$, then $\Psi(\by,s)_{r-1}=f(\tau(y_r))$ for all $s\ge r+u(y_r)$.
\end{lem}

\begin{proof} 
  Suppose first that $s\in[r+u(y_r), r+1)$, so that $P(s)=(r,v)$ for some
  $v\ge\frac{u(y_r)+1}{2}$.  Then $\Psi(\by,s)_{r-1} = H(y_r, v)$
  by~(\ref{eq:convenient-2}); and $H(y_r,v)=f(\tau(y_r))$ by
  Lemma~\ref{lem:u-key}.

  Next suppose that $s\in[t,t+1)$ for some integer~$t$ with $r+1\le t\le r+k$,
    so that $P(s)=(t,v)$ for some $v\in[1/2,1)$. Then
  $
  \Psi(\by,s)_{r-1} =  f^{t-r}(H(y_t, v)) = f^{t-r}(\tau(B(y_t))) =
  f^{t-r}(\tau(y_{t-1})) = f(\tau(y_r))
  $
  as required. Here the first equality uses~(\ref{eq:convenient-2}), the second
  uses~Remark~\ref{rmk:barf}~(a), and the fourth uses~(\ref{eq:commute})
  applied $t-r-1$ times.

  That $\Psi(\by, r+k+1)_{r-1} = f(\tau(y_r))$ follows from the continuity
  of~$\Psi$.
\end{proof}

\subsection{Good chains of crosscuts}
\label{sec:good-chain}

In this section we establish (Theorem~\ref{thm:corr-homeo}) that there is a
natural homeomorphism between $\hS$  and the circle~$\PP$ of prime ends of
$(\hT, \hI)$, with the property that, for each $\by\in\hS$, the ray $R_{\by}$
converges (in the sense of Section~\ref{sec:prime-ends}) to the prime
end corresponding to~$\by$. Moreover (Lemma~\ref{lem:PP-conj-hS}), this
homeomorphism conjugates the natural extension $\hB\colon\hS\to\hS$ of the
outside map to the action of $\hH$ on~$\PP$, so that the prime end rotation
number of $(\hT, \hI)$ is equal to the Poincar\'e rotation number of $\hB$.

The arguments require two conditions which, while they always hold, we will
only be able to establish, on a case by case basis, later. We therefore treat
them as hypotheses for the time being. The first is that the locally uniformly
landing set~$\cR$ is dense in~$\hS$. The second is that there exist chains of
crosscuts in $(\bD, \hS_\infty)$ whose images under~$\Psi$ are well-behaved
chains of crosscuts in $(\hT, \hI)$, as expressed by
Definition~\ref{def:good-chain} below. We carry over the definitions and
notation of Section~\ref{sec:prime-ends} to the (topologically trivial)
pair~$(\bD, \hS_\infty)$: a crosscut in $(\bD,
\hS_\infty)$ is an arc~$\xi'$ in $\bD$, disjoint from~$\partial'$, which
intersects $\hS_\infty$ exactly at its endpoints; $U(\xi')$ denotes the
component of $D\setminus\xi'$ which doesn't contain~$\partial'$; $\xi_2' <
\xi_1'$ means that $U(\xi_2')\subset U(\xi_1')$; and $(\xi_k')$ is a chain of
crosscuts in $(\bD, \hS_\infty)$ if the $\xi_k'$ are disjoint crosscuts with
$\xi'_{k+1}<\xi'_k$ for each~$k$ and $\diam(\xi_k')\to 0$ as $k\to\infty$.

\begin{remark}
\label{rmk:crosscut-unif-landing-endpoint} 
  If a crosscut~$\xi_k'$ in $(\bD, \hS_\infty)$ has endpoints in $\cR_\infty$
  then, by Corollary~\ref{cor:Psi-maximal-extension}, $\Psi|_{\xi_k'}$ is a
  homeomorphism onto its image~$\xi_k$, which is therefore a crosscut in~$(\hT,
  \hI)$.
\end{remark}

\begin{defn}[Good chain of crosscuts]
\label{def:good-chain} 
Let $\by\in\hS$. A chain~$(\xi'_k)$ of crosscuts in
$(\bD, \hS_\infty)$ is called a {\em good chain for~$\by$} if
\begin{enumerate}[(a)]
  \item The endpoints of each $\xi'_k$ are in~$\cR_\infty$, so that
  $\xi_k=\Psi(\xi_k')$ is a crosscut in~$(\hT, \hI)$ by
  Remark~\ref{rmk:crosscut-unif-landing-endpoint};
  \item $(\by,\infty)\in \overline{U(\xi_k')}$ for each~$k$, so in particular
    $\xi_k'\to(\by,\infty)$ as $k\to\infty$;
  \item $\diam(\xi_k)\to 0$ as $k\to\infty$; and
  \item if $\by\not\in\cL$, then $(\xi_k)$ does not converge to a point~$\bx$
    of $\hI$.
\end{enumerate}
\end{defn}

\begin{remarks}\mbox{}
\label{rmk:good-chain}
\begin{enumerate}[(a)]
\item By Definition~\ref{def:good-chain}~(a) and~(c), if $(\xi'_k)$ is a good
  chain of crosscuts for~$\by$, then $(\xi_k)$ is a chain of crosscuts in
  $(\hT, \hI)$.
\item Suppose that there is a good chain of crosscuts $(\xi'_k)$ for
$\by\in\hS$.
\begin{itemize}
\item If $\by\in\cL$ then, since $R_{\by}$ lands at $\omega(\by)$ and
  intersects every~$\xi_k$, we have $\xi_k\to\omega(\by)$ as $k\to\infty$. It
  follows that for every ray $\sigma'\colon[0,\infty)\to D$ which lands
  at~$(\by,\infty)$, the ray $\sigma=\Psi\circ\sigma'$ either lands
  at~$\omega(\by)$, or does not land.
\item If $\by\not\in\cL$, then for every such ray $\sigma'$, the ray $\sigma =
  \Psi\circ\sigma'$ intersects $\xi_k$ for all sufficiently large~$k$, and
  therefore does not land, by condition~(d) of the definition.
\end{itemize}

\item By Corollary~\ref{cor:Psi-maximal-extension}, there is a good chain of
crosscuts for every~$\by\in\cR$.

\item We will continue to use the notational convention introduced above:
functions $f'\colon X\to\tD$ and subsets $Y'\subset\tD$ will be denoted with
primed symbols, and the corresponding functions
\mbox{$f=\Psi\circ f'\colon X\to \hT$} and subsets $Y = \Psi(Y')\subset\hT$
with the corresponding unprimed symbols.
\end{enumerate}
\end{remarks}

\begin{lem}
\label{lem:crosscut-preimage-lands} Suppose that $\cR$ is dense in $\hS$, and
let $\sigma\colon[0,\infty)\to U$ be a ray which lands at a point of~$\hI$.
Then the ray $\sigma' = \Psi^{-1}\circ\sigma$ lands at a point of $\hS_\infty$.
\end{lem}
\begin{proof} The remainder $\Rem(\sigma')$ is a connected subset of
$\hS_\infty$, so if $\sigma'$ didn't land then, since~$\cR$ is open and dense
in~$\hS$, there would be a non-trivial closed subinterval~$J$ of $\cR$ with
$J_\infty\subset\Rem(\sigma')$. This would contradict the fact that $\sigma$
lands, since $\Psi\vert_{J\times[0,\infty]}$ is a homeomorphism onto its image
by Corollary~\ref{cor:Psi-maximal-extension}.
\end{proof}

\begin{cor}
\label{cor:crosscuts-pull-back} 
  Suppose that~$\cR$ is dense in~$\hS$. If $\xi$ is a crosscut in $(\hT, \hI)$, then $\xi' = \Psi^{-1}(\xi)$ is a crosscut in $(\bD, \hS_\infty)$. Moreover, if $\xi_2<\xi_1$ are crosscuts in $(\hT, \hI)$, then $\xi_2' < \xi_1'$
\end{cor}

\begin{proof}
  Immediate from Lemma~\ref{lem:crosscut-preimage-lands} and the fact that
  $U(\xi') = \Psi^{-1}(U(\xi))$.
\end{proof}

We now associate a point of $\hS$ with each prime end~$\cP$ in $(\hT,\hI)$,
under the assumption that~$\cR$ is dense in~$\hS$. Suppose that~$\cP$ is
represented by a chain~$(\xi_k)$, and write $U_k = U(\xi_k)$. Then each $\xi_k'
=
\Psi^{-1}(\xi_k)$ is a crosscut in $(\bD, \hS_\infty)$, and $U_k':=U(\xi_k') =
\Psi^{-1}(U_k)$.

Let $J_k'=\overline{U_k'}\cap\hS_\infty$, a compact arc with endpoints the
endpoints of $\xi_k'$. Then $\bigcap_{k\ge 0}J_k'$ is a single point. For if
not then, since $\cR$ is open and dense in~$\hS$, the intersection would
contain $K_\infty$ for some \mbox{$K=[\by_1,\by_2]\subset\cR$}, and every
$\xi_k$ would intersect both $\Psi(\{\by_1\}\times[0,\infty])$ and
$\Psi(\{\by_2\}\times[0,\infty])$, contradicting $\diam(\xi_k)\to0$ as
$\Psi|_{K\times[0,\infty]}$ is a homeomorphism onto its image by
Corollary~\ref{cor:Psi-maximal-extension}.

Since the point of $\bigcap_{k\ge 0}J_k'$ is independent of the choice of chain
representing~$\cP$, we can make the following definition:

\begin{defn}[$\by\colon\PP\to\hS$]
\label{def:y-corr-prime-end} Suppose that~$\cR$ is dense in~$\hS$. Let~$\cP$ be
a prime end of $(\hT, \hI)$. We write $\by(\cP)$ for the element of~$\hS$
defined by
\[\bigcap_{k\ge0}\left(\overline{\Psi^{-1}(U(\xi_k))} \cap
\hS_\infty\right) = \{(\by(\cP), \infty)\},\] where $(\xi_k)$ is a chain
representing~$\cP$.
\end{defn}

\begin{lem}
\label{lem:y-corr-prime-end-cts} Suppose that~$\cR$ is dense in~$\hS$. Then
$\by\colon\PP\to\hS$ is continuous.
\end{lem}
\begin{proof} Let~$J$ be an open subset of~$\hS$, and let~$\cP\in\by^{-1}(J)$
be represented by a chain~$(\xi_k)$. Then there is some~$k$ such that
$\overline{\Psi^{-1}(U(\xi_k))} \cap \hS_\infty \subset J$, and we have
$\cP\in\cB(\xi_k) \subset \by^{-1}(J)$, where $\cB(\xi_k)$ is the basic open
subset defined in Section~\ref{sec:prime-ends}.
\end{proof}

\begin{thm}
\label{thm:corr-homeo} Suppose that $\cR$ is dense in~$\hS$, and that there is
a good chain of crosscuts for every $\by\in\hS$. Then
\begin{enumerate}[(a)]
\item $\by\colon\PP\to\hS$ is a homeomorphism;
\item For each $\by\in\hS$, the unique prime end~$\cP$ with $\by(\cP)=\by$ is
  defined by the chain~$(\Psi(\xi_k'))$, where $(\xi_k')$ is any good chain of
  crosscuts for~$\by$: or, indeed, any chain of crosscuts which
  satisfies~(a)\,--\,(c) of Definition~\ref{def:good-chain}.
\item For each $\by\in\hS$, the ray $R_\by$ converges to the unique prime
  end~$\cP$ with $\by(\cP)=\by$; and
\item the set of accessible points of $\hI$ is $\{\omega(\by)\,:\,\by\in\cL\}$.
\end{enumerate}
\end{thm}

\begin{proof}
\begin{enumerate}[(a)]
  \item Let $\by\in\hS$, and let $(\xi_k')$ be a good chain of crosscuts
  for~$\by$. Write $U_k'=U(\xi_k')$, $\xi_k=\Psi(\xi_k')$, and $U_k = U(\xi_k)=
  \Psi(U_k')$. By Remark~\ref{rmk:good-chain}~(a), $(\xi_k)$ is a chain of
  crosscuts in $(\hT,\hI)$, which therefore represents a prime end~$\cP\in\PP$.
  By condition~(b) of Definition~\ref{def:good-chain} we have $\by(\cP)=\by$.
  In particular, $\by\colon\PP\to\hS$ is surjective.

  To show injectivity, suppose that~$(\Xi_k)$ is another chain of crosscuts in
  $(\hT, \hI)$ which defines a prime end $\cQ\in\PP$ with
  $\by(\cQ)=\by(\cP)=\by$. Write $V_k = U(\Xi_k)$, and $V_k'=\Psi^{-1}(V_k)$.
  By Corollary~\ref{cor:crosscuts-pull-back}, each $\Xi'_k = \Psi^{-1}(\Xi_k)$
  is a crosscut in $(\bD, \hS_\infty)$. In order to show that $\cQ=\cP$, we
  need to show that each $V_{\ell}$ contains all but finitely many $U_k$, and
  each $U_{\ell}$ contains all but finitely many $V_k$.

  Now for each~$\ell$, since $\by(\cQ)=\by$, we have that
  $\overline{V_{\ell}'}$ contains an arc~$J_{\ell}'$ in $\hS_\infty$ with
  $(\by,\infty)\in J_{\ell}'$. Moreover, since the $\Xi_k$ are mutually
  disjoint, so are the $\Xi_k'$ by Remark~\ref{rmk:good-chain}~(b) (this is
  where we use condition~(d) of Definition~\ref{def:good-chain}). Therefore
  $(\by,\infty)$ cannot be an endpoint of more than one of the crosscuts
  $\Xi'_k$, and hence is in the interior of $J_{\ell}'$.  Since
  $\xi_k'\to(\by,\infty)$, it follows that $U_k'\subset V_{\ell}'$ --- and
  hence $U_k\subset V_{\ell}$ --- for all sufficiently large~$k$.

  To show that each~$U_{\ell}$ contains all but finitely many $V_k$, let
  $\xi'<\xi_{\ell}'$ be a crosscut disjoint from~$\xi_{\ell}'$ whose endpoints
  are in the same components of~$\cR_\infty$ as the endpoints of $\xi_{\ell}'$,
  and which satisfies $(\by,\infty)\in \overline{U(\xi')}$. Let $X$ be the
  compact subset of $\bD$ bounded by $\xi_{\ell}'$ and $\xi'$. Since $\Psi|_X$
  is a homeomorphism onto its image, arcs which intersect both $\Psi(U(\xi'))$
  and the complement of $\Psi(U_{\ell}')$ have diameter bounded below. Now
  $\Xi_k$ intersects $\Psi(U(\xi'))$ for all sufficiently large~$k$ (since
  $\by(\cQ)=\by$), and $\diam(\Xi_k)\to 0$, so that $\Int(\Xi_k)
  \subset \Psi(U_{\ell}')=U_{\ell}$ --- and hence $V_k\subset U_{\ell}$ --- for
  all sufficiently large~$k$ as required.

  \item Follows immediately from the first paragraph of the proof of~(a), which
    doesn't make use of condition~(d) of Definition~\ref{def:good-chain}.

  \item For each~$k$ there is some~$t$ with $\{\by\}\times[t,\infty)
    \subset U(\xi_k')$, and therefore
  \[ R_{\by}([t,\infty)) = \Psi(\{\by\}\times[t,\infty)) \subset
      \Psi(U(\xi_k')) = U(\Psi(\xi_k')),
  \] so that $R_{\by}$ converges to the prime end defined by the
  chain~$(\Psi(\xi_k'))$ as required.

  \item Clearly $\omega(\by)$ is accessible for all $\by\in\cL$, since it is the
  landing point of the ray $R_\by$.

  Let $\bx$ be an accessible point of $\hI$, so that there is a ray
  $\sigma\colon[0,\infty)\to U$ which lands at~$\bx$. By
  Lemma~\ref{lem:crosscut-preimage-lands}, the ray $\sigma' =
    \Psi^{-1}\circ \sigma$ lands at some $(\by,\infty)\in
    \hS_\infty$. By Remark~\ref{rmk:good-chain}~(b), $\by\in\cL$ and $\bx =
    \omega(\by)$.
\end{enumerate}

\end{proof}

\begin{notation}[$\cP\colon\hS\to\PP$] Suppose that $\cR$ is dense in~$\hS$,
and that there is a good chain of crosscuts for every $\by\in\hS$. Then we
write $\cP\colon\hS\to\PP$ for the inverse of the homeomorphism
$\by\colon\PP\to\hS$.
\end{notation}

\begin{lem}
\label{lem:PP-conj-hS} Suppose that $\cR$ is dense in~$\hS$, and that there is
a good chain of crosscuts for every $\by\in\hS$. Then $\cP\colon\hS\to\PP$
conjugates $\hB\colon\hS\to\hS$ to $\hH\colon
\PP\to\PP$. In particular, the prime end rotation number of
$\hH\colon(\hT,\hI)\to(\hT,\hI)$ is equal to $\rho(\hB)$.
\end{lem}
\begin{proof} Let $\by\in\hS$. By Theorem~\ref{thm:corr-homeo}~(c), the ray
$R_\by$ converges to $\cP(\by)$, and hence $\hH\circ R_\by$ converges to
$\hH(\cP(\by))$. By Lemma~\ref{lem:conjugation}, $\hH\circ R_{\by}(s) =
R_{\hB(\by)}(\lambda(s))$, so that the image of $\hH\circ R_\by$ coincides with
the image of $R_{\hB(\by)}$, which converges to $\cP(\hB(\by))$ by
Theorem~\ref{thm:corr-homeo}~(c). Therefore $\hH(\cP(\by))=\cP(\hB(\by))$ as
required.
\end{proof}

We will see later (Corollary~\ref{cor:rot-no}) that
$\rho(\hB)=\rho(B)=q(\kappa(f))$. The following lemma summarizes those parts of
the results above which are relevant to the classification of prime ends, for
future reference.

\begin{lem}
\label{lem:general-summary} Suppose that $\cR$ is dense in~$\hS$, and that
there is a good chain of crosscuts for every $\by\in\hS$. Then
\begin{enumerate}[(a)]
\item If $\by\in\cL$ then $\Pi(\cP(\by)) = \{\omega(\by)\}$.
\item If $\by\in\cR$ then $\cI(\cP(\by)) = \{\omega(\by)\}$.
\end{enumerate} In particular, a prime end~$\cP\in\PP$ is of the first kind if
$\by(\cP)\in \cR$; and is of the first or second kind if $\by(\cP)\in\cL$.
\end{lem}
\begin{proof} (a) follows from the fact that $R_\by$ converges to $\cP(\by)$
and lands at $\omega(\by)$ (see Section~\ref{sec:prime-ends}). (b)~is
immediate from the homeomorphism established in
Corollary~\ref{cor:Psi-maximal-extension}.
\end{proof}

\subsection{Dynamics of the outside map}
\label{sec:outside-dynamics} In order to determine the prime ends of $(\hT,
\hI)$, it suffices, in view of the homeomorphism between $\PP$ and $\hS$
(Theorem~\ref{thm:corr-homeo}) and the triviality of prime ends $\cP(\by)$ with
$\by\in\cR$ (Lemma~\ref{lem:general-summary}), to prove that $\cR$ is dense in
$\hS$ and that there is a good chain of crosscuts for every $\by\in\hS$; and
then to analyze the prime ends which the rays $R_{\by}$ converge to in the
cases when $\by\not\in\cR$. The arguments and conclusions are quite different
depending on whether $f$ is of rational or irrational type, and we will
consider these cases separately.

In this section we state and prove the main result which will be needed about
the dynamics of the outside map $B\colon S\to S$. Because the locally uniformly
landing set~$\cR$ of Definitions~\ref{def:landing} depends on occurrences of
elements of~$\ingam$ in the threads $\by\in\hS$, it is primarily necessary to
understand the recurrence properties of $\gamma$. Since $B$ collapses~$\gamma$
to the single point~$B(a)$, the main question is: when does the orbit of~$B(a)$
first enter~$\gamma$? We will see that if~$f$ is of rational type with
$q(\kappa(f))=m/n$, then~$n$ is the smallest positive integer with
$B^n(a)\in\gamma$, except when~$f$ is of early left endpoint type; while if $f$
is of irrational type, or of early left endpoint type, then the orbit of $B(a)$
is disjoint from~$\gamma$.

\begin{notation}[$N(f)$] Let $f\colon[a,b]\to [a,b]$ be a unimodal map, and
$B\colon S\to S$ be the corresponding outside map. We define
$N(f)\in\N\cup\{\infty\}$ by $N(f)=\infty$ if $B^r(a)\not\in\gamma$ for all
$r\ge 1$, and otherwise
\[ N(f) = \min\{r\ge 1\,:\,B^r(a)\in\gamma\}.
\]

\end{notation}

Theorem~\ref{thm:outside-dynamics} below is an extension (both to more general
hypotheses and to stronger conclusions) of a result of~\cite{gpa}. Because of
the central role which this theorem plays in the paper, we prove it in full,
although we do rely on some technical lemmas from~\cite{gpa}.

Before stating the theorem, we remark that the outside map $B\colon S\to S$ is
a monotone degree~1 circle map, and therefore has a Poincar\'e rotation number
$\rho(B)$. Recall that we denote by~$\re$ the unique element of~$(c,b]$ with
$f(\re)=f(a)$. The reader is encouraged to review the notation and results of
Section~\ref{sec:height} before proceeding.

\begin{thm}[Dynamics of the outside map]
\label{thm:outside-dynamics} Let $f\colon[a,b]\to[a,b]$ be a unimodal map with
kneading sequence $\kappa(f)=\syma$, and let $B\colon S\to S$ be the
corresponding outside map. Then
\begin{enumerate}[(a)]
\item $\rho(B) = q(\syma)$.
\item If $q(\syma)=m/n$ is rational and~$f$ is not of early left endpoint type, then
\begin{enumerate}[(i)]
\item $N(f)=n$;
\item $B^n(a)=a \iff \syma = \lhe(m/n)$ and $B^n(a)=\re_u \iff \syma =
  \rhe(m/n)$; and
\item The set $S\setminus\bigcup_{r\ge0}B^{-r}(\gamma)$ of points whose orbits
  never fall into~$\gamma$ is:
  \begin{itemize}
     \item empty if $f$ is of normal endpoint type;
     \item the union of~$n$ half-open intervals, with open endpoint at a
     point of the orbit of~$B(a)$ and closed endpoint at a point of a second
     period~$n$ orbit of~$B$, if $f$ is of quadratic-like strict left endpoint
     type; and
     \item a single~period~$n$ orbit of~$B$ otherwise.
   \end{itemize} 
\end{enumerate}
\item If $q(\syma)=m/n$ is rational and~$f$ is of early left endpoint type,
then
\begin{enumerate}[(i)]
\item $N(f)=\infty$; and
\item $B$ has a period~$n$ orbit~$Q$ disjoint from~$\gamma$ which attracts the
  orbit of~$B(a)$.
\end{enumerate}
\item If $q(\syma)$ is irrational, then
\begin{enumerate}[(i)]
\item $N(f)=\infty$;
\item The set $\bigcup_{r\ge 0}B^{-r}(\gamma)$ of points whose orbits fall
  into~$\gamma$ is dense in~$S$; and
\item The orbit $\{B^r(a)\,:\, r\ge 1\}$ of $B(a)$ is dense in
  $S\setminus\bigcup_{r\ge 0}B^{-r}(\gamma)$.
\end{enumerate}
\end{enumerate}
\end{thm}

We will use two lemmas. The first, Lemma~\ref{lem:gamma-observation} below,
provides tools for determining $N(f)$ and the rotation number~$\rho(B)$.
Although the lemma is straightforward, its statement may be hard to parse, and
we start with an informal description. For $r\le N(f)$ we have that
$\tau(B^r(a))=f^r(a)$ by~\eqref{eq:commute}. In order to determine whether or
not $B^r(a)\in\gamma$, we need to decide whether $B^r(a)$ is equal to
$f^r(a)_u$ or to $f^r(a)_\ell$; and, in the former case, whether or not
$f^r(a)\le \re$. The set $\cJ$ defined in the statement of the lemma has the
property that, for $1\le r\le N(f)$, $B^r(a)=f^r(a)_u$ if and only if
$r\in\cJ$. Since $\iota(f^r(a))=\sigma^{r+1}(\kappa(f))$ and
$\iota(\re)=1\sigma^2(\kappa(f))$, the smallest~$r$ with $B^r(a)\in\gamma$ is
equal to the smallest~$r$ for which $r\in\cJ$ and
$\sigma^{r+1}(\kappa(f))\preceq 1\sigma^2(\kappa(f))$: this is the content of
parts~(a) and~(b). We will see that $\rho(B)$ depends on how many points of the
orbit of~$B(a)$ lie in the upper half of the circle, and part~(c) of the lemma
enables us to calculate this. Finally, part~(d) extends the ideas of~(a)
and~(b) to give conditions under which there is a periodic orbit of~$B$,
disjoint from~$\gamma$, above a periodic orbit of~$f$.

\begin{lem}
\label{lem:gamma-observation} 
  Let $f\colon[a,b]\to[a,b]$ be a unimodal map with kneading sequence
  $\kappa(f)=\syma$, and let $B\colon S\to S$ be the corresponding outside map.
  Write
  \begin{equation}
  \label{eq:cJ}
    \cJ = \{r\in\N\,:\,\text{there is some $0\le k\le
    (r-1)/2$ such that }\sigma^{r-(2k+1)}(\syma) =
    01^{2k+1}\sigma^{r+1}(\syma)\}.
  \end{equation}
  \begin{enumerate}[(a)]
    \item Suppose that $\sigma^{r+1}(\syma)\succ 1\sigma^2(\syma)$ for all
      $r\in\cJ$. Then $N(f)=\infty$, provided that~$c$ is not a periodic point
      of~$f$.

    \item Otherwise, let~$r$ be least such that $r\in\cJ$
      and $\sigma^{r+1}(\syma)\preceq 1\sigma^2(\syma)$. Then $N(f)=r$,
      provided that $f^i(a)\not=c$ for $1\le i<r$.

    \item For each $N\le N(f)$ we have
    \[
    \#\,\{r\le N\,:\,B^r(a) = f^r(a)_u\} = \#\,\{r\le N\,:\,r\in\cJ\},
    \] provided that $f^i(a)\not=c$ for $1\le i<N(f)$.

    \item Suppose that~$f$ has a period~$N$ point~$x$ whose orbit does not
      contain~$c$; and that $\iota(x)=\Infs{W}$, where $W=10\,V\,01^{2j+1}$ for
      some~$j\ge 0$ and some word~$V$ of length $N-2j-4$. Suppose, moreover,
      that whenever $\sigma^i(\Infs{W}) = 01^{2k+1}\symb$ for some~$k\ge 0$ and
      $\symb\in\{0,1\}^\N$, we have $\symb\succ 1\sigma^2(\syma)$. Then~$x_u$
      is a period~$N$ point of~$B$ whose orbit is disjoint from~$\gamma$.
  \end{enumerate}
\end{lem}

\begin{proof} 
  By~(\ref{eq:commute}) we have $\tau(B^r(a))=f^r(a)$ for $r\le N(f)$, so that
  $B^r(a)$ is either $f^r(a)_\ell$ or $f^r(a)_u$ when $r\le N(f)$. By the
  definition~(\ref{eq:outside-map}) of the outside map we have that, for $1\le
  r\le N(f)$,
  \[ B^r(a)=f^r(a)_u \iff B^{r-1}(a)=f^{r-1}(a)_\ell \text{ and }f^{r-1}(a)\ge c.
  \]

  Provided that $f^{r-1}(a)\not=c$ for $r\le N(f)$ (so that there is no
  ambiguity in the corresponding entries of~$\syma$) it follows that, for $r\le
  N(f)$, we have $B^r(a)=f^r(a)_u$ if and only if there is some~$k\ge 0$ with
  $f^{r-2k-2}(a)<c$ and $f^j(a)>c$ for $r-2k-1\le j < r$ (there is an odd
  number of $1$s in $\syma$ preceding the entry corresponding to $f^r(a)$).
  This in turn is equivalent to the existence of $k\ge 0$ such that
  $\sigma^{r-(2k+1)}(\syma)=01^{2k+1}\sigma^{r+1}(\syma)$. By definition of
  $\cJ$ we therefore have, under the assumption that $f^{r-1}(a)\not=c$ for
  $1\le r\le N(f)$,
  \begin{equation}
    \label{eq:in-upper} 
    B^r(a)=f^r(a)_u \iff r\in\cJ \qquad (1\le r\le N(f)).
  \end{equation}

  \begin{enumerate}[(a)]
    \item If~$c$ is not a periodic point of~$f$ then $f^r(a)\not=c$ for all
      $r\ge 0$. Since $\sigma^{r+1}(\syma)\succ 1\sigma^2(\syma)=\iota(\re)$
      whenever $r\in\cJ$ we have $f^r(a)>\re$ whenever
      $B^r(a)=f^r(a)_u$ (note that~$\re$ has a unique itinerary since
      $f^r(\re)=f^r(a)\not=c$ for all $r\ge 1$). Therefore
      $B^r(a)\not\in\gamma$ for all $r\ge 1$, i.e.\ $N(f)=\infty$ as required.
    \item Let $r$ be least such that $r\in\cJ$ and
      $\sigma^{r+1}(\syma)\preceq 1\sigma^2(\syma)$, and suppose that
      $f^i(a)\not=c$ for $1\le i<r$. As in~(a), we have $B^i(a)\not\in\gamma$
      for $1\le i<r$. On the other hand, $B^r(a)=f^r(a)_u$ and
      $\iota(f^r(a))=\sigma^{r+1}(\syma)\preceq 1\sigma^2(\syma) = 
      \iota(\re)$. Therefore $f^r(a)\le \re$ (in the borderline case
      $\iota(f^r(a))=\sigma^{r+1}(\syma) = 1\sigma^2(\syma)$ we have
      $\syma=10\Inf{\syma_2\syma_3\ldots \syma_r1}$, which is not periodic, so
      that $f^r(a)=\re$ by Convention~\ref{conv:unimodal}~(b)). Hence
      $B^r(a)\in\gamma$, and $N(f)=r$ as required.
    \item Immediate from~(\ref{eq:in-upper}).
    \item The proof is similar to that of~(a) and~(b): the condition that
      $\symb\succ 1\sigma^2(\syma)$ whenever
      $\sigma^i(\Infs{W})=01^{2k+1}\symb$ ensures that every point of the orbit
      of $x_u$ which lies on the upper half of~$S$ is not in~$\gamma$.
  \end{enumerate}

\end{proof}

It is clear from Lemma~\ref{lem:gamma-observation} that a key question is how
certain sequences compare with $1\sigma^2(\syma)$ in the unimodal order. The
next lemma, which contains and extends results of~\cite{gpa}, addresses this
and related issues.

\begin{lem}\mbox{}
\label{lem:outside-lemma}
\begin{enumerate}[(a)]
\item Let $q=m/n\in\Q\cap(0,1/2)$. For each integer~$j$ with $1\le j\le m$, the
  word
\[ 10^{\kappa_j(q)}110^{\kappa_{j+1}(q)}11\ldots 110^{\kappa_m(q)}1
\] disagrees with the word
\[ 10^{\kappa_1(q)-1}110^{\kappa_2(q)}11\ldots 110^{\kappa_m(q)}1
\] within the shorter of their lengths, and is greater than it in the unimodal
order.
\item  Let $q=m/n\in\Q\cap(0,1/2)$ and $\syma\in\KS(q)$. If $\syma=c_qd$ for
  some $d\in\{0,1\}^\N$, then $d\preceq 1\sigma^2(\syma)$.
\item Let $q=m/n\in\Q\cap(0,1/2)$ and $\syma\in\KS(q)\setminus\{
  \rhe(q)\}$. Let $\symb\in\{0,1\}^\N$ be on the $\sigma$-orbit of $\Inf{w_q1}$
  and of the form $\symb=1^k0\ldots$ with~$k$ odd. Then $\symb\succ
  1\sigma^2(\syma)$.
\item Let $q\in(0,1/2)$ be irrational. Then for each integer $j\ge 1$ we have
\[ 10^{\kappa_j(q)}110^{\kappa_{j+1}(q)}11\ldots \,\,\succ\,\,
10^{\kappa_1(q)-1}11 0^{\kappa_2(q)} 11 \ldots.
\]
\item Let $q\in(0,1/2)$ be irrational. Then for every $N\ge 1$ there is an
  $r\ge 1$ such that $\kappa_{r+i}(q)=\kappa_{i}(q)$ for $1\le i\le N$; and
  there is an $s\ge 1$ such that $\kappa_{s+1}(q) =
  \kappa_1(q)-1$, and $\kappa_{s+i}(q) = \kappa_i(q)$ for $2\le i\le N$.
\end{enumerate}
\end{lem}
\begin{proof} Statements~(a) and~(b) are lemmas~7 and 8 of~\cite{gpa}.
Statement~(c) is closely related to lemma~9 of~\cite{gpa}, whose hypotheses
allow~$\syma$ to be~$\rhe(q)$, and whose conclusion is that $\symb\succeq
1\sigma^2(\syma)$. It is easily shown that $\symb=1\sigma^2(\syma)$ is only
possible when $\syma=\rhe(q)$.  (The statements of lemmas~8 and~9 in~\cite{gpa}
have an additional hypothesis relevant to that paper, but this hypothesis is
not used in their proofs.)

To prove~(d), observe that:
\begin{enumerate}[(i)]
\item It is impossible to have $10^{\kappa_j(q)}110^{\kappa_{j+1}(q)}11\ldots =
10^{\kappa_1(q)-1}11 0^{\kappa_2(q)} 11 \ldots$, since then the sequence
$(\kappa_i(q))$ would be eventually periodic, and
$
\lim_{r\to\infty}\frac{\sum_{i=1}^r (\kappa_i(q)+2)}{r} $ would be rational:
but this limit is equal to $1/q$ by~(\ref{eq:kappa-i}).
\item It is impossible to have $10^{\kappa_j(q)}110^{\kappa_{j+1}(q)}11\ldots
 \prec 10^{\kappa_1(q)-1}11 0^{\kappa_2(q)} 11 \ldots$, since then there would
 be some~$M$ such that
\[ 10^{\kappa_j(q)}110^{\kappa_{j+1}(q)}11\ldots 110^{\kappa_{M}(q)}1
\prec 10^{\kappa_1(q)-1}110^{\kappa_2(q)}11 \ldots 11 0^{\kappa_{M}(q)}1.
\] Taking a rational approximation $m/n$ to $q$ with $m\ge M$ and
$\kappa_i(m/n)=\kappa_i(q)$ for $i\le M$ would give a contradiction to~(a).
\end{enumerate}

For~(e), recall (Definition~\ref{def:cq}) that the $\kappa_i(q)$ are defined by
intersections of a straight line~$L_q$ of slope~$q$ with lines of the
coordinate grid.  Since~$L_q$ passes arbitrarily close to integer lattice
points below the lattice point, any initial segment of the sequence
$(\kappa_i(q))$ occurs infinitely often in the sequence; and since it passes
arbitrarily close to lattice points above the lattice point, the same is true
of the sequence in which $\kappa_1(q)$ is replaced by $\kappa_1(q)-1$.
\end{proof}

\begin{proof}[Proof of Theorem~\ref{thm:outside-dynamics}] Recall
(Lemma~\ref{lem:height-intervals}~(b)) that $q(\syma)=0$ if and only if
$\syma=1\Infs{0}$, and that then~$f$ is of tent-like strict left endpoint type
by Definition~\ref{def:type}, and $\mu=\lhe(0)$ by
Definitions~\ref{def:lhe,rhe,etc}. In this case, by
Convention~\ref{conv:unimodal}~(b) and the fact that~$\syma$ is not periodic,
we have $B(a)=a$, and statements~(a) and~(b) are immediate, using
$\gamma=[b,a]$. We therefore assume in the remainder of the proof that
$q(\syma)\in(0,1/2)$.

Assume first that $q=q(\syma)=m/n$ is rational and~$f$ is not of early endpoint
type. We will suppose for the proof of~(b)(i) that $\syma\not=\lhe(q)$, so that
$\syma=c_{q}d$ for some $d\in\{0,1\}^\N$ by
Lemma~\ref{lem:height-intervals}~(c): a similar argument applies when
$\syma=\lhe(q)$ (noting that in this case we have $f^n(a)=a\in\gamma$, since
$f$ is not of early endpoint type, so that it is only necessary to show that
$B^r(a)\not\in\gamma$ for $1\le r<n$). In particular, if~$c$ is periodic then
it has period at least~$n+2$ by Lemma~\ref{lem:height-intervals}~(c) (if
$\syma=\Inf{w_q0}$ then~$c$ is not a period~$n$ point by
Definition~\ref{def:kneading-sequence}). Therefore $f^r(a)\not=c$ for $r<n$.

Recall that $c_{q} = 10^{\kappa_1(q)}110^{\kappa_2(q)}11\ldots
110^{\kappa_m(q)}1$ is a word of length~$n+1$. Defining~$\cJ$
by~(\ref{eq:cJ}), the values of $r\le n$ with $r\in\cJ$ are
\[r_i = (2i-1)+\sum_{j=1}^i \kappa_j, \qquad(1\le i\le m),\] and the
corresponding itineraries $\symb_i = \sigma^{r_i+1}(\mu)$ are
\[
\symb_i = 10^{\kappa_{i+1}(q)}110^{\kappa_{r+2}(q)}11\ldots 11
0^{\kappa_m(q)}1d \quad (1\le i\le m-1), \quad \text{ and } \quad \symb_m=d.
\] Observe that this statement is true whether or not all of the $\kappa_i(q)$
are positive: if~$\kappa_i(q)>0$, then
$\sigma^{r_i-(2k+1)}(\syma)=01^{2k+1}\symb_i$ with $k=0$, while if
$\kappa_i(q)=0$ then this equality holds for some $k>0$.

Now Lemma~\ref{lem:outside-lemma}~(a) gives $\symb_i \succ 1\sigma^2(\syma)$
for $1\le i<m$, while Lemma~\ref{lem:outside-lemma}~(b) gives $\symb_m\preceq
1\sigma^2(\syma)$. Since $r_m=n$, statement~(b)(i) follows from
Lemma~\ref{lem:gamma-observation}~(b).

\medskip

Since $B(\gamma)=B(a)$, it follows that $B(a)$ is a period~$n$ point of~$B$.
Therefore $\rho(B)$ is the rotation number of this periodic point, which we now
determine.

Let $\pi\colon \R\to S$ be a universal covering with fundamental
domain~$F=[0,1)$ and covering transformation group $\{x\mapsto
x+n\,:\,n\in\Z\}$ such that \mbox{$\pi(0)=\pi(1)=a$}, $\pi(1/2)=b$, and
$\pi(x)$ is in the lower half of~$S$ for $x\in[0,1/2]$. Let $\tB\colon\R\to\R$
be the lift of~$B$ with $\tB(0)\in F$. It follows from~(\ref{eq:outside-map})
that $\tB(x)\in F$ for $x\in[0, 1/2)$, while $\tB(x)\in F+1$ for $x\in[1/2,1)$.

Now there are exactly~$m$ points on the periodic orbit containing $B(a)$ which
lie in $\pi([1/2,1))$ by Lemma~\ref{lem:gamma-observation}~(c).  Therefore
$\rho(B)=m/n$, establishing~(a) in the rational non-early endpoint case.

\medskip

For (b)(ii), observe first that since $B^r(a)\not\in\ingam$ for $0\le r < n$,
it follows from~(\ref{eq:commute}) that $\tau(B^n(a))=f^n(\tau(a))=f^n(a)$.
Therefore $B^n(a)=a \iff
\tau(B^n(a))=a \iff f^n(a)=a$, and similarly $B^n(a)=\re_u \iff
\tau(B^n(a))=\re \iff f^n(a)=\re$ (where, for the first equivalence, we use
that $B^n(a)\in\gamma$).

Now if $\syma = \lhe(m/n)$ then, since~$f$ is not of early endpoint type, we
have $f^n(a)=a$. Conversely, if $f^n(a)=a$ then $f^{n-1}(a)=b$ (since
$\syma\not=1\Infs{0}$, so that~$a$ has only one preimage), and hence $f^n(b)=b$
and $f^n(c)=c$. Therefore~$\syma$ is a periodic kneading sequence of period~$n$
and height~$m/n$, and so is equal either to $\lhe(m/n)$ or to $\Inf{w_{m/n}0}$
by Lemma~\ref{lem:height-intervals}~(c). However, since~$c$ itself is periodic,
the latter case is impossible (Definition~\ref{def:kneading-sequence}).

If $\syma=\rhe(m/n)=10\Inf{\hw_{m/n}1}$ then
$\iota(f^n(a))=\sigma^{n+1}(\syma)=\Inf{1\hw_{m/n}} = 1\sigma^2(\syma) =
\iota(\re)$, so that $f^n(a)=\re$ by Convention~\ref{conv:unimodal}~(b).
Conversely, suppose that $f^n(a)=\re$. By the previous paragraph we have
$\syma\not=\lhe(m/n)$, so that $\syma=c_{m/n}d$ for some $d\in\{0,1\}^\N$ by
Lemma~\ref{lem:height-intervals}~(c). Therefore
\[ d = \iota(f^n(a)) = \iota(\re) =
10^{\kappa_1(m/n)-1}110^{\kappa_2(m/n)}11\ldots 110^{\kappa_m(m/n)}1d,
\] so that $d=\Inf{10^{\kappa_1(m/n)-1}110^{\kappa_2(m/n)}11\ldots
110^{\kappa_m(m/n)}1} = \Inf{1\hw_{m/n}}$, and it follows that $\syma=c_qd =
10\hw_{m/n}\Inf{1\hw_{m/n}} = 10\Inf{\hw_{m/n}1}= \rhe(m/n)$ as required.

\medskip

For (b)(iii), write $\Lambda = S\setminus\bigcup_{r\ge 0}B^{-r}(\gamma)$, and
  suppose first that~$\mu\not\in\{\lhe(m/n), \rhe(m/n)\}$, so that
  $B^n(a)\in\ingam$ by~(b)(ii). We need to show that~$\Lambda = P$, where~$P$
  is a period~$n$ orbit of~$B$. Since $\kappa(f)\succ\lhe(q)=\Inf{w_q1}$, $f$
  has a period~$n$ point~$x$ with this itinerary. By
  Lemma~\ref{lem:gamma-observation}~(d) and Lemma~\ref{lem:outside-lemma}~(c)
  (and the fact that $\Inf{w_q1}$ only contains blocks of $1$s of even length),
  $x_u$ lies on a period~$n$ orbit~$P\subset\Lambda$ of~$B$.

Since~$B$ is a monotone degree one circle map and the orbit of~$B(a)$ is an
attracting periodic orbit (as~$B$ is locally constant at $B^n(a)$), it only
remains to show that~$B$ has no other periodic orbits. 

Suppose for a contradiction that~$B$ has another periodic orbit~$R$, which must
be disjoint from~$\gamma$, have period~$n$, and have one point between each
pair of consecutive points of~$P$. By~\eqref{eq:commute}, since~$R$ is disjoint
from~$\gamma$, it lies above a periodic orbit of~$f$. Now every point of~$P$
and~$R$ in the upper half of~$S$ lies to the right of $\re_u$, and hence
of~$c_u$, so there is only one point of~$R$ which could lie either to the right
or to the left of~$c$, namely the one between the two points of~$P$ which bound
an interval containing~$c_\ell$. Therefore the periodic orbit of~$f$
corresponding to~$R$ contains a point~$y$ with either
$\iota(y)=\iota(x)=\lhe(m/n)=\Inf{w_q1}$, or
$\iota(y)=\Inf{w_q0}=\Inf{10^{\kappa_1(m/n)}11 \ldots 11 0^{\kappa_m(m/n)}}$.
The former is impossible by Convention~\ref{conv:unimodal}~(c), since
$\mu\succ\lhe(m/n)$; while the latter is impossible since~$\iota(y)$ has an
isolated~1 and so cannot be the itinerary of a point in~$\Lambda$ (we would
have $f^{n-1}(y)<c$, so the point of~$R$ above~$y=f^n(y)$ would be $y_\ell$;
but then $B(y_\ell) = f(y)_u$ since $y>c$, and hence $B(y_\ell)\in\gamma$ since
$f(y)<c$). This contradiction completes the proof of~(b)(iii) in the rational
interior case.

\smallskip

We next consider (b)(iii) in the case where $\mu=\lhe(m/n)$, so that $f$ is of
strict left endpoint type. In this case the period~$n$ orbit~$Q$ of~$a$ is
disjoint from~$\ingam$, so that $\tau(B^r(a)) = f^r(a)$ for all $r\ge 0$. As in the interior case, any other periodic orbit~$P$ of~$B$ must lie above a second period~$n$ orbit of~$f$ containing a point of itinerary $\lhe(m/n)$. 

If~$f$ is of tent-like type, then there is no such periodic orbit, so that~$Q$
is the only periodic orbit of~$B$, and is semi-stable. Since~$a\in Q$ is stable
through~$\gamma$, the orbit of any point of~$S$ eventually falls into~$\gamma$.

If~$f$ is of quadratic-like type, then~$f$ has exactly one such periodic orbit,
and there is an unstable periodic orbit~$P$ of~$B$ above it by
Lemma~\ref{lem:gamma-observation}~(d) and Lemma~\ref{lem:outside-lemma}~(c).
The $B^n$-orbits of points on one side of~$a\in Q$ converge to~$a$
through~$\gamma$, and so enter~$\gamma$; while those on the other side have
orbits which remain in the lower half of the circle, and so lie in~$\Lambda$.

\smallskip

The proof of~(b)(iii) when $\mu=\rhe(m/n)$ is similar: in this case, since
$\kappa(f)$ is not periodic, there is only one point of itinerary~$\lhe(m/n)$,
which lies on the orbit of~$B(a)$ by Lemma~\ref{lem:height-intervals}~(d):
therefore the orbit of~$B(a)$ is the only periodic orbit of~$B$, and since
$\re_u$ lies on this orbit and is stable through~$\gamma$, we have
$\Lambda=\emptyset$. This completes the proof of~(b).

\medskip\medskip

For~(c), assume that $q=q(\syma)=m/n$ and~$f$ is of early endpoint type, so
that $\syma=\Inf{w_q1}$ and $f^n(a)\not=a$. There is therefore a non-trivial
$f^n$-invariant subinterval~$J$ of~$I$, containing $a$ and $f^n(a)$, consisting
of all points with itinerary $\sigma(\syma)$. Now $f^n|_J\colon J\to J$ is
increasing, since $w_q1$ contains an even number of~$1$s, so that there is a
periodic point~$z\not=a$ in $J$ with $f^{rn}(a)\to z$ as $r\to\infty$. By the
same argument as in the previous case, every~$x\in J$ has the property that
$\{B^r(x_\ell)\,:\,r\ge 1\}$ is disjoint from~$\gamma$, and in particular~$B$
has a periodic orbit~$Q$, containing~$z_\ell$, which attracts the orbit
of~$B(a)$ and is disjoint from~$\gamma$. The rotation number of this periodic
orbit is $m/n$ by the same argument as in the previous case, and hence
$\rho(B)=m/n$. This establishes~(c), and~(a) in the early endpoint case.

\medskip\medskip

For~(d), and~(a) in the irrational case, assume that $q=q(\syma)$ is
irrational, so that $\syma=10^{\kappa_1(q)}110^{\kappa_2(q)}11\ldots$ by
Lemma~\ref{lem:height-intervals}~(a). That $B^r(a)\not\in\gamma$ for all~$r\ge
1$ is immediate from Lemma~\ref{lem:gamma-observation}~(a),
Lemma~\ref{lem:outside-lemma}~(d), and the fact that~$c$ is not periodic. By
the same argument as in the rational case, using
Lemma~\ref{lem:gamma-observation}~(c), we have
\[
\rho(B) = \lim_{m\to\infty} \frac{m}{\sum_{i=1}^m(\kappa_i(q)+2)},\] since~$m$
of the first $\sum_{i=1}^m(\kappa_i(q)+2)$ points of the orbit of~$a$ lie in
$[b,a)$. Therefore $\rho(B)=q$ by~(\ref{eq:kappa-i}).

To show that $\bigcup_{r\ge 0} B^{-r}(\gamma)$ is dense in~$S$ assume, for a
contradiction, that there is a non-trivial interval~$J = [x,y]$ in~$S$ whose
orbit is disjoint from~$\gamma$. Neither $a$ nor $b$ is in~$J$, since $B(b)=a$
and $a\in\gamma$. Therefore $\tau(x)\not=\tau(y)$ and, since $\kappa(f)$ isn't
periodic, we have $\iota(\tau(x))\not=\iota(\tau(y))$. Therefore, if~$r$ is
least with $\iota(\tau(x))_r \not= \iota(\tau(y))_r$, then $B^r(J)$ contains
either $c_\ell$ or $c_u$. In the former case we have
$B^{r+2}(J)\cap\gamma\not=\emptyset$, and in the latter we have
$B^r(J)\cap\gamma\not=\emptyset$, which is the required contradiction.

Finally, to show that the orbit of~$B(a)$ is dense in $S\setminus\bigcup_{r\ge
0}B^{-r}(\gamma)$, observe that the $\omega$-limit set $\omega(B(a), B)$
contains both $a$ and $\re_u$ by Lemma~\ref{lem:outside-lemma}~(e) and the fact
that distinct points have distinct itineraries. Let~$U$ be any interval in~$S$
which contains a point of $S\setminus\bigcup_{r\ge 0}B^{-r}(\gamma)$. Since it
also contains points of the dense set $\bigcup_{r\ge 0}B^{-r}(\gamma)$, there
is some $r\ge 0$ such that $B^r(U)$ contains a neighborhood either of~$a$ or
of~$\re_u$, and hence contains the point $B^R(a)$ for some $R>r$. Therefore
$B^{R-r}(a)\in U$ as required.

\end{proof}

\begin{cor}
\label{cor:rot-no} Suppose that $\cR$ is dense in~$\hS$, and that there is a
good chain of crosscuts for every $\by\in\hS$. Then the prime end rotation
number of $\hH\colon(\hT,\hI)\to(\hT,\hI)$ is $q(\kappa(f))$.
\end{cor}
\begin{proof} We have $\rho(\hB)=q(\kappa(f))$ by
Theorem~\ref{thm:outside-dynamics}~(a), since $B$ is a factor of~$\hB$ by the
degree one semi-conjugacy $\by\mapsto y_0$. The result follows from
Lemma~\ref{lem:PP-conj-hS}.
\end{proof}

\subsection{The irrational case}
\label{sec:irrational} 
Let $f\colon[a,b]\to[a,b]$ be a unimodal map whose kneading sequence
$\syma=\kappa(f)$ has irrational height $q=q(\syma)\in(0,1/2)$. In this section
we determine the prime ends of $(\hT, \hI)$.

We first use Theorem~\ref{thm:outside-dynamics} to analyze the dynamics of the
natural extension $\hB\colon\hS\to\hS$, showing that it is a Denjoy
counterexample (i.e.\ it has an orbit of wandering intervals). It is
straightforward to show that the landing set $\cL=\hS$
(Lemma~\ref{lem:landing-irrational}), and that the locally uniformly landing
set $\cR$ is the union of the interiors of the wandering intervals
(Lemma~\ref{lem:cantor}), the complement of $\cR$ being a Cantor set~$\Lambda$.

In particular, this establishes that $\cR$ is dense in~$\hS$.
Lemma~\ref{lem:irrat-good-chain} asserts the existence of a good chain of
crosscuts for every $\by\in\hS$. Therefore, by Lemma~\ref{lem:general-summary},
the prime ends $\cP(\by)$ with $\by\not\in\Lambda$ are of the first kind, while
those with $\by\in\Lambda$ are of the first or second kind. We complete the
analysis by showing that these are of the second kind, and that in fact
$\cI(\cP(\by))=\hI$ when $\by\in\Lambda$
(Lemma~\ref{lem:irrational-prime-ends-type-2}).

\medskip\medskip

Let $\cO = \{B^r(a)\,:\,r\ge 1\}$ be the orbit of~$B(a)$, which is disjoint
from~$\gamma$ by Theorem~\ref{thm:outside-dynamics}. Since $B(\gamma)=B(a)$,
and $B$ is injective away from~$\gamma$, the backwards orbit
$\{B^{-r}(y)\,:\,r\ge 0\}$ of any point $y\in S\setminus\cO$ is well-defined,
and is disjoint from~$\gamma$ except perhaps at its first point~$y$. On the
other hand, the backwards orbits of points of $\cO$ are ill-defined at one
point only: the preimage of $B(a)$ is $\gamma$. The elements of $\hS$ can
therefore be described straightforwardly.

\begin{notations}[Threads $\bt(y,r)$ and $\bt(y)$ in~$\hS$] \mbox{}
\label{notn:threads-irrat}
\begin{enumerate}[(a)]
\item For every $y\in\gamma$ and $r\in\Z$, define $\bt(y,r)\in \hS$ by
\begin{equation}
\label{eq:thread-hS-1}
\bt(y,r) =
\begin{cases}
\thr{ B^r(a), \ldots, B(a), y, B^{-1}(y), B^{-2}(y), \ldots } & \text{ if }r>
0,\\
\thr{ B^r(y), B^{r-1}(y), B^{r-2}(y), \ldots } & \text{ if }r\le 0.
\end{cases}
\end{equation}
\item For every $y\in S\setminus \bigcup_{r\in\Z} B^{-r}(\gamma)$, define
  $\bt(y)\in\hS$ by
\begin{equation}
\label{eq:thread-hS-2}
\bt(y) = \thr{y, B^{-1}(y), B^{-2}(y), \ldots}.
\end{equation}
\end{enumerate}
\end{notations}

Every element $\by$ of~$\hS$ can be written in exactly one way as either
$\bt(y,r)$ or $\bt(y)$: $\by$ is of the form~(\ref{eq:thread-hS-1}) if and only
if there is some (unique) $r\in\Z$ with $\hB^r(\by)_0\in\gamma$, in which case
$\by = \bt(\hB^r(\by)_0, -r)$; and $\by=\bt(y_0)$ otherwise. We have
$\hB(\bt(y,r)) = \bt(y, r+1)$, and $\hB(\bt(y)) = \bt(B(y))$.

\begin{lem}
\label{lem:landing-irrational} $\cL = \hS$.
\end{lem}
\begin{proof} $\bt(y,r)$ is landing of level at most~$\max(r,0)$ (it is landing
of level exactly~$\max(r,0)$ if $y\in\ingam$, and of level~$0$ if $y=a$ or
$y=\re$), and $\bt(y)$ is landing of level~$0$. The result follows from
Corollary~\ref{cor:ray-lands}.
\end{proof}

\begin{defn}[The gaps~$G_r$]
\label{def:gaps} For each~$r\in\Z$, define the {\em gap} $G_r\subset\hS$ by
$
G_r = \{\bt(y,r)\,:\,y\in\gamma\}.
$
\end{defn}

The gaps are compact intervals, since the functions $y\mapsto
\bt(y,r)$ are homeomorphisms $\gamma\to G_r$. Since $\hB(G_r)=G_{r+1}$ for
each~$r$, and the $G_r$ are mutually disjoint, the gaps form an orbit of
wandering intervals of $\hB$, which is therefore a Denjoy counterexample.

\begin{remark}
\label{rmk:irrat-order} The map $\pi_0\colon \hS\to S$ defined by
$\pi_0(\by)=y_0$ is continuous and surjective. Moreover,
$\pi_0(\by)=\pi_0(\by')$ for $\by\not=\by'$ if and only if $\by$ and $\by'$
belong to the same gap $G_r$ for some $r>0$. Therefore $\pi_0$ is a monotone
circle map which collapses these gaps. It follows that threads are ordered
around $\hS$ in the same way that points are ordered around $S$, except that
the points $B^r(a)$ of~$S$ for $r>0$ are blown up into gaps $G_r$.
\end{remark}

\begin{defn}[The set~$\Lambda$]
\label{def:lambda} The set $\Lambda\subset\hS$ is defined by
\[\Lambda = \displaystyle{\hS\setminus \bigcup_{r\in\Z}
  \accentset{\circ}{G}_r}.\]
\end{defn}

\begin{lem}
\label{lem:cantor} $\Lambda$ is a Cantor set, and $\cR = \hS\setminus\Lambda =
\bigcup_{r\in\Z}\accentset{\circ}{G}_r$. In particular, $\cR$ is dense
in~$\hS$.
\end{lem}
\begin{proof} $\Lambda$ is compact, and is perfect since it is the complement
of a union of open intervals with disjoint closures. To show that it is totally
disconnected, it is enough to show that $\bigcup_{r\in\Z}G_r$ is dense in
$\hS$. To do this, let $\bt(y)$ be a point in the complement of this set. By
Theorem~\ref{thm:outside-dynamics}~(d)(ii), there is a sequence $y_i\to y$
in~$S$ with $B^{r_i}(y_i)\in\gamma$ for some~$r_i\ge0$. Then the sequence
$\bt(B^{r_i}(y_i), -r_i) = \thr{y_i, B^{-1}(y_i), \ldots}$ in
$\bigcup_{r\in\Z}G_r$ converges to $\bt(y) = \thr{y, B^{-1}(y), \ldots}$. (Note
that, for each $k>0$, when $i$ is sufficiently large~$y_i$ lies in a
neighborhood~$N$ of~$y$ which doesn't contain any point $B^r(a)$ with $r\le k$,
so that $B^{-r}$ is well-defined and continuous in~$N$ for all $r\le k$.)

Each $\accentset{\circ}{G}_r$ is uniformly landing of level~$\max(r,0)$, so
that $\bigcup_{r\in\Z}\accentset{\circ}{G}_r\subset\cR$. For the converse,
suppose that $\by\in\Lambda$. Consider first the case where~$\by$ is not a gap
endpoint, so that $\by=\bt(y)$ for some $y\in
S\setminus\bigcup_{r\in\Z}B^{-r}(\gamma)$. By
Theorem~\ref{thm:outside-dynamics}~(d)(iii) there is a sequence $r_i\to\infty$
of positive integers with $B^{r_i}(a) \to y$. Then for any $z\in\ingam$,
$
\left(\bt(z,r_i)\right)_{i\ge0} = \left(\thr{ B^{r_i}(a), \ldots, B(a), z,
B^{-1}(z),\ldots }\right)_{i\ge0}
$
is a sequence converging to $\by$ which is not uniformly landing.

The proof in the case where~$\by$ is a gap endpoint is similar. We have
$\by=\bt(e,r)$ where $e=a$ or $e=\re_u$, and~$r\in\Z$. As in the proof of
Theorem~\ref{thm:outside-dynamics}~(d)(iii), there is a sequence $r_i\to\infty$
with $B^{r_i}(a)\to e$ (and $B^{r_i}(a)\not\in\gamma$). Then for any
$z\in\ingam$, $(\bt(z, r_i+r))_{i\ge0}$ is a sequence converging to~$\by$ which
is not uniformly landing.
\end{proof}

We next show that there is a good chain of crosscuts for every $\by\in\hS$. The
following notation will be convenient when defining chains of crosscuts.

\begin{notation}[The crosscuts $\xi'(\bJ, s)$ and $\xi(\bJ,s)$]
\label{notn:crosscuts-xi-J-s} Let~$\bJ$ be an interval in $\hS$ with endpoints
\mbox{$\by_1,\by_2\in\cL$}, and let $s\in(0,\infty)$. Write $\xi'(\bJ,s)$ for
the crosscut
\[
\xi'(\bJ,s) = (\{\by_1\}\times[s,\infty]) \,\,\cup\,\, (\bJ\times\{s\})\,\,
\cup\,\, (\{\by_2\}\times[s,\infty])
\] in $(\bD, \hS_\infty)$; and $\xi(\bJ,s)$ for the crosscut
$\Psi(\xi'(\bJ,s))$ in $(\hT, \hI)$.
\end{notation}

The requirement that $\by_1,\by_2\in\cL$ is automatically satisfied in the
irrational case, but this definition will be used later in situations in which
$\cL\not=\hS$.

\begin{lem}
\label{lem:irrat-good-chain} Let $\by\in\hS$. Then there is a good chain of
crosscuts for~$\by$.
\end{lem}
\begin{proof} We can assume that $\by\not\in\cR$, i.e.\ that $\by\in\Lambda$
(Remark~\ref{rmk:good-chain}~(c)), so that $\by$ is either a gap endpoint or is
in the complement of the gaps.

\noindent\textbf{Case 1: }$\by = \bt(y)$ for some $y\in
S\setminus\bigcup_{r\in\Z} B^{-r}(\gamma)$, that is, $\by$ is in the complement
of the gaps. We construct crosscuts $\xi_k'$ in $(\bD,
\hS_\infty)$ inductively for $k\ge 1$.
\begin{enumerate}[(a)]
\item Choose $\epsilon_k>0$ small enough that if $x,z\in I$ with
  $|x-z|<2\epsilon_k$ then $|f^r(x)-f^r(z)|<1/2^k$ for $0\le r\le k$.
\item Pick a closed interval~$J_k\subset S$ with~$y$ in its interior, of length
  less than~$\epsilon_k$, which is small enough that it doesn't contain any of
  the points $B^r(a)$ with $1\le r\le 2k$; and that $J_k\subset\Int(J_{k-1})$
  if $k>1$. We may shrink~$J_k$ in step~(c), and we do this in such a way that
  $y$ remains in its interior.
\item It follows that $B^{-k}$ is well-defined and continuous on~$J_k$, and we
  make~$J_k$ smaller if necessary in order to ensure that $|\tau(B^{-k}(\eta))-
  \tau(B^{-k}(y))|<\epsilon_k$ for all $\eta\in J_k$. We shrink~$J_k$ again so
  that its endpoints~$L$ and~$R$ are preimages of~$c_u$ (which is possible by
  Theorem~\ref{thm:outside-dynamics}~(d)(ii)). Let $i$ and $j$ be such that
  $B^i(L)=c_u$ and $B^j(R)=c_u$.
\item Let $\bJ_k$ be the interval in $\hS$, containing~$\by$, with endpoints
  $\bt(c_u, -i) = \thr{L,\ldots}$ and \mbox{$\bt(c_u, -j) =
  \thr{R,\ldots}$}.
\item Set $\xi'_k = \xi'(\bJ_k,2k)$.
\end{enumerate}

By Remark~\ref{rmk:irrat-order} and~(b) above, the points $\bv\in\bJ_k$ are
exactly the following:

\begin{enumerate}[(i)]
\item $\bv=\bt(v) = \thr{v,\ldots}$, for $v\in J_k\setminus\bigcup_{r\in\Z}
  B^{-r}(\gamma)$;
\item $\bv=\bt(B^r(v), -r)=\thr{v,\ldots}$ where $v\in J_k$ and
  $B^r(v)\in\gamma$ for some $r\ge 0$; and
\item $\bv=\bt(Y, r)=\thr{B^r(a),\ldots}$ where $Y\in\gamma$ and $B^r(a)\in
J_k$ for some $r>2k$.
\end{enumerate} In particular, $\bJ_k \subset \Int\bJ_{k-1}$ when $k>1$, so
that $(\xi_k')$ is a chain of crosscuts in $(\bD, \hS_\infty)$.

$(\xi_k')$ satisfies conditions~(a) and~(b) of Definition~\ref{def:good-chain},
so, since $\cL=\hS$, it only remains to show that \mbox{$\diam(\xi_k)\to 0$} as
$k\to\infty$, where $\xi_k = \Psi(\xi'_k)$. To do this we will show that for
all $\bx\in\xi_k$ we have \mbox{$|x_k-\tau(B^{-k}(y))|<\epsilon_k$.} It will
follow that if $\bx,\bz\in\xi_k$ we have $|x_k-z_k|<2\epsilon_k$, so that
$|x_r-z_r|<1/2^k$ for all $r\le k$ by choice of~$\epsilon_k$, establishing the
result.

Consider first points $\bx=\Psi(\bv, 2k)\in \Psi(\bJ_k\times\{2k\})$.
By~(\ref{eq:convenient-2}) we have $x_k = f^{k-1}(H(v_{2k}, 1/2))$, and
$H(v_{2k}, 1/2) = \Upsilon\circ\barf(v_{2k},1/2) = \tau(B(v_{2k}))$ by~(U1) of
Definition~\ref{def:unwrapping-unimodal}. Therefore
\[ x_k = f^{k-1}(\tau(B(v_{2k}))) = \tau(B^k(v_{2k})) = \tau(B^{-k}(v_0)),
\] where we use~(\ref{eq:commute}) together with the fact that
$v_r\not\in\ingam$ for $0\le r\le 2k$.

By (i)\,--\,(iii) above, every $\bv\in\bJ_k$ satisfies $v_0\in J_k$, so that
\[|x_k-\tau(B^{-k}(y))|=|\tau(B^{-k}(v_0)) -
\tau(B^{-k}(y))|<\epsilon_k\] by~(c) as required.

\medskip

Now consider points $\bx = \Psi(\bt(c_u, -i), \,s)$ or $\bx =
\Psi(\bt(c_u, -j), \,s)$ with $s\in[2k,\infty]$. Since $\bt(c_u, -i)$ and
$\bt(c_u, -j)$ are landing of level~0 and $s>k$,
Lemma~\ref{lem:formula-disjoint-gamma} gives $x_k =
\tau(B^{-k}(L))$ or $x_k = \tau(B^{-k}(R))$, and the argument goes through as
before.

\medskip\medskip

\noindent\textbf{Case 2: }$\by = \bt(e, r)$ (with $e=a$ or $e=\re_u$), i.e.\
$\by$ is an endpoint of $G_r$ for some~$r$. Choose $\bJ_k$ to have one endpoint
$\bt(c_u, -i)$ as above, and the other endpoint $\bt(v_k, r)$, where $(v_k)$ is
a sequence in $\ingam$ converging to~$e$. Then $\diam(\xi_k \cap \Psi(G_r\times
[0,\infty]))$ converges to~$0$ since $\Psi|_{G_r\times[0,\infty]}$ is a
homeomorphism; while
$\diam\left(\overline{\xi_k\setminus\Psi(G_r\times[0,\infty])}\right)$
converges to~$0$ by the same argument as in case~1.
\end{proof}

It follows from Theorem~\ref{thm:corr-homeo} and
Lemma~\ref{lem:general-summary} that $\cP\colon\hS\to\PP$ is a homeomorphism;
that the prime end $\cP(\by)$ is of the first kind if $\by\not\in\Lambda$; and
that $\Pi(\cP(\by)) = \{\omega(\by)\}$ for all~$\by$. It therefore only remains
to calculate the impressions of the prime ends $\cP(\by)$ for $\by\in\Lambda$.

\begin{lem}
\label{lem:irrational-prime-ends-type-2} 
  $\cI(\cP(\by)) = \hI$ for all $\by\in\Lambda$.
\end{lem}
\begin{proof} By Theorem~\ref{thm:corr-homeo}~(b), $\cP=\cP(\by)$ is defined by
the chain $(\Psi(\xi_k'))$, where $(\xi_k')$ is the good chain of crosscuts
constructed in the proof of Lemma~\ref{lem:irrat-good-chain}. Write $\xi_k =
\Psi(\xi_k')$ and $U_k = U(\xi_k)$. Fix~$k$, and any element $\bx\in\hI$. We
show that $\bx\in\overline{U_k}$, which will establish the result. We use the
notation of the proof of Lemma~\ref{lem:irrat-good-chain}.

By Lemma~\ref{lem:eventually-onto} in Appendix~\ref{app:technical}, there is
some~$N$ with $f^N([a,c])=I$. For each~$i\ge 1$ there exists, by
Theorem~\ref{thm:outside-dynamics}~(d)(iii), an integer $r_i>i+N$ with
$B^{r_i}(a)\in J_k$, so that $G_{r_i} \subset\bJ_k$. Since $r_i-i>N$, there is
some $z\in[a,c]$ with $f^{r_i-i}(z)=x_i$. Then $\bt(z_u,r_i)\in
G_{r_i}\subset\bJ_k$, and by Corollary~\ref{cor:ray-lands} we have
$\omega(\bt(z_u,r_i))_i = f^{r_i-i}(\tau(z_u))=x_i$.

Therefore $\omega(\bt(z_u,r_i)) \to \bx$ as $i\to\infty$: since all points of
this sequence are in $\overline{U_k}$, we have $\bx\in\overline{U_k}$ as
required.
\end{proof}

The following theorem provides a summary of what we have proved in the
irrational case.

\begin{thm}[Prime ends in the irrational case]
\label{thm:summary-irrational} Let~$f$ be a unimodal map satisfying the
conditions of Convention~\ref{conv:unimodal}, and suppose that $f$ is of
irrational type. Then
\begin{enumerate}[(a)]
\item There is a Cantor set of prime ends of $(\hT, \hI)$ of the second kind,
  for which the impression is~$\hI$. All of the other prime ends are of the
  first kind.
\item The prime end rotation number is $q(\kappa(f))$.
\end{enumerate}
\end{thm}

\begin{remark} \label{rmk:irrat-accessible}
By Theorem~\ref{thm:corr-homeo}~(d), the set of accessible
points of~$\hI$ is precisely $\{\omega(\by)\,:\,\by\in\hS\}$. This set is
partitioned into countably many compact arcs $\omega(G_r)$ for $r\in\Z$, and
uncountably many points $\omega(\bt(y))$ for $y\in
S\setminus\bigcup_{r\in\Z}B^{-r}(\gamma)$.
\end{remark}
\subsection{The rational interior case}
\label{sec:rational} Let $f\colon[a,b]\to[a,b]$ be a unimodal map whose
kneading sequence $\syma=\kappa(f)$ has rational height
$q=q(\syma)=m/n\in(0,1/2)$; and suppose that $\Inf{w_q0}\prec \syma\prec
\rhe(q)$, so that~$f$ is of rational interior type. In this section we
determine the prime ends of $(\hT, \hI)$.

By Theorem~\ref{thm:outside-dynamics}~(b)(i) we have that $B^r(a)\not\in\gamma$
for $1\le r<n$, and $B^n(a)\in\gamma$. Therefore~$B(a)$ is a period~$n$ point
of~$B$, whose orbit~$Q$ contains a single point of~$\gamma$. There is a
corresponding period~$n$ orbit~$\bQ$ of the natural extension
$\hB\colon\hS\to\hS$. By Theorem~\ref{thm:outside-dynamics}~(b)(iii), $B$ has
exactly one other periodic orbit~$P$, which has period~$n$ and is disjoint
from~$\gamma$; and therefore $\hB$ has exactly one other periodic orbit~$\bP$,
of period~$n$.

After describing the threads of~$\hS$, we will show that the points of~$\bQ$
are not landing, and that every other point of $\hS$ is locally uniformly
landing, so that $\cR$ is dense in~$\hS$
(Lemma~\ref{lem:R-dense-rat-interior}). We then construct good chains of
crosscuts for each $\by\in\hS$ (Lemma~\ref{lem:good-chain-NML-rat-interior}).
In the irrational case the construction of the good chains was rather ad hoc;
here, by contrast, there are natural choices for the crosscuts about the points
of~$\bQ$, which form an invariant system of subsets of stable sets
(Lemmas~\ref{lem:crosscut-orbit} and~\ref{lem:gamma-stable}).

By Lemma~\ref{lem:general-summary}, all of the prime ends $\cP(\by)$ with
$\by\not\in\bQ$ are of the first kind. We show that if $\by\in\bQ$ then
$\cI(\cP(\by))=\hI$ (Lemma~\ref{lem:impression-rat-interior}); and that
$\Pi(\cP(\by))$ is equal to~$\hI$ except in the case where~$f$ can be subjected
to a particular type of renormalization, in which case $\Pi(\cP(\by))$ is
homeomorphic to the inverse limit of the renormalized map
(Lemmas~\ref{lem:principal-rat-int-non-renorm}
and~\ref{lem:principal-rat-int-renorm}). Therefore these prime ends may be of
either the third or the fourth kind.

\medskip\medskip

Write $q_i = B^{i+1}(a)$ for $0\le i\le n-1$. By
Theorem~\ref{thm:outside-dynamics}~(b)(ii) we have $q_{n-1}\not\in\{a,\re_u\}$,
so that $q_{n-1}\in\ingam$, while the other $q_i$ are not in~$\gamma$. We have
$B(q_i)=q_{i+1\bmod{n}}$ for each~$i$, so that $Q=\{q_0,q_1,\ldots,q_{n-1}\}$
is a period~$n$ orbit of~$B$. Since $B^{-1}$ is well defined on
$S\setminus\{q_0\}$, the backwards orbit $\{B^{-r}(y)\,:\,r\ge 0\}$ of any
point $y\in S\setminus Q$ is well-defined. Moreover, $B^{-r}(y)\not\in\gamma$
for all $r\ge 1$ for such points~$y$.

 Since $\rho(B)=m/n$, there are~$m-1$ points of $Q$ in each interval
 $(q_i,q_{i+1\bmod{n}})$, and the first point of~$Q$ which is encountered when
 moving counterclockwise around~$S$ from~$q_i$ is $q_{i+m^{-1} \bmod{n}}$.
 Write~$p_0$ for the point of~$P$ between~$q_0$ and~$q_{m^{-1}\bmod{n}}$, and
 $p_i = B^i(p_0)$ for $1\le i\le n-1$.

\begin{notations}[Threads $\bq_i$, $\bp_i$, and $\bt(y,k,i)$ in $\hS$; the
   periodic orbits~$\bQ$ and~$\bP$]\mbox{}
\label{notn:threads-rat-interior}
\begin{enumerate}[(a)]
\item For each $0\le i\le n-1$, define $\bq_i, \bp_i\in\hS$ by
\begin{eqnarray*}
\bq_i &=& \thr{(q_i,q_{i-1},\ldots,q_0,q_{n-1},\ldots, q_{i+1})^\infty} \quad\text{ and} \\
\bp_i &=& \thr{(p_i,p_{i-1},\ldots,p_0,p_{n-1},\ldots, p_{i+1})^\infty}.
\end{eqnarray*}
\item For each $y\in\gamma\setminus\{q_{n-1}\}$, $k\in\Z$, and $0\le i\le n-1$,
  define $\bt(y,k,i) = \hB^{kn+i}\left(\thr{q_0, y, B^{-1}(y),
  \ldots}\right)\in\hS$, so that
\begin{equation}
\label{eq:tyki}
\bt(y,k,i) = \thr{ q_i,q_{i-1},\ldots,q_0, (q_{n-1},\ldots,q_0)^k, y,
B^{-1}(y),\ldots }\quad\text{ when }k\ge 0.
\end{equation}
\end{enumerate} 
Write $\bQ = \{\bq_0,\ldots,\bq_{n-1}\}$ and $\bP =
\{\bp_0,\ldots,\bp_{n-1}\}$, period~$n$ orbits of~$\hB$.
\end{notations}

\medskip

Every element $\by$ of $\hS$ can be written in exactly one way as $\bq_i$,
$\bp_i$, or $\bt(y,k,i)$. To see this, observe that the set
$Z(\by)=\{r\in\Z\,:\,\hB^{-(r+1)}(\by)_0\in\gamma\}$ is empty if and only if
$\by\in\bP$, and is not bounded above if and only if $\by\in\bQ$. For any other
$\by\in\hS$, let~$R=\max Z(\by)$, and let $y=\hB^{-(R+1)}(\by)_0\in\gamma$.
Then $\by=\bt(y, \floor{R/n}, R
\bmod{n})$.

\begin{defns}[The intervals $L_{k,i}$ and $R_{k,i}$]
\label{def:intervalsL,R,I} For each $k\in\Z$ and $0\le i\le n-1$, define
subsets $L_{k,i}$ and $R_{k,i}$ of~$\hS$ by
\[ L_{k,i} = \{\bt(y,k,i)\,:\, y\in(q_{n-1},a]\} \qquad\text{ and }\qquad
R_{k,i} =
\{\bt(y,k,i)\,:\, y\in[\re_u,q_{n-1})\}.
\] 
\end{defns} 
These subsets partition $\hS\setminus(\bQ\cup\bP)$, and are half-open intervals
since $y\mapsto \bt(y,k,i)$ defines homeomorphisms $(q_{n-1},a]\to L_{k,i}$ and
$[\re_u,q_{n-1})\to R_{k,i}$.

 The following lemma, which describes the ordering of the intervals $L_{k,i}$,
 and $R_{k,i}$ around the circle~$\hS$, is illustrated by
 Figure~\ref{fig:circle-order}.

\begin{lem}
\label{lem:ordering-intervals} Let $0\le i\le n-1$, and write
$j=i-m^{-1}\bmod{n}$.
\begin{enumerate}[(a)]
\item For each $k\in\Z$, the open endpoint of $L_{k,i}$ (respectively
  $R_{k,i}$) is equal to the closed endpoint of $L_{k+1,i}$ (respectively
  $R_{k+1,i}$).
\item As $k\to\infty$ we have $L_{k,i}\to \bq_i$ and $R_{k,i}\to
  \bq_i$; while as $k\to-\infty$ we have $L_{k,i}\to \bp_i$ and
  $R_{k,i}\to\bp_j$.
\end{enumerate}
\end{lem}
\begin{proof} (a) is a straightforward computation of the open endpoints
of the intervals, using the facts that $B^{-1}(y)\to a$ as $y\to q_0$ through
$(q_0,b)$, and $B^{-1}(y)\to \re_u$ as $y\to q_0$ through $(a,q_0)$. For~(b),
the limits as $k\to\infty$ are immediate from~\eqref{eq:tyki} and those as
$k\to-\infty$ from the choice of labeling of the $p_i$ (and hence of the
$\bp_i$).
\end{proof}

\begin{figure}[htbp]
\begin{center}
\includegraphics[width=0.95\textwidth]{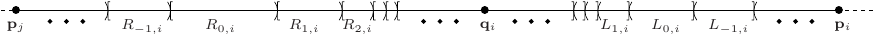}
\end{center}
\caption{Ordering of intervals around the circle in the rational interior
case.}
\label{fig:circle-order}
\end{figure}

\begin{remark}
\label{rmk:alternative-intervals} For each $y\in\gamma\setminus\{q_{n-1}\}$,
$k\in\Z$, and $0\le i\le n-1$ we have
\[
\hB(\bt(y,k,i)) =
\begin{cases}
\bt(y,k,i+1) & \text{ if }i<n-1,\\
\bt(y,k+1,0) & \text{ if }i=n-1.
\end{cases}
\] Therefore
\begin{equation}
\label{eq:alternative-intervals}
\hB(L_{k,i}) =
\begin{cases} L_{k,i+1} & \text{ if }i<n-1,\\ L_{k+1,0} & \text{ if }i=n-1,
\end{cases}
\end{equation} and analogously for $\hB(R_{k,i})$.
\end{remark}

\begin{lem}
\label{lem:R-dense-rat-interior} $\cR = \cL = \hS\setminus\bQ$. In particular,
$\cR$ is dense in~$\hS$.
\end{lem}
\begin{proof} 
Points of $L_{k,i}$ and $R_{k,i}$ are landing of level $\max(kn+i+1, 0)$, and
points of~$\bP$ are landing of level~0. Therefore, by
Lemma~\ref{lem:ordering-intervals}, every point of $\hS\setminus\bQ$ has a
uniformly landing neighborhood, so that $\hS\setminus\bQ\subset\cR\subset\cL$.

We next show that $\bq_{n-1}\not\in\cL$ (and hence $\bq_{n-1}\not\in\cR$).
Write $\theta = \min(f^n(a), \widehat{f^n(a)})\in(a, c]$, and let~$x$ and~$z$
be any two distinct elements of $[a,\theta]$. By
Lemma~\ref{lem:images-of-tight}~(a), if $\mu\preceq w_q0\Inf{w_q1}$ then
$f^n([a,\theta))=[a,\theta]$; and by Lemma~\ref{lem:images-of-tight}~(b), if
$\mu\succ w_q0\Inf{w_q1}$ then there is some~$N$ with $f^N([a,\theta))=[a,b]$,
so that $f^{N+i}([a, \theta)) \supset [a, \theta]$ for all~$i\ge 0$. Hence in
either case there are sequences~$(x^{(k)})$ and~$(z^{(k)})$ in $[a, \theta)$
with $f^{kn}(x^{(k)})=x$ and $f^{kn}(z^{(k)})=z$ for all sufficiently
large~$k$.

Since $x^{(k)}\in[a, \theta)$ and $\theta\le c$, we have $f(a)\le f(x^{(k)}) <
f(\theta) = f^{n+1}(a)$, and hence, by
Definition~\ref{def:unwrapping-unimodal}, there is some $v_1^{(k)}\in[1/2,
\phi^{-1}(f^{n+1}(a))]$ with $\phi(v_1^{(k)}) = f(x^{(k)})$. Since
$\tau(q_{n-1})=f^n(a)$, it follows from~(U3) and~(U4) of
Definition~\ref{def:unwrapping-unimodal} that $\barf(q_{n-1}, v_1^{(k)}) =
(f(x^{(k)})_\ell, v_1^{(k)})$, and hence that $H(q_{n-1}, v_1^{(k)}) =
f(x^{(k)})$. Similarly, there is a sequence $(v_2^{(k)})$ in $[1/2, 1)$ with
$H(q_{n-1}, v_2^{(k)}) = f(z^{(k)})$ for each~$k$.

By~(\ref{eq:convenient-2}) we have that, for sufficiently large~$k$,
$R_{\bq_{n-1}}(kn + 2v_1^{(k)}-1)_0 = f^{kn}(x^{(k)})=x$ and similarly
$R_{\bq_{n-1}}(kn+ 2v_2^{(k)}-1)_0 = f^{kn}(z^{(k)})=z$. Therefore
$R_{\bq_{n-1}}$ does not land, so that $\bq_{n-1}\not\in\cL$ as required.

Since $\hH^{i+1}$ maps $R_{\bq_{n-1}}$ onto $R_{\bq_i}$ for $0\le i< n-1$, it
follows that $\bq_i\not\in\cL$ for all~$i$.
\end{proof}

For future reference, we record the landing points corresponding to the
elements $\bt(y,k,i)$ of $\hS$ with $k\ge 0$ which are given, using
(\ref{eq:omega}) together with $\bt(y,k,i)\in\cL_{kn+i+1}$ and
$\bt(y,k,i)_{kn+i+1}=y$, by
\begin{equation}
\label{eq:rat-thread-land}
\omega(\bt(y,k,i)) = \thr{f^{kn+i+1}(\tau(y)), \ldots, f(\tau(y)),
  \tau(y), \tau(B^{-1}(y)), \ldots}\qquad(k\ge 0).
\end{equation}

\medskip

We next define some good chains of crosscuts for each $\bq_i$. The class of
crosscuts which we use is introduced in Definition~\ref{defn:crosscuts-gamma},
and it will be shown in Lemma~\ref{lem:good-chain-NML-rat-interior} how these
combine to give good chains.

\begin{notations}[$\hy$, $\min(y, \hy)$.] For each~$y\in\gamma$, denote by
$\hy$ the ``symmetric'' element of~$\gamma$ satisfying
\begin{enumerate}[(a)]
\item $f(\tau(\hy)) = f(\tau(y))$, and
\item $\hy\not=y$ unless $y=c_u$.
\end{enumerate} We set $\min(y,\hy)=y$ if $y\in[c_u,a]$, and $\min(y,\hy)=\hy$
otherwise. (This convention is so that $\tau(\min(y, \hy)) =
\min(\tau(y), \tau(\hy))$).
\end{notations}

\begin{defn}[The crosscuts $\Gamma'(y, k, i)$ and $\Gamma(y,k,i)$]
\label{defn:crosscuts-gamma} 
  For each $y\in(c_u,a]\setminus\{\min(q_{n-1},\hq_{n-1})\}$, each $k\in\Z$,
  and each $0\le i\le n-1$, we define a crosscut $\Gamma'(y,k,i)$ in $(\bD,
  \hS_\infty)$ by
  \[
    \Gamma'(y,k,i) = 
    \begin{cases}
      \xi'([\bt(\hy,k,i), \bt(y,k,i)], \,\, kn + i + 1 + u(y)) 
                        & \text{ if }k \ge 0, \\
      \xi'([\bt(\hy,k,i), \bt(y,k,i)], \,\, (1+u(y))/2^{|nk+i|}) 
                        & \text{ if }k < 0,
    \end{cases}
  \] 
  where $u(y)$ is given by \Notation~\ref{notn:u-y} and $\xi'(\bJ, s)$ is as in
  Definition~\ref{notn:crosscuts-xi-J-s}. Here $[\bt(\hy,k,i),
  \bt(y,k,i)]$ is the interval in~$\hS$ with the given endpoints which is
  disjoint from~$\bP$. 

  We define $\Gamma(y,k,i)=\Psi(\Gamma'(y,k,i))$, a crosscut in $(\hT, \hI)$.
\end{defn}

\begin{remark}
\label{rmk:gammas} $(\bq_i,\infty)\in U(\Gamma'(y,k,i))$ if and only if $y\in
(\min(q_{n-1}, \hq_{n-1}), a]$. See Figure~\ref{fig:gammas}.
\end{remark}

\begin{figure}[htbp]
\begin{center}
\includegraphics[width=0.95\textwidth]{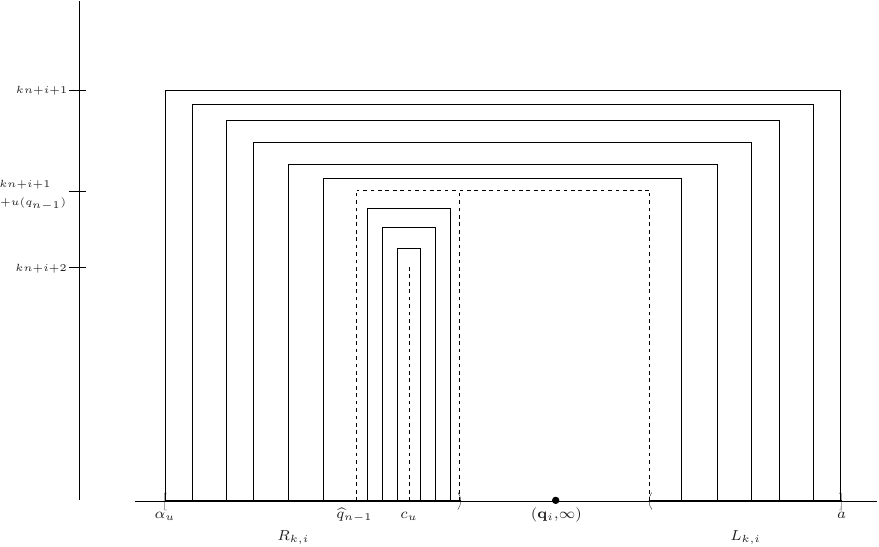}
\end{center}
\caption{The crosscuts $\Gamma'(y,k,i)$, in the case~$k\ge 0$, drawn under the
  assumption that $q_{n-1}\in (c_u,a]$, so that the interval~$L_{k,i}$ is
  shorter than the interval~$R_{k,i}$. In this figure the labels $\re_u$,
  $\hq_{n-1}$, $c_u$, and $a$ are abbreviations of $(\bt(\re_u,k,i), \infty)$,
  $(\bt(\hq_{n-1},k,i),\infty)$, $(\bt(c_u,k,i),\infty)$, and
  $(\bt(a,k,i),\infty)$. The dotted lines represent limits of the crosscuts,
  which are not themselves of the form $\Gamma'(y,k,i)$: see
  Remark~\ref{rmk:dotted}.}
\label{fig:gammas}
\end{figure}

\begin{lem}
\label{lem:crosscut-orbit} For each
$y\in(c_u,a]\setminus\{\min(q_{n-1},\hq_{n-1})\}$, each $k\in\Z$, and each
$0\le i\le n-1$, we have $\Gamma(y,k,i) = \hH^{kn+i}(\Gamma(y,0,0))$.
\end{lem}
\begin{proof} By Corollary~\ref{cor:Psi-conjugates} we have
$\hH^{kn+i}(\Gamma(y,0,0)) = \Psi(G^{kn+i}(\Gamma'(y,0,0))$,
where~$G\colon\bD\to\bD$ is given (see Definition~\ref{notn:G-et-al}) by
$G(\by,s)=(\hB(\by),
\lambda(s))$. Now $G^{kn+i}(\Gamma'(y,0,0))=\Gamma'(y,k,i)$ by
Remark~\ref{rmk:alternative-intervals} and the fact that $\lambda(s) = s+1$ for
$s\ge 1$ and $\lambda(s)=2s$ for $s<1$. The result follows.
\end{proof}

The following is a key lemma for the remainder of the paper. It implies, in
particular, that each crosscut $\Gamma(y,k,i)$ is contained in a stable set
for~$\hH$; and hence, by Lemma~\ref{lem:crosscut-orbit}, that
\mbox{$\diam(\Gamma(y,k,i))\to 0$} as $k\to\infty$.

\begin{lem}
\label{lem:gamma-stable} Let $y\in(c_u,a]\setminus\{\min(q_{n-1},
\hq_{n-1})\}$, $k\ge 0$, and $0\le i\le n-1$. Then every $\bx\in\Gamma(y,k,i)$
has $x_{kn+i}=f(\tau(y))$.

\end{lem}
\begin{proof} In view of Lemma~\ref{lem:crosscut-orbit}, we need only show that
every $\bx\in\Gamma(y,0,0)$ has $x_{0}=f(\tau(y))$. Now
\begin{eqnarray*}
\Gamma(y,0,0) &=& \Psi\left(
\xi'([\bt(\hy,0,0), \bt(y,0,0)],\,\, 1+u(y))
\right)\\ &=& \Psi\left(
\xi'([\thr{q_0,\hy,B^{-1}(\hy),\ldots}, \thr{q_0, y, B^{-1}(y),
    \ldots}],\,\, 1+u(y))
\right).
\end{eqnarray*}

\begin{enumerate}[(a)]
\item Since $\bt(y,0,0)_1=y$ and $\bt(y,0,0)_{1+i}\not\in\ingam$ for all~$i\ge
  1$, Lemma~\ref{lem:constant-block} gives that $\Psi(\bt(y,0,0), s)_0 =
  f(\tau(y))$ for all $s\ge 1+u(y)$; similarly $\Psi(\bt(\hy,0,0),s)_0 =
  f(\tau(\hy))=f(\tau(y))$ for all $s\ge 1+u(y)$.

\item It remains to show that $\Psi(\by, 1+u(y))_0 = f(\tau(y))$ for all
  $\by\in[\bt(\hy,0,0), \bt(y,0,0)]$. Now in the case $y\in(\min(q_{n-1},
  \hq_{n-1}), a]$ we have, by Lemma~\ref{lem:ordering-intervals},
\[ [\bt(\hy,0,0), \bt(y,0,0)] =
\{\bt(y',0,0)\,:\,y'\in[\hy,y]\setminus\{q_{n-1}\}\} \,\,\,\,\cup\,\,\,\,
\{\bq_0\} \,\,\,\,\cup\,\,\,\,
\bigcup_{k=1}^\infty\,\, (L_{k,0}\cup R_{k,0}),
\] while in the case $y\in(c_u, \min(q_{n-1}, \hq_{n-1}))$ we have
$[\bt(\hy,0,0),
  \bt(y,0,0)] = \{\bt(y',0,0)\,:\,y'\in[\hy,y]\}$.

If $y'\in[\hy,y]$ with $y'\not=q_{n-1}$ then
\begin{eqnarray*}
\Psi(\bt(y',0,0), 1+u(y))_0 &=& H(y', \phi^{-1}(f(\tau(y)))) \\ &=& f(\tau(y))
\end{eqnarray*} as required. Here the first equality
uses~(\ref{eq:convenient-2}) and that $(1+u(y))/2 = \phi^{-1}(f(\tau(y)))$,
while the second uses Lemma~\ref{lem:u-key} and the fact that $u(y)\le u(y')$,
since $y'\in[\hy,y]$.

On the other hand, if we are in the case $y\in(\min(q_{n-1},\hq_{n-1}),a]$, and
if $\by=\bq_0$ or $\by$ is in $L_{k,0}$ or $R_{k,0}$ for some~$k\ge 1$, then
$y_1= q_{n-1}$, and~(\ref{eq:convenient-2}) and Lemma~\ref{lem:u-key} give
\[
\Psi(\by, 1+u(y))_0 = H(q_{n-1}, \phi^{-1}(f(\tau(y)))) = f(\tau(y))
\] as required, since $\phi^{-1}(f(\tau(y))) \le u(q_{n-1})$.
\end{enumerate}
\end{proof}

\begin{remark}
\label{rmk:dotted} There are two connected components of dotted lines on
Figure~\ref{fig:gammas}, which are limits of the crosscuts $\Gamma'(y,k,i)$.
One is the arc $\{\bt(c_u,k,i)\}\times[kn+i+2,
  \infty]$, and the other is the union of the crosscut
$\xi'([\bt(\hq_{n-1},k,i), \bt(a, k+1, i)],\,\,kn+i+1+u(q_{n-1}))$ and the
crosscut $\xi'([\bt(\hq_{n-1},k,i), \bt(\re_u, k+1, i)],\,kn+i+1+u(q_{n-1}))$,
each interval in~$\hS$ being the one which contains $\bq_i$.

By the continuity of~$\Psi$ on $\tD$, every point~$(\by,s)$ of the former has
$\Psi(\by,s)_{kn+i} = f(\tau(c_u))=b$, and every point~$(\by,s)$ of the latter
has $\Psi(\by,s)_{kn+i} = f(\tau(q_{n-1})) = f^{n+1}(a)$.
\end{remark}

\begin{lem}
\label{lem:good-chain-NML-rat-interior} \mbox{}
\begin{enumerate}[(a)]
  \item Let~$0\le i\le n-1$. For every sequence~$(y^{(k)})$ in
  $(\min(q_{n-1},\hq_{n-1}),a]$, the sequence $(\Gamma'(y^{(k)},k,i))_{k\ge 0}$
  satisfies conditions (a)\,--\,(c) of Definition~\ref{def:good-chain} (of a
  good chain of crosscuts for~$\bq_i$).
  \item Let~$\by\in\hS$. Then there is a good chain of crosscuts for~$\by$.
\end{enumerate}
\end{lem}
\begin{proof}

The sequence $(\Gamma'(y^{(k)},k,i))_{k\ge 0}$ is a chain of crosscuts in
$(\bD, \hS_\infty)$ which satisfies condition~(a) of
Definition~\ref{def:good-chain} by Lemma~\ref{lem:R-dense-rat-interior}; it
satisfies condition~(b) by Lemma~\ref{lem:ordering-intervals} (see
Remark~\ref{rmk:gammas}); and it satisfies condition~(c) by
Lemma~\ref{lem:gamma-stable}, which gives that $\diam(\Gamma(y,k,i)) \le
|b-a|/2^{kn+i}$.

\medskip

For part~(b) of the lemma, it suffices by Remark~\ref{rmk:good-chain}~(c) to
find a good chain of crosscuts for each~$\bq_i$; that is, to show that we can
choose the sequence~$(y^{(k)})$ in such a way that $(\Gamma(y^{(k)}, k,
i))_{k\ge 0}$ does not converge to a point of~$\hI$. The argument is similar to
that used in the proof of Lemma~\ref{lem:R-dense-rat-interior}.

Pick two distinct points $x,z\in[a, \min(\tau(q_{n-1}),
  \tau(\hq_{n-1}))) = [a, \theta)$, where $\theta = \min(f^n(a),
    \widehat{f^n(a)})$.  By Lemma~\ref{lem:images-of-tight}, there are
    sequences $(x^{(k)})$ and $(z^{(k)})$ in~$[a, \theta)$ with
    $f^{kn}(x^{(k)})=x$ and $f^{kn}(z^{(k)})=z$ for all sufficiently large~$k$.
    Since $x^{(k)}, z^{(k)}\in[a,\theta)$ we have $x_u^{(k)}, z_u^{(k)}\in
    (\min(q_{n-1}, \hq_{n-1}),a]$ for each~$k$. Then, by
    Lemma~\ref{lem:gamma-stable}, every~$\bx\in
      \Gamma(x_u^{(k)}, k, i)$ has $x_{i+1} = x$, and
      every~$\bx\in\Gamma(z_u^{(k)}, k, i)$ has $x_{i+1}=z$, provided that~$k$
      is sufficiently large.

Choosing $y^{(k)}=x_u^{(k)}$ when~$k$ is even, and $y^{(k)}=z_u^{(k)}$ when~$k$
is odd therefore gives a good chain of crosscuts.
\end{proof}

It follows from Theorem~\ref{thm:corr-homeo} and
Lemma~\ref{lem:general-summary} that $\cP\colon\hS\to\PP$ is a homeomorphism,
and that the prime end~$\cP(\by)$ is of the first kind for all~$\by\not\in\bQ$.
It therefore only remains to calculate the principal sets and impressions of
the prime ends $\cP(\bq_i)$. We will do this for $\cP(\bq_{n-1})$: the
analogous results for the other $\cP(\bq_i)$ follow on observing that
$\cP(\bq_i) =
\hH^{i+1}(\cP(\bq_{n-1}))$ for each~$i$ by Lemmas~\ref{lem:crosscut-orbit}
and~\ref{lem:good-chain-NML-rat-interior}.

\begin{lem}
\label{lem:impression-rat-interior} Let $f$ be of rational interior type, with
$q(\kappa(f))=m/n$. Then $\cI(\cP(\bq_{n-1})) = \hI$.
\end{lem}
\begin{proof} By Theorem~\ref{thm:corr-homeo}~(b) and
Lemma~\ref{lem:good-chain-NML-rat-interior}, $\cP(\bq_{n-1})$ is defined by the
chain $(\Gamma(a,k,n-1))_{k\ge 0}$, so it suffices to show that for every fixed
$\bx\in\hI$ and $k\ge 0$, we have $\bx\in\overline{U(\Gamma(a,k,n-1))}$.

By Lemma~\ref{lem:eventually-onto} there is some~$N$ (which we take to be at
least~3) with $f^N([a,\re])\supset f^N([a,c])=[a,b]$. For each $j$ with $jn\ge
N$, we can therefore choose $z^{(j)}\in[a,\re]\setminus\{\tau(q_{n-1})\}$ with
$f^N(z^{(j)}) = x_{jn-N}$. (If $x_{jn-N}=f^N(\tau(q_{n-1}))$ then we also have
$x_{jn-N} = f^N(\tau(\hq_{n-1}))$, and either
$\tau(\hq_{n-1})\not=\tau(q_{n-1})$ or $x_{jn-N}=f^N(c)$. In the latter case we
have $x_{jn-N}=f^{N-2}(a)=f^{N-2}(\re)$, and since~$\re$ has two $f$-preimages
there is some $z^{(j)}\not=c$ with $f^N(z^{(j)})=x_{jn-N}$.)

For each such~$j$, let
\[\by^{(j)} = \bt(z_u^{(j)}, j-1,n-1) =\thr{(q_{n-1},\ldots,q_0)^j, z_u^{(j)},
B^{-1}(z_u^{(j)}), \ldots},
\] which is landing of level~$jn$. By~\eqref{eq:rat-thread-land} we have
$\omega(\by^{(j)})_{jn-N}=f^N(z^{(j)}) = x_{jn-N}$, so that
$\omega(\by^{(j)})\to\bx$ as $j\to\infty$. Since $(\by^{(j)},
\infty)\in \overline{U(\Gamma'(a,k,n-1)}$ for all $j>k$, we have
$\omega(\by^{(j)})\in\overline{U(\Gamma(a,k,n-1))}$ for all $j>k$, and hence
$\bx\in\overline{U(\Gamma(a,k,n-1))}$ as required.
\end{proof}

\begin{lem}
\label{lem:principal-rat-int-non-renorm} Let $f$ be of rational interior type,
with $q(\kappa(f))=q=m/n$. If $\kappa(f)\succ w_q0\Inf{w_q1}$ then
$\Pi(\cP(\bq_{n-1})) = \hI$.
\end{lem}
\begin{proof} Let $\bx\in\hI$. We show that $\bx\in\Pi(\cP(\bq_{n-1}))$ by
exhibiting a chain of crosscuts defining $\cP(\bq_{n-1})$ which converges
to~$\bx$.

By Lemma~\ref{lem:images-of-tight}~(b), there is some~$N\in\N$ with
$f^N([a,\theta))=[a,b]$, where $\theta=\min(f^n(a),
  \widehat{f^n(a)})$. For each~$k$ with $kn\ge N$, pick $z^{(k)}\in[a,\theta)$
  with $f^N(z^{(k)}) = x_{kn-N}$. Since $z^{(k)}\in[a,\theta)$ we have
  $z^{(k)}_u\in(\min(q_{n-1},
      \hq_{n-1}), a]$ for each~$k$. Therefore, by
    Lemma~\ref{lem:good-chain-NML-rat-interior}~(a), $\cP(\bq_{n-1})$ is
    defined by the chain $(\Gamma(z_u^{(k)}, k-1, n-1))_{k\ge N/n}$.

By Lemma~\ref{lem:gamma-stable}, every $\bv\in\Gamma(z^{(k)}_u, k-1, n-1)$ has
$v_{kn-1}=f(z^{(k)})$, and hence $v_{kn-N}=x_{kn-N}$. Therefore
$\Gamma(z^{(k)}_u, k-1,n-1)\to\bx$ as $k\to\infty$ as required.
\end{proof}

\begin{lem}
\label{lem:principal-rat-int-renorm} Let $f$ be of rational interior type, with
$q(\kappa(f))=q=m/n$. If $\kappa(f) \preceq w_q0\Inf{w_q1}$ then
$\Pi(\cP(\bq_{n-1})) =
\{\bx\in\hI\,:\,x_{\ell n}\in[a,f^n(a)]\text{ for all }\ell\ge0\}$.
\end{lem}
\begin{proof} This is a consequence of Lemma~\ref{lem:images-of-tight}~(a),
which states that $f^n([a,f^n(a))) = [a, f^n(a)]$ whenever $\Inf{w_q0} \prec
  \kappa(f) \preceq w_q0\Inf{w_q1}$.

Write $X=\{\bx\in\hI\,:\,x_{\ell n}\in[a,f^n(a)]\text{ for all }\ell\ge0\}$. To
show that $X\subset\Pi(\cP(\bq_{n-1}))$, we exhibit, for each $\bx\in X$, a
chain of crosscuts defining $\cP(\bq_{n-1})$ which converges to~$\bx$. By
Lemma~\ref{lem:images-of-tight}~(a), $f^n(a)_u=\min(q_{n-1}, \hq_{n-1})$, and
for each~$k\ge 0$ there is some $z^{(k)}\in[a, f^n(a))$ with
$f^n(z^{(k)})=x_{kn}$. By Lemma~\ref{lem:gamma-stable}, every
$\bv\in\Gamma(z_u^{(k)}, k, n-1)$ has $v_{(k+1)n-1}=f(z^{(k)})$, and hence
$v_{kn}=x_{kn}$. Therefore $(\Gamma(z^{(k)}, k, n-1))_{k\ge 0}$ is a chain of
crosscuts defining $\cP(\bq_{n-1})$ which converges to~$\bx$.

To show that $\Pi(\cP(\bq_{n-1}))\subset X$, it is enough, by
Theorem~\ref{thm:corr-homeo}~(c), to show that $\Rem(R_{\bq_{n-1}})\subset X$.
To do this, we fix $\ell\ge 0$ and show that $R_{\bq_{n-1}}(s)_{\ell n}\in[a,
f^n(a)]$ for all $s\ge\ell n+1$.

We therefore fix $s\ge\ell n+1$, and write $P(s) = (t,v)$. Recalling that
$\bq_{n-1} =
\thr{(q_{n-1}, \ldots, q_0)^\infty}$, using~(\ref{eq:convenient-2}), and
abbreviating $R_{\bq_{n-1}}$ to $R$:
\begin{enumerate}[(a)]
\item If $t=rn$ for some $r\ge\ell+1$ then
\[ R(s)_{rn-1} = H(q_{n-1},v) \in [f(a), f(\tau(q_{n-1}))] = [f(a),
f^{n+1}(a)].
\] Since $f^{n-1}([f(a), f^{n+1}(a)])=[a,f^n(a)]$ by
Lemma~\ref{lem:images-of-tight}~(a), we have $R(s)_{(r-1)n} \in[a,f^n(a)]$, and
hence $R(s)_{\ell n}\in[a,f^n(a)]$ as required.

\item If $t=rn+i$ for some $r\ge\ell$ and $1\le i\le n-1$ then
\[ R(s)_{rn+i-1} = R(s)_{t-1} = H(q_{n-1-i}, v) = \tau(B(q_{n-1-i})) =
\tau(q_{n-i}) = f^{n-i+1}(a),
\] since $q_{n-1-i}\not\in\ingam$. Therefore $R(s)_{rn} = f^n(a)$, and hence
$R(s)_{\ell n} \in [a, f^n(a)]$ as required.
\end{enumerate}
\end{proof}

\begin{remark} Since $\cP(\bq_i) = \hH^{i+1}(\cP(\bq_{n-1}))$ for $0\le i<
n-1$, it follows that, whenever we have $\Inf{w_q0} \prec \kappa(f) \preceq
w_q0\Inf{w_q1}$,
\[
\Pi(\cP(\bq_i)) = \{\bx\in\hI\,:\, x_{\ell n + i + 1}\in[a,f^n(a)] \text{ for
all }\ell\ge 0\}.
\] These principal sets are therefore homeomorphic to the inverse limit
$\invlim([a, f^n(a)], f^n)$ of the renormalized map.
\end{remark}

The following theorem provides a summary of what we have proved in the rational
interior case.

\begin{thm}[Prime ends in the rational interior case]
\label{thm:summary-rational-interior} Let~$f$ be a unimodal map satisfying the
conditions of Convention~\ref{conv:unimodal}, and suppose that
$q(\kappa(f))=m/n\in(0,1/2)$ is rational, and that
$\kappa(f)\not\in\{\lhe(m/n), (w_{m/n}0)^\infty, \rhe(m/n)\}$. Then
\begin{enumerate}[(a)]
\item All except~$n$ of the prime ends of~$(\hT,\hI)$ are of the first kind;
\item If $\kappa(f) \preceq w_{m/n}0\Inf{w_{m/n}1}$, then the~$n$ remaining
  prime ends are of the fourth kind, with principal set~$\invlim([a,f^n(a)],
  f^n)$ and impression~$\hI$;
\item If $\kappa(f) \succ w_{m/n}0\Inf{w_{m/n}1}$, then the~$n$ remaining prime
ends are of the third kind, with principal set and impression~$\hI$; and
\item The prime end rotation number is $m/n$.
\qed
\end{enumerate}
\end{thm}

\begin{remark} \label{rmk:accessible-rat-interior}
By Theorem~\ref{thm:corr-homeo}~(d), the set of accessible
points of~$\hI$ is precisely $\{\omega(\by)\,:\,\by\in\hS\setminus\bQ\}$. This
set is partitioned into~$n$ immersed copies of the line.
\end{remark}

\subsection{The rational endpoint case}
\label{sec:rat-endpoint} We finish by considering the rational endpoint case,
where
\mbox{$\syma=\kappa(f)$} has rational height $q=q(\syma)\in(0,1/2)$ and
  $\syma=\lhe(q)$, $\syma=\rhe(q)$, or $\syma=\Inf{w_q0}$; or where $q=0$. The
  following theorem summarizes this case.

\begin{thm}[Prime ends in the rational endpoint case]
\label{thm:summary-rational-endpoint} Let~$f$ be a unimodal map satisfying the
conditions of Convention~\ref{conv:unimodal}, and suppose that
$q(\kappa(f))=m/n\in[0,1/2)$ is rational, and that $\kappa(f)\in\{\lhe(m/n),
(w_{m/n}0)^\infty, \rhe(m/n)\}$. Then
\begin{enumerate}[(a)]
\item All except~$n$ of the prime ends of~$(\hT,\hI)$ are of the first kind;
\item The~$n$ remaining prime ends are of the second kind, with
  impression~$\hI$; and
\item The prime end rotation number is $m/n$.
\end{enumerate}
\end{thm}

The arguments in the four cases where~$f$ is of early endpoint, normal
endpoint, quadratic-like strict left endpoint, or late endpoint type are
different, and we consider each of them briefly in turn, pointing out how they
differ from similar arguments in the rational interior and irrational cases,
and leaving the reader to fill in some details.

\subsubsection{The normal endpoint case} In this case either $\syma=\lhe(m/n)$
and $B^n(a)=a$, or $\syma=\rhe(m/n)$ and $B^n(a)=\re_u$; and the orbit
of~$B(a)$ is the only periodic orbit of~$B$. We will consider the case where
$\syma=\lhe(m/n)$: the other case can be treated in exactly the same way. Minor
modifications are needed in the particular case $m/n=0$ (i.e.\ when
$\syma=1\Infs{0}$): we will assume that $m/n>0$.

The analysis starts in the same way as the rational interior case. We write
$q_i=B^{i+1}(a)$ for $0\le i\le n-1$, so that $Q=\{q_0,q_1,\ldots,q_{n-1}\}$ is
a period~$n$ orbit of~$B$, with $q_{n-1}=a\in\gamma$.  Threads $\bq_i$ and
$\bt(y,k,i)$ in~$\hS$, and the periodic orbit~$\bQ$ of $\hB$, are introduced
exactly as in Definitions~\ref{notn:threads-rat-interior}. 

Intervals $R_{k,i}$ can then be constructed as in
Definitions~\ref{def:intervalsL,R,I}. However, since $q_{n-1}=a$, the
intervals~$L_{k,i}$ of the rational interior case are empty. This means that
the intervals~$R_{k, i}$ converge to $\bq_i$ as $k\to\infty$, and to $\bq_j$ as
$k\to-\infty$, where $j=i - m^{-1}\bmod{n}$
(Figure~\ref{fig:circle-order-endpoint}).

\begin{figure}[htbp]
\begin{center}
\includegraphics[width=0.85\textwidth]{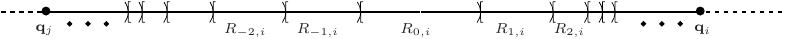}
\end{center}
\caption{The intervals~$R_{k, i}$ in the normal endpoint case when $B^n(a)=a$.}
\label{fig:circle-order-endpoint}
\end{figure}

Since the threads~$\bq_i$ do not contain any entries from~$\ingam$, the points
of~$\bQ$ are landing of level~0, and hence $\cL=\hS$. On the other hand,
$\cR=\hS\setminus\bQ$, since the interior points of~$R_{k,i}$ are not landing
of any level less than $kn+i+1$.

The construction of good chains of crosscuts for each~$\bq_i$ is reminiscent of
the irrational gap endpoint case.

\begin{lem}
\label{lem:good-crosscut-rat-endpoint} Let $0\le i\le n-1$.
Write~$V=\bigcup_{k<0}R_{k, i+m^{-1}\bmod{n}}$, and let $\left(\by^{(k)}\right)$
be any sequence in $V$ which converges strictly monotonically to~$\bq_i$. For
each~$k\ge 1$, let $\bJ_k$ be the interval in~$\hS$ with endpoints $\by^{(k)}$
and $\bt(\re_u, k, i)$ which contains $\bq_i$. Let
\[\xi'_k = \xi'(\bJ_k, \,nk+i).\] Then $(\xi'_k)$ is a good chain of crosscuts
for~$\bq_i$.
\end{lem}

\begin{proof} Conditions (a) and (b) of Definition~\ref{def:good-chain} are
immediate, and condition~(d) is vacuous. It is therefore only necessary to show
that $\diam(\xi_k)\to 0$ as $k\to\infty$, where $\xi_k=\Psi(\xi'_k)$. 

We have $\diam(\xi_k \cap \Psi\left(\overline{V}
\times[0,\infty]\right))\to 0$, since $\Psi$ is continuous on $\overline{V}
\times[0,\infty]$ by Lemma~\ref{lem:Psi-extension-continuous}. To show that the
diameters of the remaining parts of $\xi_k$ go to zero, we will show that
every~$\bx$ belonging to the arc $\Psi(\{\bt(\re_u,k,i)\} \times [nk+i,
  \infty])$ or to the arc $\Psi((\bJ_k\setminus V) \times\{nk+i\})$ has
$x_{nk+i-1}=\tau(q_1)$.

For the former, we have $\Psi(\bt(\re_u,k,i),s)_{nk+i-1} =
\tau(\bt(\re_u,k,i)_{nk+i-1}) = \tau(q_1)$ for all~$s\ge nk+i$ by
Lemma~\ref{lem:formula-disjoint-gamma}, since $\bt(\re_u,k,i)$ is landing of
level~0 (and hence of level~$nk+i-1$).

For the latter, observe that if $\by\in \bJ_k\setminus V$ then $y =
\thr{q_i,\ldots,q_0,(q_{n-1},\ldots,q_0)^k,\ldots}$ and so $y_{nk+i}=q_0$.
Therefore, by~(\ref{eq:convenient-2}), $\Psi(\by, nk+i)_{nk+i-1} = H(q_0, 1/2)
= \tau(q_1)$ as required.
\end{proof}

It follows from Theorem~\ref{thm:corr-homeo} and
Lemma~\ref{lem:general-summary} that $\cP\colon \hS\to\PP$ is a homeomorphism;
that the prime end $\cP(\by)$ is of the first kind for all $\by\not\in\bQ$; and
that $\Pi(\cP(\by))=\{\omega(\by)\}$ is a point for $\by\in\bQ$. It therefore
only remains to calculate the impressions of the prime ends $\cP(\bq_i)$. The
proof of the following result works in exactly the same way as that of
Lemma~\ref{lem:impression-rat-interior}, using the chain of crosscuts from
Lemma~\ref{lem:good-crosscut-rat-endpoint} in place of the chain $(\Gamma(a, k,
n-1))_{k\ge 0}$.

\begin{lem}
\label{lem:strict-impression} 
Let~$f$ be of normal endpoint type, with $q(\kappa(f))=m/n$. Then
$\cI(\cP(\bq_{n-1}))=\hI$. \qed
\end{lem}

\subsubsection{The quadratic-like strict left endpoint case}
\label{sec:quadratic-like-endpoint}
In this case $\mu=\lhe(m/n)$ and $B^n(a)=a$, but~$B$ has a second period~$n$
orbit~$P$ in addition to the orbit of~$B(a)$. As in the rational interior case,
we write $q_i= B^{i+1}(a)$ for $0\le i\le n-1$, $p_0$ for the point of~$P$
between $q_0$ and $q_{m^{-1}\bmod{n}}$, and $p_i=B^i(p_0)$ for $1\le i\le n-1$.
Threads $\bq_i$, $\bp_i$, and $\bt(y,k,i)$ in $\hS$, and the periodic orbits
$\bQ$ and $\bP$ of~$\hB$ are introduced exactly as in
Definitions~\ref{notn:threads-rat-interior}. However, since the $B$-orbits of
points in each interval $(q_i, p_i]$ are disjoint from~$\gamma$, there are
 threads
\[
  \bt(y) = \thr{y, B^{-1}(y), B^{-2}(y),\ldots} \qquad\text{ for }y\in
  \bigcup_{i=0}^{n-1}(q_i, p_i]
\]
in $\hS$ (with $\bt(p_i)=\bp_i$). We write $I_i = \{\bt(y)\,:\,y\in(q_i,
p_i]\}$ for $0\le i\le n-1$, half-open intervals in~$\hS$ with
$\hB(I_i)=I_{i+1\bmod{n}}$. Defining half-open intervals~$R_{k,i}$ as in
Definitions~\ref{def:intervalsL,R,I}, the intervals are arranged around~$\hS$
as depicted in Figure~\ref{fig:quadratic-order-endpoint} (where $j=i -
m^{-1}\bmod{n}$).

\begin{figure}[htbp]
\begin{center}
\includegraphics[width=0.85\textwidth]{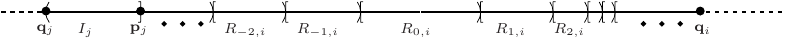}
\end{center}
\caption{Intervals around~$\hS$ in the quadratic-like left strict endpoint
case.}
\label{fig:quadratic-order-endpoint}
\end{figure}

The remainder of the analysis proceeds exactly as in the normal endpoint case,
except that in the statement of Lemma~\ref{lem:good-crosscut-rat-endpoint} we
take $V=I_i$ rather than $V=\bigcup_{k<0}R_{k, i+m^{-1}\bmod{n}}$.

\subsubsection{The late endpoint case}
\label{sec:late-endpoint} 
Here $q=m/n>0$ and $\syma=\Inf{w_q0}$. In this case $q_{n-1}=B^n(a)\in\ingam$
by Theorem~\ref{thm:outside-dynamics}~(b)(ii), and the treatment is identical
to that of the rational interior case up until
Lemma~\ref{lem:R-dense-rat-interior}. Here, because
Lemma~\ref{lem:images-of-tight} doesn't apply when $\kappa(f)=\Inf{w_q0}$, the
proof that $\cL=\hS\setminus\bQ$ breaks down. Instead we have:

\begin{lem}
\label{lem:late-endpoint-landing} Let~$f$ be of late endpoint type. Then
$\cL=\hS$.
\end{lem}
\begin{proof} The points~$a$ and $f^n(a)$ are distinct, but both have itinerary
$\sigma(\Inf{w_q0})$. Since $w_q0$ is a word of length~$n$ with an odd number
of $1$s, $f^n|_{[a,f^n(a)]}$ is decreasing, with $a<f^{2n}(a)<f^n(a)$. There is
therefore a unique fixed point~$p$ of $f^n$ in $[a, f^n(a)]$. Now the
increasing map $f^{2n}\colon [a, f^n(a)]\to[a,f^n(a)]$ also has~$p$ as its
unique fixed point (any other fixed points would be period~2 points of~$f^n$
and so would come in pairs, contradicting Convention~\ref{conv:unimodal}~(c)),
so that $f^{kn}(x)\to p$ as $k\to\infty$ for every $x\in[a,f^n(a)]$.

Since every point of $\hS\setminus\bQ$ is landing of some level, it is only
necessary to prove that the rays $R_{\bq_i}$ land. It is enough to show this
for $i=n-1$ since $R_{\bq_i} = \hH^{i+1}\circ R_{\bq_{n-1}}$ for $0\le i\le
n-2$.

Fix $r\ge 0$ and let~$s\ge r+1$. Write $P(s)=(t,v)$ and $t=kn+i$ with $0\le
i\le n-1$: then~(\ref{eq:convenient-2}) gives $\Psi(\bq_{n-1},s)_r =
f^{t-1-r}(H(q_{n-1-i}, v))$. If $i\not=0$ then $H(q_{n-1-i}, v) = \tau(q_{n-i})
= f^{n-i+1}(a)$, so that $\Psi(\bq_{n-1}, s)_r = f^{(k+1)n-r}(a)$. On the other
hand, if $i=0$ then $H(q_{n-1-i},v) \in[f(a), f^{n+1}(a)]$ by
Definition~\ref{def:unwrapping-unimodal}, so that $\Psi(\bq_{n-1}, s)_r =
f^{kn-r}(x)$ for some $x\in[a, f^n(a)]$.

Therefore $\Psi(\bq_{n-1}, s)_r \to f^{n-r}(p)$ as $s\to\infty$. It follows
that
\[
\Psi(\bq_{n-1},s) \to \omega(\bq_{n-1}) = \thr{(p, f^{n-1}(p), \ldots,
  f(p))^\infty} \text{ as } s\to\infty,
\] so that $\bq_{n-1}\in\cL$ as required.
\end{proof}

\begin{remark}
\label{rmk:unusual-landing} 
Thus, in the late endpoint case, the points of $\bQ$ are landing, despite not
being landing of any level. This is the only case in which
$\cL\not=\bigcup_{N\ge 0}\cL_N$.
\end{remark}

Since the interior points of $R_{k,i}$ are not landing of any level less
than~$kn+i+1$, the locally uniformly landing set is given by
$\cR=\hS\setminus\bQ$. The proof of
Lemma~\ref{lem:good-chain-NML-rat-interior}~(a) goes through without change to
show that $(\Gamma'(a, k, i))_{k\ge0}$ is a good chain of crosscuts for~$\bq_i$
(observing that condition~(d) of Definition~\ref{def:good-chain} is vacuous);
and the proof of Lemma~\ref{lem:impression-rat-interior} likewise carries over
to show that $\cI(\cP(\bq_{n-1})) = \hI$.

\subsubsection{The early endpoint case}
\label{sec:early-endpoint} In this case $q=m/n>0$ and $\mu=\lhe(m/n)$, but
$f^n(a)\not=a$. According to Theorem~\ref{thm:outside-dynamics}~(c)(ii), the
orbit $\cO=\{B^r(a)\,:\,r\ge 1\}$ is disjoint from~$\gamma$, and is attracted
to a period~$n$ orbit~$Q\subset S\setminus\gamma$ of~$B$: in particular,
$B^{-r}(y)$ is defined for all $r\ge 0$ provided that $y\not\in\cO$. There are
two possibilities: either~$Q$ is semi-stable and is the only periodic orbit
of~$B$, in which case the backwards orbit $\{B^{-r}(\gamma)\,:\,r\ge 0\}$
of~$\gamma$ is attracted to~$Q$; or~$Q$ is stable, and there is a repelling
period~$n$ orbit~$P\subset S\setminus\gamma$ of~$B$, which attracts the
backwards orbit of~$\gamma$.

The analysis initially follows that of the irrational case. Elements of $\hS$
can be written either as $\bt(y,r)$, with $y\in\gamma$ and $r\in\Z$; or as
$\bt(y)$ with $y\in S\setminus\bigcup_{r\in\Z} B^{-r}(\gamma)$, these threads
being defined exactly as in Definitions~\ref{notn:threads-irrat}. It follows,
as in the proof of Lemma~\ref{lem:landing-irrational}, that $\cL=\hS$. ``Gaps''
$G_r=\{\bt(y,r)\,:\,y\in\gamma\}$ can be defined as in
Definition~\ref{def:gaps}.

The difference with the irrational case is that the gaps $G_r$ converge as
$r\to\infty$ to the periodic orbit~$\bQ$ of~$\hB$ corresponding to~$Q$; and
they converge as $r\to-\infty$ either to $\bQ$ from the other side (in the case
where~$Q$ is the unique periodic orbit of~$B$), or to the periodic orbit~$\bP$
of~$\hB$ corresponding to~$P$. Since $G_r$ is uniformly landing of level
$\max(r,0)$, we have in either case that $\cR=\hS\setminus\bQ$.

The construction of a good chain of crosscuts for each point $\bq$ of $\bQ$ can
be carried out in exactly the same way as in the irrational case
(Lemma~\ref{lem:irrat-good-chain}); and the proof that $\cI(\cP(\bq))=\hI$ for
each such~$\bq$ is identical to the proof of
Lemma~\ref{lem:irrational-prime-ends-type-2}.

\begin{remark} \label{rmk:accessible-rat-endpoint}
By Theorem~\ref{thm:corr-homeo}~(d), the set of accessible
points of~$\hI$ is precisely $\{\omega(\by)\,:\,\by\in\hS\}$. Since the landing
function~$\omega$ is continuous from one side, but not from the other, at the
points of~$\bQ$, the set of accessible points is partitioned into~$n$ immersed
copies of $[0,\infty)$.
\end{remark}

\section{Semi-conjugacy to sphere homeomorphisms}
\label{sec:semi-conj} 

In this section we will prove (Theorem~\ref{thm:cont-var}) that if $\{f_t\}$ is
a continuously varying family of unimodal maps, then there is a corresponding
family $\{\chi_t\colon S^2\to S^2\}$ of sphere homeomorphisms such that each
$\chi_t$ is a factor of $\hf_t\colon \hI_t\to\hI_t$ by a semi-conjugacy with
mild point preimages. In order to simplify the exposition, we start by treating
the case of a single unimodal map~$f$ (Theorem~\ref{thm:semi-conj}).

We will also show (Theorem~\ref{thm:gpA}) that if~$\{f_t\}$ is a family of tent
maps, then $\chi_t$ is a {\em generalized pseudo-Anosov map} for those values
of~$t$ for which $f_t$ is post-critically finite (and is pseudo-Anosov
when~$f_t$ is of NBT type). Therefore, in the tent map case, $\{\chi_t\}$ is a
completion of the family of generalized pseudo-Anosovs constructed
in~\cite{gpa}.

\medskip

In order to construct the semi-conjugacy, we will define a non-separating
monotone upper semi-continuous decomposition~$\cG$ of~$\hT$, whose elements are
permuted by~$\hH$, and each of which intersects $\hI$. By Moore's
theorem~\cite{moore}, the quotient space $\Sigma=\hT/\cG$ is again a sphere,
and $F=\hH/\cG\colon
\Sigma\to\Sigma$ is a homeomorphism. Since each of the decomposition elements
intersects~$\hI$, the natural projection $\hT\to\Sigma$ induces a surjection
$\hI\to\Sigma$, which semi-conjugates $\hf=\hH|_{\hI}$ to~$F$.

To define the decomposition~$\cG$, we first introduce the {\em strongly stable}
  equivalence relation on~$\sD$ (Definition~\ref{def:strong-stable}). (Recall
  that~$\sD=D\cup\cL_\infty\subset\bD$ is the maximal domain of~$\Psi$.) The
  idea is that a strongly stable component of~$\hT$ is a maximal connected
  subset~$X$ of~$\hT$ with the property that, for all $\bx^{(1)},\bx^{(2)}\in
  X$, there is some $N\ge 0$ with $\hH^N(\bx^{(1)})_0 = \hH^N(\bx^{(2)})_0$. A
  consequence of this is that $d(\hH^i(\bx^{(1)}), \hH^i(\bx^{(2)}))\to 0$ as
  $i\to\infty$, so that strongly stable sets are stable: the converse is not
  true in general, since the unimodal map~$f$ may itself have non-trivial
  connected stable sets. Now such a subset~$X$ may intersect~$\hI$ (and hence
  leave the image of~$\Psi$) many times. A strongly stable component of~$\sD$
  is a component of the preimage $\Psi^{-1}(X)$. Our decomposition will be
  based on these components.

In Lemmas~\ref{lem:ssc-irrational}, \ref{lem:rat-general-strong-stable},
\ref{lem:rat-NBT-strong-stable}, \ref{lem:rat-endpoint-strong-stable}, and
\ref{lem:rat-quadratic-strong-stable} we describe the structure of the strongly
stable equivalence classes for each of the types of unimodal map of
Definition~\ref{def:type}. We then use these to construct a
decomposition~$\cG'$ of~$\bD$: one of the decomposition elements is the union
of $\partial'$ and all of the strongly stable equivalence classes whose closure
contains $\partial'$, while all of the other decomposition elements are single
strongly stable equivalence classes or single points not in~$\sD$. The
decomposition~$\cG$ is obtained by carrying over $\cG'$ with~$\Psi$, and adding
single inaccessible points of $\hI$ (which, by
Theorem~\ref{thm:corr-homeo}~(d), are precisely the points which are not in the
image of~$\Psi$).

\medskip

\begin{defn}[Strongly stable, strongly stable component]
\label{def:strong-stable} A subset~$Y$ of $\sD$ is said to be {\em strongly
stable} if, for all $\eta_1, \eta_2\in Y$, there is some~$N\ge 0$ such that
$\hH^N(\Psi(\eta_1))_0=\hH^N(\Psi(\eta_2))_0$.

The {\em strongly stable component} of $\eta\in\sD$ is the largest
connected strongly stable set which contains~$\eta$ (i.e.\ the union of all
such connected strongly stable sets).
\end{defn}

\begin{remarks}\mbox{}

\begin{enumerate}[(a)]
  \item 
    Since~$G$ and $\hH$ are topologically conjugate
    (Corollary~\ref{cor:Psi-conjugates}), the homeomorphism $G\colon
    \sD\to\sD$ permutes the strongly stable components.
  \item 
    $\{\partial'\}$ is a strongly stable component, since if $(\by,
    s)\not=\partial'$ then $\hH^N(\Psi(\partial'))_0 = \partial \not=
    \hH^N(\Psi(\by,s))_0$ for all $N\ge 0$.
\end{enumerate}
\end{remarks}

\subsection{Strongly stable components in the irrational and the rational early
endpoint cases}
\label{sec:semi-conj-irrat}

Recall from Section~\ref{sec:irrational} that if $f$ is of irrational type then
$\cL=\hS$, so that $\sD=\bD$; that $\hB\colon\hS\to\hS$ is a Denjoy
counterexample, having an orbit~$\{G_r\,:\,r\in\Z\}$ of wandering intervals;
and that the complement of the union of the interiors of these intervals is a
Cantor set~$\Lambda$, which is the set of points which are not locally
uniformly landing.

If $f$ is of rational early endpoint type (Section~\ref{sec:early-endpoint})
then the description is the same as in the irrational case, except that the
orbit of the intervals~$G_r$ converges as $r\to\infty$ and as $r\to -\infty$ to
periodic orbits~$\bQ$ and $\bP$ of $\hB$, and $\cR=\hS\setminus\bQ$. If $\bQ$
and $\bP$ are distinct, then the former is stable and the latter is unstable;
while if $\bP=\bQ$, then this is a semi-stable orbit, which is the limit on
one side of the intervals $G_r$ as $r\to\infty$, and on the other side of the
intervals $G_r$ as $r\to -\infty$.

The following lemma is illustrated in the irrational case by
Figure~\ref{fig:irrat-strong-stab}. The picture in the early endpoint case is
discussed in Remark~\ref{rmk:early-endpoint-strong-stable}.

\begin{lem}
\label{lem:ssc-irrational} Let $f$ be of irrational type or of rational early
endpoint type. Then the strongly stable components of~$\sD$ are
$\{\partial'\}$ and:
\begin{enumerate}[(a)]
\item for each $\by\in\hS\setminus\bigcup_{r\in\Z}G_r$, the line
  $L_{\by}=\{\by\}\times(0,\infty]$;
\item for each $r\in\Z$:
\begin{enumerate}[(i)]
\item the arc $A_r=\{\bt(c_u,r)\}\times [s_r, \infty]$; where
  $s_r=\lambda^r(1)$, i.e.
\[ s_r =
\begin{cases} r+1 & \text{ if }r\ge 1,\\ 1/2^{|r|} & \text{ if }r\le 0.
\end{cases}
\]
\item for each $y\in(c_u,a)$, the crosscut $C_{r,y}=\xi'(\bJ_{r,y}, t_{r,y})$;
  where $\bJ_{r,y}\subset G_r$ has endpoints $\bt(y,r)$ and $\bt(\hy,r)$, and
  $t_{r,y} = \lambda^r((1+u(y))/2)$, i.e.
\[ t_{r,y} =
\begin{cases} r+u(y) & \text{ if }r\ge 1,\\ (1+u(y))/2^{|r|+1} & \text{ if
}r\le 0.
\end{cases}
\]
\item the union~$D_r$ of the arcs $\bt(a,r)\times[u_r, \infty]$ and
  $\bt(\re_u,r)\times[u_r,\infty]$, and the set $G_r\times(0,u_r]$; where $u_r
  = s_{r-1} = \lambda^r(1/2)$, i.e.
\[ u_r=
\begin{cases} r & \text{ if }r\ge 1,\\ 1/2^{|r|+1} & \text{ if }r\le 0.
\end{cases}
\]
\end{enumerate}
\end{enumerate}
\end{lem}

\begin{figure}[htbp]
\begin{center}
\includegraphics[width=0.7\textwidth]{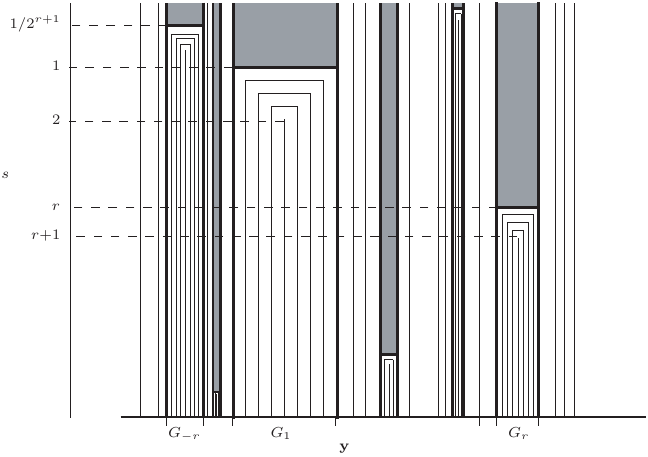}
\end{center}
\caption{Strongly stable components in the irrational case. The types of
  components are: (a) vertical lines~$L_\by$ above the buried points~$\by$ of
  the Cantor set~$\Lambda$; (b)~packets of crosscuts~$C_{r,y}$ above each
  gap~$G_r$, together with the ``central'' arc~$A_r$ which only intersects
  $\hS_\infty$ at a single point; and (c) for each gap $G_r$, the set $D_r$
  consisting of the outermost crosscut above $G_r$ together with the rectangle
  above this crosscut extending up to (but not including) $s=0$. Each set~$D_r$
  is shown as the union of a shaded rectangle and the bold arcs which intersect
  it.}
\label{fig:irrat-strong-stab}
\end{figure}

\begin{proof} We first show that each of the sets listed is strongly stable.
\begin{enumerate}[(a)]
\item Let $\by\in\hS\setminus\bigcup_{r\in\Z}G_r$, so that $\by=\thr{y,
  B^{-1}(y), B^{-2}(y), \ldots}$ for some~$y$ whose orbit under~$B$ is disjoint
  from~$\gamma$. Since~$\by$ is landing of level~0,
  Lemma~\ref{lem:formula-disjoint-gamma} gives $\Psi(\by,s)_0=\tau(y)$ for all
  $s\ge 1$. Applying~$G$ repeatedly (or arguing directly using that
  $\Psi(\by,s)=\thr{(y,s), (B^{-1}(y),s/2), \ldots}$ for $s\in[0,1)$) gives
  that, for each $m\ge 1$, $\hH^m(\Psi(\by,s))_0 = f^m(\tau(y))$ for all $s\in
  [1/2^m,
      \infty]$. Therefore $L_{\by}$ is strongly stable.
\item $G(A_r)=A_{r+1}$, $G(C_{r,y})=C_{r+1,y}$, and $G(D_r)=D_{r+1}$ for all
  $r\in\Z$ and $y\in(c_u,a)$. Since $G$ permutes strongly stable components, it
  suffices to consider the case~$r=1$.
\begin{enumerate}[(i)]
\item We have $\bt(c_u,1)=\thr{B(a), c_u, B^{-1}(c_u),\ldots}$, which is
  landing of level~1. By Lemma~\ref{lem:formula-disjoint-gamma},
  $\Psi(\bt(c_u,1), s)_0 = f(\tau(c_u))=b$ for all $s\in[2,\infty] =
  [s_1,\infty]$, so that $A_1$ is strongly stable.
\item Let $y\in(c_u,a)$. Since $\bt(y,1)_1=y$ and $\bt(y,1)_{1+i}\not\in\ingam$
  for all $i\ge 1$, Lemma~\ref{lem:constant-block} gives that
  $\Psi(\bt(y,1),s)_0 = f(\tau(y))$ for all $s\ge 1+u(y)$; similarly
  $\Psi(\bt(\hy,1),s)_0=f(\tau(\hy))=f(\tau(y))$ for all such~$s$.

Now let $z\in[\hy,y]$, so that $(\bt(z,1),\, 1+u(y))$ is on the horizontal
segment of the crosscut. Then $\Psi(\bt(z,1),\, 1+u(y))_0 = H(z,
\phi^{-1}(f(\tau(y))))$ by~(\ref{eq:convenient-2}) and the definition
of~$u(y)$. Since $z\in[\hy,y]$ we have $\phi^{-1}(f(\tau(y))) \le
\phi^{-1}(f(\tau(z)))$, so that $H(z, \phi^{-1}(f(\tau(y)))) = f(\tau(y))$ by
  Lemma~\ref{lem:u-key}, as required.

Therefore $\Psi(\eta)_0 = f(\tau(y))$ for all $\eta\in C_{1,y}$, so that
$C_{1,y}$ is strongly stable.

\item We have $\Psi(\by,s)_0=f(\tau(a))=f(a)$ for $(\by,s)\in
  \{\bt(a,1),\bt(\re_u,1)\}\times[1, \infty]$ as in~(ii), and hence
  $(\hH^m(\Psi(\by,s)))_0 = f^{m+1}(a)$ for all such $(\by,s)$ and all $m\ge
  0$.

 Now suppose that $s\in[1/2^m, 1/2^{m-1})$ for some $m\ge 1$, and that $\by\in
   G_1$, so that we have $\by=\thr{B(a), y, B^{-1}(y), \ldots}$ for some
   $y\in\gamma$. Then $\Psi(\by, s)_0=(B(a),s)$ by~(\ref{eq:Psi}), so that
   \[\hH^m(\Psi(\by,s))_0=\tau(B^{m+1}(a)) = f^{m+1}(a),\]
   using~(\ref{eq:commute}) and that the orbit of~$B(a)$ is disjoint
   from~$\gamma$. Therefore $D_1$ is strongly stable.
\end{enumerate}
\end{enumerate} The proof that there are no connected strongly stable sets
which strictly contain one of these sets is a routine consideration of cases.
We will only show that $A_1$ is a strongly stable component, and omit the
entirely analogous proofs in the other cases.

From the argument above, we have $\Psi(\eta)_0 = b$ for all $\eta\in A_1$.
Therefore, if $\eta'\in\sD$ satisfies $\hH^N(\Psi(\eta'))_0 =
\hH^N(\Psi(\eta))_0$ for some $N\ge 0$ then $f^N(\Psi(\eta')_0) = f^N(b)$.
There are therefore only countably many possible values which $\Psi(\eta')_0$
can take if $\eta'$ is in the strongly stable component containing~$A_1$.

Now any connected set~$Y$ which strictly contains~$A_1$ must intersect
$C_{1,y}$ for all~$y$ in some interval $(c_u, e)\subset (c_u,a)$. Since
$\Psi(\eta')_0=f(\tau(y))$ for all $\eta'\in C_{1,y}$, and since~$f$ is not
locally constant, it follows that~$\{\Psi(\eta')_0\,:\,\eta'\in Y\}$ is
uncountable, and hence~$Y$ cannot be strongly stable.
\end{proof}

\begin{remark}
\label{rmk:early-endpoint-strong-stable} In the early endpoint case, the
strongly stable components above each interval~$G_r$, and above points $\by$ of
$\hS\setminus\bigcup_{r\in\Z}G_r$, are exactly as depicted in
Figure~\ref{fig:irrat-strong-stab}, but the intervals~$G_r$ are arranged
differently. For each $0\le i\le n-1$, the intervals $G_{i+kn}$ converge
strictly monotonically to a point of~$\bQ$ as $k\to\infty$, and the intervals
$G_{i-kn}$ converge strictly monotonically to a point of~$\bP$. The open
intervals between each $G_{i+kn}$ and $G_{i+(k+1)n}$ are contained in
$\hS\setminus\bigcup_{r\in\Z}G_r$, so that the strongly stable components above
them are vertical lines. If $\bQ$ and $\bP$ are distinct, then there are also
intervals with one endpoint in~$\bQ$ and one in~$\bP$ which are likewise
contained in $\hS\setminus\bigcup_{r\in\Z}G_r$.
\end{remark}

\subsection{Strongly stable components in the rational case}
\label{sec:semi-conj-rat} 
Consider now the case where~$f$ is of rational type but not of early endpoint
type, and let $q(\kappa(f))=m/n\in\Q\cap[0,1/2)$. Recall that we write
$q_{n-1}=B^n(a)\in\gamma$. We will treat in turn the general case together with
the late endpoint case (Lemma~\ref{lem:rat-general-strong-stable}), the NBT
case (Lemma~\ref{lem:rat-NBT-strong-stable}), the normal endpoint case
(Lemma~\ref{lem:rat-endpoint-strong-stable}), and the quadratic-like strict
left endpoint case (Lemma~\ref{lem:rat-quadratic-strong-stable}).

Recall that in the interior case, we have $\cL=\cR = \hS\setminus\bQ$, while in
the endpoint case we have $\cL=\hS$ and $\cR = \hS\setminus\bQ$; and that in
the interior, quadratic-like endpoint, and late endpoint cases, there is a
second period~$n$ orbit~$\bP$ of $\hB\colon\hS\to\hS$, while in the normal
endpoint case $\bQ$ is the only periodic orbit of $\hB$.

The following lemma is illustrated by Figure~\ref{fig:rat-gen-strong-stab}.

\begin{lem}
\label{lem:rat-general-strong-stable} Let~$f$ be of rational general or late
endpoint type, with $q(\kappa(f))=m/n\in(0,1/2)\cap\Q$.  Then the strongly
stable components of $\sD$ are $\{\partial'\}$ and:
\begin{enumerate}[(a)]
\item for each $\bp\in\bP$, the line $L_{\bp}=\{\bp\}\times(0,\infty]$;
\item for each $k\in\Z$ and $0\le i\le n-1$:
\begin{enumerate}[(i)]
\item the arc $A_{k,i}=\{\bt(c_u, k, i)\} \times [r_{k,i}, \infty]$, where
  $r_{k,i} = \lambda^{kn+i}(2)$, i.e.
\[ r_{k,i} =
\begin{cases} kn+i+2 & \text{ if } k\ge 0,\\ 1/2^{|k|n-i-1} & \text{ if } k<0.
\end{cases}
\]
\item for each $y\in(c_u,a]\setminus\{\min(q_{n-1},\hq_{n-1})\}$, the crosscut
 $\Gamma'(y,k,i)$.
\item the union~$D_{k,i}$ of
\begin{itemize}
\item the crosscut $\xi'([\bt(\hq_{n-1},k,i), \bt(a, k+1, i)],\,\,u_{k,i})$;
\item the crosscut $\xi'([\bt(\hq_{n-1},k,i), \bt(\re_u, k+1,
  i)],\,\,u_{k,i})$; and
\item the set $[\bt(a, k+1, i), \bt(\re_u, k+1, i)] \,\,\,\times\,\,\,
  [u_{k,i}, v_{k,i}]$.
\end{itemize} Here $u_{k,i}=\lambda^{kn+i}(1+u(q_{n-1}))$,
$v_{k,i}=\lambda^{kn+i}(n+1)$, and all three of the intervals in~$\hS$ are
those with the given endpoints which are disjoint from~$\bP$.
\end{enumerate}
\end{enumerate}
\end{lem}

\begin{figure}[htbp]
\begin{center}
\includegraphics[width=\textwidth]{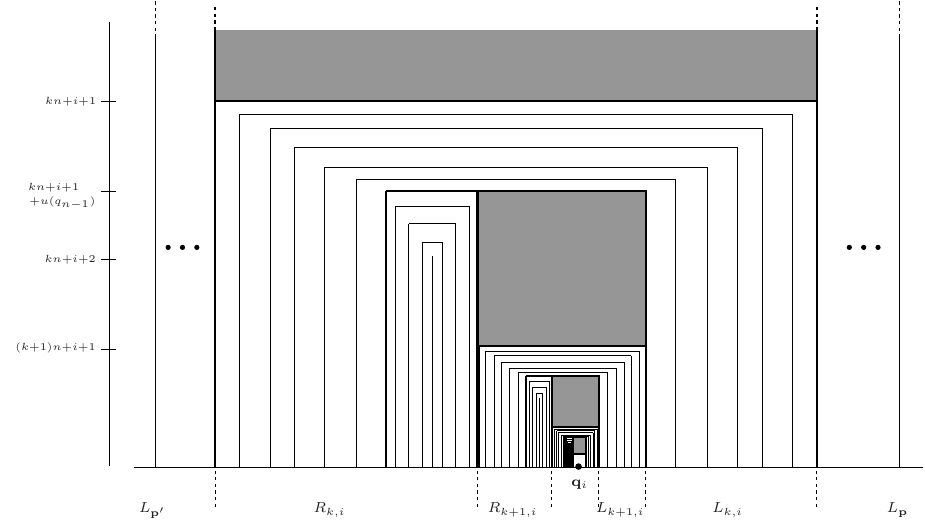}
\end{center}
\caption{Strongly stable components in the rational general or late endpoint
  cases. The types of components are: (a) vertical lines $L_{\bp}$ above the
  points~$\bp$ of the periodic orbit~$\bP$; (b) crosscuts~$\Gamma'(y,k,i)$
  joining points of each~$R_{k,i}$ to the corresponding points of $L_{k,i}$;
  (c) packets of crosscuts $\Gamma'(y,k,i)$ above the subinterval of $R_{k,i}$
  consisting of points which don't correspond to points of $L_{k,i}$, together
  with the ``central'' arc $A_{k,i}$ which only intersects $\hS_\infty$ at a
  single point; and (d) the sets $D_{k,i}$ which consist of a shaded region
  together with the bold arcs connected to it, and which intersect $\hS_\infty$
  at three points. The figure depicts the case where $q_{n-1}\in(c_u,a)$; if
  $q_{n-1}\in (\re_u,c_u)$, then the packets of crosscuts are in $L_{k,i}$
  rather than in $R_{k,i}$. The labeling on the $s$-axis is for the case $k\ge
  0$.}
\label{fig:rat-gen-strong-stab}
\end{figure}

\begin{proof}
\begin{enumerate}[(a)]
\item If $\bp\in\bP$ then $\bp$ is landing of level~0, and the proof that
  $L_{\bp}$ is strongly stable is identical to that of part~(a) of
  Lemma~\ref{lem:ssc-irrational}.
\item Since $A_{k,i}=G^{kn+i}(A_{0,0})$,
  $\Gamma'(y,k,i)=G^{kn+i}(\Gamma'(y,0,0))$, and $D_{k,i} = G^{kn+i}(D_{0,0})$
  for all $k\in\Z$, $0\le i\le n-1$, and $y\in (c_u,a]\setminus\{\min(q_{n-1},
  \hq_{n-1})\}$, it suffices to consider the case $k=i=0$.

That the sets~$A_{0,0}$ and~$\Gamma'(y,0,0)$, and the crosscuts of~$D_{0,0}$,
  are strongly stable is immediate from Lemma~\ref{lem:gamma-stable} and
  Remark~\ref{rmk:dotted}, which gives that $\Psi(\eta)_{0} = b$ for all
  $\eta\in A_{0,0}$; $\Psi(\eta)_{0}=f(\tau(y))$ for all
  $\eta\in\Gamma'(y,0,0)$; and $\Psi(\eta)_{0}=f(\tau(q_{n-1}))$ for all~$\eta$
  in the crosscuts of~$D_{0,0}$.

To complete the proof that~$D_{0,0}$ is strongly stable, it is therefore only
required to show that for all $\by\in[\bt(a,1, 0),\,\bt(\re_u,1,0)]$ and all
$s\in[1+u(q_{n-1}), n+1]$ we have $\Psi(\by,s)_{0} = f(\tau(q_{n-1}))$. Given
$\by\in [\bt(a,1, 0),\,\bt(\re_u,1,0)]=
\{\bq_0\}\cup\bigcup_{k=1}^\infty(L_{k,0}\cup R_{k,0})$, we have $y_0=q_0$ and
$y_r = q_{n-r}$ for $1\le r\le n$. Since $y_1=q_{n-1}$ and
$y_{1+i}\not\in\ingam$ for $1\le i\le n-1$, Lemma~\ref{lem:constant-block}
gives that $\Psi(\by,s)_0 = f(\tau(q_{n-1}))$ for all $s\in[1+u(q_{n-1}), n+1]$
as required.
  
\end{enumerate} The proof that there are no connected strongly stable sets
which strictly contain one of these sets proceeds in the same way as in the
proof of Lemma~\ref{lem:ssc-irrational}.
\end{proof}

In the NBT case, where $q_{n-1}=c_u$, the interval $(c_u,\min(q_{n-1},
\hq_{n-1}))$ degenerates, leaving the simpler situation described in the
following lemma, whose proof works in exactly the same way as that of
Lemma~\ref{lem:rat-general-strong-stable}. Its conclusions are illustrated by
Figure~\ref{fig:rat-NBT-strong-stab}.

\begin{lem}
\label{lem:rat-NBT-strong-stable} Let $f$ be of rational NBT type, with
$q(\kappa(f))=m/n\in(0,1/2)\cap\Q$. Then the strongly stable components of
$\sD$ are $\{\partial'\}$ and:
\begin{enumerate}[(a)]
\item for each $\bp\in\bP$, the line $L_{\bp}=\{\bp\}\times(0,\infty]$;
\item for each $k\in\Z$ and $0\le i\le n-1$:
\begin{enumerate}[(i)]
\item for each $y\in(c_u,a]$, the crosscut $\Gamma'(y,k,i)$.
\item the union $D_{k,i}$ of
\begin{itemize}
\item the crosscut $\xi'([\bt(\re_u, k+1, i), \bt(a, k+1, i)],\,\,u_{k,i})$,
  and
\item the set $[\bt(a, k+1, i), \bt(\re_u, k+1, i)] \,\,\times\,\, [u_{k,i},
  v_{k,i}]$.
\end{itemize} Here $u_{k,i} = \lambda^{kn+i}(2)$, $v_{k,i} =
\lambda^{kn+i}(n+1)$, and both of the intervals in~$\hS$ are those with the
given endpoints which are disjoint from~$\bP$. \qed
\end{enumerate}
\end{enumerate}
\end{lem}

\begin{figure}[htbp]
\begin{center}
\includegraphics[width=0.85\textwidth]{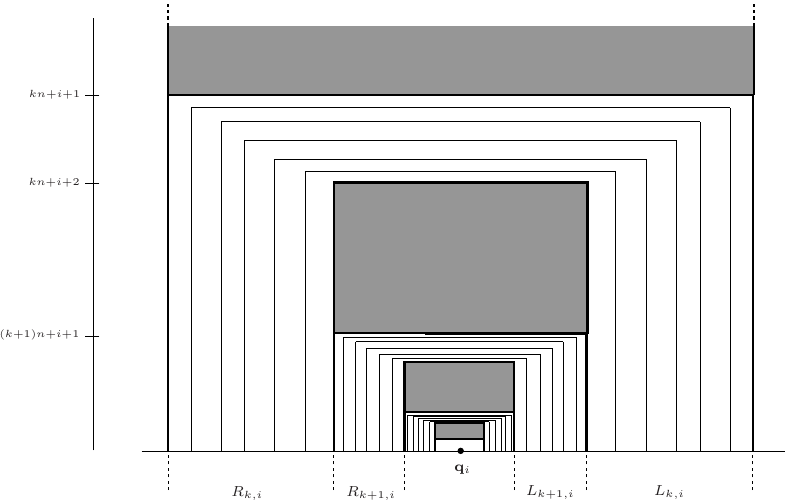}
\end{center}
\caption{Strongly stable components in the rational NBT case. The packets of
  crosscuts which do not surround~$\bq_i$ have degenerated, and every component
  except for the lines~$L_\bp$ ($\bp\in\bP$) touches $\hS_\infty$ at two
  points. The labeling on the $s$-axis is for the case $k\ge 0$.}
\label{fig:rat-NBT-strong-stab}
\end{figure}

In the rational normal endpoint case, on the other hand, the interval
$(\min(q_{n-1}, \hq_{n-1}), a]$ degenerates, giving rise to a more substantial
modification to the description. The following lemma is illustrated by
Figure~\ref{fig:rat-end-strong-stab}.

\begin{lem}
\label{lem:rat-endpoint-strong-stable} 
  Let $f$ be of rational normal endpoint type, with
  $q(\kappa(f))=m/n\in[0,1/2)\cap\Q$, and suppose that $f$ is of right hand
  endpoint type, so that $q_{n-1}=\re_u$ (the modifications in the left hand
  endpoint case are given at the end of the lemma statement). Then the strongly
  stable components of $\sD$ are $\{\partial'\}$ and:
  \begin{enumerate}[(a)]
    \item for each $k\in\Z$ and $0\le i\le n-1$:
    \begin{enumerate}[(i)]
      \item the arc $A_{k,i}=\{\bt(c_u, k, i)\} \times [r_{k,i}, \infty]$, where
        $r_{k,i}=\lambda^{kn+i}(2)$, i.e.
        \[ r_{k,i} =
          \begin{cases} 
              kn + i + 2 & \text{ if }k\ge 0, \\ 
              1/2^{|k|n - i-1} & \text{ if }k<0.
          \end{cases}
        \]
      \item for each $y\in(c_u,a)$, the crosscut $\Gamma'(y,k,i)$.
    \end{enumerate}
      \item For each $0\le i\le n-1$, the set
      \[ 
        D_i=L_{\bq_i} \cup\bigcup_{k\in\Z}L_{\bt(a,k,i)} \cup \bigcup_{k\in\Z}
          B_{k,i},
      \] 
      where $L_{\by}$ is the line $\{\by\}\times(0,\infty]$ and $B_{k,i} =
      L_{k,i}\times (0, u_{k,i}]$ with $u_{k,i}=\lambda^{kn+i}(1)$, i.e.
      \[ 
        u_{k,i} =
        \begin{cases} 
          kn + i + 1 & \text{ if }k\ge 0, \\ 
          1/2^{|k|n - i} & \text{ if }k<0.
        \end{cases}
      \]
  \end{enumerate} 
  In the left hand endpoint case $q_{n-1}=a$, the strongly stable
  components are given by replacing $\bq_i$ with $\bq_{i-m^{-1}\bmod n}$,
  $\bt(a,k,i)$ with $\bt(\re_u, k, i)$, and $L_{k,i}$ with $R_{k,i}$ in~(b).
\end{lem}

\begin{figure}[htbp]
\begin{center}
\includegraphics[width=0.75\textwidth]{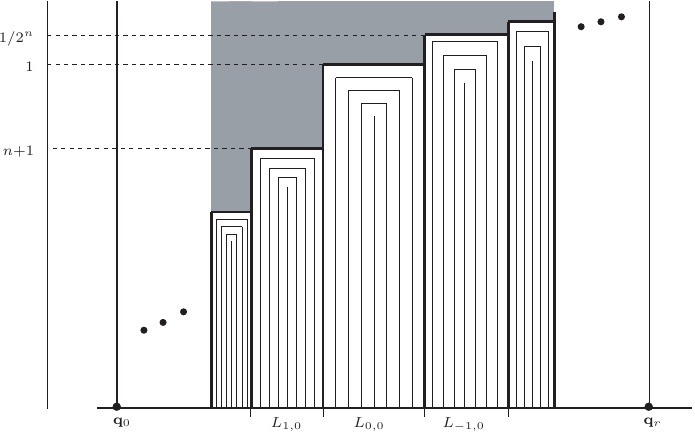}
\end{center}
\caption{Strongly stable components in the normal right hand endpoint case. The types of components are: (a) packets of crosscuts
  $\Gamma'(y,k,i)$ above each interval $L_{k,i}$, together with the ``central''
  arc $A_{k,i}$ which only intersects $\hS_\infty$ at a single point; and (b),
  for each~$i$, the set $D_i$ consisting of the outermost crosscut above
  each~$L_{k,i}$, the rectangles above these crosscuts extending up to (but not
  including) $s=0$, and the line $L_{\bq_i}$ above~$\bq_i$. In the figure,
  $D_0$ is the union of the shaded region and all of the bold arcs: it
  intersects $\hS_\infty$ at infinitely many points.}
\label{fig:rat-end-strong-stab}
\end{figure}

\begin{proof} We suppose that $q_{n-1}=\re_u$, so that $m/n>0$. The
modifications needed for the case $q_{n-1}=a$ are straightforward (including
for the special sub-case $m/n=0$, when $q_0=a$ is fixed by~$B$, and $\re_u=b$).
Notice that, since the orbit of~$a$ is disjoint from $\ingam$, we have
$\tau(B^r(a))=f^r(a)$ for all $r\ge 0$ by~(\ref{eq:commute}) (here the $a$ on
the left hand side is $a\in S$, while the $a$ on the right hand side is $a\in
I$). In particular, $f^n(a)=\tau(\re_u)=\re$, $f(a)$ is periodic of period~$n$,
and $\tau(q_i) = f^{i+1}(a)$ for $0\le i\le n-1$.

That the sets $A_{k,i}$ and $\Gamma'(y,k,i)$ are strongly stable is immediate
from Lemma~\ref{lem:gamma-stable} and Remark~\ref{rmk:dotted}.

To show that the sets~$D_i$ are strongly stable it suffices, since
$G(D_i)=D_{i+1\bmod n}$, to consider the case $i=0$.

\begin{enumerate}[(i)]
\item Let $\by=\bq_0 = \thr{\Inf{q_0, q_{n-1}, q_{n-2}, \ldots, q_1}}$. Since
    $y_1=q_{n-1}=\re_u$ and $y_{1+i}\not\in\ingam$ for all $i\ge 1$,
    Lemma~\ref{lem:constant-block} gives that $\Psi(\by,s)_0 =
    f(\tau(\re_u))=f(a)$ for all $s\in[1, \infty]$. On the other hand, if
    $s\in[1/2^r, 1/2^{r-1})$ for some $r\ge 1$, then $\Psi(\by, s)_0 = (q_0,
    s)$, and hence $\hH^r(\Psi(\by, s))_0 =
    \tau(B^r(q_0))=\tau(q_{r\bmod n}) = f^{r+1}(a)$. Therefore, for each $r\ge
    0$,
\[
\hH^r(\Psi(\bq_0, s))_0 = f^{r+1}(a) \qquad\text{ for all }s\in [1/2^r,
   \infty].
\]
\item By a similar argument applied to $\by = \bt(a,0,0) = \thr{q_0, a,
  B^{-1}(a), \ldots}$, we obtain that, for each~\mbox{$r\ge 0$},
$
\hH^r(\Psi(\bt(a,0,0), s))_0 = f^{r+1}(a)\text{ for all }s\in[1/2^r, \infty]$.

Now for each~$k\in\Z$ we have $G^{kn}(\bt(a,0,0), s) = (\bt(a,k,0),
\lambda^{kn}(s))$. By Corollary~\ref{cor:Psi-conjugates}, we obtain that for
all $k\in\Z$ and all $r\ge 0$,
\begin{eqnarray*}
\hH^r(\Psi(\bt(a,k,0),s))_0 &=& \hH^r(\Psi(G^{kn}(\bt(a,0,0),
  \lambda^{-kn}(s))))_0 \\ &=& \hH^{kn+r}(\Psi(\bt(a,0,0), \lambda^{-kn}(s)))_0
\\ &=& f^{kn}(f^{r+1}(a))=f^{r+1}(a) \qquad\text{ provided that
}\lambda^{-kn}(s)\in[1/2^r, \infty].
\end{eqnarray*} Therefore, for each $k\in\Z$ and each~$r\ge 0$, we have
\[
\hH^r(\Psi(\bt(a,k,0), s))_0 = f^{r+1}(a) \qquad\text{ for all }
s\in[\lambda^{kn}(1/2^r), \infty].
\]
\item Now let $\by = \bt(y,0,0)=\thr{q_0, y, B^{-1}(y), \ldots}\in L_{0,0}$,
  where $y\in [a,\re_u)$. Then $\Psi(\by,1)_0 = H(y,1/2) = f(a)$
  by~(\ref{eq:convenient-2}) and~(U1) of
  Definition~\ref{def:unwrapping-unimodal}. On the other hand, if $s\in(0,1)$
  we have $\Psi(\by,s) = \thr{(q_0,s), \ldots}$ and hence, as in~(i), if
  $s\in[1/2^r, 1/2^{r-1})$ for some $r\ge 1$ then $\hH^r(\Psi(\by,s))_0 =
  f^{r+1}(a)$. Therefore, for each~$r\ge 0$,
\[
\hH^r(\Psi(\by,s))_0 = f^{r+1}(a) \qquad\text{ for all }s\in[1/2^r, 1]\text{
  and all }\by\in L_{0,0}.
\] Since $G^{kn}(L_{0,0}) = L_{k,0}$ for each $k\in\Z$, a similar argument to
that of part~(ii) establishes that for all $k\in\Z$ and all $r\ge 0$,
\[
\hH^r(\Psi(\by,s))_0 = f^{r+1}(a) \qquad\text{ for all
}s\in[\lambda^{kn}(1/2^r), u_{k,0}] \text{ and all }\by\in L_{k,0},
\] where we have used that $\lambda^{kn}(1) = u_{k,0}$ for all $k\in\Z$.

\end{enumerate} Therefore, for all $\eta_1,\eta_2\in D_0$, there is some $r\ge
0$ such that $\hH^r(\Psi(\eta_1))_0 = \hH^r(\Psi(\eta_2))_0 = f^{r+1}(a)$,
establishing that $D_0$ is strongly stable as required.

\medskip

The proof that there are no connected strongly stable sets which strictly
contain one of these sets proceeds in the same way as in the proof of
Lemma~\ref{lem:ssc-irrational}.
\end{proof}

\medskip

The quadratic-like strict left endpoint case
(Section~\ref{sec:quadratic-like-endpoint}) is identical to the normal endpoint
case, except that there are additional half-open intervals
$I_i=\{\bt(y)\,:\,y\in(q_i, p_i]\}$ in $\hS$ (for $0\le i\le n-1$) whose points
satisfy $\hB^r(\bt(y))_0\not\in\gamma$ for all $r\in\Z$. The strongly stable
component containing $(\bt(y),\infty)$ is the line
$\{\bt(y)\}\times(0,\infty]$, exactly as in the irrational case; and other
strongly stable components are as in the normal endpoint case. We therefore
have the following description.

\begin{lem}
\label{lem:rat-quadratic-strong-stable}
  Let $f$ be of rational quadratic-like strict left endpoint type, with
  $q(\kappa(f))=m/n\in(0,1/2)\cap\Q$.
  Then the strongly stable components of $\sD$ are $\{\partial'\}$ and:
  \begin{enumerate}[(a)]
    \item for each $k\in\Z$ and $0\le i\le n-1$:
    \begin{enumerate}[(i)]
      \item the arc $A_{k,i}=\{\bt(c_u, k, i)\} \times [r_{k,i}, \infty]$, where
        $r_{k,i}=\lambda^{kn+i}(2)$, i.e.
        \[ r_{k,i} =
          \begin{cases} 
              kn + i + 2 & \text{ if }k\ge 0, \\ 
              1/2^{|k|n - i-1} & \text{ if }k<0.
          \end{cases}
        \]
      \item for each $y\in(c_u,a)$, the crosscut $\Gamma'(y,k,i)$.
    \end{enumerate}
      \item For each $0\le i\le n-1$, the set
      \[ 
        D_i=L_{\bq_{i-m^{-1}\bmod{n}}} \cup\bigcup_{k\in\Z}L_{\bt(\re_u,k,i)} \cup \bigcup_{k\in\Z}
          B_{k,i},
      \] 
      where $L_{\by}$ is the line $\{\by\}\times(0,\infty]$ and $B_{k,i} =
      R_{k,i}\times (0, u_{k,i}]$ with $u_{k,i}=\lambda^{kn+i}(1)$, i.e.
      \[ 
        u_{k,i} =
        \begin{cases} 
          kn + i + 1 & \text{ if }k\ge 0, \\ 
          1/2^{|k|n - i} & \text{ if }k<0.
        \end{cases}
      \]
    \item For each $0\le i\le n-1$ and each $y\in(q_i, p_i]$, the line $L_{\bt(y)}$.
  \end{enumerate} 
\end{lem}

The following straightforward consequence of the above proofs will be useful in
Section~\ref{sec:cont-var}.

\begin{lem}\mbox{}
  \label{lem:rational-diameter-bound}
  \begin{enumerate}[(a)]
    \item Let~$f$ be of irrational or rational early endpoint type. Then for
    each~$r\ge 1$ and each $y\in(c_u, a)$, the diameters of the strongly stable
    component images $\Psi(A_r)$ and $\Psi(C_{r,y})$ are bounded above by
    $|b-a|/2^{r-1}$.

    \item Let~$f$ be of rational interior or late endpoint type with
    $q(\kappa(f))=m/n\in(0,1/2)\cap\Q$. Then for each $k\ge 0$, each $0\le i\le
    n-1$, and each $y\in(c_u, a] \setminus\{\min(q_{n-1}, \hq_{n-1})\}$, the
    diameters of the strongly stable component images $\Psi(A_{k,i})$,
    $\Psi(\Gamma'(y,k,i))$ and $\Psi(D_{k, i})$ are bounded above by
    $|b-a|/2^{kn+i}$.

    \item Let~$f$ be of rational normal endpoint type or quadratic-like strict
    left endpoint type, with $q(\kappa(f))=m/n\in(0,1/2)\cap\Q$. Then for each
    $k\ge 0$, each $0\le i\le n-1$, and each $y\in(c_u, a)$, the diameters of
    the strongly stable component images $\Psi(A_{k,i})$ and
    $\Psi(\Gamma'(y,k,i))$ are bounded above by $|b-a|/2^{kn+i}$.

  \end{enumerate}
\end{lem}

\begin{proof} In the irrational or early endpoint case, the proof of
    Lemma~\ref{lem:ssc-irrational} shows that every element~$\xi$ of
    $\Psi(A_1)$ (respectively $\Psi(C_{1,y})$) has $\xi_0=b$ (respectively
    $\xi_0=f(\tau(y))$). Therefore any two elements of $\Psi(A_r)$ or of
    $\Psi(C_{r,y})$ have equal $(r-1)^\text{th}$ entries, and so are within
    distance $|b-a|/2^{r-1}$ of each other. This establishes~(a). Parts~(b)
    and~(c) follow similarly from the proofs of
    Lemmas~\ref{lem:rat-general-strong-stable}
    and~\ref{lem:rat-endpoint-strong-stable}, which show that every
    element~$\xi$ of $\Psi(A_{0,0})$ (respectively $\Psi(\Gamma'(y, 0, 0))$,
    $\Psi(D_{0,0})$) has $\xi_0 = b$ (respectively $\xi_0=f(\tau(y))$, $\xi_0 =
    f(\tau(q_{n-1}))$).
\end{proof}

\subsection{Construction of the sphere homeomorphism}

\begin{defn}[The decomposition $\cG'$ of $\bD$]
\label{def:decomp-primed} 
  Let~$\cG'$ be the decomposition of $\bD$ whose elements are:
  \begin{itemize}
    \item $\{\eta\}$ for each $\eta\in\bD\setminus\sD$;
    \item Strongly stable components whose closures don't contain~$\partial'$; and
    \item The set~$X$ which is the union of the strongly stable components
      whose closures contain~$\partial'$ (including the strongly stable
      component $\{\partial'\}$ itself).
  \end{itemize}
\end{defn}

\begin{remark}
\label{rmk:decomp-primed-permuted} It follows from the explicit descriptions of
the strongly stable components that those whose closures don't contain
$\partial'$ are compact; and that the set~$X$ is also compact. Therefore $\cG'$
is the largest partition of~$\bD$ into compact sets with the property that
every strongly stable component is contained in a single partition element.

Moreover, the elements of $\cG'$ are connected, and are permuted by~$G$, since
\mbox{$G(X)=X$}.
\end{remark}

\begin{lem}
\label{lem:Psi-cont-decomp} 
  The restriction of~$\Psi$ to each element of~$\cG'$, apart from the single
  points of $\bD\setminus\sD$ (where it is not defined), is a
  homeomorphism onto its image.
\end{lem}
\begin{proof} This is immediate from the descriptions of the strongly stable
components and Corollary~\ref{cor:Psi-maximal-extension} in all cases except
for the element~$X$ of~$\cG'$ in the irrational, early endpoint, normal
endpoint, and quadratic-like strict left endpoint cases.

Suppose that~$f$ is of irrational type, so that the strongly stable components
are given by Lemma~\ref{lem:ssc-irrational}. Then $X\cap \hS_\infty$ is the
Cantor set~$\Lambda_\infty$. $\Psi|_{X}$ is injective since~$\Psi$ is injective
on~$D$ and on~$\Lambda_\infty$ (Corollary~\ref{cor:Psi-homeo} and
Lemma~\ref{lem:landing-injective}), and $\Psi(\eta)\in\hI$ if and only if
$\eta\in\hS_\infty$. Since~$\Psi$ is continuous away from~$\hS_\infty$
(Corollary~\ref{cor:Psi-homeo}), it is only necessary to show that $\Psi|_X$ is
continuous at the points of~$\Lambda_\infty$ (of course, $\Psi$ itself is not
continuous at these points).

Now~$\Psi|_{\Lambda\times[0,\infty]}$ is continuous by
Lemma~\ref{lem:Psi-extension-continuous}, since $\Lambda$ is uniformly landing
of level~0. Thus it suffices to show that if $(\by,s)\in
\bigcup_{r\in\Z}(G_r\times(0,u_r])$ is sufficiently close to a point
    $(\by',\infty)$ of $\Lambda_\infty$, then $\Psi(\by,s)$ is close to
    $\Psi(\by',\infty)$, where $G_r\times(0,u_r]$ are the rectangles of
    Lemma~\ref{lem:ssc-irrational}~(b)(iii). In order to do this we will show
    that, for all $N\ge2$, if $(\by,s)\in G_r\times(0,u_r]$ for some~$r$, and
    $s>N$, then $\Psi(\by,s)_{N-2} =
  \Psi(\bt(a,r), \infty)_{N-2}$. This will establish the result, since if
  $(\by,s)$ is close to $(\by',\infty)$ then $(\bt(a,r),
  \infty)\in\Lambda_\infty$ is also close to $(\by', \infty)$.

Recall that $u_r<1$ for $r<1$, and $u_r=r$ for $r\ge 1$. So if $(\by,s)\in
G_r\times(0,u_r]$ and $s>N$, we have \mbox{$N<s\le r$}. Observe that $
G^{-(r-1)}(\by,s) = \left( \hB^{-(r-1)}(\by), \lambda^{-(r-1)}(s)
  \right) \in G_1\times (0,1] $.

Let $m\ge 0$ be such that $r-m\le s < r-m+1$, so that $\lambda^{-(r-1)}(s)\in
[1/2^m, 1/2^{m-1})$.  By the proof of Lemma~\ref{lem:ssc-irrational}~(b)(iii),
we have $\hH^m(\Psi(G^{-(r-1)}(\by, s)))_0 = f^{m+1}(a)$: therefore, by
Corollary~\ref{cor:Psi-conjugates}, $ \hH^{m-r+1}(\Psi(\by, s))_0 = f^{m+1}(a)
$, and so $\Psi(\by,s)_{r-m-1} = f^{m+1}(a)$. Since $r-m-1>s-2>N-2$, it follows
that
\[
\Psi(\by, s)_{N-2} = f^{r-(N-2)}(a) \qquad\text{ whenever }(\by,s)\in
    G_r\times(0,r] \text{ with } s>N.
\]

On the other hand, $\Psi(\bt(a,r), \infty) = \omega(\bt(a,r)) =
\thr{f^r(a), \ldots, f(a), a, \tau(B^{-1}(a)), \ldots}$ by~(\ref{eq:omega}),
since $\bt(a,r)\in\cL_r$. Therefore $\Psi(\bt(a,r), \infty)_{N-2} =
f^{r-(N-2)}(a)$, as required.

The proof when $f$ is of early endpoint type is identical, with the periodic
orbit~$\bQ$ in place of the Cantor set~$\Lambda$; and the proof when~$f$ is of
normal endpoint or quadratic-like strict left endpoint type involves only minor
modifications.
\end{proof}

\begin{defn}[The decomposition $\cG$ of $\hT$]
\label{def:decomp} Let~$\cG$ be the decomposition of $\hT$ whose
elements are:
\begin{itemize}
\item the images under $\Psi$ of the elements of $\cG'$, other than points of
  $\bD\setminus\sD$; and
\item single points which are not in the image of $\Psi$.
\end{itemize}
\end{defn}

By Corollary~\ref{cor:Psi-conjugates} and
Remark~\ref{rmk:decomp-primed-permuted}, the elements of~$\cG$ are permuted by
$\hH$.

\begin{lem}
\label{lem:G-musc}
  $\cG$ is a non-separating monotone upper semi-continuous decomposition
  of~$\hT$.
\end{lem}
\begin{proof}
  The elements of~$\cG$ are compact, connected, and do not separate~$\hT$ since
  the elements of~$\cG'$ are compact, connected, and contractible, and the
  restriction of~$\Psi$ to each of them is a homeomorphism by
  Lemma~\ref{lem:Psi-cont-decomp}.

  That $\cG$ is upper semi-continuous is a special case of
  Lemma~\ref{lem:G-ast-usc} below (where~$\cG$ is a single slice of a sliced
  decomposition which is shown to be upper semi-continuous).
\end{proof}

\begin{thm}[Semi-conjugacy to sphere homeomorphisms]
\label{thm:semi-conj} 
  Let~$f$ be a unimodal map satisfying the conditions of
  Convention~\ref{conv:unimodal}. Then there is a sphere homeomorphism
  $F\colon\Sigma\to\Sigma$ and a continuous surjection $g\colon
  \hI\to\Sigma$ which semi-conjugates $\hf\colon\hI\to\hI$ to
  $F\colon\Sigma\to\Sigma$.

  Every fiber of~$g$ except for at most one contains three or fewer points, and
  only countably many fibers contain three points.
\end{thm}

\begin{proof}
By Lemma~\ref{lem:G-musc} and Moore's theorem, $\Sigma = \hT/\cG$ is a sphere;
and since $\hH$ permutes the elements of~$\cG$, the map $F=\hH/\cG\colon\Sigma
\to
\Sigma$ is a homeomorphism. Since every element of $\cG$ intersects $\hI$, and
$\hH|_{\hI} =
\hf$, it follows that $\Sigma$ is also the quotient of $\hI$ by the equivalence
relation $\sim_{\cG|_{\hI}}$ on~$\hI$ induced by~$\cG$; and the canonical
projection $g\colon \hI\to \hI/\!\!\sim_{\cG|_{\hI}} =
\Sigma$ is a semi-conjugacy between $\hf$ and~$F$.

The statement about the cardinalities of the fibers of~$g$ is immediate from
the descriptions of the elements of~$\cG'$ (Definition~\ref{def:decomp-primed}
and Lemmas~\ref{lem:ssc-irrational}, \ref{lem:rat-general-strong-stable},
\ref{lem:rat-NBT-strong-stable}, \ref{lem:rat-endpoint-strong-stable}, and
\ref{lem:rat-quadratic-strong-stable}), every one of which except for~$X$
intersects $\hS_\infty$ in three or fewer points, and only countably many of
which can intersect $\hS_\infty$ in three points.
\end{proof}

\begin{remark}
\label{rmk:same-entropies} Since $\Psi|_{X}$ is a homeomorphism onto its image
(Lemma~\ref{lem:Psi-cont-decomp}), the restriction of~$\hf$ to the exceptional
fiber~$\Psi(X\cap \hS_\infty)$ of~$g$ is topologically conjugate to the action
of the circle homeomorphism~$\hB$ on an invariant subset in the circle, and
therefore has topological entropy zero. It follows from a result of Bowen
(Theorem~17 of~\cite{bowen}), using the fact that all other fibers are finite,
that $\hf\colon\hI\to\hI$ and $F\colon\Sigma\to\Sigma$ have the same
topological entropy. Since~$\hf$ and~$f$ also have the same topological entropy
(this follows from the same result of Bowen), we conclude that the sphere
homeomorphism $F\colon\Sigma\to\Sigma$ has the same topological entropy as the
unimodal map $f\colon I\to I$.
\end{remark}

\begin{remark}
\label{rmk:semiconj-fibres} 
  By the proof of Theorem~\ref{thm:semi-conj}, the fibers of the semi-conjugacy
  $g\colon\hI\to\Sigma$ can be described explicitly. Every non-trivial fiber is
  contained in the set of accessible points and, conversely, all but countably
  many trivial fibers are contained in the set of inaccessible points.

  The accessible fibers of~$g$ are as follows:
  \begin{enumerate}[(a)]
  \item 
    In the irrational case, there is one fiber equal to $\omega(\Lambda)$,
    where $\Lambda$ is the Cantor set of Definition~\ref{def:lambda}; and there
    are countably many accessible trivial fibers $\{\omega(\bt(c_u,r))\}$ for
    $r\in\Z$. All other accessible fibers are of the form $\{\omega(\bt(y,r)),
    \omega(\bt(\hy,r))\}$, where $y\in\ingam\setminus\{c_u\}$ and $r\in\Z$.
  \item 
    In the early endpoint case there is one fiber
    \[\omega(\hS_\infty)\setminus \{\omega(\bt(y,r))\,:\,y\in\ingam, r\in\Z\}\]
     which is either a countable union of disjoint intervals, or such a union
     together with finitely many isolated points; and there are countably many
     accessible trivial fibers $\{\omega(\bt(c_u,r))\}$ for $r\in\Z$. All other
     accessible fibers are of the form $\{\omega(\bt(y,r)),
     \omega(\bt(\hy,r))\}$, where $y\in\ingam\setminus\{c_u\}$ and $r\in\Z$.
  \item 
    In the normal endpoint case with $q(\kappa(f))=m/n$ there is one countable
      fiber
    \[
      \omega(\bQ)\,\cup\, \{\omega(\bt(A, k, i))\,:\, k\in\Z,\, 0\le i\le n-1\},
    \] 
    where $A=a$ if $\kappa(f)=\rhe(m/n)$, and $A=\re_u$ if
    $\kappa(f)=\lhe(m/n)$; and there are countably many accessible trivial
    fibers $\{\omega(\bt(c_u,k,i))\}$ for $k\in\Z$ and $0\le i\le n-1$. All
    other accessible fibers are of the form $\{\omega(\bt(y,k,i)),
    \omega(\bt(\hy,k,i))\}$, where $y\in\ingam\setminus\{c_u\}$, $k\in\Z$ and
    $0\le i\le n-1$.
  \item
    In the quadratic-like strict left end point case with $q(\kappa(f))=m/n$,
    there is one fiber
    \[
      \omega(\bQ) 
            \cup \bigcup_{i=0}^{n-1} \omega(I_i)
            \cup \{\omega(\bt(\re_u, k, i))\,:\, k\in\Z, 0\le i\le n-1\}
    \]
    which is the union of~$n$ disjoint compact intervals and countably many
    points; and there are countably many accessible trivial fibers
    $\{\omega(\bt(c_u,k,i))\}$ for $k\in\Z$ and $0\le i\le n-1$. All other
    accessible fibers are of the form $\{\omega(\bt(y,k,i)),
    \omega(\bt(\hy,k,i))\}$, where $y\in\ingam\setminus\{c_u\}$, $k\in\Z$ and
    $0\le i\le n-1$.
  \item 
    In the rational general and late endpoint cases with $q(\kappa(f))=m/n$,
    there is one fiber $\omega(\bP)$ of cardinality~$n$; countably many
    $3$-element fibers of the form
    \[
      \{\omega(\bt(\hq_{n-1}, k, i)), \omega(\bt(a, k+1, i)), \omega(\bt(\re_u,
      k+1, i)) \}\quad\text{ for $k\in\Z$ and $0\le i\le n-1$};
    \] 
    and countably many accessible trivial fibers $\{\omega(\bt(c_u,k,i))\}$ for
    $k\in\Z$ and $0\le i\le n-1$. All other accessible fibers are of the form
    $\{\omega(\bt(y,k,i)),  \omega(\bt(\hy,k,i))\}$, where
    $y\in\ingam\setminus\{q_{n-1}, \hq_{n-1}, c_u\}$, $k\in\Z$, and $0\le i\le
    n-1$.
  \item 
    In the rational NBT case with $q(\kappa(f))=m/n$, there is one fiber
    $\omega(\bP)$ of cardinality~$n$. All other accessible fibers are of the
    form $\{\omega(\bt(y,k,i)), \omega(\bt(\hy,k,i))\}$, where
    $y\in\gamma\setminus\{c_u\}$, $k\in\Z$, and $0\le i\le n-1$.
  \end{enumerate}

  Note that the exceptional fiber of~$g$ can only be infinite when~$f$ is of
  irrational or endpoint type. In particular, for the tent family $\{f_t\}$,
  the semi-conjugacy has only finite fibers for an open dense subset of
  parameters.

\end{remark}

\medskip

An immediate consequence of Theorem~\ref{thm:semi-conj} is that {\em any}
  unimodal map with topological entropy greater than $\frac{1}{2}\log2$ has
  natural extension semi-conjugate to a sphere homeomorphism, although the
  fibers of the semi-conjugacy may not be so well behaved when the conditions
  of Convention~\ref{conv:unimodal} are not satisfied.

\begin{cor} Let $f$ be any unimodal map (not necessarily satisfying the
conditions of Convention~\ref{conv:unimodal}) with topological entropy
$h(f)>\frac{1}{2}\log2$. Then the natural extension $\hf$ is semi-conjugate to
a sphere homeomorphism with the same topological entropy as~$f$.
\end{cor}
\begin{proof} $h(f)>\frac{1}{2}\log2$ is equivalent to $\kappa(f)\succ
10\Infs{1}$. Now $f$ is semi-conjugate to a tent map~$F$ with
$\kappa(F)=\kappa(f)$ and $h(F)=h(f)$ (see~\cite{MT,Parry}), and hence
$\hf\colon\hI_f\to\hI_f$ is semi-conjugate to $\hF\colon\hI_F\to\hI_F$, which
is semi-conjugate to a sphere homeomorphism of topological entropy~$h(F)$ by
Theorem~\ref{thm:semi-conj} and Remark~\ref{rmk:same-entropies}.
\end{proof}

\subsection{Continuously varying families}
\label{sec:cont-var} 
Our aim in this section is the following result, which shows that the above
construction of sphere homeomorphisms can be carried out continuously.

\begin{thm}
\label{thm:cont-var} Let $J$ be a compact parameter interval, and
$\{f_t\}_{t\in J}$ be a continuously varying family of unimodal maps, all of
which are defined on the same core interval~$I$ and satisfy the conditions of
Convention~\ref{conv:unimodal}. For each~$t$, let
$F_t\colon\Sigma_t\to\Sigma_t$ be the sphere homeomorphism constructed from
$f_t$ in the proof of Theorem~\ref{thm:semi-conj}. Then there is a continuously
varying family $\{\chi_t\colon S^2\to S^2\}_{t\in J}$ of sphere homeomorphisms
such that $\chi_t$ is topologically conjugate to $F_t$ for each~$t$.

In particular, each natural extension $\hf_t\colon \hI_t\to\hI_t$ is
semi-conjugate to $\chi_t$, by a semi-conjugacy all but one of whose fibers
contains three or fewer points, and only countably many of whose fibers contain
three points.
\end{thm}

\begin{remark}
  If $f_t\colon I_t\to I_t$, where $\{I_t\}$ is a continuously varying family
  of compact intervals --- as occurs naturally when families of unimodal maps
  are restricted to their core intervals --- then the theorem applies after
  conjugating by a continuously varying affine coordinate change.
\end{remark}

\medskip

Throughout this section, $\{f_t\}_{t\in J}$ will denote a fixed family of
unimodal maps as in the statement of Theorem~\ref{thm:cont-var}. Because the
domain $I=[a,b]$ is fixed, the circle~$S$ and the sphere~$T$ are independent of
the parameter~$t$. However, almost every other object is parameter dependent.
This dependence will generally be indicated with a subscript~$t$, but will
sometimes be suppressed, particularly when it doesn't serve to illuminate
continuity or convergence arguments, in order to avoid excessive notation. For
example, we will not normally make explicit the parameter dependence of $c_u$
and $\re_u$.

Recall that $\{\barf_t\colon T\to T\}$ is a continuously varying family of
unwrappings of $\{f_t\}$, and that the homeomorphisms $\hH_t\colon
\hT_t\to\hT_t$ are the natural extensions of the near-homeomorphisms
\[H_t=\Upsilon\circ \barf_t \colon T\to T.\]

Let
\[
  \hT_\ast = \bigsqcup_{t\in J} \left(\hT_t \times\{t\}\right),
\]
topologized as a compact subset of $T^\N\times J$. The following result
from~\cite{param-fam} --- which is the key lemma used in the proof of
Theorem~\ref{thm:unwrap-family} --- tells us that $\hT_\ast$ is homeomorphic to
$S^2\times J$.

\begin{thm}
\label{thm:fat-homeo} There is a slice-preserving homeomorphism $\beta\colon
\hT_\ast \to T\times J$.
\end{thm}

Here \emph{slice-preserving} means that $\beta(\hT_t\times\{t\}) =
T\times\{t\}$ for each~$t$. In~\cite{param-fam} this result is stated not
for~$\hT_\ast$, but for the inverse limit $\widehat{T\times J}$ of the fat map
$T\times J\to T\times J$ defined by $(x, t) \mapsto
(H_t(x), t)$. However $\hT_\ast$ and $\widehat{T\times J}$ are homeomorphic
by the slice-preserving homeomorphism $\left(\thr{x_0, x_1, \ldots}, t\right)
\mapsto \thr{(x_0, t), (x_1, t), \ldots}$.

Let $\hH_\ast\colon \hT_\ast\to \hT_\ast$ be the slice-preserving homeomorphism
defined by $\hH_\ast(\bx, t) = (\hH_t(\bx), t)$, and $\cG_\ast$ be the
$\hH_\ast$-invariant decomposition of $\hT_\ast$ induced in each slice by
$\cG_t$: that is,
\[
  \cG_\ast = \{g\times\{t\}\,:\,t\in J\text{ and }g\in\cG_t\}.
\]

The elements of~$\cG_\ast$ are each contained in a slice of $\hT_\ast$, and
moreover are compact, connected, and do not separate their slices, since these
properties are inherited from the $\cG_t$ (see Lemma~\ref{lem:G-musc}).
Lemma~\ref{lem:G-ast-usc} below states that $\cG_\ast$ is upper
semi-continuous. We now assume this lemma and show how to complete the proof of
the Theorem~\ref{thm:cont-var}. The key ingredient is the following theorem of
Dyer and Hamstrom~\cite{DH} (both statement and proof of this result are
contained in the proof of Theorem~8 of~\cite{DH}: note that a decomposition is
upper semi-continuous if and only if its quotient mapping is closed).

\begin{thm}[Dyer\,--\,Hamstrom]
\label{thm:DH} Let $\cG$ be a monotone upper semi-continuous decomposition of
  $S^2\times J$ into compact subsets, each of whose elements lies in, and does
  not separate, some slice $S^2\times\{t\}$. Suppose also that there is an
  arc~$L$ in $S^2\times J$ which intersects each slice $S^2\times\{t\}$ in a
  singleton decomposition element. Then there is a slice-preserving
  homeomorphism $K\colon (S^2\times J)/\cG \to S^2\times J$.
\end{thm}

It follows that, assuming the upper semi-continuity of~$\cG_\ast$, we have the
commutative diagram of Figure~\ref{fig:cont-fam-cd}. Here $\pi$ is the quotient
mapping of the decomposition~$\cG_\ast$;  $K$ is the homeomorphism of
Theorem~\ref{thm:DH} (which exists since, by Theorem~\ref{thm:fat-homeo},
$\hT_\ast$ is slice-preserving homeomorphic to $S^2\times J$: a suitable
arc~$L$ is the one which intersects each $\hT_t\times\{t\}$ at $(\bz_t, t)$,
where $\bz_t\in\hI_t$ is the fixed point of $\hH_t$ which lies above the fixed
point~$z_t$ of~$f_t$ in $(c,b)$); and $\hH_{\ast}/\cG_\ast$, and $\chi_\ast$
are the homeomorphisms which make the diagram commute. All of the maps in the
diagram are slice-preserving, and in particular $\chi_\ast\colon S^2\times J\to
S^2\times J$ defines a continuously varying family $\{\chi_t\}_{t\in J}$ of
sphere homeomorphisms. Restricting the diagram to a single slice, we see that
$\chi_t$ is topologically conjugate to the homeomorphism $F_t = \hH_t/\cG_t$ of
Theorem~\ref{thm:semi-conj}, which completes the proof of
Theorem~\ref{thm:cont-var}.

\begin{figure}[htbp]
\[
  \begin{CD} 
    \hT_\ast    @>\hH_\ast>>      \hT_\ast\\
    @V\pi VV   @VV\pi V\\ 
    \hT_\ast / \cG_\ast    @>\hH_{\ast}/\cG_\ast
    >>      \hT_\ast / \cG_\ast\\
    @V K VV   @VV K V\\ S^2\times J   @>\chi_\ast>>      S^2\times J.\\
  \end{CD}
\]
\caption{Construction of the continuous family of sphere homeomorphisms.}
\label{fig:cont-fam-cd}
\end{figure}

It therefore only remains to show that the decomposition $\cG_\ast$ of
$\hT_\ast$ is upper semi-continuous.  We do this by first considering the
decompositions $\cG'_t$ of the spaces $\bD_t$ --- which are described
explicitly by Lemmas~\ref{lem:ssc-irrational},
\ref{lem:rat-general-strong-stable}, \ref{lem:rat-NBT-strong-stable},
\ref{lem:rat-endpoint-strong-stable}, and~\ref{lem:rat-quadratic-strong-stable}
--- and then transferring the results using the maps $\Psi_t$.

Recall from Definition~\ref{def:decomp-primed} that, for each~$t\in J$, the
decomposition $\cG'_t$ of $\bD_t$ has as elements
\begin{itemize}
  \item Strongly stable components whose closures are disjoint from
  $\partial'_t$,
  \item The union~$X_t$ of strongly stable components whose closures contain
  $\partial'_t$, and
  \item Single points at which $\Psi_t$ is not defined.
\end{itemize} 
We write
\[
  \bD_\ast = \bigsqcup_{t\in J}\left(\bD_t \times \{t\}\right),
\] 
topologized as a compact subset of $\left(S^\N
  \times[0,\infty]\right)/\left(S^\N \times\{0\}\right) \,\times\, J$, and
  define 
  \[
    \cG'_\ast = \{g'\times\{t\}\,:\, t\in J\text{ and } g'\in\cG'_t\}
  \]
   to be
  the sliced decomposition of $\bD_\ast$ induced on each slice by the
  decompositions $\cG'_t$.

Recall also (Definitions~\ref{def:landing-etc}) that, for each~$t$, we
denote by $\sDt$ the subset of $\bD_t$ on which $\Psi_t$ is defined:
that is, $\sDt = \bD_t$ if $f_t$ is of irrational or rational endpoint
type, and otherwise $\sDt = \bD_t\setminus\bQ_t$. Writing $\sDast$
for the subset $\bigsqcup_{t\in J}(\sDt \times \{t\})$ of $\bD_\ast$, we
can then define the function $\Psi_\ast\colon \sDast \to \hT_\ast$ by
$\Psi_\ast(\eta, t) = (\Psi_t(\eta), t)$. With these definitions, the
non-trivial elements of the decomposition $\cG_\ast$ of $\hT_\ast$ are
precisely the images under $\Psi_\ast$ of the non-trivial elements of
$\cG'_\ast$.

The proof of the following is essentially the same as that of
Corollary~\ref{cor:Psi-homeo}.

\begin{lem}
\label{lem:psi-ast-cont} 
  $\Psi_\ast$ is continuous at $\left((\by, s), t\right)\in \sDast$ whenever
    $s<\infty$.
\end{lem}

\begin{proof} 
For each~$N\in\N$, we have that, whenever $s<N+1$,
  \[
    \Psi_\ast\yst = \thr{H_t^N(y_N, \lambda^{-N}(s)), H_t^N(y_{N+1},
    \lambda^{-(N+1)}(s)), \ldots}
  \] by \eqref{eq:HGPsi}, \eqref{eq:Psi}, and \eqref{eq:thread1}. This
  expression is clearly continuous in $\yst$.
\end{proof}

The main technique in the proof of upper semi-continuity of $\cG_\ast$ is to
take certain convergent sequences in~$\hT_\ast$, transfer them to $\sDast$
using $\Psi_\ast^{-1}$, draw conclusions about the limit in $\sDast$, and
transfer back to~$\hT_\ast$. In order to do this, we need to know that
$\Psi_\ast$ respects the limits of certain sequences, although it may not be
continuous at those limits. The following lemma enables us to do this:
parts~(a) and~(b) are natural, while part~(c), which is more esoteric, is
motivated by the specific requirements of the proof.

\begin{lem}
  \label{lem:convergence-Psi}
  Let $\ystj\to\yst$ be a convergent sequence in $\sDast$. Then
  $\Psi_\ast\ystj\to \Psi_\ast\yst$ if one of the following holds:
  \begin{enumerate}[(a)]
    \item $s<\infty$;
    \item $\by$ and all of the $\by\sj$ are landing of level~$1$; or
    \item $\by$ is landing of level~1, and there is a sequence
    $n_j\to \infty$ such that for each~$j$ we have
    $s_j\le n_j+1$ and $y\sj_{i}\not\in\ingam_{t_j}$ for $2\le i\le n_j$.
  \end{enumerate}
\end{lem}

\begin{proof}\mbox{}
  \begin{enumerate}[(a)]
    \item By Lemma~\ref{lem:psi-ast-cont}.
    \item We can assume that $s=\infty$ since otherwise~(a) applies. Fix $r\ge
    1$. Since $\by\in\cL_r$, Corollary~\ref{cor:ray-lands} gives $\Psi_t(\by,
    \infty)_r = \tau(y_r)$; and since $\by\sj\in\cL_r$,
    Lemma~\ref{lem:formula-disjoint-gamma} gives $\Psi_{t_j}(\by\sj, s_j)_r =
    \tau(y\sj_r)$, provided that~$j$ is large enough that $s_j\ge r+1$.
    Therefore for each~$r\ge 1$ we have $\Psi_{t_j}(\by\sj, s_j)_r \to
    \Psi_t(\by, \infty)_r$ as $j\to\infty$, and the result follows.
    \item We can again assume that~$s=\infty$. Fix $r\ge 1$. As for~(b), we
    have $\Psi_t(\by, \infty)_r = \tau(y_r)$. For each~$j$ large enough that
    $n_j\ge r+2$, we have $y\sj_{(r+1)+i}\not\in\ingam_{t_j}$ for $1\le i\le
    n_j- (r+1)$, so that Lemma~\ref{lem:constant-block} gives
    \[
      \Psi_{t_j}(\by\sj, s')_r = f_{t_j}(\tau(y\sj_{r+1})) = \tau(y\sj_r)
      \quad\text{ for all }s'\in[r+2, n_j + 1].
    \]
    In particular, since $s_j\le n_j+1$ for all~$j$, we have
    $\Psi_{t_j}(\by\sj, s_j)_r = \tau(y\sj_r)$ whenever $j$ is large enough
    that $s_j\ge r+2$, so that $\Psi_{t_j}(\by\sj, s_j)_r \to \Psi_t(\by,
    \infty)_r$ as $j\to\infty$.
  \end{enumerate}
\end{proof}

The following abbreviated language will be convenient.

\begin{defn}[Type and height of a parameter] 
  We say that a parameter $t\in J$ is of \emph{irrational}, \emph{rational},
  and rational \emph{interior}, \emph{early endpoint}, \emph{normal endpoint},
  \emph{quadratic-like strict left endpoint}, or \emph{late endpoint} type
  according as $f_t$ is. We define the \emph{height} of~$t$ to be
  $q(\kappa(f_t))$.
\end{defn}

\begin{lem} 
\label{lem:G-ast-usc}
The decomposition $\cG_\ast$ of $\hT_\ast$ is upper semi-continuous.
\end{lem}
\begin{proof} Let $(t_j)$ be a sequence in~$J$ converging to $t\in J$, and for
  each $j$, let $g_j$ be a decomposition element of $\cG_{t_j}$. We need to
  show that there is a decomposition element $g\in\cG_t$ with the property
  that, whenever $(\xi_j, t_j)$ is a sequence in~$\hT_\ast$ with $(\xi_j,
  t_j)\to (\xi, t)$ and $\xi_j\in g_j$ for all~$j$, then we have $\xi\in g$.

  This is clearly the case if infinitely many of the $g_j$ are singletons, so
  we can assume that there are decomposition elements $g_j'\in\cG'_{t_j}$ with
  $\Psi_{t_j}(g_j') = g_j$ for each~$j$.

  Observe that if the required property holds for some subsequence of $(t_j,
  g_j)$, then it also holds for the full sequence. By taking such a
  subsequence, we can therefore further assume that all of the~$t_j$ are of
  rational interior or late endpoint type (type~A); or that they are all of
  rational normal or quadratic-like endpoint type (type~B); or that they are
  all of irrational or rational early endpoint type (type~C). We will consider
  each of these three cases in turn. The arguments will also depend on the type
  of the limiting parameter~$t$. We note that~$t$ can only be of type~A if the
  $t_j$ are also of type~A, and if the sequence $m_j/n_j$ of their heights is
  eventually constant, since the set~$J_q$ of parameters~$t$ of rational
  interior or late endpoint type with prescribed height~$q$ is open in~$J$.
  This is because $J_q=\{t\in J\,:\, (w_q1)^\infty \prec \kappa(f_t)\prec
  10(\hw_q1)^\infty \}$ (see Definitions~\ref{def:lhe,rhe,etc}
  and~\ref{def:type}); because if $\kappa(f_t)=(w_q0)^\infty$ then the turning
  point of~$f_t$ is not periodic by Definition~\ref{def:kneading-sequence}; and
  because $10(\hw_q1)^\infty$ is not a periodic sequence by
  Lemma~\ref{lem:height-intervals}~(d).

  The method of proof is the same in all three cases. We first use the explicit
  description of the decompositions~$\cG'_t$ provided by
  Lemmas~\ref{lem:ssc-irrational}, \ref{lem:rat-general-strong-stable},
  \ref{lem:rat-NBT-strong-stable}, \ref{lem:rat-endpoint-strong-stable},
  and~\ref{lem:rat-quadratic-strong-stable} to show that either (i) the
  diameters of (a subsequence of) the decomposition elements $g_j$ converge to
  zero, in which case the result is obvious; or (ii) there is a decomposition
  element $g'\in\cG'_t$ with the property that, whenever $(\eta_j, t_j)$ is a
  sequence in~$\sDast$ with $(\eta_j, t_j)\to (\eta, t)$ and $\eta_j\in g_j'$
  for all~$j$, we have $\eta\in g'$. Then if $(\xi_j, t_j)\to(\xi, t)$ with
  $\xi_j\in g_j$ for all~$j$, we define~$\eta_j = \Psi_{t_j}^{-1}(\xi_j)\in
  g_j'$, and take a subsequence if necessary to ensure that $(\eta_j, t_j)\to
  (\eta, t)$ with $\eta\in g'$. Writing $g=\Psi_t(g')$, it only remains to
  show, using Lemma~\ref{lem:convergence-Psi}, that $(\xi_j, t_j) =
  \Psi_\ast(\eta_j, t_j) \to  \Psi_\ast(\eta, t)$, so that $\xi\in g$ as
  required.

  We will therefore assume, in each case, that $(\eta_j, t_j) = \ystj\to(\eta,
  t)=\yst$ is a sequence in~$\sDast$, and show that the decomposition element
  $g'\in\cG'_t$ which contains $\eta$ only depends on the decomposition
  elements $g'_j\in  \cG'_{t_j}$ which contain the $\eta_j$. The proof in the
  given case will then be completed by showing (or observing) that one of the
  conditions of Lemma~\ref{lem:convergence-Psi} holds.

  The decomposition elements $X_t\in \cG'_t$ (see
  Definition~\ref{def:decomp-primed}) which contain $\partial'$ play a special
  role in the arguments. Observe that in all cases these contain the
  verticals $\{\by\}\times[0,\infty]$ above the points $\by\in
  \hS_t$ which have the property that $\hB^r_t(\by)_0\not\in\ingam_t$ for all
  $r\in\Z$; and that in the rational interior and late endpoint cases, $X_t$ is
  equal to the union of these verticals.

  If infinitely many of the $s_j$ are equal to~$0$ then $s=0$, and hence (using
  Lemma~\ref{lem:convergence-Psi}~(a)), $\xi\in \Psi_t(X_t)$. We will therefore
  always assume that $s>0$ and that $s_j>0$ for all~$j$.

  Sequences of type~A are the hardest to treat, mainly because there are
  decomposition elements (other than~$X_t)$ which are not uniformly landing.
  They also involve the most subcases, since limits of type~A are possible, and
  because limits of type~B can be approached in two quite different ways,
  either from within their height interval or from outside it. We will treat
  this case in detail --- although methods of arguments will be successively
  abbreviated as they recur --- and treat sequences of types~B and~C more
  briefly. Although the proof is quite long, it involves nothing more than the
  careful enumeration of cases and their analysis using the explicit
  description of the decomposition $\cG'_\ast$.

\medskip\medskip

  \noindent\textbf{Case A: }All $t_j$ are of rational interior or late endpoint
  type.

  In this case the decompositions~$\cG'_{t_j}$ are given by
  Lemma~\ref{lem:rat-general-strong-stable} or, in the NBT case, by
  Lemma~\ref{lem:rat-NBT-strong-stable}. Let the height of~$t_j$ be $m_j/n_j$.
  Suppose first that infinitely many (and so, without loss of generality, all)
  of the $\eta_j$ are contained in the $n_j$-stars $X_{t_j}$, so that for
  each~$j$ we have $\hB^r_{t_j}(\by\sj)_0\not\in\ingam_{t_j}$ for all~$r\in\Z$.
  Since $\hB_t$ and $\gamma_t$ vary continuously with~$t$, it follows that
  $\hB^r_t(\by)_0\not\in\ingam_t$ for all $r\in\Z$: therefore $\eta\in X_t$.
  Observing that $\by$ and $\by\sj$ are landing of level~0, and hence of
  level~1,  completes the proof using Lemma~\ref{lem:convergence-Psi}~(b).

  We can therefore assume that none of the $\eta_j$ lie in $X_{t_j}$, and so
  can define $k_j\in\Z$ and $0\le i_j\le n_j-1$ such that
  \[
    \eta_j\in A_{k_j, i_j} \cup D_{k_j, i_j} \cup \bigcup_{y\in (c_u,
    a]\setminus\{\min(q_{n_j-1}, \hq_{n_j-1})\}} \Gamma'(y, k_j, i_j)
  \] for each~$j$, where $A_{k_j, i_j}$, $D_{k_j, i_j}$, and $\Gamma'(y, k_j,
  i_j)$ are the elements of $\cG'_{t_j}$ defined in the statements of
  Lemmas~\ref{lem:rat-general-strong-stable}
  and~\ref{lem:rat-NBT-strong-stable} (we suppress the dependence of these
  decomposition elements, as well as that of~$c_u$ and $q_{n_j-1}$, on~$t_j$).

  There are three possibilities:
  \begin{enumerate}[(a)]
    \item The sequence $(n_jk_j + i_j)$ is not bounded above. Then, by
    Lemma~\ref{lem:rational-diameter-bound}~(b), there is a subsequence of the
    $\xi_j = \Psi_{t_j}(\eta_j)$ contained in decomposition elements whose
    diameters go to zero.

    \item The sequence $(n_jk_j +i_j)$ is not bounded below so that, taking a
    subsequence, we can assume that $k_j<0$ for all~$j$ and that
    $n_jk_j+i_j\to-\infty$. For each~$j$, we therefore have either that
    $\hB^r_{t_j}(\by\sj)_0\not\in\ingam_{t_j}$ for all $r\in\Z$ with $r\le
    -(n_jk_j+i_j+1)$, or that $s_j\le 2^{n_jk_j + i_j + 1}$ (depending on
    whether $\eta_j$ is in a vertical of its decomposition element, or is in
    one of the horizontals, or disks of $D_{k_j, i_j}$). Since $s\not=0$ it
    follows that $\hB^r_t(\by)_0\not\in\ingam_t$ for all $r\in\Z$. Thus
    $\eta\in X_t$ and the proof is completed using Lemma~\ref{lem:convergence-Psi}~(b).

    \item The sequence $(n_jk_j + i_j)$ is bounded so that, taking a
    subsequence, we can assume it to be a constant~$N$. Acting on $\hT_\ast$ by
    the decomposition-preserving homeomorphism $\hH_\ast^{-N}$, we can further
    assume that~$N=0$ (i.e. that $k_j=i_j=0$ for all~$j$), so that
    \[
      \eta_j \in A_{0, 0} \cup D_{0, 0} \cup
      \bigcup_{y\in(c_u, a]\setminus\{\min(q_{n_j-1}, \hq_{n_j-1})\}}\Gamma'(y,
      0, 0)\quad\text{ for all }j.
    \] 
    Taking another subsequence, we can assume that $\eta_j\in A_{0,0}$ for
    all~$j$; or $\eta_j\in D_{0,0}$ for all~$j$; or $\eta_j\in \bigcup_y
    \Gamma'(y, 0, 0)$ for all~$j$.

    \noindent\textbf{1. }If $\eta_j\in A_{0,0}$ for all~$j$, then $\by\sj =
    \bt(c_u, 0, 0) = \thr{B_{t_j}(a), c_u, B_{t_j}^{-1}(c_u),
    B_{t_j}^{-2}(c_u), \ldots}$
    (Definition~\ref{notn:threads-rat-interior}~(b)) and $s_j\in[2,\infty]$.
    Therefore $s\in[2,\infty]$, and since $B_t\colon S\to S$, $B_t^{-1}\colon
    S\setminus\{B_t(a)\}\to S$, and~$c_u$ all depend continuously on~$t$, we
    have $\by =
    \thr{B_t(a), c_u, B_t^{-1}(c_u), B_t^{-2}(c_u), \ldots}$ provided that
    $c_u$ is not in the $B_t$-orbit of $a$: that is, provided that~$t$ is not
    of NBT type.

    If~$t$ is of rational interior, late endpoint, or normal or quadratic-like
    endpoint type but not of NBT type, then $\by=\bt(c_u, 0, 0)$, and hence
    $\eta\in A_{0,0}$ (compare with
    Definition~\ref{notn:threads-rat-interior}~(b) and
    Lemmas~\ref{lem:rat-general-strong-stable},
    \ref{lem:rat-endpoint-strong-stable},
    and~\ref{lem:rat-quadratic-strong-stable}); while if~$t$ is of irrational
    or early endpoint type then $\by=\bt(c_u, 1)$ and hence $\eta\in A_{1,t}$
    (compare with Definition~\ref{notn:threads-irrat}~(a) and
    Lemma~\ref{lem:ssc-irrational}). If~$t$ is of NBT type, then~$t$ and $t_i$
    (for sufficiently large~$i$) all have the same height~$m/n$, and by taking
    a subsequence we can assume that either $B_{t_j}^n(a)\in(c_u, a]$ for
    all~$i$, or that $B_{t_j}^n(a)\in[\re_u, c_u)$ for all~$i$ (it is
    impossible to have $B_{t_j}^n(a)=c_u$, since there is no decomposition
    element $A_{0,0}$ in the NBT case). In the former case we have that
    $B_{t_j}^{-n}(c_u)\to \re_u$, so that $\by = \thr{B_t(a),c_u,
    B_t^{-1}(c_u), \ldots, B_t^{-(n-1)}(c_u), \re_u, B_t^{-1}(\re_u), \ldots} =
    \thr{q_0, q_{n-1}, q_{n-2}, \ldots, q_0, \re_u, B_t^{-1}(\re_u), \ldots} =
    \bt(\re_u, 1, 0)$; and in the latter case we have $B_{t_j}^{-n}(c_u)\to a$,
    so that $\by = \bt(a, 1, 0)$. Hence $\eta\in D_{0,0}$ (compare with
    Lemma~\ref{lem:rat-NBT-strong-stable}). 

    Since $\by$ and all of the $\by\sj$ are landing of level~1, the proof when
    $\eta_j\in A_{0,0}$ for all~$j$ is complete by
    Lemma~\ref{lem:convergence-Psi}~(b).

\medskip

    \noindent\textbf{2. }If $\eta_j\in D_{0,0}$ for all~$j$, then, referring to
    the descriptions of $D_{0,0}$ in the statements of
    Lemmas~\ref{lem:rat-general-strong-stable}
    and~\ref{lem:rat-NBT-strong-stable}, we have that for each~$j$ one of the
    following occurs (and so, taking a subsequence, one of them occurs for
    all~$j$):
    \begin{enumerate}[(i)]
      \item $\eta_j$ is in one of the verticals of the crosscuts in the
      description of~$D_{0,0}$: that is, $\by\sj$ is one of $\bt(\hq_{n_j-1},
      0, 0)$, $\bt(a, 1, 0)$, and $\bt(\re_u, 1, 0)$, and
      $s_j\in[1+u(q_{n_j-1}), \infty]$ (with the case $\bt(\hq_{n_j-1}, 0, 0)$
      omitted if~$t_j$ is of NBT type). By
      Definition~\ref{notn:threads-rat-interior}, we have
      \begin{eqnarray*}
        \bt(\hq_{n_j-1}, 0, 0) &=& \thr{B_{t_j}(a), \hq_{n_j-1},
        B_{t_j}^{-1}(\hq_{n_j-1}), B_{t_j}^{-2}(\hq_{n_j-1}), \ldots}  , \\
        \bt(a, 1, 0) &=& \thr{B_{t_j}(a), q_{n_j-1}, \ldots, q_1, q_0, a,
        B_{t_j}^{-1}(a), B_{t_j}^{-2}(a), \ldots}, \quad\text{ and}\\
        \bt(\re_u, 1, 0) &=& \thr{B_{t_j}(a), q_{n_j-1}, \ldots, q_1, q_0,
        \re_u, B_{t_j}^{-1}(\re_u), B_{t_j}^{-2}(\re_u), \ldots}.
      \end{eqnarray*} Note also that the function $u = u_t\colon S\to[0,1]$ of
      Definition~\ref{notn:u-y} varies continuously with~$t$.

      \item $\eta_j$ is in a horizontal of the crosscuts in the description
      of~$D_{0,0}$, but does not lie in the set \mbox{$[\bt(a, 1, 0),
      \bt(\re_u, 1, 0)] \,\times\,[u_{0,0}, v_{0,0}]$}: that is,
      $\by\sj = \bt(y, 0, 0) = \thr{B_{t_j}(a), y, B_{t_j}^{-1}(y), \ldots}$
      for some $y$ between $q_{n_j-1}$ and $\hq_{n_j-1}$, and
      $s_j=1+u(q_{n_j-1})$. (This case does not occur if~$t_j$ is of NBT type.)

      \item $\by\sj\in[\bt(a, 1, 0), \bt(\re_u, 1, 0)]$ and $s_j\in
      [1+u(q_{n_j-1}), n_j+1]$. (Here the interval is the one with the given
      endpoints which contains $\bq_0$: therefore $\by\sj$ lies in this
      interval if and only if its first $n_j+1$ entries are $\thr{B_{t_j}(a),
      B_{t_j}^{n_j}(a), \ldots B_{t_j}(a), \ldots}$.)
    \end{enumerate}

    \medskip

    Consider first the case where~$t$ is of rational interior or late endpoint
    type with height $m/n$, so that $m_j/n_j = m/n$ for all (sufficiently
    large)~$j$. Since~$n_j=n$, sequences $(\eta_j)$ of type~(i) converge to
    $(\by, s)$ with $\by\in\{\bt(\hq_{n-1}, 0, 0), \bt(a, 1, 0),
    \bt(\re_u, 1, 0)\}$ and $s\in [1+u(q_{n-1}), \infty]$ (notice that if~$t$
    is of NBT type and $t_j$ is not, then sequences of the form $\bt(\hq_{n-1},
    0, 0)$ --- which depend on~$j$ since both $q_{n-1}$ and $B_{t_j}$ do ---
    converge either to $\bt(a, 1, 0)$ or $\bt(\re_u, 1, 0)$). Sequences
    $(\eta_j)$ of type~(ii) converge either to $(\bt(y, 0, 0), s)$ with $y$
    between $q_{n-1}$ and $\hq_{n-1}$ and $s=1+u(q_{n-1})$, or to limits of
    type~(i). Finally, sequences $(\eta_j)$ of type~(iii) converge to $(\by,
    s)$ with $\by\in[\bt(a, 1, 0), \bt(\re_u, 1, 0)]$ and $s\in [1+u(q_{n-1}),
    n+1]$.

    Therefore $\eta\in D_{0,0}$. Since $\by$ and all of the $\by\sj$ are
    landing of level~1 in type~(i), and $s<\infty$ in types~(ii) and~(iii), the
    proof in this case is complete.

    \smallskip

    Now suppose that $t$ is of rational normal or quadratic-like endpoint type
    with height~$m/n$. We will assume that this is a right hand endpoint, so
    that $q_{n-1} = B_t^n(a) = \re_u$: the left hand endpoint cases are
    similar. By taking a subsequence, we can reduce to one of two
    possibilities: first, that $m_j/n_j=m/n$ for all~$j$ (we approach the
    endpoint from inside the height interval); or second, that $n_j\to\infty$
    as $j\to\infty$ (we approach the endpoint from outside the height
    interval).
    \begin{itemize}
      \item Suppose that we approach the endpoint from inside the height
      interval. Then sequences~($\eta_j$) of type~(i) converge to $(\by,s)$
      with $\by\in\{\bt(a, 0, 0), \bt(a, 1, 0), \bt(a, 2, 0)\}$ and
      $s\in[1,\infty]$. Sequences~$(\eta_j)$ of type~(ii) converge either to
      $(\bt(y, 0, 0), 1)$ for some $y\in\ingam_t$, or to limits of type~(i).
      Finally, sequences~$(\eta_j)$ of type~(iii) converge to $(\by, s)$ with
      $s\in[1, n+1]$ and the first~$n+1$ entries of $\by$ being $\thr{B_t(a),
      \re_u, \ldots, B_t(a), \ldots}$: that is, $\by\in \bigcup_{k\ge 1}L_{k,
      0}$.

      Therefore $\eta$ is in the set $D_0$ of
      Lemma~\ref{lem:rat-endpoint-strong-stable} (or of
      Lemma~\ref{lem:rat-quadratic-strong-stable} in the quadratic-like
      endpoint case), and hence $\eta\in X_t$, and the proof is completed since
      $\by$ and all of the $\by\sj$ are landing of level~1 in type~(i), and
      $s<\infty$ in types~(ii) and~(iii).

      \item Suppose that we approach the endpoint from outside of the height
      interval, so that $n_j\to\infty$ as $j\to\infty$. By taking a
      subsequence, we can assume that $q_{n_j-1}\to y\in\gamma_t$. Since the
      $q_{n_j-1}$ are determined by the~$t_j$, i.e.\ by the decomposition
      elements~$g_j'$ containing the~$\eta_j$, it is enough to show that the
      decomposition element containing~$\eta$ depends only on~$y$: in fact, we
      will show that this decomposition element is $A_{0,0}$ if $y=c_u$; is
      $X_t$ if $y=a$ or $y=\re_u$; and is $\Gamma'(y, 0, 0)$ otherwise.

      If $y\in\ingam_t$, then sequences $(\eta_j)$ of type~(i) converge to
      $\eta=(\by, s)$ with $\by = \bt(y, 0, 0)$ or $\by=\bt(\hy, 0, 0)$, and
      $s\in[1+u(y),
      \infty]$: therefore $\eta$ is contained in $\Gamma'(y,0,0)$ if
      $y\not=c_u$, and in $A_{0,0}$ if $y=c_u$ (see
      Lemma~\ref{lem:rat-endpoint-strong-stable}). Sequences~$(\eta_j)$ of
      type~(ii) converge to $\eta = (\bt(z, 0, 0), 1+ u(y))$ for some~$z$
      between $y$ and $\hy$: therefore $\eta$ is contained in $\Gamma'(y, 0,
      0)$ if $y\not=c_u$, and in $A_{0,0}$ if $y=c_u$ (in which case $z=y$).
      Sequences~$(\eta_j)$ of type~(iii) converge to $(\bt(y, 0, 0), s)$ with
      $s\in[1+u(y), \infty]$, since the first $n_j+1$ entries of $\by\sj$ are
      $\thr{B_{t_j}(a), q_{n_j-1}, B_{t_j}^{-1}(q_{n_j-1}), \ldots
      B_{t_j}^{-(n_j-1)}(q_{n_j-1}), \ldots}$, and again $\eta$ is in
      $\Gamma'(y, 0, 0)$ if $y\not=c_u$, and in $A_{0,0}$ if $y=c_u$.

      As before, $\by$ and all of the $\by\sj$ are landing of level~1 in
      type~(i), and $s<\infty$ in type~(ii). For sequences of type~(iii), we
      have that $\by = \bt(y,0,0)$ is landing of level~1, $s_j\le n_j+1$, and
      $y\sj_{i}\not\in\ingam_{t_j}$ for $2\le i\le n_j$, so that the proof is
      completed using Lemma~\ref{lem:convergence-Psi}~(c).

      \smallskip

      If $y = a$ or $y=\re_u$, then sequences of type~(i) and of type~(iii)
      converge to $\eta = (\by, s)$ with $\hB^r_t(\by)_0\not\in\ingam_t$ for
      all~$r\in\Z$ (and with $s\in[1,\infty]$): that is, to $\eta\in X_t$;
      while sequences of type~(ii) converge to $\eta=(\bt(z, 0, 0), 1)$ for
      some $z\in\gamma_t$, which is contained in the set~$B_{0,0}$ of
      Lemma~\ref{lem:rat-endpoint-strong-stable}, and hence in~$X_t$. That
      $\Psi_{t_j}(\by\sj, s_j)\to \Psi_t(\by, s)$ follows as when $y\in\ingam_t$.
    \end{itemize}

    \medskip

    The argument when~$t$ is of irrational or rational early endpoint type is
    similar. In this case we must have $n_j\to\infty$. Taking a subsequence so
    that $q_{n_j-1}\to y\in\gamma_t$, and referring to the notation of
    Lemma~\ref{lem:ssc-irrational}, we see that:
    \begin{itemize}
      \item If $y=c_u$, then sequences of types~(i), (ii), and (iii) all
      converge to elements of $A_1$;

      \item If $y\in\ingam_t\setminus\{a, \re_u\}$, then sequences of all three
      types converge to elements of $C_{1, \min(y, \hy)}$; and

      \item If $y\in\{a, \re_u\}$, then sequences of all three types converge
      to elements of $D_1\subset X_t$.
    \end{itemize} The argument that $\Psi_{t_j}(\by\sj, s_j)\to \Psi_t(\by, s)$
    for sequences of types~(i),~(ii), and~(iii) uses parts~(b),~(a), and~(c) of
    Lemma~\ref{lem:convergence-Psi} respectively.

\medskip

    \noindent\textbf{3. }If $\eta_j = (\by\sj, s_j)\in\bigcup_y \Gamma'(y, 0,
    0)$ for all~$j$, then let $y_j\in(c_u, a]\setminus\{\min(q_{n_j-1},
    \hq_{n_j-1})\}$ be such that $\eta_j\in\Gamma'(y_j, 0, 0)$, and take a
    subsequence so that $y_j\to y\in [c_u, a]$ (as usual, $c_u$ and $q_{n_j-1}$
    have a suppressed dependence on~$t_j$). Taking a further subsequence if
    necessary, we can assume that one of the following occurs for all~$j$:
    \begin{enumerate}[(i)]
      \item $\by\sj$ is either $\bt(y_j, 0, 0)$ or $\bt(\hy_j, 0, 0)$, and $s_j
      \in[1+u(y_j), \infty]$; or

      \item $\by\sj \in [\bt(y_j, 0, 0), \bt(\hy_j, 0, 0)]$, and $s_j =
      1+u(y_j)$. (The interval, as usual, is the one with the given endpoints
      which is disjoint from $\bP$.)
    \end{enumerate}

    If~$t$ is of irrational or early endpoint type; or if~$t$ is of rational
    interior or normal or quadratic-like endpoint type with height $m/n$ and
    $y\not=\min(q_{n-1},
    \hq_{n-1})$, it then follows straightforwardly that:
    \begin{itemize}
      \item If~$t$ is of rational interior or late endpoint type, then the
      limit~$\eta$ lies in $A_{0,0}$ if $y=c_u$; in $D_{0,0}$ if $y=a$; and in
      $\Gamma'(y, 0, 0)$ otherwise.

      \item If~$t$ is of rational normal or quadratic-like endpoint type,
      then~$\eta$ lies in $A_{0,0}$ if $y=c_u$; in $X_t$ if $y=a$; and in
      $\Gamma'(y, 0, 0)$ otherwise.

      \item If~$t$ is of irrational or rational early endpoint type,
      then~$\eta$ lies in $A_1$ if $y=c_u$; in $X_t$ if $y=a$; and in $C_{1,y}$
      otherwise.
    \end{itemize}

    \small

    Suppose, then, that~$t$ is of rational interior or normal or quadratic-like
    endpoint type with height~$m/n$, and that we have~$y=\min(q_{n-1},
    \hq_{n-1})$.
    \begin{itemize}
      \item If~$t$ is of interior type, then $m_j/n_j=m/n$ for all sufficiently
      large~$j$. Sequences $\by\sj$ of type~(i) converge to $\bt(a, 1, 0)$, to
      $\bt(\re_u, 1, 0)$, or to $\bt(\hq_{n-1}, 0, 0)$, while $s_j\to s\in
      [1+u(q_{n-1}), \infty]$, and hence $\eta_j\to\eta\in D_{0, 0}$.
      Similarly, sequences of type~(ii) converge to $\eta=(\by, 1+u(q_{n-1}))$,
      where $\by\in[\bt(\hq_{n-1}, 0, 0), \bt(a, 1, 0)] \cup [\bt(\hq_{n-1}, 0,
      0), \bt(\re_u, 1, 0)]$, so that $\eta\in D_{0,0}$.

      \item If~$t$ is of endpoint type and $m_j/n_j=m/n$ for all
      sufficiently large~$j$, then similarly $\eta\in D_0\subset X_t$.

      \item If~$t$ is of endpoint type and $n_j\to\infty$, then
      sequences of both types~(i) and~(ii) converge to $\eta=(\by, s)$ with
      $\hB_t^r(\by)_0\not\in\ingam_t$ for all $r\in\Z$: that is, to $\eta\in
      X_t$. 
    \end{itemize}

\medskip

    In all cases either $s<\infty$, or $\by$ and the $\by\sj$ are all landing
    of level~1, so that $\Psi_{t_j}(\by\sj, s_j)\to \Psi_t(\by, s)$ as
    $j\to\infty$ by Lemma~\ref{lem:convergence-Psi}~(a) and~(b).

  \end{enumerate}

\medskip\medskip

  \noindent\textbf{Case B: }All $t_j$ are of rational normal or quadratic-like
  endpoint type.

  In this case the decompositions~$\cG'_{t_j}$ are given by
  Lemmas~\ref{lem:rat-endpoint-strong-stable}
  and~\ref{lem:rat-quadratic-strong-stable}, and the limit parameter~$t$ cannot
  be of rational interior or late endpoint type. We will assume that all of the
  $t_j$ are of strict left hand endpoint type (either tent-like or
  quadratic-like): the right hand endpoint case is similar. Let the height
  of~$t_j$ be $m_j/n_j$.  Suppose first that infinitely many (and so, taking a
  subsequence, all) of the~$\eta_j$ are contained in the decomposition elements
  $X_{t_j}$: that is, in one of the sets~$D_i$ of
  Lemmas~\ref{lem:rat-endpoint-strong-stable}~(b)
  or~\ref{lem:rat-quadratic-strong-stable}~(b), or in one of the
  verticals~$L_{\bt(y)}$ of Lemma~\ref{lem:rat-quadratic-strong-stable}~(c). We
  will show that $\eta\in X_t$.

  If infinitely many of the~$\eta_j$ are contained in the lines $L_{\bq_i -
  m_j^{-1}\bmod n_j}$, $L_{\bt(\re_u, k, i)}$, or $L_{\bt(y)}$  then
  $\hB_t^r(\by)_0\not\in\ingam_t$ for all $r\in\Z$, so that $\eta\in X_t$ as
  required. We can therefore assume that there are $k_j\in\Z$ and $0\le i_j\le
  n_j-1$ such that $\eta_j = (\by\sj, s_j)$ satisfies $\by\sj\in\inR_{k_j,
  i_j}$ (that is, $\by\sj = \bt(y_j, k_j, i_j)$ for some $y_j\in\ingam_{t_j}$),
  and $s_j\in(0, u_{k_j, i_j}]$, where
  \[ u_{k_j,i_j} =
    \begin{cases} k_jn + i_j + 1 & \text{ if }k_j\ge 0, \\ 1/2^{|k_j|n - i_j} & \text{ if
      }k_j<0.
    \end{cases}
  \] 
  The sequence $(n_jk_j+i_j)$ must therefore be bounded below since $s>0$. If
  it is not bounded above then, since the first $n_jk_j+i_j+1$ entries of
  $\bt(y_j, k_j, i_j)$ are disjoint from $\ingam_{t_j}$, we have
  $\hB_t^r(\by)_0\not\in\ingam_t$ for all $r\in\Z$, and hence $\eta\in X_t$. We
  can therefore assume that $n_jk_j+i_j$ is constant and, acting on $\hT_\ast$
  by the decomposition-preserving homeomorphism $\hH_\ast^{-n_jk_j-i_j}$, that
  it is equal to~$0$, so that $k_j=i_j=0$, for all~$j$, and $s\in(0,1]$. Take a subsequence so
  that $y_j\to y\in\gamma_t$, and either $m_j/n_j$ is constant or
  $n_j\to\infty$.

  If $m_j/n_j$ is constant, then $t$ is of rational endpoint type and either
  $\eta_j\to (\bt(y, 0, 0), s) \in B_{0,0}$, or (if $y\not\in\ingam_t$)
  $\hB_t^r(\by)_0\not\in\ingam_t$ for all $r\in\Z$. Therefore $\eta\in X_t$.

  Suppose then that $n_j\to\infty$ as $j\to\infty$. If $y$ is not on the
  $B_t$-orbit of $B_t(a)$ then $\eta = (\by, s)$ with $s\in(0, 1]$ and $\by =
  \thr{B_t(a), y, B_t^{-1}(y), B_t^{-2}(y), \ldots}$. If~$t$ is of rational
  normal endpoint type then $\eta\in B_{0,0}\subset X_t$, while if~$t$ is of
  irrational or rational early endpoint type then $\by=\bt(y, 1)$ and (see
  Lemma~\ref{lem:ssc-irrational}) $\eta\in D_1\subset X_t$. On the other hand,
  if $y$ is on the $B_t$-orbit of $B_t(a)$, then $t$ is of rational normal
  endpoint type and $y\in\{a, \re_u\}$, so that $\hB_t^r(\by)_0\not\in\ingam_t$
  for all $r\in\Z$. Therefore $\eta\in X_t$.

  This completes the proof that if $\eta_j\in X_{t_j}$ for all~$j$, then
  $\eta\in X_t$. We can therefore assume that there are $k_j\in\Z$ and $0\le
  i_j\le n_j-1$ such that
  \[
    \eta_j\in A_{k_j, i_j} \cup \bigcup_{y\in(c_u, a)} \Gamma'(y, k_j, i_j)
  \] for each~$j$. The remainder of the proof is now similar to but simpler
  than that in case A. By the same argument as in that case (using part~(c)
  rather than part~(b) of Lemma~\ref{lem:rational-diameter-bound}), we can
  reduce to having $k_j=i_j=0$ for all~$j$.

  \begin{itemize}
    \item If $\eta_j\in A_{0,0}$ for all~$j$ then $\eta\in A_{0,0}$ if~$t$ is
    of rational normal or quadratic-like endpoint type, and $\eta\in A_1$ if
    $t$ is of irrational or rational early endpoint type.
    \item If $\eta_j\in \bigcup_y\Gamma'(y_j, 0, 0)$ for some sequence $y_j\in
    (c_u, a)$, then take a subsequence so that $y_j\to y\in[c_u, a]$. If~$y=a$
    then $\eta\in X_t$. If~$y=c_u$ then $\eta\in A_{0,0}$ if~$t$ is of rational
    normal or quadratic-like endpoint type, and $\eta\in A_1$ if $t$ is of
    irrational or rational early endpoint type. If $y\in(c_u, a)$, then
    $\eta\in\Gamma'(y, 0, 0)$ if~$t$ is of rational normal or quadratic-like
    endpoint type, and $\eta\in C_{1,y}$ if~$t$ is of irrational or rational
    early endpoint type.
  \end{itemize}

\medskip\medskip

  \noindent\textbf{Case C: }All $t_j$ are of irrational or rational early
  endpoint type.

  In this case the decompositions $\cG'_{t_j}$ are given by
  Lemma~\ref{lem:ssc-irrational}. If all of the $\eta_j$ are contained in the
  decomposition elements $X_{t_j}$, then $\eta\in X_t$ by an argument exactly
  analogous to that in case~B. We can therefore assume that there are
  integers~$r_j$ such that
  \[
    \eta_j\in A_{r_j} \cup \bigcup_{y\in(c_u, a)} C_{r_j, y}
  \] for each~$j$. If~$(r_j)$ is not bounded above, then by
  Lemma~\ref{lem:rational-diameter-bound}~(a) there is a subsequence of the
  $\xi_j=\Psi_{t_j}(\eta_j)$ contained in decomposition elements whose
  diameters go to zero; while if~$(r_j)$ is not bounded below then
  $\hB_t^r(\by)_0\not\in\ingam_t$ for all $r\in\Z$, so that $\eta\in X_t$. We
  can therefore assume that $r_j$ is constant and, acting on $\hT_\ast$ by the
  decomposition-preserving homeomorphism $\hH_\ast^{1-r_j}$, that $r_j=1$ for
  all~$j$. The analysis of the different cases then proceeds exactly as in
  case~B.
\end{proof}

This completes the proof of Theorem~\ref{thm:cont-var}.

\subsection{The post-critically finite tent map case}
\label{sec:gpa} In this section we consider the case in which~$f$ is a tent map
(of slope $t>\sqrt{2}$) for which the orbit of~$b$ is either periodic or
preperiodic. In particular (see Lemma~\ref{lem:height-intervals}~(a) and
Remark~\ref{rmk:terminology}), $q(\kappa(f))=m/n$ is rational, and $f$ is of
interior or normal endpoint type.

We will show that the sphere homeomorphism $F\colon\Sigma\to\Sigma$ constructed
in the proof of Theorem~\ref{thm:semi-conj} is pseudo-Anosov when
$\kappa(f)=\NBT(m/n)$; and otherwise is generalized pseudo-Anosov, in the sense
of the following definition from~\cite{gpa}.

\begin{defn}[Generalized pseudo-Anosov]
\label{def:gpA} A sphere homeomorphism $\Phi\colon S^2\to S^2$ is {\em
generalized pseudo-Anosov} if there exist
\begin{enumerate}[(a)]
\item a finite $\Phi$-invariant set~$Z$;
\item a pair $(\cF^s, \mu^s)$, $(\cF^u, \mu^u)$ of transverse measured
  foliations of $S^2\setminus Z$ (whose transverse measures are non-atomic and
  positive on open subsets on transversals) with countably many pronged
  singularities, which accumulate on each point of~$Z$ and have no other
  accumulation points; and
\item a real number $\lambda>1$ such that $\Phi(\cF^s, \mu^s) = (\cF^s,
\frac{1}{\lambda}\, \mu^s)$ and $\Phi(\cF^u, \mu^u) = (\cF^u, \lambda\,\mu^u)$.
\end{enumerate}
\end{defn}

We will do this by proving (Theorem~\ref{thm:gpA}) that~$F$ is topologically
conjugate to the explicit generalized pseudo-Anosov~$\Phi$ constructed
in~\cite{gpa} corresponding to the kneading sequence $\kappa(f)$. The existence
of the conjugacy will be a consequence of the following list of properties
of~$\Phi$ (see Figure~\ref{fig:ugpa}).

\begin{enumerate}[(P1)]
\item The homeomorphism~$\Phi$ is given by $\Phi = \tPhi/\!\!\sim\colon
R/\!\!\sim\,\to\, R/\!\!\sim$, where $\tPhi\colon R\to R$ is a continuous
self-map of a metric disk~$R$, and $\sim$ is a $\tPhi$-invariant equivalence
relation on~$R$ for which $R/\!\!\sim$ is a sphere.
\item There is a projection $\pi\colon R\to[a,b]$ which semi-conjugates $\tPhi$
  to the tent map~$f$.
\item For each~$x\in[a,b]$, the fiber $\fF_x:=\pi^{-1}(x)$ is a compact
  interval if~$x$ is not on the (finite) orbit of~$b$, and is a dendrite
  otherwise.
\item The map $x\mapsto \fF_x$ is upper semi-continuous with respect to the
  Hausdorff metric (that is, for every~$x_0\in[a,b]$ and every neighborhood~$U$
  of $\cF_{x_0}$, there is a neighborhood~$V$ of~$x_0$ with $\cF_x\subset U$
  for all $x\in V$).
\item $\tPhi$ is injective on each fiber $\fF_x$, and contracts it uniformly by
  a factor~$1/t$ (where~$t$ is the slope of the tent map~$f$).
\item The dynamics of $\tPhi$ on the boundary $\partial R$ is given by the
  outside map $B\colon S\to S$ corresponding to~$f$. More precisely (Lemma~16
  of~\cite{gpa}), there is a homeomorphism $\theta\colon S\to \partial R$ with
  the property that $\tPhi(\theta(y))\in\partial R$ if and only if
  $y\not\in\ingam$, and in this case $\tPhi(\theta(y)) = \theta(B(y))$.
  Moreover, $\tau(y)=\pi(\theta(y))$ for each~$y\in S$. We will suppress the
  homeomorphism~$\theta$, and label points and subsets of~$\partial R$ with the
  same symbols as the corresponding points and subsets of~$S$. With this
  convention, we have $\tPhi=B$ on $\partial R\setminus\ingam$.
\item If $x\in[a,f(a))$ or $x=b$ then $\tPhi(\fF_{z})=\fF_x$, where $z$ is the
  unique element of~$[a,b]$ with $f(z)=x$.
\item If $x\in[f(a),b)$, then $x$ has two $f$-preimages $z,\hz\in[a,\re]$. We
  have $\tPhi(\fF_{z}) \cup \tPhi(\fF_{\hz}) =
  \fF_x$; and $\tPhi(\fF_{z})$ and $\tPhi(\fF_{\hz})$ intersect at exactly one
  point, which is $\tPhi(z_u)=\tPhi(\hz_u)$. (Notice that $z_u,
  \hz_u\in\gamma$.)
\item The equivalence relation~$\sim$ is defined as follows: if $\xi,\xi'\in
  R$, then $\xi\sim \xi'$ if and only if there is some~$r\ge 0$ such that
  either $\tPhi^r(\xi)=\tPhi^r(\xi')$, or $\tPhi^r(\xi)$ and $\tPhi^r(\xi')$
  both belong to the periodic orbit~$P$ of $\tPhi$ on $\partial R$.  (This
  periodic orbit is given by Theorem~\ref{thm:outside-dynamics}~(b)(iii) in the
  interior case; and is the orbit of~$B(a)$ in the endpoint case.)
\end{enumerate} We will also use the following consequences of these
properties:
\begin{enumerate}[(P1)]
\setcounter{enumi}{9}
\item It follows from~(P7) and~(P8) that $\tPhi$ is injective away from
$\gamma\setminus\{c_u\}$, while if $y\in\gamma\setminus\{c_u\}$ then
$\tPhi^{-1}(\tPhi(y)) = \{y, \hy\}$. In particular the only point of $\partial
R$ which has more than one preimage is~$q_0$.
\item It follows from~(P7) and~(P8) that $\tPhi$ is surjective; and
\item It follows from~(P6), (P9), and (P10) that all of the non-trivial
  equivalence classes of~$\sim$ are contained in $\partial R$.
\end{enumerate}

\begin{figure}[htbp]
\begin{center}
\includegraphics[width=0.75\textwidth]{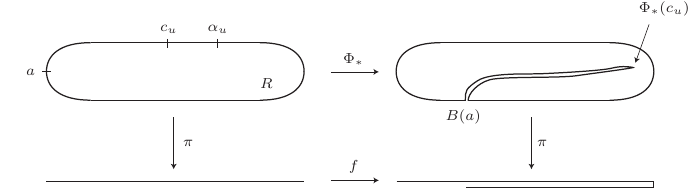}
\end{center}
\caption{Schematic representation of $\tPhi\colon R\to R$. The boundary of~$R$
  is identified with the circle~$S$. The map $\tPhi$ is injective except
  on~$\gamma\setminus\{c_u\}=[\re_u,a]\setminus\{c_u\}$, where it is
  two-to-one. We have $\tPhi(a)=\tPhi(\re_u)=B(a)$, where~$B$ is the outside
  map.}
\label{fig:ugpa}
\end{figure}

\begin{defns}[$\htPhi\colon \hR\to\hR$, $\hpi\colon\hR\to\hI$] Write
$\hR=\invlim(R, \tPhi)$, and let $\htPhi\colon\hR\to\hR$ be the natural
extension of~$\tPhi$. Let $\hpi\colon\hR\to\hI$ be the function induced by the
semi-conjugacy $\pi\colon R\to I$ of~(P2), that is, $\hpi(\bxi)_i =
\pi(\xi_i)$.
\end{defns}

We will show (Theorem~\ref{thm:gpA}) that $F\colon\Sigma\to\Sigma$ is
topologically conjugate to $\Phi\colon R/\!\!\sim\,\to R/\!\!\sim$. The proof
is structured as follows. We first show (Lemma~\ref{lem:hpi-conjugates}) that
$\hpi$ is a homeomorphism which conjugates~$\htPhi$ and~$\hf$. There are
therefore commutative diagrams
\begin{equation}
\label{eq:CD}
\begin{CD}
\hR    @>\htPhi>>   \hR\\
@V\hpi VV   @VV\hpi V\\
\hI    @>\hf>>      \hI\\
@VgVV       @VVgV\\
\Sigma @>F>>        \Sigma
\end{CD}
\qquad\qquad\text{ and }\qquad\qquad
\begin{CD}
\hR    @>\htPhi>>   \hR\\
@Vp_0VV   @VVp_0V\\ R    @>\tPhi>>      R\\
@Vp_\sim VV       @VVp_\sim V\\ R/\!\sim @>\Phi>>     R/\!\sim
\end{CD}
\end{equation} where $p_0(\bxi)=\xi_0$, and $p_\sim$ is the canonical
projection of~$\sim$. In order to show that $F$ and $\Phi$ are conjugate, it
therefore suffices to show that the fibers of $g\circ\hpi$ agree with those of
$p_\sim\circ p_0$: in other words, that $g(\hpi(\bxi))=g(\hpi(\bxi'))$ if and
only if $\xi_0\sim\xi_0'$. This will be done using the description of the
fibers of~$g$ given in Remark~\ref{rmk:semiconj-fibres}, together with the
technical Lemma~\ref{lem:xi-vary}.

\begin{lem}
\label{lem:hpi-conjugates} $\hpi$ is a homeomorphism which conjugates $\htPhi$
and $\hf$.
\end{lem}
\begin{proof} $\hpi$ is clearly continuous, and semi-conjugates $\htPhi$ and
$\hf$ since $\pi$ semi-conjugates $\tPhi$ and~$f$. We will exhibit an explicit
inverse $v\colon\hI\to\hR$ of~$\hpi$, which will establish the result since
$\hR$ and $\hI$ are compact metric spaces.

To do this, we first define a function $h\colon\hI\to R$. Let $\bx\in\hI$. Then
$\cF_{x_0}\supset \tPhi(\cF_{x_1}) \supset
\tPhi^2(\cF_{x_2}) \supset \cdots$ by~(P7) and~(P8). Since each
$\tPhi^j(\cF_{x_j})$ is compact and non-empty by~(P3) and~(P1), the
intersection $\bigcap_{j\ge 0}\tPhi^j(\cF_{x_j})$ is non-empty; moreover, it
contains a single point by~(P5). We define $h(\bx)\in\cF_{x_0}$ to be the
unique point of this intersection. Then $h\circ\hf = \tPhi\circ h$ by
construction. Moreover, $h$ is continuous: for if $U$ is a neighborhood
of~$h(\bx)$, then by~(P5) and the definition of~$h$ there is some~$N$ with
$\tPhi^N(\cF_{x_N})\subset U$. By~(P4), if $\bx'$ is sufficiently close
to~$\bx$ then we have also $\tPhi^N(\cF_{x'_N})\subset U$, and hence
$h(\bx')\in U$.

Define $v\colon\hI\to\hR$ by $v(\bx)_i = h(\hf^{-i}(\bx))$. That $v(\bx)\in\hR$
follows from $\tPhi\circ h = h\circ\hf$, which gives
$\tPhi(h(\hf^{-(i+1)}(\bx))) = h(\hf^{-i}(\bx))$ for each~$i$.

We now show that $v$ is inverse to~$\hpi$. First, let $\bx\in\hI$. Then for
each~$i\ge 0$,
\[
\hpi(v(\bx))_i = \pi(v(\bx)_i) = \pi(h(\hf^{-i}(\bx))) = x_i,
\] since $h(\hf^{-i}(\bx)) = h(\thr{x_i, x_{i+1}, \ldots}) \in
\cF_{x_i}$. On the other hand, if $\bxi\in\hR$, then for each $i\ge 0$,
\[ v(\hpi(\bxi))_i = h(\hf^{-i}(\hpi(\bxi))) = h(\thr{\pi(\xi_i),
  \pi(\xi_{i+1}), \ldots}) = \xi_i,
\] since for every $j\ge 0$ we have $\xi_i=\tPhi^{j}(\xi_{i+j})\in
\tPhi^j(\cF_{\pi(\xi_{i+j})})$, so that $\xi_i$ is the unique element of
$\bigcap_{j\ge 0}\tPhi^j(\cF_{\pi(\xi_{i+j})})$.
\end{proof}

The following lemma expresses the connection between the equivalence
relation~$\sim$ defined in~(P9) and the identifications on $\hI$ described in
Remark~\ref{rmk:semiconj-fibres}.

\begin{lem}
\label{lem:xi-vary} Let $\bxi, \bxi'\in\hR$.
\begin{enumerate}[(a)]
\item If~$f$ is of rational general type, then $\xi_0=\xi_0'$ but
  $\xi_1\not=\xi_1'$ if and only if either
\begin{eqnarray*}
\{\hpi(\bxi), \,\hpi(\bxi')\} &=& \{\omega(\bt(y,0,0)),\,
\omega(\bt(\hy,0,0))\} \text{ for some } y\in\gamma\setminus\{c_u, q_{n-1},
\hq_{n-1}\};\text{ or}\\
\{\hpi(\bxi), \,\hpi(\bxi')\} &=& \{\omega(\bt(\hq_{n-1},0,0)),\,
  \omega(\bt(a,1,0))\};\text{ or}\\
\{\hpi(\bxi), \,\hpi(\bxi')\} &=& \{\omega(\bt(\hq_{n-1},0,0)),\,
  \omega(\bt(\re_u,1,0))\}.
\end{eqnarray*}

\item If~$f$ is of rational NBT type, then $\xi_0=\xi_0'$ but
  $\xi_1\not=\xi_1'$ if and only if \[\{\hpi(\bxi),\, \hpi(\bxi')\} =
  \{\omega(\bt(y,0,0)),\, \omega(\bt(\hy,0,0))\}\text{ for some }
  y\in\gamma\setminus\{c_u\}.\]

\item If~$f$ is of rational (normal) endpoint type, then $\xi_0=\xi_0'$ but
  $\xi_1\not=\xi_1'$ if and only if either
\begin{eqnarray*}
\{\hpi(\bxi), \,\hpi(\bxi')\} &=& \{\omega(\bt(y,0,0)),\,
\omega(\bt(\hy,0,0))\} \text{ for some } y\in\ingam\setminus\{c_u\};\text{
or}\\
\{\hpi(\bxi), \,\hpi(\bxi')\} &=& \{\omega(\bt(A,0,0)),\,
  \omega(\bt(A,\ell,0))\} \text{ for some }\ell>0;\text{ or}\\
\{\hpi(\bxi), \,\hpi(\bxi')\} &=& \{\omega(\bt(A,0,0)),\,
  \omega(\bq_0)\},
\end{eqnarray*} where $A=a$ if $\kappa(f)=\rhe(m/n)$, and $A=\re_u$ if
$\kappa(f)=\lhe(m/n)$.

\item If~$f$ is of rational (normal) endpoint type, then there is some~$r\ge
  0$ with $\tPhi^r(\xi_0)\in P$ (the periodic orbit of~(P9)) if and only if
  there is some $i$ with $0\le i\le n-1$ such that either
  $\hpi(\bxi)=\omega(\bq_i)$, or $\hpi(\bxi)=\omega(\bt(A,k,i))$ for some
  $k\in\Z$, where $A=a$ if $\kappa(f)=\rhe(m/n)$, and $A=\re_u$ if
  $\kappa(f)=\lhe(m/n)$.
\end{enumerate}
\end{lem}

\begin{proof} Assume first that~$f$ is of general type. By~(P10), we have
$\xi_0=\xi_0'$ but $\xi_1\not=\xi_1'$ if and only
if~$\{\xi_1,\xi_1'\}=\{y,\hy\}$ for some $y\in\gamma\setminus\{c_u\}$.

Since~(i) $\tPhi^{-1}(\partial R)\subset \partial R$; (ii) the only point
of~$\partial R$ which has more than one preimage is
$q_0=\tPhi(a)=\tPhi(\re_u)$; and (iii) the only point of~$\gamma$ on the orbit
of $q_0$ is $q_{n-1}$, it follows that, for $y\in\gamma\setminus\{c_u, q_{n-1},
\hq_{n-1}\}$, we have
\begin{eqnarray*}
\{\xi_1,\, \xi_1'\} = \{y,\, \hy\}  &\iff& \{\bxi,\, \bxi'\} =
\{\thr{\tPhi(y), y, B^{-1}(y), \ldots} ,\, \thr{\tPhi(y), \hy, B^{-1}(\hy),
  \ldots}\} \\ &\iff& \{\hpi(\bxi),\, \hpi(\bxi')\} = \{\thr{f(\tau(y)),
  \tau(y),
  \tau(B^{-1}(y)),\ldots},\, \thr{ f(\tau(y)), \tau(\hy),
  \tau(B^{-1}(\hy)), \ldots } \}\\ &\iff& \{\hpi(\bxi),\, \hpi(\bxi')\} =
\{\omega(\bt(y,0,0)),\, \omega(\bt(\hy,0,0)) \}.
\end{eqnarray*} Here we have used (P6) (in particular that $\tPhi=B$ on
$S\setminus\ingam$ in the first line, and that $\pi=\tau$ on $S$ in the second
line); we have used that $\pi\circ\tPhi = f\circ\pi$ in the second line; and we
have used~(\ref{eq:rat-thread-land}) in the final line.

In the case $y\in\{q_{n-1}, \hq_{n-1}\}$, we have $\xi_1=q_{n-1}$ if and only
if
\[\bxi = \thr{\tPhi(q_{n-1}), q_{n-1}, q_{n-2}, \ldots, q_0, a, B^{-1}(a),
  \ldots}\text{ or }\bxi = \thr{\tPhi(q_{n-1}), q_{n-1}, q_{n-2}, \ldots, q_0,
\re_u, B^{-1}(\re_u),
  \ldots};\] while $\xi_1=\hq_{n-1}$ if and only if $\bxi =
\thr{\tPhi(\hq_{n-1}), \hq_{n-1}, B^{-1}(\hq_{n-1}), \ldots}$. These give the
other two possibilities in the statement of~(a),
using~(\ref{eq:rat-thread-land}).

\medskip

The argument in the NBT case is identical, except that the case $y\in\{q_{n-1},
\hq_{n-1}\}$ does not arise since $q_{n-1}=c_u$.

\medskip

Suppose then that~$f$ is of normal endpoint type. As in the general case, we
have $\xi_0=\xi_0'$ but $\xi_1\not=\xi_1'$ if and only if $\{\xi_1,\,\xi_1'\} =
\{y,\,\hy\}$ for some $y\in\gamma\setminus\{c_u\}$. The argument for
$y\not\in\{c_u, q_{n-1}, \hq_{n-1}\}$ (i.e.\ for $y\in\ingam\setminus\{c_u\}$)
is identical to the general case. Suppose then, without loss of generality,
that $\xi_1=a$ and $\xi_1'=\re_u$.

Consider first the case where $\kappa(f)=\rhe(m/n)$, so that we have
$B^n(a)=B^n(\re_u)=\re_u$. Then
\[\bxi=\thr{q_0, a, B^{-1}(a),  \ldots},\] while either
\[\bxi'=\thr{(q_0, \re_u, q_{n-2}, \ldots, q_1)^\infty}\quad\text{ or }\quad
  \bxi'=\thr{(q_0, \re_u, q_{n-2},
  \ldots, q_1)^\ell, q_0, a, B^{-1}(a), \ldots} \text{ for some }\ell>0.\]
Therefore $\hpi(\bxi)=\omega(\bt(a,0,0))$; while either
$\hpi(\bxi')=\omega(\bq_i)$ or
\mbox{$\hpi(\bxi')=\omega(\bt(a,\ell,0))$} for some~$\ell>0$. Here we have
used~(\ref{eq:rat-thread-land}); and we have used~(\ref{eq:omega}) to show that
$\hpi(\thr{(q_0, \re_u, q_{n-2},
  \ldots, q_1)^\infty}) = \omega(\bq_i)$.

The case where $\kappa(f)=\lhe(m/n)$ works analogously, and the result follows.

\medskip

For~(d), there is some~$r\ge0$ with $\tPhi^r(\xi_0)\in P$ if and only if either
$\xi_0\in P$ (i.e.\ $\xi_0=q_i$ for some~$i$), or $\xi_0 = B^{-s}(A)$ for some
$s\ge 0$. This is equivalent to
\begin{eqnarray*}
\bxi &=& \thr{(q_i, q_{i-1}, \ldots , q_0, q_{n-1}, \ldots, q_{i+1})^\infty}
  \quad\text{ for some~$i$, or}\\
\bxi &=& \thr{q_i, q_{i-1}, \ldots, q_0, (q_{n-1}, \ldots, q_0)^k, A,
  B^{-1}(A), \ldots} \quad\text{ for some~$i$ and some $k\ge 0$, or}\\
\bxi &=& \thr{B^{-s}(A), B^{-(s+1)}(A), \ldots}\quad \text{ for some $s\ge 0$}.
\end{eqnarray*} The first of these is equivalent to $\hpi(\bxi)=\omega(\bq_i)$
for some~$i$; the second to $\hpi(\bxi)=\omega(\bt(A,k,i))$ for some~$i$ and
some $k\ge 0$; and the third (noting Remark~\ref{rmk:alternative-intervals}) to
$\hpi(\bxi)=\omega(\bt(A,k,i))$ for some~$i$ and some~$k<0$.
\end{proof}

\begin{thm}
\label{thm:gpA} 
  Let~$f$ be a post-critically finite tent map of slope $t>\sqrt2$. Then the
  sphere homeomorphism~$F\colon\Sigma\to\Sigma$ constructed in the proof of
  Theorem~\ref{thm:semi-conj} is topologically conjugate to the generalized
  pseudo-Anosov $\Phi\colon R/\!\!\sim\,\to\, R/\!\!\sim$ constructed
  in~\cite{gpa}.
\end{thm}

\begin{proof} By~(\ref{eq:CD}) and the accompanying discussion, it is only
necessary to show that for all pairs $\bxi,\bxi'$ of distinct elements
of~$\hR$, we have $g(\hpi(\bxi))=g(\hpi(\bxi'))$ if and only if
$\xi_0\sim\xi_0'$.

\medskip

We start with the case where~$f$ is of general type. In this case the points
of~$P$ (the periodic orbit of~(P9)) have unique \mbox{$\tPhi$-preimages} ---
which are, of course, in ~$P$ --- and hence, by~(P9), $\xi_0\sim\xi_0'$ if and
only if either
\begin{enumerate}[(i)]
\item $\tPhi^r(\xi_0)=\tPhi^r(\xi'_0)$ for some $r\ge 0$; or
\item $\xi_0,\xi'_0\in P$.
\end{enumerate} Since $\bxi\not=\bxi'$, condition~(i) is equivalent to the
condition
\begin{equation}
\label{eq:condition} \text{There is some }s\in\Z\text{ with
}\quad\htPhi^s(\bxi)_0 =
\htPhi^s(\bxi')_0,\quad\text{ but }\quad\htPhi^s(\bxi)_1 \not=
\htPhi^s(\bxi')_1.
\end{equation} For if~(\ref{eq:condition}) holds, then either $s\le 0$, in
which case $\xi_0=\xi_0'$; or $s>0$, in which case
$\tPhi^r(\xi_0)=\tPhi^r(\xi'_0)$ for $r=s$. Conversely, suppose that there is
some $r\ge 0$ for which $\tPhi^r(\xi_0)=\tPhi^r(\xi_0')$, and pick this~$r$ to
be as small as possible. If $r>0$ then we have $\htPhi^r(\bxi)_0 =
\htPhi^r(\bxi')_0$, but $\tPhi^{r-1}(\xi_0)\not=\tPhi^{r-1}(\xi_0')$, so that
$\htPhi^{r-1}(\bxi)_0 \not=\htPhi^{r-1}(\bxi')_0$; that is,
$\htPhi^r(\bxi)_1\not=\htPhi^r(\bxi')_1$, and hence~(\ref{eq:condition}) holds
with $s=r$. On the other hand, if $r=0$ then $\xi_0=\xi'_0$. Since
$\bxi\not=\bxi'$, there is some greatest~$i\ge0$ with $\xi_i=\xi'_i$. Then we
have $\htPhi^{-i}(\bxi)_0=\htPhi^{-i}(\bxi')_0$ but
$\htPhi^{-i}(\bxi)_1\not=\htPhi^{-i}(\bxi')_1$, so that~(\ref{eq:condition})
holds with $s=-i$.

Therefore condition~(i) holds if and only if there is some $s\in\Z$ such that
$\{\hpi(\htPhi^s(\bxi)),\,\hpi(\htPhi^s(\bxi'))\}$ is one of the pairs from the
statement of Lemma~\ref{lem:xi-vary}~(a). By Lemma~\ref{lem:hpi-conjugates},
this is equivalent to the existence of $s\in\Z$ such that
$\{\hf^s(\hpi(\bxi)),\,\hf^s(\hpi(\bxi'))\}$ is one of these pairs.
Setting~$r=-s$, this in turn is equivalent to the existence of $r\in\Z$ such
that $\{\hpi(\bxi),\,\hpi(\bxi')\}$ is the image under $\hf^r$ of one of these
pairs. That is, condition~(i) holds if and only if there is some $k\in\Z$ and
$0\le i\le n-1$ such that
\begin{eqnarray*}
\{\hpi(\bxi),\,\hpi(\bxi')\} &=&
\{\omega(\bt(y,k,i)),\,\omega(\bt(\hy,k,i))\} \text{ for some
}y\in\gamma\setminus\{c_u, q_{n-1}, \hq_{n-1}\}; \text{ or}\\
\{\hpi(\bxi),\,\hpi(\bxi')\} &=&
\{\omega(\bt(\hq_{n-1},k,i),\,\omega(\bt(a,k+1,i))\}; \text{ or}\\
\{\hpi(\bxi),\,\hpi(\bxi')\} &=&
\{\omega(\bt(\hq_{n-1},k,i),\,\omega(\bt(\re_u, k+1, i))\}.
\end{eqnarray*} Observing that condition~(ii) holds if and only if $\hpi(\bxi),
\hpi(\bxi')\in\omega(\bP)$ and comparing with
Remark~\ref{rmk:semiconj-fibres}~(e), we see that $\xi_0\sim\xi_0'$ if and only
if $g(\hpi(\bxi))=g(\hpi(\bxi'))$.

\medskip

The proof in the NBT case is similar, using Lemma~\ref{lem:xi-vary}~(b) and
Remark~\ref{rmk:semiconj-fibres}~(f).

\medskip

Suppose, then, than~$f$ is of (normal) endpoint type, so that $q_{n-1}$ is
either $a$ or $\re_u$: as before, we will write $A=\hq_{n-1}$, so that
$q_{n-1}=\re_u$ and $A=a$ when $\kappa(f)=\rhe(m/n)$; while $q_{n-1}=a$ and
$A=\re_u$ when $\kappa(f)=\lhe(m/n)$. The periodic orbit~$P$ of $\tPhi$
on~$\partial R$ is therefore $P=\{q_0, q_1, \ldots, q_{n-1}\}$ (that is, it is
equal to~$Q$), and the two $\tPhi$-preimages of $q_0$ are $a$ and $\re_u$.
By~(P9), $\xi_0\sim\xi_0'$ if and only if either
\begin{enumerate}[(i)]
\item $\tPhi^r(\xi_0) = \tPhi^r(\xi_0')$ for some $r\ge 0$; or
\item $\tPhi^r(\xi_0), \tPhi^r(\xi_0')\in P$ for some $r\ge 0$.
\end{enumerate}

As in the general case, condition~(i) is equivalent to~(\ref{eq:condition}),
which in turn is equivalent to the existence of $r\in\Z$ such that
$\{\hpi(\bxi), \hpi(\bxi')\}$ is the image under $\hf^r$ of one of the pairs
from the statement of Lemma~\ref{lem:xi-vary}~(c). That is, condition~(i) holds
if and only if there is some $k\in\Z$ and $0\le i\le n-1$ such that
\begin{eqnarray*}
\{\hpi(\bxi),\,\hpi(\bxi')\} &=&
\{\omega(\bt(y,k,i)),\,\omega(\bt(\hy,k,i))\} \text{ for some
}y\in\ingam\setminus\{c_u\}; \text{ or}\\
\{\hpi(\bxi),\,\hpi(\bxi')\} &=&
\{\omega(\bt(A,k,i),\,\omega(\bt(A,k+\ell,i))\} \text{ for some }\ell>0; \text{
or}\\
\{\hpi(\bxi),\,\hpi(\bxi')\} &=&
\{\omega(\bt(A,k,i),\,\omega(\bq_i)\}.
\end{eqnarray*}

On the other hand, by Lemma~\ref{lem:xi-vary}~(d), condition~(ii) holds if and
only if there are integers $k$ and $k'$, and $0\le i,i'\le n-1$ such that
$\hpi(\bxi)=\omega(\bq_i)$ or $\hpi(\bxi)=\omega(\bt(A,k,i))$; and $\hpi(\bxi')
= \omega(\bq_{i'})$ or $\hpi(\bxi')=\omega(\bt(A, k', i'))$.

Combining conditions~(i) and~(ii), we obtain that $\xi_0\sim\xi_0'$ if and only
if either there exist $y\in\ingam\setminus\{c_u\}$, $k\in\Z$, and $0\le i\le
n-1$ such that
$\{\hpi(\bxi),\,\hpi(\bxi')\}=\{\omega(\bt(y,k,i)),\,\omega(\bt(\hy,k,i))\}$;
or
\[
\{\hpi(\bxi),\,\hpi(\bxi')\} \subset \omega(\bQ) \cup
\{\omega(\bt(A,k,i))\,:\,k\in\Z,\, 0\le i\le n-1\}.
\] Comparing with Remark~\ref{rmk:semiconj-fibres}~(c), we see that
$\xi_0\sim\xi_0'$ if and only if $g(\hpi(\bxi))=g(\hpi(\bxi'))$, as required.
\end{proof}

A detailed description of the dynamics of the sphere homeomorphism $F$ in the
case where $f$ is a tent map but is not post-critically finite is the subject
of ongoing research. Here we only give a straightforward statement about the
way in which dynamical properties of a general~$f$ carry over to~$F$. Recall
that a Borel probability measure on a topological manifold~$M$ is called an
{\em Oxtoby-Ulam} measure, or {\em OU-measure}, if it is non-atomic, positive
on open sets, and assigns zero measure to the boundary of~$M$ (if it has one).

\begin{thm}[Sphere homeomorphism dynamics]
\label{thm:sphere-dynamics} Let~$f$ be a unimodal map satisfying the conditions
of Convention~\ref{conv:unimodal}, and $F\colon\Sigma\to\Sigma$ be the
corresponding sphere homeomorphism given by Theorem~\ref{thm:semi-conj}. Then
\begin{enumerate}[(a)]
\item if~$f$ is topologically transitive then so is~$F$;
\item if~$f$ has dense periodic points, then so does~$F$;
\item $f$ and~$F$ have the same number of periodic orbits of each period, with
the exception that, provided $\kappa(f)\not=1\Infs{0}$,
\begin{itemize}
  \item $F$ has one more fixed point than~$f$, and
  \item if~$f$ is of rational type with
$q(\kappa(f))=m/n\in(0,1/2)$, then~$F$ has either one or two fewer period~$n$
orbits than~$f$.
\end{itemize}
\item $f$ and $F$ have the same topological entropy; and
\item if~$f$ preserves an ergodic OU-measure, then~$F$ preserves an ergodic
  OU-measure with the same metric entropy.
\end{enumerate} In particular, if $f$ is a tent map of slope $t\in(\sqrt2,2]$
restricted to its dynamical interval, then~$F$ is topologically transitive, has
dense periodic points, has topological entropy~$\log(t)$, and has an invariant
ergodic OU-measure with metric entropy~$\log(t)$.
\end{thm}
\begin{proof} It is well known that if $f$ is topologically transitive or has
dense periodic points, then the same is true of its natural extension~$\hf$.
Since these properties are preserved by semi-conjugacy, (a) and~(b) follow. 

(c) follows from the explicit descriptions of the fibers of the semi-conjugacy
$g\colon \hI\to\Sigma$ given in Remark~\ref{rmk:semiconj-fibres}. The only
fiber which contains periodic points is the exceptional fiber in the rational
case: in the normal endpoint case this fiber contains the period~$n$
orbit~$\omega(\bQ)$; in the interior and late endpoint cases it is equal to the
period~$n$ orbit $\omega(\bP)$; and in the early or quadratic-like strict
endpoint cases, it contains both period~$n$ orbits $\omega(\bP)$ and
$\omega(\bQ)$ (or the single semi-stable orbit $\omega(\bP) =\omega(\bQ)$). In all cases, the exceptional fiber is collapsed to a fixed point of~$F$.

(d) is established in Remark~\ref{rmk:same-entropies}.

For~(e), it is also well known that if~$\mu$ is an $f$-invariant Borel
probability measure, then there is a unique $\hf$-invariant Borel probability
measure $\hmu$ on $\hI$ characterized by $(\pi_n)_*(\hmu)=\mu$ for all~$n$
(here $\pi_n\colon\hI\to I$ is defined by $\pi_n(\bx)=x_n$); moreover, $\hmu$
is ergodic if and only if $\mu$ is; and $\mu$ and $\hmu$ have the same metric
entropy. If $\hmu$ were atomic then $\mu$ would be also. Moreover, since a base
for the topology on~$\hI$ is given by the set of all $\pi_n^{-1}(U)$, where~$U$
is non-empty and open in~$I$, and since $\hmu(\pi_n^{-1}(U))=\mu(U)$, we have
that $\hmu$ is positive on open sets if $\mu$ is.

Write $\tmu=g_*(\hmu)$, so
that $\tmu$ is $F$-invariant and ergodic. Since $g$ is continuous, $\tmu$ is
also positive on open sets. To show that $\tmu$ is non-atomic, suppose for a
contradiction that there were some $z\in\Sigma$ with $\tmu(z)>0$. Then, since
$\tmu$ is $F$-invariant and ergodic, $z$ would belong to a periodic orbit of
full measure, a contradiction since the complement of this periodic orbit would
be open and non-empty.

Except perhaps for a single point (which we now know has $\tmu$-measure zero),
the point preimages of the semi-conjugacy $g$ are finite. Thus $g$ has finite
preimages $\tmu$-almost everywhere and so $\hmu$ and $\tmu$ have the same
metric entropy.

Therefore $\tmu$ is an $F$-invariant ergodic OU-measure as required. Since tent
maps of slope $t>\sqrt{2}$ restricted to their dynamical intervals are
transitive, have dense periodic points, have topological entropy $\log(t)$, and
have invariant ergodic OU-measures with metric entropy $\log(t)$, the final
statement follows.
\end{proof}

\begin{remark} A theorem of Oxtoby and Ulam (see~\cite{OU} and Appendix~2
of~\cite{AP}) states that if $m_1$ and $m_2$ are OU-measures on a manifold~$M$,
then there is a homeomorphism $h\colon M\to M$ with $h_*(m_1)=m_2$. Therefore
by conjugating the sphere homeomorphism~$F$, we can make it preserve any
OU-measure; in particular, one coming from an area form on~$\Sigma$.
\end{remark}

\appendix

\section{The embedding is independent of the unwrapping}
\label{app:independence} 
In this appendix we will prove the following result, which establishes that the
prime ends of $(\hT, \hI)$ are independent of the choice of unwrapping of the
unimodal map~$f$.

\medskip

\noindent\textbf{Theorem~\ref{thm:independence-unwrapping}.}\emph{ Any two
unwrappings of a unimodal map~$f$ are equivalent. }

\medskip

Recall that unwrappings $\barf_0$ and $\barf_1$ of~$f$, which have associated
Barge-Martin homeomorphisms \mbox{$\hH_0\colon(\hT_0,\hI)\to (\hT_0, \hI)$} and
$\hH_1\colon(\hT_1,\hI)\to(\hT_1, \hI)$, are equivalent if there is a
homeomorphism \mbox{$\lambda\colon\hT_0\to\hT_1$} which restricts to the
identity on $\hI$. Theorem~\ref{thm:independence-unwrapping} is a consequence
of the following recent deep result (Theorem~7.3 of~\cite{OT}).

\begin{thm}[Oversteegen--Tymchatyn]
\label{thm:OT} Let~$K$ be a continuum and $\{\alpha_t\}_{t\in[0,1]}\colon
K\to S^2$ be an isotopy of embeddings of~$K$ into the sphere. Then there is an
ambient isotopy $\{A_t\}_{t\in[0,1]}\colon S^2\to S^2$ with $A_0=\id$ and
$\alpha_t=A_t\circ \alpha_0$ for all~$t$.
\end{thm}

\begin{defn}[Weak unwrapping] A {\em weak unwrapping} of a unimodal map~$f$ is
an orientation-preserving near-homeomorphism $\barf\colon T\to T$ which is
injective on~$I$ with $\barf(I)\subset\{(y,s)\,:\,s\ge 1/2\}$, and satisfies
$\Upsilon\circ\barf|_I=f$.
\end{defn}

\begin{remark}
\label{rmk:weak-unwrap} A weak unwrapping differs from an unwrapping in that it
is not required that the second component of $\barf(y,s)$ is equal to~$s$
whenever $s\in[0,1/2]$. A Barge-Martin homeomorphism can be associated to a
weak unwrapping in exactly the same way as to an unwrapping, although it may
not have $\hI$ as a global attractor (if $\barf$ pushes points with
$s\in[0,1/2]$ outwards more strongly than the smash~$\Upsilon$ pulls them
inwards).
\end{remark}

\begin{defn}[u-near-isotopy] A homotopy $\{\barf_t\}\colon T\to T$ of weak
unwrappings of a fixed unimodal map~$f$ is called a {\em u-near-isotopy} (for
``unwrapping near-isotopy'').
\end{defn}

\begin{remark}
\label{rmk:EK} It is not obvious that a homotopy of near-homeomorphisms of~$T$
can be uniformly approximated by isotopies, and therefore merits the name {\em
near-isotopy}. That this is the case follows from a theorem of Edwards and
Kirby (Corollary~1.1 of~\cite{edwardskirby}), which states that the
homeomorphism group of any manifold is locally contractible, and hence locally
path connected.
\end{remark}

\begin{lem}
\label{lem:u-n-i-equiv} Any two u-near-isotopic unwrappings $\barf_0\colon T\to
T$ and $\barf_1\colon T\to T$ of a unimodal map~$f$ are equivalent.
\end{lem}

\begin{proof} Let $\{\barf_t\}$ be a u-near isotopy from~$\barf_0$
to~$\barf_1$, and for each~$t$ let $\hH_t \colon (\hT_t, \hI)\to (\hT_t,\hI)$
be the Barge-Martin homeomorphism associated with $\barf_t$. By
Theorem~\ref{thm:unwrap-family}, there are homeomorphisms $h_t\colon
\hT_t\to S^2$ with the property that $\{\Phi_t=h_t\circ \hH_t\circ
h_t^{-1}\}\colon S^2\to S^2$ is an isotopy. It follows from the construction of
the homeomorphisms $h_t$ (as compositions $h_t = p_1\circ\beta\circ\iota_t$) in
the proof of Corollary~2.3 of~\cite{param-fam} --- which applies equally well
to weak unwrappings as to unwrappings --- that $h_t|_{\hI}$ varies continuously
with~$t$, so that $\{h_t|_\hI\}$ is an isotopy of embeddings of the
continuum~$\hI$ into $S^2$.

By Theorem~\ref{thm:OT}, there is an isotopy $\{A_t\}\colon S^2\to S^2$ with
$A_0=\id$ and $h_t|_{\hI} = A_t\circ h_0|_{\hI}$ for all~$t$. Let~$\lambda =
h_1^{-1}\circ A_1\circ h_0\colon
\hT_0\to\hT_1$. Then $\lambda$ is the required homeomorphism which restricts to
the identity on~$\hI$.
\end{proof}

\begin{remark} Lemma~\ref{lem:u-n-i-equiv} extends to apply to any continuous
surjection $f\colon K\to K$, where~$K$ is a Peano non-separating planar
continuum. A theorem of Brechner and Brown~\cite{BB} states that such a
continuum~$K$ has a {\em Disk Mapping Cylinder Neighborhood}: that is, it can
be embedded in a disk~$D$ in such a way that there is a continuous map
$\phi\colon S^1\times[0,1]\to D$ with~$K=\phi(S^1\times\{1\})$, and
$\phi(y_1,s_1)=\phi(y_2,s_2) \implies s_1=s_2=1$. (In fact, Brechner and Brown
show that a planar continuum has a Disk Mapping Cylinder Neighborhood if and
only if it is Peano and non-separating.) Using the associated ``coordinates''
$(y,s)$, the Barge-Martin construction can be carried out exactly as in the
case~$K=I$, and the proof of Lemma~\ref{lem:u-n-i-equiv} goes through without
change.
\end{remark}

\medskip

We next show that there are only two u-near-isotopy classes of unwrappings of a
given unimodal map~$f$. Since a weak unwrapping~$\barf$ is injective on~$I$,
one of the following two cases must occur
\begin{enumerate}[(a)]
\item For every $x\in[a,c)$ either $\barf(x) = (f(x)_u,s_1)$ and
  $\barf(\hx)=(f(x)_\ell, s_2)$; or $\barf(x)=(f(x)_u, s_1)$ and
  $\barf(\hx)=(f(x)_u, s_2)$ with $s_1<s_2$; or $\barf(x)=(f(x)_\ell, s_1)$ and
  $\barf(\hx)=(f(x)_\ell, s_2)$ with $s_1>s_2$.
\item For every $x\in[a,c)$ either $\barf(x) = (f(x)_\ell,s_1)$ and
  $\barf(\hx)=(f(x)_u, s_2)$; or $\barf(x)=(f(x)_u, s_1)$ and
  $\barf(\hx)=(f(x)_u, s_2)$ with $s_1>s_2$; or $\barf(x)=(f(x)_\ell, s_1)$ and
  $\barf(\hx)=(f(x)_\ell, s_2)$ with $s_1<s_2$.
\end{enumerate} We say that $\barf$ is of type ``up'' in case~(a), and of type
``down'' in case~(b) (our preferred unwrapping is therefore of type ``down'').

\begin{lem}
\label{lem:up-down} If two unwrappings $\barf_0$ and $\barf_1$ of a unimodal
map~$f$ are of the same type, then they are u-near-isotopic.
\end{lem}

\begin{proof} We show first that any unwrapping $\barf$ of~$f$ is
u-near-isotopic to a weak unwrapping which is a homeomorphism, and which has
the same type as $\barf$. Since $\barf$ is a near-homeomorphism, it follows
from the theorem of Edwards and Kirby stated in Remark~\ref{rmk:EK} that it is
the endpoint of a pseudo-isotopy: that is, that there is a homotopy
$\{g_t\}\colon T\to T$ with $g_0=\barf$ and $g_t$ a homeomorphism for $t>0$.
Define $\alpha_t = g_t|_I\colon I\to T$. Since $\barf|_I$ is a homeomorphism by
definition, $\{\alpha_t\}$ is an isotopy of embeddings. By
Theorem~\ref{thm:OT}, there is an isotopy $\{A_t\}\colon T\to T$ with $A_0=\id$
and $\alpha_t=A_t\circ \alpha_0$ for all~$t$. Define $\barg_t = A_t^{-1}\circ
g_t$, so that $\barg_0 = \barf$. If $t>0$ then $\barg_t$ is a homeomorphism
with $\barg_t|_I = A_t^{-1}\circ\alpha_t = \alpha_0 =
\barf|_I $, so that $\barg_t$ is a weak unwrapping. Therefore $\barf$ is
u-near-isotopic to the homeomorphism weak unwrapping~$\barg_1$, which is
clearly of the same type as $\barf$.

We can therefore complete the proof by showing that if $\barf_0$ and $\barf_1$
are homeomorphism weak unwrappings of the same type, then they are
u-near-isotopic. Write~$\zeta\colon T\to I$ for the map $(y,s) \mapsto (y,1)$.
A {\em fiber} in~$T$ is an arc of the form $\zeta^{-1}(x)$ for some $x\in I$.
Since $\barf_0$ and $\barf_1$ have the same type, we can slide from one to the
other along fibers to produce an isotopy of embeddings $\{\alpha_t\}\colon I\to
T$ with $\alpha_0=\barf_0|_I$, $\alpha_1 =
\barf_1|_I$, and $\zeta\circ\alpha_t=f$ for all $t$. By Theorem~\ref{thm:OT}
there is an isotopy $\{A_t\}\colon T\to T$ with $A_0=\id$ and
$\alpha_t=A_t\circ\alpha_0$ for all~$t$. Let $\{B_t\}\colon T\to T$ be the
isotopy defined by $B_t = A_t\circ \barf_0$. Now
\[
\tau\circ B_t|_I \,=\,\tau\circ A_t\circ\barf_0|_I \,=\, \tau\circ
A_t\circ\alpha_0 \,=\, \tau\circ\alpha_t \,=\, f.
\] Therefore $\{B_t\}$ is a u-near-isotopy (consisting of homeomorphisms) from
$\barf_0$ to $B_1$.

Now $B_1|_I = A_1\circ \barf_0|_I = A_1\circ\alpha_0 = \alpha_1 =
\barf_1|_I$. Therefore $B_1$ and $\barf_1$ are orientation-preserving
homeomorphisms $T\to T$ which agree on the embedded arc~$I$: by the Alexander
trick, they are isotopic by an isotopy which is constant on~$I$. This isotopy
is a u-near-isotopy from $B_1$ to~$\barf_1$, so that $\barf_0$ and $\barf_1$
are u-near-isotopic as required.
\end{proof}

\begin{proof}[Proof of Theorem~\ref{thm:independence-unwrapping}] Let~$f$ be a
unimodal map. By Lemmas~\ref{lem:u-n-i-equiv} and~\ref{lem:up-down}, any two
unwrappings of~$f$ of the same type are equivalent. It therefore only remains
to show that there exist unwrappings $\barf_0$ and $\barf_1$ of different types
which are equivalent.

Let $\Gamma\colon T\to T$ be the involution defined by
$\Gamma(x_u,s)=(x_\ell,s)$ and $\Gamma(x_\ell,s)=(x_u,s)$. Let~$\barf_0$ be any
unwrapping of type ``down'', and let $\barf_1 = \Gamma\circ \barf_0\circ
\Gamma$, which is of type ``up''. Since $\Gamma$ commutes with the smash, if
$H_i=\Upsilon\circ\barf_i\colon T\to T$ for each~$i$, then $H_1 =
\Gamma\circ H_0\circ \Gamma$.

It follows that the map $\lambda\colon\hT_0\to \hT_1$ defined by
$\lambda(\thr{x_0,x_1,\ldots}) = \thr{\Gamma(x_0), \Gamma(x_1), \ldots}$ is a
homeomorphism which restricts to the identity on $\hI$.
\end{proof}

\section{Technical lemmas}
\label{app:technical} 
In this appendix we prove two lemmas about the dynamics of unimodal maps. As is
the case throughout the paper, we assume that unimodal maps satisfy the
conditions of Definition~\ref{def:unimodal} and Convention~\ref{conv:unimodal}.

\begin{lem}
\label{lem:eventually-onto} Let~$f\colon[a,b]\to[a,b]$ be unimodal with turning
point~$c$. Then there is some~$N$ such that $f^N([a,c])=[a,b]$.
\end{lem}
\begin{proof} $\kappa(f)\succ 101^\infty$ by
Convention~\ref{conv:unimodal}~(a), so that $\kappa(f) = 10(11)^\ell 0\ldots$
for some~$\ell\ge 0$. We shall show that $f^{2\ell+2}([a,c])=[a,b]$.

We first show by finite induction on $i$ that $[a,c] \subset f^{2i}([a,c])$ for
$0\le i\le \ell$. The base case is trivial. If $1\le i\le\ell$ then in
particular $\ell\ge 1$, so that $f^2(a)\ge c$. We have by the inductive
hypothesis that $f^{2i-2}([a,c])\supset [a,c]$. Therefore
$f^{2i-1}([a,c])\supset [f(a),b]$, and $f^{2i}([a,c])\supset [a,f^2(a)] \supset
[a,c]$ as required.

Therefore $f^{2\ell}([a,c])\supset [a,c]$. Since $f^{2\ell}(a)\in
f^{2\ell}([a,c])$, it follows that $f^{2\ell}([a,c])\supset [c, f^{2\ell}(a)]$
and so $f^{2\ell+1}([a,c]) \supset [f^{2l+1}(a),b]
\supset [c,b]$, the latter inclusion coming from the fact that
$\kappa(f)_{2\ell+2}=0$. Therefore $f^{2\ell+2}([a,c])=[a,b]$ as required.
\end{proof}

If $\{f_t\}$ is a monotonic family of unimodal maps then, for each rational
$m/n\in(0,1/2)$, the height~$m/n$ parameter interval starts when the kneading
sequence is $\Inf{w_q1}$ (with the saddle-node creation of a periodic orbit
whose rightmost point has this itinerary). In a full family, this is followed
by an interval of parameters in which~$f_t$ is renormalizable (starting with a
period-doubling cascade). This interval ends when the kneading sequence exceeds
$w_q0\Inf{w_q1}$. In the tent family, by contrast, kneading sequences~$\kappa$
with $\Inf{w_q1} \prec \kappa \preceq  w_q0\Inf{w_q1}$ do not occur.

\begin{lem}
\label{lem:images-of-tight} Let $f\colon[a,b]\to[a,b]$ be unimodal with
$q(\kappa(f))=q=m/n\in(0,1/2)$, and write $\theta=\min(f^n(a),
\widehat{f^n(a)})$.
\begin{enumerate}[(a)]
\item If $\Inf{w_q0}\prec \kappa(f)\preceq w_q0\Inf{w_q1}$ then $\theta=f^n(a)$
  and $f^n([a,f^n(a)))=[a, f^n(a)]$. \\Moreover, $f^{n-1}([f(a), f^{n+1}(a)]) =
  [a,f^n(a)]$.
\item If $ w_q0\Inf{w_q1}\prec \kappa(f) \prec \rhe(q)$ then there is
  some~$N\in\N$ with $f^N([a,\theta)) = [a,b]$.
\end{enumerate}
\end{lem}
\begin{proof} Recall that $w_q = 10^{\kappa_1(q)}110^{\kappa_2(q)}11\ldots 11
0^{\kappa_{m-1}(q)} 11 0^{\kappa_m(q)-1}$ is a word of length~$n-1$ containing
an odd number of $1$s. Moreover $\kappa_1(q)=\floor{1/q}-1>0$ so that
$w_q=10\ldots$.

\begin{enumerate}[(a)]
\item Since $\Inf{w_q0}$ and $\Inf{w_q0w_q1}$ are consecutive kneading
  sequences (the latter is obtained from the former by period doubling), we
  have $\Inf{w_q0w_q1}\preceq\kappa(f)\preceq w_q0\Inf{w_q1}$, and hence
  \mbox{$\kappa(f)=w_q0w_q1\ldots$}. In particular, $\iota(f^n(a))=0\ldots$, so
  that $f^n(a)\le c$ and $\theta = f^n(a)$.

We will first show that $f^{n-1}(a) \le f^{2n-1}(a)$. If
$\kappa(f)=w_q0\Inf{w_q1}$ then both $f^{n-1}(a)$ and $f^{2n-1}(a)$ have
itinerary $\Inf{w_q1}$, and since $\kappa(f)$ is not periodic it follows from
Convention~\ref{conv:unimodal}~(b) that $f^{n-1}(a)=f^{2n-1}(a)$. On the other
hand, if $\kappa(f)\not= w_q0\Inf{w_q1}$, then let~$\ell\ge 1$ and
$\symb\in\{0,1\}^\N$ be such that $\kappa(f)=w_q0(w_q1)^\ell \symb$,
where~$\symb$ does not start with the symbols $w_q1$. Since $\kappa(f)\prec
w_q0\Inf{w_q1}$ we have $\symb\succ\Inf{w_q1}$. Then $\iota(f^{n-1}(a)) =
(w_q1)^\ell \symb \prec (w_q1)^{\ell-1} \symb =
\iota(f^{2n-1}(a))$, so that $f^{n-1}(a)\le f^{2n-1}(a)$ as required.

Now $\iota(a)=\sigma(w_q0w_q1\ldots)$ and $\iota(f^n(a))=\sigma(w_q1\ldots)$.
Therefore $f^i(a)$ and $f^i(f^n(a))$ lie in the same monotone piece of~$f$ for
$0\le i\le n-3$, and they lie in the decreasing piece of~$f$ for an even number
of values of~$i$. Therefore $f^{n-2}([a, f^n(a))) = [f^{n-2}(a), f^{2n-2}(a))$.
Since $f^{n-2}(a)\le c\le f^{2n-2}(a)$ (as $\kappa(f)_{n-1}=0$ and
$\kappa(f)_{2n-1}=1$), and $c\le f^{n-1}(a) \le f^{2n-1}(a)$ (as shown in the
previous paragraph), it follows that $f^{n-1}([a, f^n(a))) = [f^{n-1}(a), b]$
and hence $f^n([a,f^n(a))) = [a,f^n(a)]$ as required. Since $f([a,f^n(a)) =
[f(a), f^{n+1}(a))$, the final statement is immediate.

\item Recall that $\rhe(q)=10\Inf{\hw_q1}$. Using $10\hw_q=w_q01$ (which is
  true since $c_q=w_q01$ is palindromic), we have
\[ w_q0\Inf{w_q1} \prec \kappa(f) \preceq w_q01\Inf{1\hw_q}.
\] In particular, $\kappa(f)=w_q01\ldots$. We consider separately the case
where $\kappa(f)=w_q010\ldots$ (so that $\theta=f^n(a)$), and the case where
$\kappa(f)=w_q011\ldots$ (so that $\theta =
\widehat{f^n(a)}$).

\medskip\medskip

\noindent\textbf{Case 1:} $\kappa(f) = w_q010\ldots$ and $\theta=f^n(a)$.

\medskip

Let $\ell\ge 0$ and $\symb\in\{0,1\}^\N$ be such that
\[\kappa(f)=w_q0(w_q1)^\ell \symb,\] where $\symb$ does not start with the
symbols $w_q1$. Since $\kappa(f)\succ w_q0\Inf{w_q1}$ we have
$\symb\prec\Inf{w_q1}$.

\medskip

\noindent{\em Step 1:} We will show that
\begin{equation}
\label{eq:step1} f^{\ell n}([a, \theta)) =
\begin{cases} [a,b] & \text{ if }\symb = 0\ldots, \\ [a, \sigma(\symb)) &
\text{ if }\symb = 1\ldots.
\end{cases}
\end{equation} (Here, $[a,\sigma(\symb))$ means an interval whose left hand
endpoint is~$a$, and whose right hand endpoint has itinerary $\sigma(\symb)$:
the fact that there may be more than one point with this itinerary will not be
important. We will use this notation throughout the remainder of the proof.)

If~$\ell=0$ then this is immediate: $\theta=f^n(a)$ has itinerary
$\sigma(\symb)$, and $\symb=1\ldots$ since $\kappa(f)=w_q010\ldots$. We
therefore suppose that $\ell\ge 1$, so that
\[
\iota(a) = \sigma(w_q)0(w_q1)^\ell \symb \qquad\text{ and }\qquad \iota(\theta)
= \sigma(w_q) 1 (w_q1)^{\ell-1} \symb.
\] We show by finite induction on~$i$ that
\[ f^{(i+1)n - 2}([a,\theta)) = [0(w_q1)^\ell \symb,
\,\,1(w_q1)^{\ell-i-1}\symb)
    \quad\text{ for }0\le i\le \ell-1.
\] The case~$i=0$ is straightforward, since the first~$n-2$ symbols of
$\iota(a)$ and $\iota(\theta)$ agree, and contain an even number of~$1$s.
Suppose then that $0<i\le\ell-1$, and assume, by the inductive hypothesis, that
$f^{in-2}([a,\theta)) = [0(w_q1)^\ell \symb, \,\, 1(w_q 1)^{\ell-i}\symb)$.
Since $\symb\prec\Inf{w_q1}$ we have $(w_q1)^{\ell-i}\symb \prec (w_q1)^\ell
\symb$, and hence $f^{in-1}([a,\theta)) = ((w_q1)^{\ell-i}\symb, b]$, and
    \[ f^{in}([a,\theta)) = [a,\,\, \sigma(w_q) 1 (w_q1)^{\ell-i-1}\symb) =
[\sigma(w_q)0(w_q1)^\ell \symb,\,\, \sigma(w_q) 1 (w_q1)^{\ell-i-1}\symb).
\] Applying $f^{n-2}$ gives the required result.

Setting $i = \ell-1$ gives $f^{\ell n-2}([a,\theta)) = [0(w_q1)^\ell
    \symb,\,\, 1\symb)$ and hence $f^{\ell n-1}([a,\theta)) = (\symb,b]$ (since
    $\symb\prec (w_q1)^\ell \symb$). Applying~$f$ once more
    gives~(\ref{eq:step1}).

\medskip

\noindent{\em Step 2: }We therefore assume that~$\symb=1\ldots$, and show that
$f^N([a,\sigma(\symb)))=[a,b]$ for some~$N$, which will complete the proof in
case~1.

Recall that
\[ w_q1 = 10^{\kappa_1}110^{\kappa_2}11 \ldots 11 0^{\kappa_{m-1}}11
0^{\kappa_m-1}1,
\] where $\kappa_i=\kappa_i(q)$ is given by~(\ref{eq:kappa-i}). Since~$\symb$
does not begin with the symbols $w_q1$, and satisfies $\symb\prec\Inf{w_q1}$,
one of the following possibilities must occur:
\begin{enumerate}[(i)]
\item There is some $i$ with $1\le i\le m$, and an integer~$k$ with $0\le
  k<\kappa_i$ (or $0\le k<\kappa_m-1$ in the case $i=m$), such that
\[
\symb = 10^{\kappa_1}110^{\kappa_2}11 \ldots 11 0^{\kappa_{i-1}} 11 0^k
1\ldots.
\]
\item There is some $i$ with $1\le i < m$ such that
\[
\symb = 10^{\kappa_1}110^{\kappa_2}11 \ldots 11 0^{\kappa_i} 1 0 \ldots.
\]
\end{enumerate} For~(i), write~$M=k + \sum_{j=1}^{i-1}(\kappa_j+2)$. If $i<m$
we have
\[ f^M([a,\sigma(\symb))) = [0^{\kappa_i-k}11\ldots,\,\, 1\ldots).
\] Therefore $f^{M+1}([a,\sigma(\symb))) \supset [0^{\kappa_i-k-1}11\ldots,\,\,
b]$. If $\kappa_i>k+1$ then we have $f^{M+2}([a,\sigma(\symb)))=[a,b]$, while
if $\kappa_i=k+1$ then $f^{M+2}([a,\sigma(\symb))) \supset [a, 1\ldots]\supset
[a,c]$, and the result follows by Lemma~\ref{lem:eventually-onto}. The
case~$i=m$ works similarly, remembering that $\iota(a)$ starts with the word
$\sigma(w_q)0$: we always have $f^{M+2}([a.\sigma(\symb)))=[a,b]$ in this case.

\medskip

\noindent For (ii), write~$M=\left(\sum_{j=1}^i (\kappa_j+2)\right) - 1$. Then
$f^M([a,\sigma(\symb))) = (0\ldots, 1^{2r+1}0\ldots]$ for some $r\ge 0$  (which
will be greater than~$0$ if and only if $\kappa_{i+1}=0$). Therefore
$f^{M+1}([a,\sigma(\symb)))\supset [1^{2r}0\ldots,\,\,b]$. If~$r=0$ then
$f^{M+2}([a,\sigma(\symb))) = [a,b]$; while if $r>0$ then
$f^{M+2}([a,\sigma(\symb)))=[a, 1\ldots]\supset [a,c]$, and the result follows
by Lemma~\ref{lem:eventually-onto}.

\medskip\medskip

\noindent\textbf{Case 2:} $\kappa(f) = w_q011\ldots$ and
$\theta=\widehat{f^n(a)}$.

\medskip

We have $\kappa(f)=w_q011\ldots\prec\rhe(q) = w_q01\Inf{1\hw_q}$: equivalently
$\kappa(f) = 10\hw_q1\ldots \prec 10\Inf{\hw_q1}$, since $w_q01=10\hw_q$.  Let
$\ell\ge1$ and $\symb\in\{0,1\}^\N$ be such that
\[
\kappa(f) = 10(\hw_q1)^\ell \symb,
\] where~$\symb$ does not start with the symbols $\hw_q1$. Since
$\kappa(f)\prec 10\Inf{\hw_q1}$ we have $\symb\succ\Inf{\hw_q1}$.

Now $\iota(a)=0(\hw_q1)^\ell \symb$, and
$\iota(f^n(a))=1(\hw_q1)^{\ell-1}\symb$, so that
$\iota(\theta)=0(\hw_q1)^{\ell-1}\symb$. Therefore
\[ f^{(\ell-1)n+1}([a,\theta)) = [\hw_q1\symb,\,\,\symb).
\] We have
\[
\hw_q1 = 0^{\kappa_m-1}110^{\kappa_{m-1}}11 \ldots
110^{\kappa_2}110^{\kappa_1}11.
\] Since~$\symb$ does not begin with the symbols $\hw_q1$ and
$\symb\succ\Inf{\hw_q1}$, one of the following possibilities must occur:
\begin{enumerate}[(i)]
\item There is some~$i$ with $1\le i\le m$ and an integer~$k$ with $0\le
  k<\kappa_i$ (or $0\le k < \kappa_m-1$ in the case~$i=m$), such that
\[
\symb = 0^{\kappa_m-1}110^{\kappa_{m-1}}11 \ldots 110^{\kappa_{i+1}}110^k
1\ldots.
\]
\item There is some~$i$ with $1\le i\le m$ such that
\[
\symb = 0^{\kappa_m-1}110^{\kappa_{m-1}}11 \ldots 11 0^{\kappa_i} 10\ldots.
\]
\end{enumerate} For~(i), we suppose that $i<m$ to avoid complicating the
notation: the case $i=m$ is similar. Write~$M=\bigl((\ell-1)n+1\bigr) +
\bigl(k-1+\sum_{j=i+1}^m (\kappa_j+2)\bigr)$. Then
\[ f^M([a,\theta)) = [0^{\kappa_i-k}11\ldots,\,1\ldots),
\] so that $f^{M+1}([a,\theta))\supset [0^{\kappa_i-k-1}11\ldots,\,b]$. If
$\kappa_i>k+1$ then we have $f^{M+2}([a,\theta)) = [a,b]$, while if
$\kappa_i=k+1$ then $f^{M+2}([a,\theta)) \supset [a, 1\ldots]\supset [a,c]$,
and the result follows by Lemma~\ref{lem:eventually-onto}.

For~(ii), write~$M=\bigl((\ell-1)n+1\bigr) +
\bigl(\sum_{j=i}^m(\kappa_j+2))-2\bigr)$. Then $f^M([a,\theta)) = (0\ldots,\,
1^{2r+1}0\ldots]$ for some $r\ge 0$. Therefore $f^{M+1}([a,\theta))\supset
[1^{2r}0\ldots,\,b]$. If~$r=0$ then $f^{M+2}([a,\theta))=[a,b]$; while if $r>0$
then $f^{M+2}([a,\theta))=[a,\,1\ldots]\supset[a,c]$, and the result follows by
Lemma~\ref{lem:eventually-onto}.

\end{enumerate}
\end{proof}

\bibliographystyle{amsplain}
\bibliography{neumrefs}

\end{document}